\documentclass[10pt,a4paper,reqno]{article}

\usepackage[latin,english]{babel}
\usepackage[utf8]{inputenc}
\usepackage[T1]{fontenc}
\usepackage{eufrak, color, mathrsfs,dsfont}
\usepackage{geometry}
\usepackage{amsmath}
\usepackage{amssymb}
\usepackage{amsthm}
\usepackage{mathtools}
\usepackage{mathabx}
\usepackage{cases}
\usepackage{braket}
\usepackage[toc,page]{appendix}
\usepackage{pdfsync}
\usepackage{hyperref}
\usepackage{psfrag}
\usepackage{times}
\usepackage[mathscr]{euscript}
\usepackage{math}
\usepackage{upgreek}
\usepackage[normalem]{ulem}
\usepackage{mathtools}
\mathtoolsset{showonlyrefs}

\numberwithin{equation}{section}

\theoremstyle{plain}
\newtheorem{thm}{Theorem}[section]
\newtheorem{lem}[thm]{Lemma}
\newtheorem{prop}[thm]{Proposition}
\newtheorem{cor}[thm]{Corollary}
\theoremstyle{definition}
\newtheorem{defn}[thm]{Definition}

\theoremstyle{remark}
\newtheorem{rem}[thm]{Remark}

\newcommand{\LieTr}[2]{e^{-#1} #2 e^{#1}}

\DeclareMathOperator{\diag}{diag}

\DeclareMathOperator{\dist}{dist}

\newcommand{\Op}{{\rm Op}}
\newcommand{\Ops}{{\rm OP}S}

\renewcommand{\bar}{\overline}

\newcommand{\sgn}{{\rm sgn}}

\newcommand{\even}{{\rm even}}
\newcommand{\odd}{{\rm odd}}
\newcommand{\ora}[1]{\vec{#1}}
\newcommand{\pa}{\partial}
\newcommand{\vs}{\varsigma}
\newcommand{\vphi}{\varphi}

\newcommand{\sign}{{\rm sign}\,}
\newcommand{\normk}[2]{\| #1 \|_{#2}^{k_0,\upsilon}}
\newcommand{\absk}[2]{| #1 |_{#2}^{k_0,\upsilon}}

\renewcommand{\whi}{\widehat \imath}
\renewcommand{\wti}{\widetilde \imath}
\newcommand{\acca}{\fH}

\definecolor{lblu}{rgb}{0.5,0.5,1}
\definecolor{rosso}{rgb}{1,0,0}
\definecolor{verde}{rgb}{0,0.7,0.1}
\definecolor{blu}{rgb}{0,0.1,0.7}
\definecolor{viola}{rgb}{0.9,0.2,0.5}
\definecolor{bluverde}{rgb}{0.,0.5,1.}
\definecolor{acq}{rgb}{0,0.5,0.7}
\definecolor{giallo}{rgb}{0.9,0.9,0}
\definecolor{rossos}{rgb}{0.9,0,0.05}
\definecolor{verdes}{rgb}{0,0.5,0.2}
\definecolor{blus}{rgb}{0,0,0.4}
\definecolor{orange}{rgb}{1.0, 0.49, 0.0}

\providecommand{\vect}[2]{{\bigl(\begin{smallmatrix}#1\\#2\end{smallmatrix}\bigr)}}

\title{\bf Pure gravity traveling quasi-periodic   \\
water waves with constant vorticity}

\begin{document}

\author{M. Berti\footnote{
		% International School for Advanced Studies (
		SISSA, Via Bonomea 265, 34136, Trieste, Italy. 
		\textit{Email:} \texttt{berti@sissa.it}; 
	}, 
 L. Franzoi\footnote{
NYUAD Research Institute, New York University Abu Dhabi, PO Box 129188, Abu Dhabi, United Arab Emirates.
 	\textit{Email:} \texttt{lf2304@nyu.edu}; 
 },
 A. Maspero\footnote{
%International School for Advanced Studies (
SISSA, Via Bonomea 265, 34136, Trieste, Italy. 
 \textit{Email:} \texttt{alberto.maspero@sissa.it} 
 }}

\maketitle

\noindent
{\bf Abstract.}
We prove the existence of small amplitude 
time quasi-periodic  solutions of the {\it pure gravity} water waves equations 
with  {\it constant vorticity},  for a bidimensional fluid over a flat bottom 
delimited by a space periodic free interface. 
Using a Nash-Moser implicit function iterative scheme  
we construct   traveling nonlinear waves which pass through each other slightly deforming and retaining forever a quasiperiodic structure. 
These 
solutions exist for any fixed value of depth and gravity
and restricting 
the vorticity  parameter to a Borel set of asymptotically full Lebesgue measure.

\smallskip


\smallskip


\tableofcontents

\section{Introduction}

A problem of fundamental importance in fluid mechanics
regards 
the search for traveling 
 surface waves.
Since the pioneering work of Stokes \cite{stokes} in 1847, 
a huge literature has established the existence of  steady traveling  waves,
namely solutions (either periodic or localized in space)  which look stationary in a moving frame. 
 The majority of the results concern 
  bidimensional fluids.  At the end of the section we shortly report 
 on the vast literature on 
 this problem. 

In the recent 
work \cite{BFM} we proved the first bifurcation result of 
 {\it time quasi-periodic 
traveling} solutions 
of the  water waves equations  
under the effects of gravity, constant vorticity,  and exploiting the capillarity effects 
at the free surface.  
These solutions can not be reduced to steady solutions in any moving frame. 
For pure gravity  irrotational water waves 
with infinite depth, 
quasi-periodic traveling waves has been  obtained in Feola-Giuliani \cite{FG}. 

\smallskip
The goal of this paper is to prove 
the existence of time quasi-periodic traveling water waves, 
 also in the physically important  case of  the {\it pure gravity}
equations with non zero {\it constant vorticity}, for any  value of the 
{\it depth} of the water, finite or infinite. 
In this work we are able to use the vorticity as  a parameter:  
 the 
 solutions that we construct  exist for any value of  gravity and 
 depth of the fluid,  provided the vorticity is restricted to a Borel set of asymptotically full measure, see Theorem \ref{thm:main0}. 
We also remark that, in case of 
non zero vorticity, one can not expect the bifurcation of standing waves since they are not allowed by the linear theory. 

\smallskip

It is well known that 
 this is a subtle small divisor problem. Major difficulties 
are that: 
($i$) the vorticity parameter enters 
in  the dispersion relation only at the zero order; 
($ii$) there are  resonances among the linear frequencies which can be avoided 
only for traveling waves; 
$(iii)$ the dispersion relation of the pure gravity equations 
is sublinear at infinity;
($iv$)  the nonlinear transport term is a singular perturbation of the unperturbed linear water waves vector field. 
Related difficulties appear in the search of pure gravity time periodic {\it standing} waves
which have been constructed  in the last  years for irrotational fluids
by Iooss, Plotnikov, Toland \cite{PlTo, IPT, IP-SW1}, 
extended to time quasi-periodic standing waves solutions
in Baldi-Berti-Haus-Montalto \cite{BBHM}.  In presence of surface tension,  time periodic standing waves solutions
 were constructed  by 
Alazard-Baldi \cite{AB}, extended to  time quasi-periodic solutions 
 by Berti-Montalto  \cite{BM}.  
We  mention that also the construction of
gravity steady 
traveling waves periodic  in space
presents  small divisor difficulties  for three dimensional  fluids.
These solutions, in a moving frame,
look steady bi-periodic waves and    
have been constructed for irrotational fluids
by Iooss-Plotnikov \cite{IP-Mem-2009,IP2} using the 
speed as a bidimensional  parameter {(for capillary waves in \cite{CN} is not a small divisor problem)}.

We now recall the pure gravity water waves equations with constant vorticity. 

\paragraph{The water waves equations.}

We consider the Euler equations of hydrodynamics for a 2-dimensional 
incompressible and inviscid fluid with constant vorticity $\gamma$, under the action of 
{\it pure gravity}.
The fluid 
occupies the region
\begin{equation}
\label{domain}
\cD_{\eta, \tth} := \big\{ (x,y)\in \T\times \R \ : \ -\tth <  y<\eta(t,x) \big\} \, , 
\quad \T := \T_x :=\R/ (2\pi\Z) \, ,
\end{equation} 
with a (possibly infinite) depth $\tth > 0 $  and 
 space periodic boundary conditions.
The unknowns of the problem are the free surface  $ y = \eta (t, x)$
of the time dependent domain $\cD_{\eta,\tth} $ and the 
divergence free  velocity field $\vect{u(t,x,y)}{v(t,x,y)}$.
If the  fluid has constant vorticity 
$$ 
v_x - u_y = \gamma \, , 
 $$
the velocity field   is the sum of the Couette flow $\vect{-\gamma y}{0}$ (recently studied in \cite{BeMa}, \cite{WZZ} and references therein),  which carries
all the  vorticity $ \gamma $ of the fluid,  and an irrotational field, 
expressed as the gradient of a harmonic function $\Phi $, called the generalized velocity potential. 
Denoting $ \psi (t,x) := \Phi (t,x, \eta(t,x)) $ 
 the evaluation of the generalized velocity potential at the free interface, 
one recovers $ \Phi $ by solving the elliptic problem
\begin{equation}
\label{dir}
\Delta \Phi = 0  \ \mbox{ in } \cD_{\eta, \tth} \, , \quad
\Phi = \psi \  \mbox{ at } y = \eta(t,x) \, , \quad
\Phi_y \to  0  \  \mbox{ as } y \to  - \tth \, .
\end{equation}
The third condition in \eqref{dir}  means the impermeability  property of the bottom
$ \Phi_y ( t, x, - \tth) = 0 $ if $ \tth < \infty $, and 
$ \lim\limits_{y \to - \infty } \Phi_y ( t, x, y) = 0 $, if $ \tth = + \infty $.
Imposing  
that the fluid particles at the free surface remain on it along the evolution
(kinematic boundary condition), and that
the pressure of the fluid 
 is equal to the constant 
atmospheric pressure  at the free surface (dynamic boundary condition),  the 
time evolution of the fluid is determined by the following 
system of equations 
\begin{equation}
\label{ww}
\begin{cases}
\eta_t = G(\eta)\psi + \gamma \eta \eta_x \\
\displaystyle{\psi_t = - g\eta  - \frac{\psi_x^2}{2} + 
	\frac{(  \eta_x \psi_x + G(\eta)\psi)^2}{2(1+\eta_x^2)}  + \gamma \eta \psi_x  + \gamma \partial_x^{-1} G(\eta) \psi} \, .
\end{cases}
\end{equation}
Here $ g $ is the gravity and 
$G(\eta)$ is the  Dirichlet-Neumann operator  
\begin{equation}
\label{DN}
G(\eta)\psi := G(\eta,\tth)\psi := \sqrt{1+\eta_x^2} \, (\partial_{\vec n} \Phi )\vert_{y = \eta(x)} = (- \Phi_x \eta_x + \Phi_y)\vert_{y = \eta(x)} \, . 
\end{equation}
As observed in the irrotational case  by Zakharov \cite{Zak1}, and  in presence of constant vorticity by Wahlén \cite{Wh}, 
the water waves equations \eqref{ww} are the Hamiltonian system 
\begin{equation}\label{ham_ww}
	\eta_t = \nabla_{\psi} H(\eta,\psi)\,, \quad \psi_t = (-\nabla_{\eta} + \gamma \pa_x^{-1}\nabla_{\psi})H(\eta,\psi)\,,
\end{equation}
where $\nabla$ denotes the $L^2$-gradient, with Hamiltonian
\begin{equation}\label{ham1}
	H(\eta,\psi) = \frac12 \int_{\T} \Big( \psi \,G(\eta) \psi + g \eta^2 \Big) \wrt x 
	 + \frac{\gamma}{2} 
	\int_{\T} 
	\Big( -   \psi_x \eta^2 + \frac{\gamma}{3} \eta^3 \Big) \wrt x\,.
\end{equation}
For any  value of the vorticity $\gamma\neq 0$, the system \eqref{ham_ww} is endowed with a non canonical Poisson structure, discussed in detail in Section \ref{ham.s}.
The equations \eqref{ww} enjoy two important symmetries. 
First of all, they are time reversible. 
We say that a solution of \eqref{ww} is \emph{reversible} if  
\begin{equation}\label{time-rev}
	\eta(-t,-x) = \eta (t,x) \, , \quad  \psi (-t,-x) = - \psi (t,x) \, . 
\end{equation}
Second, since the bottom of the fluid domain is flat, they 
are \emph{invariant by space translations}.

The variables $ (\eta, \psi) $ of system \eqref{ww} belong to some Sobolev space
$ H^s_0(\T) \times \dot H^s (\T) $ for some $ s $ large.  
Here  $H^s_0(\T)$, $s \in \R$, 
denotes the Sobolev space of functions with zero average
$ H^s_0(\T) := \big\{
u \in H^s(\T) \ \colon 
\  
\int_\T u(x) \di x = 0 \big\} $
and  $\dot H^s(\T)$, $s \in \R$, the corresponding 
homogeneous Sobolev space, namely the quotient space obtained by identifying  
the  functions in $H^s(\T)$ which differ  by a constant. 
This choice of the phase space is allowed because $\int_\T \eta(t,x)\wrt x$ is a prime integral of \eqref{ww} and the right hand side of \eqref{ww} depends only on $ \eta $ and $  \psi - \frac{1}{2 \pi}\int_\T \psi \, \di x  $.

\paragraph{Linear water waves.}
 
Linearizing \eqref{ww} at the equilibrium $(\eta, \psi) = (0,0)$ gives the system 
\begin{equation}
\label{lin.ww1}
\begin{cases}
\partial_t \eta & =G(0) \psi \\
\partial_t \psi & = -g \eta + \gamma\partial_x^{-1} G(0) \psi \, ,
\end{cases} 
\end{equation}
where $G(0) $ is the  Dirichlet-Neumann operator at the flat surface $\eta = 0$.
A direct computation reveals that  $G(0) $   is  the  Fourier 
multiplier operator
\begin{equation}\label{G(0)}
G(0) := G(0,\tth) = 
\begin{cases}
D \, \tanh( \tth D) & {\rm if } \ \tth < \infty \\
|D| & {\rm if } \ \tth = + \infty \, , 
\end{cases} \qquad {\rm where} \qquad D := \frac{1}{\im} \partial_x  \, , 
\end{equation}
with symbol, for any $ j \in \Z $, 
\begin{equation}\label{def:Gj0}
G_j(0):= 	G_j(0,\tth)= \begin{cases}
j\tanh(\tth j) & \text{ if }\tth<\infty \\
\abs j & \text{ if } \tth=+\infty \, . 
\end{cases} 
\end{equation}
As we will show in Section \ref{lin:op}, 
all reversible solutions, i.e. satisfying \eqref{time-rev}, of \eqref{lin.ww1}
are the   linear superposition  of plane waves, traveling 
either to the right 
or to the left,  given by 
{\begin{equation}
	\label{linz200}
	\begin{aligned}
	\begin{pmatrix}
	\eta(t,x) \\ \psi(t,x)
	\end{pmatrix} & =  \sum_{n \in \N}  \begin{pmatrix}
	M_n \rho_n \cos ( n x - \Omega_n (\gamma) t) \\ 
	P_n \rho_n \sin ( n x - \Omega_n (\gamma) t) 
	\end{pmatrix}  + 
	\begin{pmatrix}
	M_n \rho_{-n} \cos ( n x + \Omega_{-n}(\gamma) t) \\ 
	P_{-n} \rho_{-n} \sin ( n x + \Omega_{-n} (\gamma) t) 
	\end{pmatrix} 
	\,,
	\end{aligned}
	\end{equation}}
where $\rho_n\geq 0$ are arbitrary amplitudes and   $M_n$, $P_{\pm n}$ 
are the real coefficients
\begin{equation}\label{def:Mn}
M_j :=\left( \frac{G_j(0)}{g +
	\frac{\gamma^2}{4} \frac{G_j(0)}{j^2}} \right)^{\frac14}, \ 
j \in \Z \setminus \{0\} \,  , 
\quad 
P_{\pm n} := \frac{\gamma}{2} \frac{M_n}{n} \pm M_n^{-1}, \ 
n \in \N \, .
\end{equation}
The frequencies $\Omega_{\pm n}(\gamma)$ in \eqref{linz200} are 
\begin{equation}
\label{def:Omegajk}
\Omega_j(\gamma) := 
\sqrt{ \Big( g  + \frac{\gamma^2}{4}\frac{G_j(0)}{j^2}   \Big) 
	G_j(0) } + \frac{\gamma}{2}\frac{G_j(0)}{j} \, , \quad
j \in \Z \setminus\{0 \} \, . 
\end{equation}
Note that the map $ j \mapsto \Omega_j (\gamma )$
is not even due to the vorticity term $ \gamma G_j (0) / j $, which is odd in $j $.

All the linear solutions \eqref{linz200} 
are either time periodic, quasi-periodic or almost-periodic, 
depending on the irrationality properties of the frequencies 
$ \Omega_{\pm n} (\gamma) $ and the number of non zero 
amplitudes $\rho_{\pm n}  $. 
%We first provide  the notion of  quasi-periodic traveling wave.

The problem of the existence of the traveling quasi-periodic in time water waves
is formulated as follows.

\begin{defn} {\bf (Quasi-periodic traveling wave)} \label{def:TV}
	We say that $ (\eta (t,x), \psi(t,x)) $ is a 
	time quasi-periodic {\it traveling} wave  
	with irrational frequency vector  $ \omega = ( \omega_1, \ldots, \omega_\nu)  \in \R^\nu $, $ \nu \in \N $,  i.e.
	$ \omega \cdot 	\ell \neq 0 $ for any $ \ell \in \Z^\nu \setminus \{0 \} $, 
	and ``wave vectors'' $ ( j_1, \ldots, j_\nu)  \in \Z^\nu $,  if there exist 
	functions
	$ ( \breve\eta, \breve\psi ) : \T^\nu \to \R^2 $ such that 
	\begin{equation} \label{trav-etazeta}
		\begin{pmatrix}
			\eta ( t, x) \\ \psi ( t, x) 
		\end{pmatrix} = 
		\begin{pmatrix}
			\breve\eta( \omega_1  t- j_1 x ,\ldots, \omega_\nu t- j_\nu x )  \\
			\breve\psi( \omega_1  t- j_1 x ,\ldots, \omega_\nu t- j_\nu x )
		\end{pmatrix} \, . 
	\end{equation}
\end{defn}

Note that, if $ \nu = 1 $, such functions are  time periodic and indeed stationary in a moving frame
	with speed $ \omega_1 / j_1 $.
	If the number of the irrational frequencies in greater or equal than $ 2 $,  
	the waves \eqref{trav-etazeta} cannot be reduced to steady waves by any appropriate choice of the moving frame. 

\smallskip

We  shall construct traveling quasi-periodic  
solutions of the nonlinear equations  \eqref{ww} with a diophantine frequency vector $ \omega $ belonging to an open bounded 
subset $  {\mathtt \Omega} $ in $ \R^\nu $, 
namely, for some 
$ \upsilon \in (0,1) $, $ \tau >  \nu - 1 $, 
\begin{equation}\label{def:DCgt}
	\tD\tC (\upsilon, \tau) := \Big\{ \omega\in {\mathtt \Omega} \subset\R^\nu \ : \ 
	\abs{\omega\cdot \ell}\geq \upsilon \braket{\ell}^{-\tau} \ , 
	\ \forall\,\ell\in\Z^\nu \setminus \{0\} \Big\}\,, \quad \langle \ell \rangle :=  \max\{1, |\ell|\} 
	\, . 
\end{equation}
Regarding regularity, we will prove the existence of 
quasi-periodic traveling waves $ (\breve\eta,  \breve\psi ) $
belonging to some Sobolev space 
\begin{equation} \label{unified norm}
	H^s(\T^{\nu}, \R^2)
	= \Big\{ \breve f (\vf) = 
	\sum_{\ell \in \Z^{\nu}} f_{\ell} \, e^{\im \ell \cdot \vf } \ , \ \ \ f_\ell \in \R^2 \ \ : \, 
	\| \breve f \|_s^2 := \sum_{\ell \in \Z^{\nu}} | f_{\ell}|^2 \langle \ell \rangle^{2s} < \infty 
	\Big\} \, . 
\end{equation}
Fixed   finitely  many  arbitrary  
{\em  distinct} natural numbers 
\begin{equation}
	\label{Splus}
	\S^+ := \{ \bar n_1, \ldots, \bar n_\nu \}\subset \N  \ , \quad
	1 \leq \bar n_1 < \ldots < \bar n_\nu \, , 
\end{equation}
and signs
\begin{equation}
	\label{signs}
	\Sigma := \{ \sigma_1 , \ldots , \sigma_\nu \} , \quad \sigma_a \in \{ -1, 1 \} \, ,
	\quad a = 1, \ldots, \nu \, ,  
\end{equation}
we consider reversible quasi-periodic traveling wave 
solutions   of the linear system \eqref{lin.ww1}, given by 
\begin{equation}
	\label{sel.sol}
	\begin{aligned}
		\begin{pmatrix}
			\eta(t,x) \\ \psi(t,x)
		\end{pmatrix} & =  \sum_{a \in \{1, \ldots, \nu \colon   \sigma_a = + 1\}}  
		\begin{pmatrix}
			M_{\bar n_a} \sqrt{\xi_{\bar n_a}} \cos ( \bar n_a x - \Omega_{\bar n_a} (\gamma) t) \\ 
			P_{\bar n_a} \sqrt{\xi_{\bar n_a}} \sin ( \bar n_a x - \Omega_{\bar n_a} (\gamma) t) 
		\end{pmatrix} \\ 
		& +   \sum_{a \in \{1, \ldots, \nu \colon   \sigma_a = - 1\}}
		\begin{pmatrix}
			M_{\bar n_a} \sqrt{\xi_{- \bar n_a}} \cos ( \bar n_a x + \Omega_{- \bar n_a}(\gamma) t) \\ 
			P_{-\bar n_a} \sqrt{\xi_{- \bar n_a}} \sin ( \bar n_a x + \Omega_{- \bar n_a} (\gamma) t) \ 
		\end{pmatrix}  
	\end{aligned}
\end{equation}
where $ \xi_{\pm \bar n_a} >  0 $, $ a  = 1, \ldots, \nu $. The frequency vector 
of \eqref{sel.sol}  is given by 
\begin{equation}\label{Omega-gamma}
	\vec \Omega (\gamma) := (\Omega_{\sigma_a \bar n_a} (\gamma))_{a=1, \ldots, \nu}  
	\in \R^\nu\,.
\end{equation}

Theorem \ref{thm:main0} shows that the  linear 
solutions \eqref{sel.sol}
can be continued to  quasi-periodic traveling wave  solutions of the nonlinear 
water waves equations
\eqref{ww}, 
for most values of the  vorticity $ \gamma \in [\gamma_1, \gamma_2 ]\subset\R$,
with a  frequency vector 
$ \widetilde  \Omega  := ( \widetilde  \Omega_{\sigma_a \bar n_a})_{a=1, \ldots, \nu}  $,   
close to 
$  \vec \Omega (\gamma) := (\Omega_{\sigma_a \bar n_a} (\gamma))_{a =1, \ldots, \nu} $.

\begin{thm} \label{thm:main0}  {\bf (KAM for traveling gravity 
		water waves with constant vorticity)}
	Consider finitely many tangential sites $ \S^+ \subset \N  $
	as in \eqref{Splus} and signs $ \Sigma $ as in \eqref{signs}. Then 
	there exist $ \bar s >  0 $,  
	$ \varepsilon_0 \in (0,1) $ such that,  
	for any $ |\xi |   \leq \varepsilon_0^2  $, 
	$  \xi := (\xi_{ \sigma_a {\bar n}_a} )_{a = 1, \ldots, \nu} \in \R_+^\nu $, the following hold:
	\begin{enumerate}
		\item
		there exists 
		a Cantor-like set  $ {\cal G}_\xi \subset [\gamma_1, \gamma_2] $ 
		with asymptotically full measure as $ \xi \to 0 $, i.e. 
		$ \lim_{\xi \to 0} | {\cal G}_\xi |  = {\gamma}_2- {\gamma}_1 $;
		\item
		for any  $ \gamma \in {\cal G}_\xi $,  the 
		gravity water waves equations \eqref{ww}
		have a  reversible 
		quasi-periodic traveling wave solution (according to 
		Definition \ref{def:TV}) of the form 
		\begin{equation}
			\label{QP:soluz}
			\begin{aligned}
				\begin{pmatrix}
					\eta( t ,x) \\ \psi( t ,x)
				\end{pmatrix} & =  \sum_{a \in \{1, \ldots, \nu\} \colon   \sigma_a = + 1}  
				\begin{pmatrix}
					M_{\bar n_a} \sqrt{\xi_{\bar n_a}} \cos ( \bar n_a  x - \widetilde{\Omega}_{ \bar n_a}(\gamma)  t) \\ 
					P_{\bar n_a} \sqrt{\xi_{\bar n_a}} \sin ( \bar n_a x - \widetilde{\Omega}_{\bar n_a}(\gamma) t) 
				\end{pmatrix}   \\ 
				& +   \sum_{a \in \{1, \ldots, \nu\} \colon   \sigma_a = - 1}
				\begin{pmatrix}
					M_{\bar n_a} \sqrt{\xi_{- \bar n_a}} \cos ( \bar n_a x + \widetilde{\Omega}_{-\bar n_a}(\gamma)  t) \\ 
					P_{-\bar n_a} \sqrt{\xi_{- \bar n_a}} \sin ( \bar n_a x + \widetilde{\Omega}_{-\bar n_a}(\gamma)  t) 
				\end{pmatrix} + r ( t, x ) 
			\end{aligned}
		\end{equation}
		where
		$$
		r ( t, x )
		= 
		\breve r(\wt\Omega_{\sigma_1 \bar n_1}(\gamma)t-\sigma_1\bar n_1 x,\ldots, \wt\Omega_{\sigma_\nu \bar n_\nu}(\gamma)t-\sigma_\nu\bar n_\nu x)  \, , \quad \breve r \in H^{\bar s} ( \T^\nu , \R^2) \, ,
		\quad  \lim_{\xi \to 0} \frac{\| \breve r \|_{\bar s}}{\sqrt{|\xi|}} = 0 \, , 
		$$
		with a Diophantine 
		frequency vector $ \widetilde  \Omega  := ( \widetilde  \Omega_{\sigma_a \bar n_a})_{a=1, \ldots, \nu} \in \R^\nu $, depending on $\gamma, \xi$, and  satisfying 
		$ \lim_{\xi \to 0}{\widetilde \Omega} = \vec \Omega (\gamma) $. 
		In addition these quasi-periodic solutions are  linearly stable.  
	\end{enumerate}
\end{thm}

Let us make some comments about the result.
\\[1mm]
\indent 1) {\it Vorticity as parameter and irrotational quasi-periodic traveling waves.}
We are able to use the vorticity $\gamma$ as a parameter, even though 
the dependence of the linear frequencies $ \Omega_j (\gamma)$ in 
 \eqref{def:Omegajk}  with respect to $ \gamma$ affects only the order $ 0 $.
In Section \ref{sec:deg_kam} we prove the non-degeneracy and the transversality of the linear frequencies $ \Omega_j (\gamma)$ with respect to $\gamma$.
If $ \gamma_1 < 0 < \gamma_2 $ we do not know if 
the value $ \gamma=0$ belongs to the set $\cG_\xi$ for which the quasi periodic solutions \eqref{QP:soluz} exist.
Nevertheless,  irrotational quasi-periodic traveling solutions for the gravity water waves equations \eqref{ww} exist for most values of the depth $\tth\in[\tth_1,\tth_2]$,
see Remark \ref{acca.param}. 
These 
  traveling wave solutions
  do not clearly reduce to  the standing wave solutions constructed 
in \cite{BBHM}, which are even in 
the space variable.  
\\[1mm]
\indent 2) {\it More general traveling solutions}.
The Diophantine condition \eqref{dioph0} could be weakened requiring only 
$  | \omega \cdot \ell | \ge\ \upsilon \langle \ell \rangle^{-\tau} $
for any $\ell\in\Z^\nu\setminus\{0\}$ with 
$ \ell_1\,\sigma_1\bar n_1 +...+\ell_\nu\,\sigma_\nu\bar n_\nu = 0 $.  In such a case 
the vector $ \omega $ could admit one non-trivial resonance.
This is the natural minimal requirement 
to look for traveling  solutions of the form 
$ U( \omega t - \vec \jmath x )$, see Definition \ref{QPTW} and Remark \ref{rem:inclu}.
For $ \nu = 2 $ solutions of these kind could be time periodic,  with clearly  a completely  different shape with respect to the classical Stokes traveling waves \cite{stokes}. 
\\[1mm]
\indent Let us make some comments about the proof.
\\[1mm]
\indent 3) {\it Symmetrization and reduction in order of the linearized operator.}
The leading order of the linearization of the water waves system \eqref{ww} at any quasi-periodic traveling wave is given by the 
Hamiltonian transport operator (see  \eqref{cL00}) 
$$
\cL_{\rm TR} :=
\omega\cdot\pa_\vf +\begin{pmatrix}
	\pa_x \wtV & 0 \\ 0 & \wtV\pa_x
\end{pmatrix}
$$ 
where $\wtV(\vf,x)$ is a small quasi-periodic traveling wave. 
By the  almost-straightening result of Lemma \ref{conju.tr}, 
for any $ (\omega, \gamma) $ satisfying  suitable 
non-resonance conditions as in \eqref{0meln}, we 
conjugate  $ \cL_{\rm TR} $  via a symplectic transformation to a  transport 
operator of the form 
$$
\omega\cdot\pa_\vf +\begin{pmatrix}
	\tm_{1} \pa_x & 0 \\ 0 & \tm_{1} \pa_x
\end{pmatrix} + \begin{pmatrix}
				\pa_y \,p_{\bar\tn} & 0 \\ 0 &  p_{\bar\tn}\, \pa_y
			\end{pmatrix}\,,
$$ 
where $\tm_{1}\in\R$ is a constant to be determined and $ p_{\bar\tn} (\vphi,x) $ is 
an  exponentially small function,
see \eqref{stime.pn.not.w}. 
For the standing waves problem in \cite{BBHM} we have that  $ \tm_{1} = 0 $
and the complete conjugation of $ {\cal L}_{\rm TR} $ 
is proved for any $ \omega $ diophantine. % applying a Nash-Moser-H\"ormander theorem. 
The almost-straightening Theorem \ref{thm:as}, which implies the conjugation in Lemma \ref{conju.tr}, is performed in the same spirit of the almost-reducibility Theorem \ref{iterative_KAM}. 
The KAM algebraic reduction scheme 
is like in \cite{FGMP} and in \cite{BM20}.
Here we do not perform the full straightening of the transport operator $\cL_{\rm TR} $
(i.e. we have $ \bar \tn <  \infty $) 
in order to 
formulate a simple non-resonance condition as in \eqref{0meln}.
The resulting almost-remainders are then considered
along  the Nash-Moser nonlinear iteration 
 (the estimates obtained in the 
 proof in \cite{FGMP}  after finitely many iterative steps 
 are not sufficient for our purposes).

As the almost-straightening above,
we also perform  in a symplectic way 
other steps of the  reduction to constant coefficients of the lower 
order terms of the linearized operator. This is needed in order to 
prevent  the appearance of  unstable  operators. 
Since Section \ref{sec:block_dec} we preserve only the reversible structure.
%Finally we mention that % On a more technical level, 
Due to  the pseudo-differential nature of   the vorticity vector field in \eqref{ww},  
there appear along the reduction process, since Section \ref{subsec:straight},  
smoothing tame remainders (cfr.  \eqref{AS.est3}),
unlike \cite{BBHM}.
%The notion of tame operators is reported  in Section \ref{subsec:semiphi}.
\\[1mm]
\indent 4) {\it Traveling waves and  
Melnikov  non-resonance conditions.}
We strongly use 
the invariance under space translations of the Hamiltonian 
nonlinear water waves vector field
\eqref{ww}, i.e. the ``momentum conservation'', in the construction of the traveling quasi-periodic waves. 
%For Theorem \ref{thm:main0}, this symmetry is even more striking. 
%In \cite{BFM} we strongly used 
%the invariance under space translations of the Hamiltonian 
%nonlinear water waves vector field
%\eqref{ww}, i.e. the ``momentum conservation'', in the construction of the traveling quasi-periodic waves. For Theorem \ref{thm:main0}, this symmetry is even more striking. 
We list the main points in which it occurs:

({\it i}) The Floquet exponents \eqref{def:FE} of the quasi-periodic solutions \eqref{QP:soluz} are a singular perturbation of the unperturbed linear frequencies in \eqref{def:Omegajk}, with 
leading terms  of order $1$. The  Melnikov non-resonance conditions 
formulated in the Cantor-like set 
  $ {\cal C}_\infty^{\upsilon} $   
  in \eqref{dioph0}-\eqref{2meln+} 
 hold on a set of large measure only thanks to the conservation of the momentum,
 see Section \ref{subsec:measest}.

({\it ii})
Thanks to the restriction on the Fourier indexes coming from the space translation invariance, we can impose Melnikov conditions 
 that \emph{do not lose} space derivatives, see \eqref{2meln-}. 
 This simplifies considerably the 
 reduction in decreasing orders of Section \ref{sec:linnorm} and 
the KAM reducibility scheme of Section \ref{sec:KAM}. Indeed, it is enough to reduce to constant coefficients 
the linearized vector operator until the order $0$
(included, in order to have a sufficiently good asymptotic expansion of the perturbed frequencies to prove the inclusion 
Lemma \ref{lemma:inWave}).  
%We remark that, 
Conversely, in \cite{BBHM} the second order Melnikov conditions verified 
for the  standing pure gravity waves  
lose several space derivatives and many more steps of regularization are needed.  
 %For balancing this losses, they had to regularize to constant coefficient the linearized operator up to arbitrary negative orders. 

($iii$) 
The invariance by space translations in the construction of the quasiperiodic {\it traveling} waves
allows to avoid  resonances between the linear frequencies.
For example, with infinite depth  $ {\mathtt h} = + \infty $, these are given by 
 $ \Omega_j (\gamma) = \omega_j (\gamma) + \frac{\gamma}{2}  {\rm sign}(j) $.
In this case there exist   $\ell \in \Z^\nu \setminus\{0\}$ and  $j, j' \not\in \{\sigma_a \bar n_a \}_{a=1,\ldots,\nu}$, with $j \neq j'$, such that
\begin{equation}
\label{ex.resonance}
\sum_{a=1}^\nu \ell_a \,\Omega_{\sigma_a \bar n_a}(\gamma) + \Omega_j(\gamma) - \Omega_{j'}(\gamma)  \equiv 0 \quad \forall\, \gamma \, .
\end{equation}
For example if $\sigma_1 = \sigma_2$, it is sufficient to  take $  \ell = (\ell_1, \ell_2, 0, \ldots, 0) = (-1, 1, 0, \ldots, 0)$ and 
 $   j=- \sigma_1 \bar n_1 $, $   j' = - \sigma_2 \bar n_2 $.
To exclude this  resonance we exploit  the conservation of momentum, which  guarantees that resonances of the form \eqref{ex.resonance} have to be checked only on indexes fulfilling  $ \sum_{a=1}^\nu \ell_a\,\sigma_a \bar n_a +  j - j' = 0 $. 
The indexes above violate this constraint, as $\bar n_1 \neq \bar n_2$ by  \eqref{Splus}.
Along the proof, we %are able to 
systematically use this kind of arguments to exclude nontrivial resonances.

\smallskip

Before concluding this introduction, we  shortly describe the huge literature 
regarding  time periodic traveling wave solutions, 
which are steady in a moving frame. 
\\[1mm]
{\bf Literature about  time periodic traveling wave solutions.}
After the pioneering work of Stokes  \cite{stokes}, the first rigorous construction 
of small amplitude  space periodic steady traveling waves goes back to the 1920's with the papers of 
Nekrasov \cite{Nek}, Levi-Civita \cite{LC} and Struik \cite{Struik}, in case of  irrotational bidimensional flows  under the action of pure gravity.
In the presence of vorticity,  Gerstner \cite{gerstner} in 1802 gave an explicit example of periodic traveling wave, in infinite depth, and non-zero vorticity, but 
it is only with Dubreil-Jacotin \cite{dubreil} in 1934 the first bifurcation  result of 
periodic traveling waves with small vorticity, subsequently 
extended  by  Goyon \cite{goyon} and Zeidler \cite{Zei2} for large vorticity.
More recently we point out the works of Wahl\'en \cite{Wh0}  for capillary-gravity waves and non-constant vorticity, and of Martin  \cite{Martin} and Walh\'en \cite{Wh} for constant vorticity. 
All these results deal with 2d water waves, and  can ultimately 
be  deduced by the classical Crandall-Rabinowitz 
bifurcation theorem from a simple eigenvalue. 

We also mention that these local bifurcation results can be extended to  global branches 
of steady traveling waves by the theory % applying the methods 
of global analytic, or topological, 
bifurcation. We refer to 
Keady-Norbury \cite{KN}, Toland \cite{To}, McLeod \cite{ML}
for irrotational flows and  Constantin-Strauss \cite{CS}
for fluids with non-constant vorticity. We suggest  the reading of 
 \cite{const_book} for  further results.  
 
We finally quote the  recent numerical 
work of  Wilkening-Zhao \cite{W1} 
about  spatially quasi-periodic  gravity-capillary $ 1$d-water waves.

\section{Hamiltonian structure and linearization at the origin }\label{ham.s}

The Hamiltonian formulation of the water waves equations \eqref{ww} 
with non-zero constant 
vorticity was obtained by Constantin-Ivanov-Prodanov  \cite{CIP} and Wahl\'en 
\cite{Wh} in the case of finite depth. 
For irrotational flows it reduces to the classical Craig-Sulem-Zakharov formulation
in \cite{Zak1}, \cite{CrSu}. 

On the phase space $H^1_0(\T) \times \dot H^1(\T)$, endowed with the non canonical Poisson tensor
\begin{equation}
\label{eq:magn_sympl}
J_M(\gamma) := \begin{pmatrix}
0 & {\rm Id} \\
- {\rm Id}  & 
\gamma \partial_x^{-1} 
\end{pmatrix} \, ,
\end{equation}
we
consider the Hamiltonian $ H $ defined in \eqref{ham1}. 
Such Hamiltonian 
is well defined on $ H^1_0(\T) \times \dot H^1(\T) $ since $ G(\eta ) [1] = 0 $ 
and $ \int_{\T} G(\eta) \psi \, \di x = 0 $.
It turns out \cite{CIP, Wh} 
that equations \eqref{ww} are the Hamiltonian system generated by
$H(\eta, \psi)$ with respect to the Poisson tensor $J_M(\gamma)$, namely 
\begin{equation}
\label{ham.eq1}
\partial_t 
\begin{pmatrix}
\eta \\
\psi  
\end{pmatrix} = J_M (\gamma) 
\begin{pmatrix}
\nabla_\eta H  \\
\nabla_\psi H
\end{pmatrix}
\end{equation}
where  $ (\nabla_\eta H, \nabla_\psi H) \in \dot L^2(\T) \times L^2_0(\T) $  denote the $ L^2 $-gradients.  
	The non canonical Poisson tensor $J_M(\gamma)$ in \eqref{eq:magn_sympl} has to be regarded as an operator from (subspaces of) $(L_0^2\times \dot L^2)^* = \dot{L}^2\times L_0^2$ to $L_0^2\times \dot{L}^2$, that is
	\begin{equation*}
	J_M(\gamma) = \begin{pmatrix}
	0 & {\rm Id}_{L_0^2\to L_0^2} \\ -{\rm Id}_{\dot{L}^2\to \dot{L}^2} & \gamma\pa_x^{-1}
	\end{pmatrix}\,.
	\end{equation*}
%	The operator $ \pa_x^{-1} $ maps a dense subspace of $ L^2_0 $ in $ \dot L^2 $. 
	For sake of simplicity, throughout the paper we omit this detail,  see   \cite{BFM} for a more precise analysis.  
%	Above the dual space $(L_0^2\times \dot L^2)^* $ with respect to the 
%	scalar product in $L^2$ is identified with 
%	$ \dot L^2\times L_0^2 $. 

We describe now some symmetries of the Hamiltonian \eqref{ham1}. % which we  now describe.
\\[1mm]
{\bf Reversible structure.}
Defining on the phase space $H_0^1(\T) \times \dot{H}^1(\T)$ the involution
\begin{equation}\label{rev_invo}
\cS\left( \begin{matrix}
\eta \\ \psi 
\end{matrix} \right) := \left( \begin{matrix}
\eta^\vee \\ -  \psi^\vee 
\end{matrix} \right) \, , \quad \eta^\vee (x) :=  \eta (-  x) \, ,  
\end{equation}
the Hamiltonian \eqref{ham1} is invariant under $\cS$, that is 
$  H \circ \cS = H $. 
Equivalently, the water waves vector field 
$ X $   in the right hand side on \eqref{ww} satisfies 
\begin{equation}\label{revNL}
X\circ \cS = - \cS \circ X \, .
\end{equation}
This property follows since the Dirichlet-Neumann operator satisfies 
$ G(  \eta^\vee ) [ \psi^\vee ] = \left( G(\eta) [\psi ] \right)^\vee  $.
\\[1mm]
{\bf Translation invariance.}
Since the bottom of the fluid domain \eqref{domain} is flat (or in case of infinite depth there is no bottom), the water waves 
equations   \eqref{ww}  are  invariant under space translations. Specifically, 
defining the translation operator  
\begin{equation}\label{trans}
\tau_\vs \colon u(x) \mapsto u(x+\vs) \, ,  \quad \varsigma \in \R \, , 
\end{equation} 
the Hamiltonian \eqref{ham1} satisfies 
$ H \circ \tau_\vs = H $ for any $\vs \in \R $.
Equivalently, the water waves vector field 
$ X $  in the right hand side on \eqref{ww} satisfies 
\begin{equation}\label{eq:mom_pres}
X\circ \tau_\vs = \tau_\vs\circ X  
\, , \quad \forall\, \vs \in \R \,  .
\end{equation}
This property follows  since 
$ \tau_\vs \circ G( \eta ) = G(\tau_\vs  \eta ) \circ \tau_\vs $ for any 
$\vs \in \R $.
\\[1mm]
{\bf Wahl\'en coordinates.} 
We introduce the Wahl\'en   \cite{Wh}  coordinates   $(\eta, \zeta)$ via the map  
\begin{equation}\label{eq:gauge_wahlen}
\left( \begin{matrix}
\eta \\ \psi
\end{matrix} \right) =  W \left(\begin{matrix}
\eta \\ \zeta
\end{matrix}\right) \, , \quad
W :=  \left(\begin{matrix}
{\rm Id} & 0 \\  \frac{\gamma}{2}\partial_x^{-1} & {\rm Id}
\end{matrix}\right)  \, , \quad
W^{-1}  := 
\left(\begin{matrix}
{\rm Id} & 0 \\ - \frac{\gamma}{2}\partial_x^{-1} & {\rm Id}
\end{matrix}\right) \ . 
\end{equation}
The change of coordinates $W$ maps the  phase space $H^1_0 \times \dot H^1  $ into itself, and it  conjugates the 
 Poisson tensor
 $J_M(\gamma)$ to the canonical one 
 \begin{equation}\label{Jtensor}
W^{-1} J_M(\gamma) (W^{-1})^{*} =  J \, , \quad 
J := \begin{pmatrix}
0 & {\rm Id} \\ - {\rm Id}  & 0
\end{pmatrix}   \, ,
\end{equation}
 so that $(\eta, \zeta)$ are Darboux coordinates. 
The Hamiltonian 
\eqref{ham1} becomes 
\begin{equation}\label{Ham-Wal}
\cH := H \circ W\,, \quad \text{ i.e. } \quad  {\cal H}(\eta,\zeta):=H\Big(\eta,\zeta + \frac{\gamma}{2}\partial_x^{-1}\eta\Big) \, , 
\end{equation}
and the Hamiltonian equations \eqref{ham.eq1} (i.e. \eqref{ww})  are transformed into
\begin{equation}\label{eq:Ham_eq_zeta}
\partial_t\left(\begin{matrix}
\eta \\ \zeta 
\end{matrix}\right) = X_{\cal H} (\eta, \zeta) \, , \quad 
X_{\cal H} (\eta, \zeta) := J
\begin{pmatrix} \nabla_\eta {\cal H} \\ \nabla_\zeta {\cal H} \end{pmatrix}  ( \eta, \zeta )\,. 
\end{equation}
By \eqref{Jtensor}, the symplectic form of  \eqref{eq:Ham_eq_zeta} is  the standard one,
\begin{align} \label{sympl-form-st}
{\cal W}  
\left( \begin{pmatrix}
\eta_1 \\
\zeta_1
\end{pmatrix}, 
\begin{pmatrix}
\eta_2 \\
\zeta_2
\end{pmatrix}
\right) 
:= 
\left( J^{-1} \left(\begin{matrix}
\eta_1 \\ \zeta_1
\end{matrix}\right), 
\left(\begin{matrix}
\eta_2 \\ \zeta_2
\end{matrix}\right) \right)_{L^2}  
= (  -  \zeta_1 , \eta_2 )_{L^2} + (\eta_1  , \zeta_2 )_{L^2}  \, , 
\end{align}
where  $ J^{-1}  =  \begin{pmatrix}
0 & - {\rm Id} \\  {\rm Id}  & 0
\end{pmatrix} $ is 
regarded as a map 
from $L_0^2\times \dot{L}^2$ into $ \dot{L}^2\times L_0^2 $. 
% Note that $ J J^{-1} = {\rm Id}_{L_0^2\times \dot{L}^2}$ and $ J^{-1}J  = {\rm Id}_{\dot{L}^2\times L_0^2}$. 

%The  Hamiltonian vector field 
%$ X_{\cal H} (\eta, \zeta) $	 in  \eqref{eq:Ham_eq_zeta} is characterized by the identity
%$$
%d {\cal H} (\eta, \zeta) [ \widehat u ]= 
%{\cal W} \big( X_{\cal H} (\eta, \zeta),  \widehat u \big) \, , \quad  \forall\, \widehat u :=  \begin{pmatrix}
%\widehat \eta \\
%\widehat \zeta
%\end{pmatrix} \, . 
%$$
The transformation $W$ defined in \eqref{eq:gauge_wahlen} is reversibility preserving, 
namely it commutes with the involution
$ {\cal S} $ in \eqref{rev_invo} (see Definition \ref{rev_defn} below), and thus also the Hamiltonian $\cH$ in \eqref{Ham-Wal} is invariant under the involution $\cS$.
% as well as 
%$ H$ in \eqref{ham1}. 
For this reason  we look for solutions $(\eta(t,x),\zeta(t,x))$
of \eqref{eq:Ham_eq_zeta} that are reversible, i.e., see  \eqref{time-rev}, 
\begin{equation}\label{rev:soluz}
\left(\begin{matrix}
\eta \\ \zeta
\end{matrix}\right)(-t)= \cS\left(\begin{matrix}
\eta \\ \zeta
\end{matrix}\right)(t)\, . 
\end{equation}
The corresponding solutions $(\eta(t,x), \psi (t,x))$ of \eqref{ww} induced by 
\eqref{eq:gauge_wahlen} are reversible as well. 

We finally note that  the transformation $W$ defined in \eqref{eq:gauge_wahlen} 
commutes with  the translation operator $ \tau_\vs$, 
therefore the Hamiltonian $\cH$ in \eqref{Ham-Wal} is invariant under
$ \tau_\vs $. 

\subsection{Linearization at the equilibrium}\label{lin:op}

We now prove that the reversible solutions of the linear system \eqref{lin.ww1} have the form \eqref{linz200}; we proceed in a similar way as in \cite{BFM}, Section 2.1, and we refer to it for details.
The linear system \eqref{lin.ww1} is Hamiltonian and it is  generated by the quadratic Hamiltonian
$$
H_L(\eta,\psi) := \frac{1}{2} \int_\T\left( \psi G(0)\psi + g \eta^2  \right) \wrt x =
\frac12 
\left( \b\Omega_L  \begin{pmatrix}
\eta \\ \psi
\end{pmatrix}, 
\begin{pmatrix}
\eta \\ \psi
\end{pmatrix} \right)_{L^2} 
\, .
$$
Thus, recalling \eqref{ham.eq1}, the linear system \eqref{lin.ww1} is 
\begin{equation}\label{Lin:HS}
\pa_t \begin{pmatrix}
\eta \\ \psi
\end{pmatrix} = J_M(\gamma) \b\Omega_L \begin{pmatrix}
\eta \\ \psi
\end{pmatrix} \ , \qquad 
\b\Omega_L:=\begin{pmatrix}
 g & 0 \\ 0 & G(0)
\end{pmatrix} \, .
\end{equation} 
In the Wahl\'en  coordinates  \eqref{eq:gauge_wahlen}, 
 system  \eqref{Lin:HS} is transformed into  the  linear Hamiltonian system 
\begin{equation}\label{eq:lin00_wahlen} 
\begin{aligned}
&	\pa_t\begin{pmatrix}
\eta \\ \zeta
\end{pmatrix}= J \b\Omega_W  \begin{pmatrix}
\eta \\ \zeta
\end{pmatrix}  \ ,  \\
& \b\Omega_W := W^* \b  \Omega_L W =\begin{pmatrix}
g -  \left( \frac{\gamma}{2}\right)^2 \partial_x^{-1}G(0)\partial_x^{-1} & - \frac{\gamma}{2}\partial_x^{-1}G(0) \\ \frac{\gamma}{2}G(0)\partial_x^{-1} & G(0)
\end{pmatrix}	
\end{aligned}
\end{equation}
generated by the  quadratic Hamiltonian  
\begin{equation}\label{lin_real}
\cH_L (\eta, \zeta) := (H_L \circ W) (\eta, \zeta)  = \frac12 
\left( \b\Omega_W  \left(\begin{matrix}
\eta \\ \zeta
\end{matrix}\right), 
\left(\begin{matrix}
\eta \\ \zeta
\end{matrix}\right) \right)_{L^2} \, .
\end{equation}
Let us diagonalize  \eqref{eq:lin00_wahlen}.
We first conjugate \eqref{eq:lin00_wahlen} under  the  symplectic transformation
(with respect to the standard symplectic form $ {\cal W}   $ in \eqref{sympl-form-st}) of the phase space
$$
\begin{pmatrix}
\eta \\
\zeta
\end{pmatrix} = \cM \begin{pmatrix}
u \\
v
\end{pmatrix} 
$$
where $ \cM $ is the diagonal matrix of self-adjoint Fourier multipliers 
\begin{equation}\label{eq:T_sym}
\cM:= \left(\begin{matrix}
M(D) & 0 \\ 0 & M(D)^{-1}
\end{matrix}  \right) \ , \quad 
M(D) :=\left( \frac{G(0)}{ g - 
	\frac{\gamma^2}{4} \partial_x^{-1}G(0)\partial_x^{-1}} \right)^{1/4} \, , 
\end{equation}
with the real valued symbol $ M_j $ defined in \eqref{def:Mn}.
The map $ \cal {M} $ is reversibility preserving. 

\begin{rem}
%	In \eqref{eq:T_sym} the  Fourier multiplier $ M(D) $  
%	acts in $ H^1_0 $. On the other hand, 
%With a slight abuse of notation,  
	$M(D)^{-1} $  denotes  the Fourier multiplier operator in  $ \dot H^1 $ defined as 
	$
	M(D)^{-1} [\zeta] :=  \big[ \sum_{j \neq 0} M_j^{-1} \zeta_j e^{\im j x } \big] $, 
	$  \zeta (x) =  \sum_{j \in \Z} \zeta_j e^{\im j x } $
	where $[\zeta] $ is the element in $ \dot H^1 $ with representative 
	$ \zeta (x)  $. 
\end{rem}
By a direct computation, 
the Hamiltonian system  \eqref{eq:lin00_wahlen} assumes the symmetric form
\begin{eqnarray}
\label{lin.ww3}
\pa_t \begin{pmatrix}
u \\
v
\end{pmatrix} = J \b\Omega_S \begin{pmatrix}
u \\
v
\end{pmatrix} \, , \ \  \b\Omega_S:=\cM^*\b\Omega_W\cM = \left( \begin{matrix}
\omega (\gamma, D) & -\frac{\gamma}{2}\partial_x^{-1}G(0) \\ \frac{\gamma}{2}G(0)\partial_x^{-1} & \omega (\gamma, D)
\end{matrix}\right)  \, ,
\end{eqnarray}
where
\begin{equation}\label{eq:omega0}
\omega(\gamma,D):= \sqrt{ g\,G(0) - \left( \frac{\gamma}{2}\pa_x^{-1} G(0) \right)^2 } \, . 
\end{equation}
Now we introduce   complex coordinates  by  the transformation 
\begin{equation}\label{C_transform}
\begin{pmatrix}
u \\
v
\end{pmatrix} = \cC 
\begin{pmatrix}
z \\
\bar z
\end{pmatrix} \, , \qquad 
\cC := \frac{1}{\sqrt 2} \left(\begin{matrix}
{\rm Id} & {\rm Id} \\ -\im & \im
\end{matrix}\right) \ ,\quad \cC^{-1}:= \frac{1}{\sqrt 2}\left(\begin{matrix}
{\rm Id} & \im \\ {\rm Id} & -\im
\end{matrix}\right) \, .
\end{equation}
In these variables, the Hamiltonian system  \eqref{lin.ww3} becomes the diagonal system 
\begin{eqnarray}\label{eq:lin00_ww_C}
\partial_t\left( \begin{matrix}
z \\ \bar z
\end{matrix} \right)=
\begin{pmatrix}
- \im & 0 \\ 0 & \im
\end{pmatrix} 
\b\Omega_D\left( \begin{matrix}
z \\ \bar z
\end{matrix} \right) \, , \quad \b\Omega_D := \cC^*\b\Omega_S\cC = \begin{pmatrix}
\Omega (\gamma, D) & 0 \\ 0 & \bar \Omega (\gamma, D)
\end{pmatrix} \,,
\end{eqnarray}
where 
\begin{equation}
\label{Omega}
\Omega (\gamma, D) := \omega (\gamma, D) + \im \,\frac{\gamma}{2}\partial_x^{-1} G(0)
\end{equation}
is the Fourier multiplier with symbol  $  \Omega_j(\gamma)   $ defined in \eqref{def:Omegajk} and
$\bar \Omega(\gamma, D)$ is  defined by 
$$
\bar \Omega(\gamma, D) z:= \bar{\Omega(\gamma, D) \bar z} \, , 
\quad \bar \Omega(\gamma, D)  = 
\omega (\gamma, D) - \im \,\frac{\gamma}{2}\partial_x^{-1} G(0)  \, . 
$$
Note that $\bar \Omega(\gamma, D)$ is the Fourier multiplier with symbol $\{\Omega_{-j}(\gamma)\}_{j \in \Z\setminus\{0\}}$. 

	We regard the system  \eqref{eq:lin00_ww_C} in $ \dot H^1 \times \dot H^1 $. 
The diagonal system 
\eqref{eq:lin00_ww_C} amounts to the scalar equation
\begin{equation}\label{zjF}
\pa_t z = - \im \Omega (\gamma, D) z \, , \quad 
z(x) =  \sum_{j \in \Z\setminus\{0\}} z_j e^{\im j x } \, , 
\end{equation}
which, written in the 
exponential Fourier basis,  
is an 
infinite collections of  decoupled harmonic oscillators 
\begin{equation}\label{eq:zj}
\dot z_j = - \im \Omega_j (\gamma) z_j \, , \quad j \in \Z \setminus \{0\} \, .
\end{equation}
Note that, in these complex coordinates,   
the involution $\cS$ defined in \eqref{rev_invo} reads as the map
\begin{equation}\label{inv-complex}
\begin{pmatrix}
z(x) \\ \bar{z(x)}
\end{pmatrix} \mapsto \begin{pmatrix}
\, \bar{z(-x)} \\ z(-x) \, 
\end{pmatrix} \ , 
\end{equation}
whereas, 
in the Fourier coordinates introduced in \eqref{zjF}, it amounts to 
\begin{equation}\label{inv-zj}
z_j \mapsto \overline{z_j} \, , \quad \forall j \in \Z \setminus \{ 0 \} \,  .
\end{equation}
In view of \eqref{eq:zj} and \eqref{inv-zj}
any {\em reversible} solution (which is characterized as in \eqref{rev:soluz}) of  
\eqref{zjF} has  the form 
\begin{equation}
\label{linz}
z(t,x):= \frac{1}{\sqrt{2}}\sum_{j\in \Z\setminus\{0\}} \rho_{j} \,  e^{-\im\,\left( \Omega_j(\gamma)t- j\,x \right)} 
\quad {\rm with} \quad  \rho_j \in \R \, .
\end{equation}
Let us see the form of these solutions 
back 
in  the original variables $ (\eta, \psi)$.
First, by \eqref{eq:T_sym}, \eqref{C_transform},
\begin{equation}
\label{linz2}
\begin{pmatrix}
\eta \\
\zeta
\end{pmatrix} = \cM \, \cC 
\begin{pmatrix}
z \\
\bar z
\end{pmatrix} = 
\frac{1}{\sqrt{2}}\begin{pmatrix}
M(D) & M(D)  \\
-\im M(D)^{-1} & \im M(D)^{-1} 
\end{pmatrix}\begin{pmatrix}
z \\
\bar z
\end{pmatrix} =
\frac{1}{\sqrt{2}}\begin{pmatrix}
M(D) (z + \bar z)  \\
-\im M(D)^{-1} (z - \bar z )
\end{pmatrix} \, , 
\end{equation} 
and the solutions \eqref{linz}  assume the form
$$
\begin{aligned}
\begin{pmatrix}
\eta(t,x) \\ \zeta(t,x)
\end{pmatrix} & =  \sum_{n \in \N}  \begin{pmatrix}
M_n \rho_n \cos ( n x - \Omega_n (\gamma) t) \\ 
M_n^{-1} \rho_n \sin ( n x - \Omega_n (\gamma) t) 
\end{pmatrix}  +  
\begin{pmatrix}
M_n \rho_{-n} \cos ( n x + \Omega_{-n}(\gamma) t) \\ 
- M_n^{-1} \rho_{-n} \sin ( n x + \Omega_{-n} (\gamma) t) 
\end{pmatrix} 
\,.
\end{aligned}
$$
Back to the variables $(\eta, \psi)$ with the change of coordinates  \eqref{eq:gauge_wahlen} one obtains formula \eqref{linz200}.

\paragraph{Decomposition of the phase space in Lagrangian  subspaces invariant 
	under
	\eqref{eq:lin00_wahlen}.}
We express the Fourier coefficients $ z_j \in \C $ in  \eqref{zjF} as 
$$
z_j = \frac{\alpha_j + \im \beta_j}{\sqrt{2}}, \quad (\alpha_j, \beta_j) \in \R^2 \, , \quad j \in \Z\setminus\{0\} \, .
$$
In the new coordinates $ (\alpha_j, \beta_j)_{ j \in \Z\setminus\{0\}} $,
we  write  \eqref{linz2} as (recall that $ M_j = M_{-j} $)
\begin{equation}\label{deco-real}
\begin{pmatrix}
\eta (x) \\
\zeta (x) 
\end{pmatrix}  = \sum_{j \in \Z\setminus\{0\}}  
\begin{pmatrix}
M_j ( \alpha_j \cos (jx) - \beta_j  \sin (jx) )  \\
M_j^{-1} (  \beta_j  \cos (jx) +  \alpha_j \sin (jx) ) 
\end{pmatrix} 
\end{equation}
with
\begin{equation}\label{proj-Vj}
\begin{aligned}
& \alpha_j = \frac{1}{2\pi} 
\Big( M_j^{-1} (\eta, \cos (jx))_{L^2} + M_j (\zeta, \sin (jx))_{L^2}  \Big) \, , \\
& \beta_j = 
\frac{1}{2\pi} 
\Big( M_j (\zeta, \cos (jx))_{L^2} - M_j^{-1} (\eta, \sin (jx))_{L^2}  \Big)\, .
\end{aligned}
\end{equation}
The symplectic form \eqref{sympl-form-st}  then becomes
$ 2\pi\sum_{j \in \Z\setminus\{0\}} \di \alpha_j \wedge \di \beta_j $. 
Each $ 2$-dimensional 
subspace in the sum \eqref{deco-real}, spanned by 
$ (\alpha_j, \beta_j ) \in \R^2 $ is therefore a symplectic subspace. 
The quadratic Hamiltonian $ \cH_L $ in \eqref{lin_real} reads 
\begin{equation}\label{QFH}
2 \pi \sum_{j \in \Z\setminus\{0\}} \frac{\Omega_j(\gamma)}{2} (\alpha_j^2 + \beta_j^2 )\, . 
\end{equation}
In view of \eqref{deco-real}, the involution $ \cS $ defined in \eqref{rev_invo} reads
$ (\alpha_j, \beta_j) \mapsto (\alpha_j, - \beta_j) $, $ \forall j \in \Z \setminus \{ 0 \} $.

We may also  enumerate the independent variables $ (\alpha_j, \beta_j )_{j \in \Z\setminus\{0\}} $ as
$  \big( \alpha_{-n}, \beta_{-n}, \alpha_{n}, \beta_{n} \big) $, $  n \in \N $. 
Thus  the phase space $\acca := L^2_0 \times \dot L^2 $ of 
\eqref{eq:Ham_eq_zeta} decomposes as the direct sum 
\begin{equation}
\label{acca}
\acca = \sum_{n \in \N} V_{n,+}  \oplus V_{n,-} 
\end{equation}
of $ 2 $-dimensional Lagrangian symplectic subspaces 
\begin{align} 
V_{n,+} & :=
\left\{ 
\begin{pmatrix}
\eta \\ \zeta
\end{pmatrix} 
= 
\begin{pmatrix}
M_{n} ( \alpha_{n} \cos (nx) - \beta_n \sin (nx) ) \\ 
M_{n}^{-1} ( \beta_{n} \cos (nx) + \alpha_n \sin (nx) ) 
\end{pmatrix} \, , (\alpha_n , \beta_n) \in \R^2  
\right\} \, , \label{Vn+} \\
V_{n,-} & :=
\left\{ 
\begin{pmatrix}
\eta \\ \zeta
\end{pmatrix}  = 
\begin{pmatrix}
M_{n} ( \alpha_{-n} \cos (nx) + \beta_{-n} \sin (nx) ) \\ 
M_{n}^{-1} ( \beta_{-n} \cos (nx) - \alpha_{-n} \sin (nx) ) 
\end{pmatrix} \, , (\alpha_{-n} , \beta_{-n}) \in \R^2  
\right\} \, , \label{Vn-}
\end{align}  
which are invariant for the linear Hamiltonian system \eqref{eq:lin00_wahlen}, namely
$ J  \b\Omega_W : V_{n,\sigma} \mapsto V_{n,\sigma} $.
Note that the involution $ \cS $ defined in \eqref{rev_invo} and the translation operator 
$ \tau_\vs $ in \eqref{trans} leave the subspaces $ V_{n,\sigma} $, $ \sigma \in \{ \pm  \} $,  invariant.

\subsection{Tangential and normal subspaces of the phase space}
\label{sec:decomp}

We  split  the phase space $ \acca $  in \eqref{acca} 
into a direct sum of {\em tangential} and {\em normal}  Lagrangian subspaces 
$ \acca_{\S^+,\Sigma}^\intercal $ and 
$ \acca_{\S^+,\Sigma}^\angle  $. 
Note that the 
main part of the solutions \eqref{QP:soluz} that we shall obtain in Theorem 
\ref{thm:main0} is the component  in the tangential subspace 
$ \acca_{\S^+,\Sigma}^\intercal $,  whereas the component in  the normal subspace $ \acca_{\S^+,\Sigma}^\angle  $ is much smaller. 

Recalling  the definition of the sets $\S^+$ and 
$\Sigma$ defined in \eqref{Splus} respectively \eqref{signs}, we split 
\begin{equation}\label{H-split}
\acca =\acca_{\S^+,\Sigma}^\intercal\oplus\acca_{\S^+,\Sigma}^\angle 
\end{equation}
where $\acca^{\intercal}_{\S^+, \Sigma} $ is the  finite dimensional {\em tangential subspace}  
\begin{equation}
\label{cHStang}
\acca^{\intercal}_{\S^+, \Sigma} := 
\sum_{a = 1}^\nu V_{{\bar n}_a, \sigma_a} 
\end{equation}
and $\acca^{\angle}_{\S^+, \Sigma} $ is the {\em normal subspace} defined as
its symplectic orthogonal 
\begin{equation}
\label{cHSnorm}
\acca^{\angle}_{\S^+, \Sigma} := \sum_{a = 1}^\nu V_{{\bar n}_a, - \sigma_a} 
\oplus  
\sum_{n \in \N \setminus {\mathbb S}^+} \big( V_{n,+}  \oplus V_{n,-}\big) \, . 
\end{equation}
Both the subspaces 
$\acca_{\S^+,\Sigma}^\intercal$ and $\acca_{\S^+,\Sigma}^\angle$ are Lagrangian. 
We denote  by
$ \Pi_{\S^+,\Sigma}^\intercal $ and
$  \Pi_{\S^+,\Sigma}^\angle  $ 
the symplectic projections on the subspaces $\acca_{\S^+,\Sigma}^\intercal$ and $\acca_{\S^+,\Sigma}^\angle$, respectively. 
The restricted symplectic form $\cW\vert_{\acca_{\S^+, \Sigma}^\angle} $ is represented by the symplectic structure  
$$
J_\angle^{-1} : \acca_{\S^+, \Sigma}^\angle \to \acca_{\S^+, \Sigma}^\angle \, , \
\quad 
J_\angle^{-1} := \Pi^{L^2}_\angle \,   J^{-1}_{| \acca_{\S^+,\Sigma}^\angle }    \, ,  
$$
where $ \Pi^{L^2}_\angle $ is the $ L^2 $-projector on the subspace 
$ \acca_{\S^+,\Sigma}^\angle $.  
Its 
 associated  Poisson tensor is 
$$
J_\angle :  \acca_{\S^+, \Sigma}^\angle \to \acca_{\S^+, \Sigma}^\angle \, , \quad
J_\angle := \Pi^\angle_{\S^+, \Sigma}  \, J_{| \acca_{\S^+,\Sigma}^\angle } \, .
$$
By Lemma 2.6 in \cite{BFM}, we have 
$ J^{-1}_\angle  \, J_\angle  = $ $ J_\angle \, J^{-1}_\angle =  $ $ {\rm Id}_{ \acca_{\S^+, \Sigma}^\angle} $. 
\\[1mm]
{\bf Action-angle coordinates.}
We introduce action-angle coordinates on the tangential subspace 
$ \acca^{ \intercal}_{\S^+, \Sigma}$ defined in \eqref{cHStang}. 
Given the sets $\S^+$ and $\Sigma$ defined respectively in \eqref{Splus} and \eqref{signs},  
we define the set 
\begin{equation}
\label{def.S}
\S := \{ \bar \jmath_1 , \ldots, \bar \jmath_\nu \} \subset \Z \,\setminus\{0\}\, , \quad \bar 
\jmath_a := \sigma_a \bar n_a \, ,
\quad a = 1, \ldots, \nu \, , 
\end{equation} 
and  the action-angle coordinates $ (\theta_j, I_j)_{j \in {\mathbb S}} $, by the relations  
\begin{equation}\label{ajbjAA}
\alpha_j = \sqrt{\frac{1}{\pi}(I_j + \xi_j)}\cos(\theta_j) \,, \ 
\beta_j = -\sqrt{\frac{1}{\pi}(I_j + \xi_j)}\sin(\theta_j) \, , \quad \xi_j >0 \, , \ | I_j | <  \xi_j \, ,
\ \forall j \in {\mathbb S} \, . 
\end{equation}
In view of \eqref{H-split}-\eqref{cHSnorm}, we represent any function
of the phase space $ \acca $ as
\begin{align}
\label{aacoordinates}
A(\theta,I,w)& := v^\intercal(\theta,I)+ w \notag \, ,  \\
& := \frac{1}{\sqrt{\pi}}\sum_{j\in\S}
\left[ \begin{pmatrix}
M_j\sqrt{I_j+ \xi_j}\cos(\theta_j) \\ -M_j^{-1}\sqrt{I_j+\xi_j}\sin(\theta_j)
\end{pmatrix}\cos(j x)+\begin{pmatrix}
M_j\sqrt{I_j+ \xi_j}\sin(\theta_j) \\ M_j^{-1}\sqrt{I_j+\xi_j}\cos(\theta_j)
\end{pmatrix}\sin(j x) \right] + w  \nonumber \\
& = \frac{1}{\sqrt{\pi}}\sum_{j\in\S}
\left[ \begin{pmatrix}
M_j\sqrt{I_j+ \xi_j} \cos(\theta_j - j x ) \\ 
- M_j^{-1}\sqrt{I_j+\xi_j}\sin(\theta_j - j x )
\end{pmatrix} \right] + w 
\end{align}
where $ \theta := (\theta_j)_{j \in \S} \in \T^\nu $, $ I := (I_j)_{j \in \S } \in \R^\nu $
and $ w \in  \acca^{ \angle}_{\S^+, \Sigma} $.

In view of  \eqref{aacoordinates}, 
the involution $\cS$ in \eqref{rev_invo} reads
\begin{equation}\label{rev_aa}
\vec \cS: (\theta,I,w)\mapsto \left( -\theta,I,\cS w \right)\,, 
\end{equation}
the  translation operator $\tau_\varsigma$ in  \eqref{trans} reads
\begin{equation}
\label{vec.tau}
\vec \tau_\vs : 
(\theta, \, I, \, w) \mapsto 
(\theta - \ora{\jmath} \vs, \, I, \, \tau_\vs w), \quad \forall \vs \in \R \, , 
\end{equation}
where 
\begin{equation}\label{def:vecj}
\ora{\jmath}:= (j)_{j \in \S} = ( \bar{\jmath}_1,\ldots,\bar{\jmath}_\nu) \in \Z^\nu\setminus\{0\} \, , 
\end{equation}
and the symplectic 2-form \eqref{sympl-form-st} becomes
\begin{equation}\label{sympl_form}
{\cal W} = 
\sum_{j\in\S} (\di \theta_j \wedge  \di I_j)  \, \oplus \, 
{\cal W}|_{\acca_{\S^+, \Sigma}^\angle} \,  . 
\end{equation}
We also note that  ${\cal W} $ is exact, namely 
\begin{equation}\label{liouville}
{\cal W} = d \Lambda\,,   \qquad {\rm where} \qquad 
\Lambda_{(\theta,I,w)} [ \wh\theta,\whI ,\wh{w} ] := - \sum_{j\in\S} I_j \wh{\theta}_j + \tfrac12 \left( J_\angle^{-1} w, \wh{w} \right)_{L^2} 
\end{equation}
is the  associated  Liouville 1-form.
Finally, given a Hamiltonian $ K \colon \T^\nu \times \R^\nu \times \acca_{\S^+, \Sigma}^\angle \to \R$, 
the associated Hamiltonian vector field 
(with respect to the symplectic form \eqref{sympl_form}) is
$$
X_K  := 
\big( \pa_I K, -\pa_\theta K,  J_\angle \nabla_{w} K \big)  =
\big( \pa_I K,  -\pa_\theta K,  \Pi_{\S^+,\Sigma}^\angle J \nabla_{w} K \big) 	\, , 
$$
where $\grad_w K $ denotes the $L^2$ gradient of $K$ with respect to $ w
\in \acca_{\S^+, \Sigma}^\angle$. 
\\[1mm]
{\bf Tangential and normal subspaces in complex variables.}
Each $ 2 $-dimensional symplectic subspace $ V_{n,\sigma} $, 
$ n \in \N $, $ \sigma = \pm 1 $,  defined in \eqref{Vn+}-\eqref{Vn-}
is isomorphic,   
through the linear map $ {\cal M } {\cal C} $ defined in \eqref{linz2}, 
to the complex subspace 
$$
\bH_j  
:= \Big\{  \begin{pmatrix}
z_j e^{\im j x }  \\ \overline{z_j} e^{- \im j x } 
\end{pmatrix} 	\, , \  z_j \in \C   \Big\}  \qquad {\rm with } \qquad j = n \sigma \in \Z  \, .
$$
Denoting by $ \Pi_j $  the $ L^2 $-projection on $	\bH_j   $, we have that 
$\Pi_{V_{n, \sigma}}  = {\cal M } {\cal C} \,  \Pi_j \, ({\cal M } {\cal C})^{-1} $.  
Thus $ {\cal M } {\cal C} $
is an isomorphism between the tangential subspace 
$ \acca^{\intercal}_{\S^+, \Sigma} $ defined in \eqref{cHStang} and 
$$
\bH_{\mathbb S} 
:= \Big\{ \begin{pmatrix}
z \\ \bar z
\end{pmatrix}  \, : \, 
z (x) = \sum_{j \in {\mathbb S} } z_j e^{\im j x }  \Big\}
$$
and between  the  normal subspace $ \acca^{\angle}_{\S^+, \Sigma} $ defined in \eqref{cHSnorm} and 
\begin{equation}\label{def:HS0bot}
\bH_{{\mathbb S}_0}^\bot := \Big\{ \begin{pmatrix}
z \\ \bar z
\end{pmatrix} \, : \, 
z (x) = \sum_{j \in \S_0^c } z_j e^{\im j x } \in L^2   \Big\}  \, ,
\quad \S_0^c := \Z \setminus (\S \cup \{0\}) \, . 
\end{equation}
Denoting by 
$ \Pi_{\S}^\intercal $, $\Pi_{\S_0}^\perp  $, the $ L^2 $-orthogonal 
projections on the subspaces $ \bH_\S $ and 
$ \bH_{\S_0}^\perp $, we have that
\begin{equation}\label{proiez}
\Pi_{\S^+,\Sigma}^\intercal = {\cal M } {\cal C} \,  \Pi_{\S}^\intercal \, ({\cal M } {\cal C})^{-1} \, , 
\quad
\Pi_{\S^+,\Sigma}^\angle = {\cal M } {\cal C} \, \Pi_{\S_0}^\perp \, ({\cal M } {\cal C})^{-1} \, . 
\end{equation} 
From this analysis, it follows that (cfr. Lemma 2.9 in \cite{BFM})
\begin{equation}
\label{zero_term}
\left( v^\intercal , \b\Omega_W w \right)_{L^2} = 0 \ , \qquad
\forall 
 v^\intercal  \in \acca_{\S^+,\Sigma}^\intercal , \ \  \forall w \in \acca_{\S^+,\Sigma}^\angle  \ . 
\end{equation}

\noindent
{\bf Notation.} For $ a \lesssim_s b $ means that $ a \leq C(s) b $ for some positive constant 
$ C (s) $. We denote $ \N := \{1, 2, \ldots \} $ and $ \N_0 := \{0\} \cup \N $.

\section{Functional setting}\label{sec:FS}

In this section we report basic notation, definitions, and results used along the paper, concerning traveling waves, pseudo-differential operators, tame operators, 
and the algebraic properties of Hamiltonian, reversible and momentum preserving operators. 

\smallskip

We consider functions $u(\vf,x)\in L^2\left(\T^{\nu+1},\C\right)$ depending on the space variable $x\in\T=\T_x$ and the angles $\vf\in\T^\nu=\T_\vf^\nu$ (so that $\T^{\nu+1}= \T_\vf^\nu\times \T_x$) which we expand in Fourier series as
\begin{equation}\label{u_fourier}
u(\vf,x) =  \sum_{j\in\Z} u_j(\vf)e^{\im\,jx} = 
\sum_{\ell\in\Z^\nu,j\in\Z}u_{\ell,j}e^{\im (\ell\cdot \vf +jx )}\, . 
\end{equation}
We also consider real valued functions $u(\vf,x)\in\R$, as well as vector valued functions $u(\vf,x)\in\C^2$ (or $u(\vf,x)\in\R^2$). 
When no confusion appears, we denote simply by $L^2$, $L^2(\T^{\nu+1})$, $L_x^2:=L^2(\T_x)$, $L_\vf^2:= L^2(\T^\nu)$ either the spaces of real/complex valued, scalar/vector valued, $L^2$-functions.

\paragraph{Quasi-periodic traveling waves.}

We first provide the following definition: 

\begin{defn} \label{QPTW}
	{\bf (Quasi-periodic  traveling waves)} 
	Let  $ \ora{\jmath} := (\bar{\jmath}_1, \ldots ,\bar{\jmath}_\nu) \in \Z^\nu $ be the vector defined in 
	\eqref{def:vecj}. 
	A function 
	$ u (\vf, x) $  is called a {\em quasi-periodic  traveling} wave if it has 
	the form $ u(\vf,x) = U(\vf-\ora{\jmath}x)  $ 
	where $ U :  \T^\nu \to \C^K $, $ K \in \N $, is 
	a $ (2 \pi)^\nu $-periodic function. 
\end{defn}

Comparing with Definition \ref{def:TV}, we find convenient to call 
{\em quasi-periodic  traveling} wave both the function $ u(\vf,x) = U(\vf-\ora{\jmath}x)  $ and the function of time  $ u(\omega t,x) = U(\omega t-\ora{\jmath}x)  $. 

Quasi-periodic traveling waves are characterized by the relation 	
$ u(\vf - \ora{\jmath} \vs, \cdot) = \tau_\vs u $ for any $\vs \in \R  $,
where $\tau_\vs $ is the translation operator in \eqref{trans}. 

Product and composition of quasi-periodic traveling waves are quasi-periodic traveling waves. 
Expanded in Fourier series as in \eqref{u_fourier}, 
a quasi-periodic traveling wave has the form 
\begin{equation}\label{u-trav-Fourier}
u(\vf, x) = 
\sum_{\ell\in\Z^\nu,j\in\Z, j + \vec \jmath \cdot \ell = 0 }u_{\ell,j}e^{\im (\ell\cdot \vf +jx )}
\, , 
\end{equation}
namely, comparing with  Definition \ref{QPTW}, 
\begin{equation}\label{uphiU}
u(\vf, x) = U(\vf - \vec \jmath x) \, , \quad U (\psi) = \sum_{\ell \in \Z^\nu}
U_\ell e^{\im \ell \cdot \psi} \, , \quad U_\ell = u_{\ell, - \vec \jmath \cdot \ell} \, . 
\end{equation}
The quasi-periodic traveling waves $ u (\vf, x) = U (\vf - \vec \jmath x ) $ where 
$ U (\cdot ) $ belongs to the Sobolev space $ H^s (\T^\nu, \C^K ) $ 
in \eqref{unified norm} (with values in $ \C^K $, $ K \in \N $), 
form a subspace of the 
Sobolev space 
\begin{equation} \label{Sobonorm}
H^s(\T^{\nu+1})
= \Big\{ u = \sum_{(\ell,j) \in \Z^{\nu+1}} u_{\ell,j} \, e^{\im (\ell \cdot \vf + jx)} \, : \, 
\| u \|_s^2 := \sum_{(\ell,j) \in \Z^{\nu+1}} | u_{\ell, j}|^2 \langle \ell,j \rangle^{2s} < \infty 
\Big\}
\end{equation}
where  $\langle \ell,j \rangle := \max \{ 1, |\ell|, |j| \} $. 
Note the equivalence of the norms 
$ \| u \|_{H^s (\T^{\nu}_\vf \times \T_x )} \simeq_s \| U \|_{H^s (\T^{\nu})} $.
 
For 
$ s \geq s_0 := \big[ \frac{\nu +1}{2} \big] +1 \in \N  $
one has $ H^s ( \T^{\nu+1}) \subset C ( \T^{\nu+1})$, and $H^s(\T^{\nu+1})$ is an algebra. 
Along the paper we  denote by $ \| \ \|_s $ both the Sobolev norms in 
\eqref{unified norm} and \eqref{Sobonorm}.  

For $K\geq 1$ we define the smoothing operator  $\Pi_{K}$  on the quasi-periodic traveling waves
\begin{equation}\label{pro:N}
\Pi_K :  u = \sum_{\ell \in \Z^\nu,\, j \in \S_0^c, \,j + \vec \jmath \cdot \ell = 0}
u_{\ell,j}e^{\im(\ell\cdot \vf+jx)} \mapsto
\Pi_K u = \sum_{\braket{\ell}\leq K,\, j \in \S_0^c,\, j + \vec \jmath \cdot \ell = 0}
u_{\ell,j}e^{\im(\ell\cdot \vf+jx)}   \, ,  
\end{equation}
and $ \Pi_K^\perp := {\rm Id}-\Pi_K  $. 
Writing a traveling wave as in  
\eqref{uphiU}, the projector $ \Pi_K $ in \eqref{pro:N} is equal to 
$$ 
(\Pi_K u)(\vf, x) =  U_K (\vf - \vec \jmath x ) \, , \quad U_K (\psi) := 
\sum_{\ell \in \Z^\nu,\, \langle \ell \rangle \leq K} U_\ell e^{\im \ell \cdot \psi} \, . 
$$
For a function $ u(\vphi, x) $ we define the averages
\begin{equation}\label{def:avera}
\langle u \rangle_{\vphi,x} := \frac{1}{(2 \pi)^{\nu+1}}  \int_{\T^{\nu+1}} u(\vphi,x) \wrt \vphi \wrt x 
\, , \quad 
\langle u \rangle_{\vphi}(x) := \frac{1}{(2 \pi)^{\nu}}  \int_{\T^{\nu+1}} u(\vphi,x) \wrt \vphi \, ,  
\end{equation}
and we note that, if $ u (\vphi,x) $ is a quasi-periodic traveling wave then 
$ \langle u \rangle_\vphi = \langle u \rangle_{\vphi,x}  $. 
\\[1mm]
{\bf Whitney-Sobolev functions.} 
Along the paper we consider  families of Sobolev functions $\lambda\mapsto u(\lambda)\in H^s (\T^{\nu+1}) $ and  
$\lambda\mapsto U(\lambda)\in H^s (\T^{\nu}) $ which are $k_0$-times differentiable in the sense of Whitney with respect to the parameter
$ \lambda :=(\omega,\gamma) \in F \subset \R^\nu\times [\gamma_1,\gamma_2] $
where $F\subset \R^{\nu+1}$ is a closed set.
We refer to Definition 2.1 in \cite{BBHM}, for the 
definition of a Whitney-Sobolev function $ u : F \to H^s $ where 
$ H^s $ may be either the Hilbert space $ H^s (\T^\nu \times \T) $ or $ H^s (\T^\nu) $.
Here we mention that, given $ \upsilon \in (0,1) $,  
we can identify a  Whitney-Sobolev function $ u : F \to H^s $
with $ k_0 $ derivatives 
with the equivalence class of functions $ f \in W^{k_0,\infty,\upsilon}(\R^{\nu+1},H^s)/\sim$ with respect to the equivalence relation $f\sim g$ when $\pa_\lambda^j f(\lambda) = \pa_\lambda^j g(\lambda)$ for all $\lambda\in F$, $\abs j \leq k_0-1$, 
with equivalence of the norms 
\begin{equation*}
\normk{u}{s,F}  \sim_{\nu,k_0} \norm{u}_{W^{k_0,\infty,\upsilon}(\R^{\nu+1},H^s)}:= \sum_{\abs\alpha\leq k_0} \upsilon^{\abs\alpha} \| \pa_\lambda^\alpha u \|_{L^\infty(\R^{\nu+1},H^s)} \,.
\end{equation*}
The key result is the Whitney extension theorem,  
which associates to a Whitney-Sobolev function 
$u : F \to H^s $ with $ k_0 $-derivatives a 
function $\wtu : \R^{\nu+1} \to H^s $, $\wtu $ in $ W^{k_0,\infty}(\R^{\nu+1},H^s) $
(independently of the target Sobolev space $H^s$) with an equivalent norm.  
For sake of simplicity we  
often denote  $ \| \ \|_{s,F}^{k_0,\upsilon} =  \| \ \|_{s}^{k_0,\upsilon} $.

Thanks to this equivalence, all the tame 
estimates which hold for Sobolev spaces carry over for Whitney-Sobolev functions. 
For example the following classical tame estimate for the product holds: 
(see e.g.  Lemma 2.4 in \cite{BBHM}): 
for all $s\geq s_0 > (\nu+1)/2$, 
\begin{equation}\label{prod}
\normk{u v}{s} \leq C(s,k_0) \normk{u}{s}\normk{v}{s_0} + C(s_0,k_0)\normk{u}{s_0}\normk{v}{s}\,.
\end{equation}
Moreover the following  estimates hold for the smoothing operators defined 
in \eqref{pro:N}: for any quasi-periodic traveling wave $ u $
\begin{equation}\label{SM12}
\normk{\Pi_K u}{s}  \leq K^\alpha \normk{u}{s-\alpha}\,, \  0\leq \alpha\leq s \,, \quad 
\normk{\Pi_K^\perp u}{s}  \leq K^{-\alpha} \normk{u}{s+\alpha}\,, \ \alpha\geq 0 \,.
\end{equation}
We also state a standard Moser tame estimate for the nonlinear composition operator, see e.g. Lemma 2.6 in \cite{BBHM}, 
$ u(\vf,x)\mapsto \tf(u)(\vf,x) :=f(\vf,x,u(\vf,x)) $. 
Since the variables $(\vf,x)=:y$ have the same role, we state it for a generic Sobolev space $H^s(\T^d)$.
\begin{lem}{\bf (Composition operator)}\label{compo_moser}
	Let $f\in\cC^\infty(\T^d\times\R,\R)$. If $u(\lambda)\in H^s(\T^d)$ is a family of Sobolev functions satisfying $\normk{u}{s_0}\leq 1$, then, for all $s\geq s_0:=(d+1)/2$,
	$$
	\normk{\tf(u)}{s}\leq C(s,k_0,f)\big( 1+\normk{u}{s} \big) \,.
	$$
	If $ f(\varphi, x, 0) = 0 $ then 
	$	  \normk{\tf(u)}{s}\leq C(s,k_0,f) \normk{u}{s} $. 
\end{lem}

\paragraph{Constant transport equation on quasi-periodic traveling waves.} 
Let $\tm : \R^\nu \times [\gamma_1, \gamma_2]  \to \R$, $ (\omega, \gamma) \mapsto 
\tm (\omega, \gamma) $  be a real function.  For any 
%the frequency vector $\omega $ 
$ (\omega, \gamma) $ in 
% belongs  to the set
\begin{equation}\label{DCm}
	\tT\tC(\tm;\upsilon,\tau):=\big\{ (\omega, \gamma)  
	\in\R^\nu \times [\gamma_1, \gamma_2] \,:\, | \omega\cdot\ell + \tm\,j|\geq \upsilon\braket{\ell}^{-\tau}, \forall(\ell,j)\in\Z^{\nu+1}\setminus\{0\}\,, \text{ with } \ora{\jmath}\cdot\ell+j=0 \big\}\,,
\end{equation}
then, for a quasi-periodic traveling wave $u(\vf,x)$ with zero average with respect to $\vf$ (and therefore with respect to $(\vf,x)$), the transport equation $(\omega\cdot\pa_\vf + \tm \, \pa_x)v = u$ has the quasi-periodic traveling wave solution (see \eqref{u-trav-Fourier})
\begin{equation*}
	(\omega\cdot\pa_\vf+\tm\,\pa_x)^{-1}u := \sum_{(\ell,j)\in\Z^{\nu+1}\setminus\{0\}\atop \ora{\jmath}\cdot\ell+j=0}\frac{u_{\ell,j}}{\im(\omega\cdot\ell+\tm\,j)}e^{\im(\ell\cdot\vf+jx)}\,.
\end{equation*}
For any $ (\omega, \gamma) \in\R^\nu \times [\gamma_1, \gamma_2] $, 
we define its extension
\begin{equation}\label{paext}
(\omega\cdot\pa_\vf+ \tm \,\pa_x)_{\rm ext}^{-1} u(\vf,x) := \sum_{(\ell,j)\in\Z^{\nu+1}\atop \ora{\jmath}\cdot \ell + j =0}\frac{\chi((\omega\cdot\ell+\tm\, j)\upsilon^{-1}\braket{\ell}^\tau)}{\im (\omega\cdot\ell + \tm \,j)} u_{\ell,j}e^{\im(\ell\cdot\vf+jx)}\,,
\end{equation}
where $\chi\in\cC^\infty(\R,\R)$ is an even positive $\cC^\infty$ cut-off function such that
\begin{equation}\label{cutoff}
\chi(\xi) = \begin{cases}
0 & \text{ if } \ \abs\xi\leq \frac13 \\
1 & \text{ if } \ \abs\xi \geq \frac23 
\end{cases}\,, \qquad \pa_\xi \chi(\xi) >0   \quad \forall\,\xi \in (\tfrac13,\tfrac23) \,.
\end{equation}
Note that $(\omega\cdot\pa_\vf+\tm\,\pa_x)_{\rm ext}^{-1} u = (\omega\cdot\pa_\vf+\tm\,\pa_x)^{-1}u$ for all $(\omega, \gamma) \in\tT\tC(\tm;\upsilon,\tau)$. 

If $ |\tm |^{k_0,\upsilon} \leq C $ then  the following estimate holds
\begin{equation}\label{lem:diopha.eq}
\normk{(\omega\cdot\pa_\vf+\tm\,\pa_x)_{\rm ext}^{-1}u}{s,\R^{\nu+1}} \leq C(k_0)\upsilon^{-1}\normk{u}{s+\mu,\R^{\nu+1}}\,, \quad  \mu:=k_0+\tau(k_0+1) \,.
\end{equation}
Furthermore one has the estimate, for any $\omega \in \R^\nu$, $ \tm_1, \tm_2 \in\R$ and   $s\geq 0$
\begin{equation}\label{lem:diopha.eq.12}
\norm{\left( (\omega\cdot\pa_\vf+\tm_1\,\pa_x)_{\rm ext}^{-1}
- (\omega\cdot\pa_\vf+\tm_2\,\pa_x)_{\rm ext}^{-1}
\right)u}_{s} \leq C \, \upsilon^{-2}\, \abs{\tm_1 - \tm_2} \, \norm{u}_{s+2\tau+1} \, . 
\end{equation}

\paragraph{Linear operators.}
We consider $\vf$-dependent families of linear operators $A:\T^\nu\mapsto \cL(L^2(\T_x))$, $\vf\mapsto A(\vf)$, acting on subspaces of $L^2(\T_x)$. 
We also regard $A$ as an operator (which for simplicity we denote by $A$ as well) that acts on functions $u(\vf,x)$ of space and time, that is
\begin{equation}\label{linear_A}
(Au)(\vf,x) := \left( A(\vf)u(\vf,\,\cdot \,) \right)(x) \,.
\end{equation}
The action of an operator $A$ as in \eqref{linear_A} on a scalar function $u(\vf,x)\in L^2$ expanded as in \eqref{u_fourier} is
\begin{align}\label{A_expans}
Au(\vf,x) & = \sum_{j,j'\in\Z} A_j^{j'}(\vf) u_{j'}(\vf)e^{\im\, jx}  
= \sum_{j,j'\in\Z}\sum_{\ell,\ell'\in\Z^\nu} A_j^{j'}(\ell-\ell') u_{\ell',j'} e^{\im\left(\ell\cdot \vf + jx\right)} \,.
\end{align}
We identify an operator $A$ with its matrix $ \big( A_j^{j'}(\ell-\ell') \big)_{j,j'\in\Z,\ell,\ell'\in\Z^\nu}$, which is T\"oplitz with respect to the index $\ell$. In this paper we always consider T\"oplitz operators as in \eqref{linear_A}, \eqref{A_expans}. 
\\[1mm]
{\bf Real operators.} 
A linear  operator $ A $ is {\it real} if 
$ A = \bar A  $,
where $ \bar A $ is defined by  
$ \bar{A}(u):= \bar{A(\bar u)} $. 
% Equivalently $A$ is \emph{real} if it maps real valued functions into real valued functions.
We represent a real operator acting on $(\eta,\zeta) $ belonging to 
(a subspace of) $  L^2(\T_x,\R^2) $   by a matrix
\begin{equation}\label{real_matrix}
\cR = \begin{pmatrix}
A & B \\ C & D
\end{pmatrix}
\end{equation}
where $A,B,C,D$ are real operators acting on the scalar valued components 
$\eta,\zeta \in L^2(\T_x,\R) $. 

The  change of coordinates \eqref{C_transform}
transforms a real operator $\cR$ into a complex one acting on the variables
$(z,\bar z)  $, given by the matrix
\begin{equation}\label{C_transformed}
\bR  := \cC^{-1} \cR \cC = \left( \begin{matrix}
\cR_1 & \cR_2 \\ \bar\cR_2 & \bar\cR_1
\end{matrix} \right) \ , \\
\qquad 
\begin{matrix}
\cR_1  := \left\{ (A+D) -\im(B-C) \right\} /2 \, ,\\
 \cR_2:=\left\{ (A-D) + \im (B+C)  \right\} /2  \, . 
 \end{matrix}
\end{equation}
We call  \emph{real} a matrix operator acting on the complex variables $(z,\bar z)$ of the form
\eqref{C_transformed}.
We shall also consider  real operators $\bR $ of the form 
\eqref{C_transformed} acting on subspaces of $ L^2 $. 
\\[1mm]
{\bf Lie expansion.} Let $X(\vf)$ be a linear operator
with  associated  flow $\Phi^\tau ( \vf)$ defined by
$$
\pa_\tau \Phi^\tau (\vf) = X(\vf) \Phi^\tau (\vf) \, , \quad
\Phi^0 (\vf) = {\rm Id} \, ,  \quad  \tau \in[0,1] \, . 
$$
Let $ \Phi(\vf) := \Phi^\tau (\vf)_{|\tau = 1} $ denote the time-$1$ flow.
Given a linear operator $ A ( \vf ) $,
the conjugated operator 
$  A^+(\vf):=\Phi(\vf)^{-1}A(\vf) \Phi(\vf)  $
admits the  Lie expansion, for any  $ M \in \N_0 $, 
\begin{equation}\label{lie_abstract}
\begin{aligned}
& A^+(\vf) = 
\sum_{m=0}^M \frac{(-1)^m}{m!} \ad_{X(\vf)}^m(A(\vf)) +R_M(\vf)\,, \\ 
R_M(\vf)& = \frac{(-1)^{M+1}}{M!}\int_0^1 (1- \tau)^M \, (\Phi^\tau (\vf))^{-1} \ad_{X(\vf)}^{M+1}(A(\vf)) \Phi^\tau(\vf) \, \wrt \tau \,,
\end{aligned}
\end{equation}
where $\ad_{X(\vf)}(A(\vf)) := [X(\vf), A(\vf)] = X(\vf) A(\vf) - A(\vf) X(\vf) $ 
and  $\ad_{X(\vf)}^0 := {\rm Id} $. 

In particular, for $A=\omega\cdot\pa_\vf$, since $[X(\vf), \omega\cdot \pa_\vf] = 
- ( \omega\cdot \pa_\vf X)(\vf) $, we obtain
\begin{equation}\label{lie_omega_devf}
\begin{aligned}
\Phi(\vf)^{-1} \circ  \omega\cdot\pa_\vf \circ \Phi(\vf) = &\, \omega\cdot\pa_\vf 
+ 
\sum_{m=1}^{M} \frac{(-1)^{m+1}}{m!} \ad_{X(\vf)}^{m-1}(\omega\cdot \pa_\vf X(\vf)) \\
& + \frac{(-1)^M}{M!}\int_0^1(1- \tau)^{M} (\Phi^\tau(\vf))^{-1}\ad_{X(\vf)}^M(\omega\cdot\pa_\vf X(\vf))\Phi^\tau(\vf) \wrt \tau \,.
\end{aligned}
\end{equation}
For matrices of operators $\bX(\vf)$ and $\bA(\vf)$ as in \eqref{C_transformed}, the same formula \eqref{lie_abstract} holds.

\subsection{Pseudodifferential calculus}\label{subsec:pseudo_calc}

In this section we report fundamental notions of pseudodifferential calculus,
following \cite{BM}. 
\begin{defn}{\bf ($\Psi$DO)} 
	A  \emph{pseudodifferential} symbol $ a (x,j) $
	of order $m$ is 
	the restriction to $ \R \times \Z $ of a function $ a (x, \xi ) $ which is $ \cC^\infty $-smooth on $ \R \times \R $,
	$ 2 \pi $-periodic in $ x $, and satisfies, $ \forall \alpha, \beta \in \N_0 $,  
$ | \pa_x^\alpha \pa_\xi^\beta a (x,\xi ) | \leq C_{\alpha,\beta} \langle \xi \rangle^{m - \beta} $. 
	We denote by $ S^m $ 
	the class of  symbols  of order $ m $ and 
	$ S^{-\infty} := \cap_{m \geq 0} S^m $. 
	To a  symbol $ a(x, \xi ) $ in $S^m$ we associate its quantization acting on a $ 2 \pi $-periodic function 
	$ u(x) = \sum_{j \in \Z} u_j \, e^{\im j x} $
	as 
	$$
	[\Op(a)u](x) : = \sum_{j \in \Z}  a(x, j )  u_j \, e^{\im j x} \, . 
	$$
	We denote by  $ \Ops^m $ 
	the  set of pseudodifferential  operators of order $ m $ and
	$ \Ops^{-\infty} := \bigcap_{m \in \R} \Ops^{m} $.
	For a matrix of pseudodifferential operators 
	\begin{equation}\label{operatori matriciali sezione pseudo diff}
	\bA = \begin{pmatrix}
	A_1  & A_2 \\
	A_3 & A_4
	\end{pmatrix}, \quad A_i \in \Ops^m, \quad i =1, \ldots , 4
	\end{equation}
	we say that  $\bA \in \Ops^m$.
\end{defn}

When the symbol $ a (x) $ is independent of $ \xi $, the operator $ \Op (a) $ is 
the multiplication operator by the function $ a(x)$, i.e.\  $ \Op (a) : u (x) \mapsto a ( x) u(x )$. 
In such a case we also denote $ \Op (a)  = a (x) $.

We shall use 
the following notation, used also in \cite{AB,BM,BBHM}.
For any $m \in \R \setminus \{ 0\}$, we set
$$
|D|^m := \Op \big( \chi(\xi) |\xi|^m \big)\,,
$$
where $\chi$ is an  even, positive $\cC^\infty$ cut-off satisfying  \eqref{cutoff}.
We also identify the Hilbert transform $\cH$, acting on the $2 \pi$-periodic functions, defined by
\begin{equation}\label{Hilbert-transf}
\cH ( e^{\im j x} ) := - \im \, \sign(j) e^{\im jx} \,  \quad  \forall j \neq 0 \, ,  \quad \cH (1) := 0\,, 
\end{equation}
with the Fourier multiplier $\Op (- \im\, \sign(\xi) \chi(\xi) )$. Similarly we regard 
the  operator 
\begin{equation}\label{pax-1}
\pa_x^{-1}\left[ e^{\im jx}\right] := -\,\im \,j^{-1} \,e^{\im jx} \, \quad \forall\, j\neq 0\,,  \quad \pa_x^{-1}[1] := 0\,,
\end{equation}
as the  Fourier  multiplier $\pa_x^{-1} =  \Op \left( - \im\,\chi(\xi) \xi^{-1} \right)$ and the projector $\pi_0 $, defined on the $ 2 \pi $-periodic functions  as
\begin{equation}\label{defpi0}
\pi_0 u := \frac{1}{2\pi} \int_\T u(x)\, d x\, , 
\end{equation}
with the Fourier multiplier $ {\rm Op}\big( 1 - \chi(\xi) \big)$.
Finally we define, for  any $m \in \R \setminus \{ 0 \}$, 
$$
\langle D \rangle^m := \pi_0 + |D|^m
:= \Op \big( ( 1 - \chi(\xi)) + \chi(\xi) |\xi|^m \big) \, .
$$
Along the paper we consider families of pseudodifferential operators 
with a symbol $  a(\lambda;\vf,x,\xi)  $
which is $k_0$-times differentiable with respect to a parameter
$ \lambda:=(\omega,\gamma) $ in an open subset $ \Lambda_0 \subset \R^\nu\times [\gamma_1,\gamma_2] $. 
Note that $\partial_\lambda^k A = \Op\left( \partial_\lambda^k a \right) $ for any $k\in \N_0^{\nu+1}$. 

We  recall the  pseudodifferential norm introduced in Definition 2.11 in \cite{BM}. 

\begin{defn}
	{\bf (Weighted $\Psi$DO norm)}
	Let $ A(\lambda) := a(\lambda; \vf, x, D) \in \Ops^m $ 
	be a family of pseudodifferential operators with symbol $ a(\lambda; \vf, x, \xi) \in S^m $, $ m \in \R $, which are 
	$k_0$-times differentiable with respect to $ \lambda \in \Lambda_0 \subset \R^{\nu + 1} $. 
	For $ \upsilon \in (0,1) $, $ \alpha \in \N_0 $, $ s \geq 0 $, we define  
	$$
	\norm{A}_{m, s, \alpha}^{k_0, \upsilon} := \sum_{|k| \leq k_0} \upsilon^{|k|} 
	\sup_{\lambda \in {\Lambda}_0}\norm{\partial_\lambda^k A(\lambda)}_{m, s, \alpha}  
	$$
	where 
	$
	\norm{A(\lambda)}_{m, s, \alpha} := 
	\max_{0 \leq \beta  \leq \alpha} \, \sup_{\xi \in \R} \|  \partial_\xi^\beta 
	a(\lambda, \cdot, \cdot, \xi )  \|_{s} \  \langle \xi \rangle^{-m + \beta} $. 
	For a matrix of pseudodifferential operators $\bA \in \Ops^m$ as in \eqref{operatori matriciali sezione pseudo diff}, we define 
	$
	\norm{\bA}_{m, s, \alpha}^{k_0, \upsilon} 
	:= \max_{i = 1, \ldots, 4} \norm{A_i}_{m, s, \alpha}^{k_0, \upsilon}\,. 
	$
\end{defn}

Given a function $a(\lambda; \vf, x) \in \cC^\infty$ which is $k_0$-times differentiable with respect to $\lambda$, the weighted norm of the corresponding multiplication operator is
$ \normk{\Op(a)}{0,s,\alpha} = \normk{a}{s} $, $ \forall \alpha \in \N_0 $.

\paragraph{Composition of pseudodifferential operators.}
If $ {\rm Op}(a) $, ${\rm Op}(b) 
$ are pseudodifferential operators with symbols $a\in S^m$, $b\in S^{m'}$, 
$m,m'\in\R$,  
then the composition operator 
$ {\rm Op}(a) {\rm Op}(b) $ 
is a pseudodifferential operator  $ {\rm Op}(a\# b) $ with symbol $a\# b\in S^{m+m'}$. 
It 
admits the  asymptotic expansion: for any $N\geq 1$
\begin{align}\label{compo_symb}
(a\# b)(\lambda;\vf,x,\xi) & = \sum_{\beta= 0}^{N-1} \frac{1}{\im^\beta \beta!} \pa_\xi^\beta a(\lambda;\vf,x,\xi) \pa_x^\beta b(\lambda;\vf,x,\xi) + (r_N(a,b))(\lambda;\vf,x,\xi) 
\end{align}
where  $ r_N(a,b) \in S^{m+m'-N} $. 
The following result is proved in Lemma 2.13 in \cite{BM}.

\begin{lem}{\bf (Composition)} \label{pseudo_compo}
	Let $ A = a(\lambda; \vf, x, D) $, $ B = b(\lambda; \vf, x, D) $ be pseudodifferential operators
	with symbols $ a (\lambda;\vf, x, \xi) \in S^m $, $ b (\lambda; \vf, x, \xi ) \in S^{m'} $, $ m , m' \in \R $. Then $ A \circ B \in \Ops^{m + m'} $
	satisfies,   for any $ \alpha \in \N_0 $, $ s \geq s_0 $, 
	\begin{equation}\label{eq:est_tame_comp}
	\begin{split}
	\norm{A B}_{m + m', s, \alpha}^{k_0, \upsilon} &
	\lesssim_{m,  \alpha, k_0} C(s) \norm{ A }_{m, s, \alpha}^{k_0, \upsilon} 
	\norm{ B}_{m', s_0 + |m|+\alpha, \alpha}^{k_0, \upsilon}  \\
	& \ \quad \qquad + C(s_0) \norm{A }_{m, s_0, \alpha}^{k_0, \upsilon}  
	\norm{ B }_{m', s + |m|+\alpha, \alpha}^{k_0, \upsilon} \, .
	\end{split}
	\end{equation}
	Moreover, for any integer $ N \geq 1  $,  
	the remainder $ R_N := {\rm Op}(r_N) $ in \eqref{compo_symb} satisfies
	\begin{equation}
	\begin{aligned}
	\norm{\Op(r_N(a,b))}_{m+ m'- N, s, \alpha}^{k_0, \upsilon}
	&\lesssim_{m, N,  \alpha, k_0} 
	C(s) \norm{ A}_{m, s, N + \alpha}^{k_0, \upsilon} 
	\norm{ B }_{m', s_0 + \abs{m} +  2N  + \alpha,N+\alpha }^{k_0, \upsilon}  
	\\
	& \ \qquad \qquad  +  C(s_0)\norm{A}_{m, s_0   , N + \alpha}^{k_0, \upsilon}
	\norm{ B}_{m', s +|m| + 2 N  + \alpha, N+ \alpha }^{k_0, \upsilon}.
	\label{eq:rem_comp_tame} 
	\end{aligned}
	\end{equation}
	Both  \eqref{eq:est_tame_comp}-\eqref{eq:rem_comp_tame} hold  
	with the constant $ C(s_0) $ 
	interchanged with $ C(s) $. 
%	Analogous estimates hold if $\bA$ and $\bB$ are matrix operators of the form \eqref{operatori matriciali sezione pseudo diff}. 
\end{lem}

The commutator between two pseudodifferential operators $ \Op(a)\in \Ops^m$ and $\Op(b)\in\Ops^{m'}$ is a pseudodifferential operator
in $ \Ops^{m+m'-1}$ with symbol $a\star b\in S^{m+m'-1}$, 
namely $ \left[ \Op(a), \Op(b)\right] = \Op\left( a\star b \right)$, 
that admits, by \eqref{compo_symb},  the expansion
\begin{equation}\label{eq:moyal_exp}
\begin{aligned}
& a\star b= -\im\left\{ a,b \right\} + \wt{r}_2(a,b) \,, \quad \wt{r}_2(a,b):=r_2(a,b)-r_2(b,a)\in S^{m+m'-2} \,, \\
& 
{\rm where} \quad  \{ a,b \}:= \pa_\xi a \pa_x b - \pa_x a \pa_\xi b \, , 
\end{aligned}
\end{equation}
is the Poisson bracket between $a(x,\xi)$ and $b(x,\xi)$. 
As a corollary of  Lemma \ref{pseudo_compo} we have: 
\begin{lem}{\bf (Commutator)} \label{pseudo_commu}
	Let $A = {\rm Op}(a) $ and $B = {\rm Op} (b) $ be pseudodifferential operators with symbols $a(\lambda;\vf,x,\xi)\in S^{m}$, $b(\lambda;\vf,x,\xi)\in S^{m'}$, $m,m'\in \R$. Then the commutator $[A,B]:=AB-BA\in \Ops^{m+m'-1}$ satisfies
	\begin{equation}\label{eq:comm_tame_AB}
	\begin{aligned}
	\norm{[A,B]}_{m+m'-1,s,\alpha}^{k_0,\upsilon} & \lesssim_{m, m', \alpha, k_0} C(s)\norm{A }_{m,s+|m'|+\alpha+2,\alpha+1}^{k_0,\upsilon}\norm{ B }_{m',s_0+|m|+\alpha+2,\alpha+1}^{k_0,\upsilon}\\
	& \qquad \quad \ + C(s_0)\norm{A }_{m,s_0+|m'|+\alpha+2,\alpha+1}^{k_0,\upsilon}\norm{ B }_{m',s+|m|+\alpha+2,\alpha+1}^{k_0,\upsilon} \,.
	\end{aligned}
	\end{equation}
\end{lem}

Finally we consider the exponential of a pseudodifferential operator of order $0$.
The following lemma follows as in Lemma 2.12 of \cite{BKM} (or Lemma 2.17 in \cite{BM}).

\begin{lem} {\bf (Exponential map)}
	\label{Neumann pseudo diff}
	If $ A := {\rm Op}(a(\lambda; \vf, x,  \xi ))$ 
	is in $ OPS^{0} $,   
	then $e^A$ is in $ OPS^{0} $ and 
	for any $s \geq s_0$, $\alpha \in \N_0 $,
	there is a constant 
	$C(s, \alpha) > 0$ so that 
	$$
	\normk{e^A - {\rm Id} }{0, s, \alpha} \leq  \normk{A}{0, s + \alpha, \alpha} \, {\rm exp} \big( C(s, \alpha) 
	\normk{A}{0, s_0 + \alpha, \alpha}\big)\, .
	$$
	The same holds for a matrix $\bA$ of the form \eqref{operatori matriciali sezione pseudo diff}  in $\Ops^0$.
\end{lem}

\paragraph{Egorov Theorem.}

Consider the family of $ \vphi $-dependent diffeomorphisms of $ \T_x $ defined by 
%\begin{equation}
%\label{diffeo_inverso}
$ y= x + \beta(\vf, x) $, with inverse
$ x= y + \breve\beta(\vf, y) $,  
where $\beta(\vf, x)$ is a small smooth function, 
and the induced operators 
$
(\cB u)(\vf, x) :=  u(\vf, x + \beta(\vf, x)) $ and 
$ (\cB^{-1}u)(\vf, y) :=  u(\vf, y + \breve\beta(\vf, y))  $. 
\begin{lem}{\bf (Composition)}\label{product+diffeo}
	Let $\normk{\beta}{2s_0+k_0+2}\leq \delta(s_0,k_0)$ small enough. 
	Then the composition operator
	$ 	\cB $
	satisfies the tame estimates, for any $s\geq s_0$,
\begin{equation}\label{est:compo-loss}
	\normk{\cB u}{s} \lesssim_{s,k_0} \normk{u}{s+k_0} + \normk{\beta}{s} \normk{u}{s_0+k_0+1} \, , 
\end{equation}
	and the function $\breve{\beta}$ defined  by the inverse diffeomorphism
	satisfies 
	$	\normk{\breve{\beta}}{s} \lesssim_{s,k_0} \normk{\beta}{s+k_0} $. 
\end{lem}

The following result is a  small variation of Proposition 2.28 of \cite{BKM}.
\begin{prop}\label{egorov} {\bf (Egorov)}
	Let $N \in \N$, $\tq_0 \in \N
	_0 $,  $S > s_0$ and assume that 
	$\pa_\lambda^k \beta(\lambda; \cdot, \cdot)$ 
	are $ \cC^\infty$ for all $|k| \leq k_0$. There exist constants  $\sigma_N, \sigma_N(\tq_0) >0$,  $\delta = \delta(S, N, \tq_0, k_0) \in (0,1)$ such that, if
	$ \normk{\beta}{s_0 + \sigma_N(\tq_0)} \leq \delta $, 
	then the conjugated operator
	$  \cB^{-1} \circ \pa_x^{m}\circ \cB $, $ m \in \Z $,   
	is a pseudodifferential operator of order $ m $ with an expansion of the form
	\begin{equation*}
	\cB^{-1} \circ \pa_x^{m} \circ \cB = \sum_{i=0}^N p_{m - i}(\lambda; \vf, y) \pa_y^{m - i}  + \cR_N(\vf)
	\end{equation*}
	with the following properties:
	\\[1mm]
	1. The principal symbol  
	$
	p_{m}(\lambda; \vf, y) = \Big( [1+\beta_x(\lambda;\vf,x)]^{m} \Big)\vert_{x=y+
		\breve \beta(\lambda;\vf,y)} $. 
		For any $s \geq s_0$ and $i=1, \ldots, N$,
	\begin{equation}\label{norm-pk}
	\normk{p_m - 1}{s} \, , \ 
	\normk{p_{m-i}}{s} \lesssim_{s, N} \normk{\beta}{s+\sigma_N}\, . 
	\end{equation}
	2. For any $ \tq \in \N^\nu_0 $ with $ |\tq| \leq \tq_0$, 
	$n_1, n_2 \in \N_0 $  with $ n_1 + n_2 + \tq_0 \leq N + 1 - k_0 - m  $,   the  
	operator $\langle D \rangle^{n_1}\partial_{\vphi}^\tq {\cal R}_N(\vphi) \langle D \rangle^{n_2}$ is 
	$\cD^{k_0} $-tame with a tame constant satisfying, for any $s_0 \leq s \leq S $,   
	\begin{equation}\label{stima resto Egorov teo astratto}
	{\mathfrak M}_{\langle D \rangle^{n_1}\partial_{\vphi}^\tq {\cal R}_N(\vphi) \langle D \rangle^{n_2}}(s) \lesssim_{S, N, \tq_0} 
	\| \beta\|_{s + \sigma_N(\tq_0)}^{k_0,\upsilon} 
	\,. 
	\end{equation}
	3.  Let $s_0 < s_1 $ 
	and assume that  $\| \beta_j \|_{s_1 + \sigma_N(\tq_0)} \leq \delta,$ 
	$j = 1,2$. Then 
	$ \| \Delta_{12} p_{m - i} \|_{s_1} \lesssim_{s_1, N} 
	\| \Delta_{12} \beta\|_{s_1 + \sigma_N} $, $ i = 0, \ldots, N $, 
	and, for any $ |\tq| \leq \tq_0$,  $n_1, n_2 \in \N_0 $ with $n_1 + n_2 + \tq_0 \leq N  - m$, 
	$$
	\| \langle D \rangle^{n_1}\partial_{\vphi}^\tq \Delta_{12} {\cal R}_N(\vphi) \langle D \rangle^{n_2} \|_{{\cal B}(H^{s_1})} \lesssim_{s_1, N, n_1, n_2} 
	\| \Delta_{12} \beta\|_{s_1 + \sigma_N(\tq_0)} \, . 
	$$
	Finally, if $ \beta (\vf, x ) $ is a quasi-periodic traveling wave,  
	then $ \cB $ is momentum preserving (we refer to 
	Definition \ref{def:mom.pres}), 
	as well as the conjugated operator $ \cB^{-1} \circ\pa_x^m \circ\cB $,
	and each function $ p_{m-i} $, $ i = 0, \ldots, N $, is a quasi-periodic traveling wave. 
\end{prop}

%\paragraph
\noindent
{\bf Dirichlet-Neumann operator.} 
We finally remind the following decomposition of the Dirichlet-Neumann operator proved in
\cite{BM}, in the case of infinite depth, and in \cite{BBHM},  for finite depth.

\begin{lem}\label{DN_pseudo_est}
	{\bf (Dirichlet-Neumann)}
	Assume that $\pa_\lambda^k \eta(\lambda, \cdot, \cdot) $ is $\cC^\infty (\T^\nu \times \T_x)$ for all $|k| \leq k_0$.
	There exists $ \delta(s_0, k_0) >0$ such that, if
	$ \normk{\eta}{2s_0 +2k_0  +1} \leq \delta (s_0, k_0) $, 
	then the Dirichlet-Neumann operator $ G(\eta ) = G(\eta, \tth)$ may be written as 
\begin{equation}\label{DNGeta}
	G(\eta, \tth) = G(0, \tth) + \cR_G(\eta)
\end{equation}
	where $ \cR_G(\eta) := \cR_G(\eta, \tth)  \in \Ops^{-\infty}$  
	satisfies, for all $m, s, \alpha \in \N_0$, the estimate
	\begin{align}
	\label{est:RG}
	\normk{\cR_G(\eta)}{-m, s, \alpha} 
	\leq  C(s, m, \alpha, k_0) \normk{\eta}{s+s_0 +2k_0 +m+\alpha + 3} \, . 
	\end{align}
\end{lem}

\subsection{$\cD^{k_0}$-tame and $(-\tfrac12)$-modulo-tame operators}\label{subsec:semiphi}

Tame and modulo tame operators were introduced in \cite{BM}. 
Let $ A := A(\lambda) $ be a linear operator as in \eqref{linear_A}, 
$ k_0 $-times differentiable with respect to
the parameter $ \lambda $ in an open set $ \Lambda_0 \subset \R^{\nu+1}$. 

\begin{defn}{\bf ($\cD^{k_0}$-$\sigma$-tame)} \label{Dk_0sigma}
	Let $\sigma\geq 0$. A linear operator $A:=A(\lambda)$   is $\cD^{k_0}$-$\sigma$-tame if there exists a non-decreasing function $[s_0,S]\rightarrow[0,+\infty)$, $s\mapsto\fM_A(s)$, with possibly $S=+\infty$, 
	such that, for all $s_0\leq s\leq S$ and $u\in H^{s+\sigma} $, 
	$$
	\sup_{\abs k \leq k_0}\sup_{\lambda\in \Lambda_0} \upsilon^{\abs k} \norm{(\pa_\lambda^k A(\lambda))u }_s \leq \fM_A(s_0) \norm u_{s+\sigma} + \fM_A(s)\norm u_{s_0+\sigma} \,.
$$
	We say that $\fM_A(s)$ is a \emph{tame constant} of the operator $A$. The constant $\fM_A(s)=\fM_A(k_0,\sigma,s)$ may also depend on $k_0,\sigma$ but
	we shall often omit to write them.
	When the "loss of derivatives" $\sigma$ is zero, we simply write $\cD^{k_0}$-tame instead of $\cD^{k_0}$-$0$-tame. 
	For a matrix operator as in \eqref{C_transformed}, we denote the tame constant $\fM_{\bR}(s):=\max\left\{ \fM_{\cR_1}(s),\fM_{\cR_2}(s) \right\}  $.
\end{defn}

%Note that the tame constants $\fM_A(s)$ are not uniquely determined. 
%An immediate consequence of \eqref{tame_constants} is that $\norm A_{\cL\left(H^{s_0+\sigma},H^{s_0}\right)}\leq 2 \fM_{A}(s_0)$. 
%Also note that, representing the operator $A$ by its matrix elements 
%$ (A_j^{j'}(\ell-\ell') )_{\ell,\ell'\in\Z^\nu,j,j'\in\Z}$ as in \eqref{A_expans}, we have for all $\abs k \leq k_0$, $j'\in\Z$, $\ell'\in\Z^\nu$,
%$$
%\upsilon^{2\abs k} \sum_{\ell,j}\braket{\ell,j}^{2s} \big| \pa_\lambda^k A_j^{j'}(\ell-\ell') \big|^2 \leq 2 \big( \fM_A(s_0)\big)^2 
%\braket{\ell',j'}^{2(s+\sigma)} + 2 (\fM_A(s))^2\braket{\ell',j'}^{2(s_0+\sigma)} \,.
%$$

The class of $\cD^{k_0}$-$\sigma$-tame operators is closed under composition, see
Lemma 2.20 in \cite{BM}. 
\begin{lem}{\bf (Composition)}\label{tame_compo}
	Let $A,B$ be respectively $\cD^{k_0}$-$\sigma_A$-tame and $\cD^{k_0}$-$\sigma_B$-tame operators with tame constants respectively $\fM_A(s)$ and $\fM_B(s)$. Then the composed operator $A\circ B$ is $\cD^{k_0}$-$(\sigma_A+\sigma_B)$-tame with a tame constant
	$$
	\fM_{AB}(s) \leq C(k_0) \left( \fM_A(s) \fM_B(s_0+\sigma_A) + \fM_A(s_0)\fM_B(s+\sigma_A) \right)\,.
	$$
\end{lem}
It is proved in Lemma 2.22 in \cite{BM} that 
the action of a $\cD^{k_0}$-$\sigma$-tame operator $A(\lambda)$ on a 
Sobolev function $  u = u(\lambda)\in H^{s+\sigma}$
is bounded by 
$
\normk{ Au }{s} \lesssim_{k_0} \fM_A(s_0) \normk{u}{s+\sigma} + \fM_A(s)\normk{u}{s_0+\sigma} $.

Pseudodifferential operators are tame operators. We  use in particular the following lemma
which is Lemma 2.21 in \cite{BM}.
 
\begin{lem}\label{tame_pesudodiff} 
	Let $A=a(\lambda;\vf,x,D)\in\Ops^0$ be a family of pseudodifferential operators 
	satisfying $\normk{A}{0,s,0}<\infty$ for $s\geq s_0$. Then 
	$A$ is $\cD^{k_0}$-tame with a tame constant 
	satisfying $ 	\fM_A(s) \leq C(s) \normk{A}{0,s,0}  $, 
	for any $ s\geq s_0$.
\end{lem}

In view of the KAM reducibility scheme of Section \ref{sec:KAM} we also consider the notion of $\cD^{k_0}$-$(-\frac12)$-modulo-tame operator. 
We first recall that, given a linear operator $A$ acting as  in \eqref{A_expans},  
the majorant operator $| A |$ is defined to have the matrix elements  
$ (|A_j^{j'}(\ell-\ell') |)_{\ell,\ell'\in\Z^\nu, j,j'\in\Z} $.

\begin{defn}{\bf ($\cD^{k_0}$-$(-\frac12)$-modulo-tame)} \label{Dk0-modulo-12}
	A linear operator $A=A(\lambda)$ 
	is $\cD^{k_0}$-$(-\frac12)$-\emph{modulo-tame} if 
		there exists a non-decreasing function $[s_0,S]\rightarrow[0,+\infty]$, $s\mapsto \fM_{\langle D\rangle^\frac14 A \langle D\rangle^\frac14}^\sharp(s)$, such that for all $k\in\N_0^{\nu+1}$, $\abs k\leq k_0$, the majorant operator $\langle D\rangle^\frac14 |\pa_\lambda^k A| \langle D\rangle^\frac14 $ satisfies,  
		for all $s_0\leq s\leq S$ and $u\in H^s$,
		\begin{equation}\label{def:mod-tame}
			\sup_{ |k| \leq k_0}\sup_{\lambda \in {\Lambda}_0}\upsilon^{|k|} \| \langle D\rangle^\frac14  |\pa_\lambda^k A| \langle D\rangle^\frac14  u \|_s \leq \fM_{\langle D\rangle^\frac14  A \langle D\rangle^\frac14 }^\sharp(s_0) \norm u_s + \fM_{\langle D\rangle^\frac14  A \langle D\rangle^\frac14 }^\sharp(s) \norm u_{s_0} \,.
		\end{equation}
	For a matrix as in \eqref{C_transformed}, we denote $\fM_{\langle D\rangle^\frac14 \bR\langle D\rangle^\frac14 }^\sharp(s):= \max\big\{ \fM_{\langle D\rangle^\frac14 \cR_1\langle D\rangle^\frac14 }^\sharp(s),\fM_{\langle D\rangle^\frac14 \cR_2\langle D\rangle^\frac14 }^\sharp(s) \big\}$.
\end{defn}

%If $A$, $B$ are $\cD^{k_0}$-$(-\tfrac12)$-modulo-tame operators with 
%$ | A_j^{j'}(\ell)  | \leq | B_j^{j'}(\ell) | $, then $\fM_{\langle D\rangle^\frac14 A \langle D\rangle^\frac14 }^\sharp(s)\leq  \fM_{\langle D\rangle^\frac14  B \langle D\rangle^\frac14 }^\sharp(s)$. 
Given a linear operator $A$ acting as  in \eqref{A_expans}, we define the operator
$\braket{\pa_{\vf}}^\tb A$, $ \tb\in\R$, whose matrix elements are $\braket{\ell-\ell'}^\tb A_j^{j'}(\ell-\ell')$.
From Lemma A.5-(iv) in \cite{FGP}, we deduce the following lemma.
\begin{lem}{\bf (Sum and composition)} 
	\label{modulo_sumcomp-12}
	Let $A$, $B$, $\braket{\pa_{\vf}}^\tb A$, $\braket{\pa_{\vf}}^\tb B$ 
	be $\cD^{k_0}$-$(-\frac12)$-modulo-tame operators. Then 
	$A+B$, $A\circ B$ and $\braket{\pa_{\vf}}^\tb(AB)$  are $\cD^{k_0}$-$(-\frac12)$-modulo-tame with
	\begin{align*}
		& \fM_{\langle D\rangle^\frac14(A+B)\langle D\rangle^\frac14}^\sharp(s)\leq \fM_{\langle D\rangle^\frac14 A \langle D\rangle^\frac14}^\sharp (s)+ \fM_{\langle D\rangle^\frac14 B \langle D\rangle^\frac14}^\sharp(s) \\
		&
		\fM_{\langle D\rangle^\frac14AB\langle D\rangle^\frac14}^\sharp(s) \lesssim_{k_0}\Big(  \fM_{\langle D\rangle^\frac14 A \langle D\rangle^\frac14}^\sharp(s)\fM_{\langle D\rangle^\frac14 B \langle D\rangle^\frac14}^\sharp(s_0) + \fM_{\langle D\rangle^\frac14 A \langle D\rangle^\frac14}^\sharp(s_0)\fM_{\langle D\rangle^\frac14 B \langle D\rangle^\frac14}^ \sharp(s) \Big) \\
		&	\fM_{\langle D\rangle^\frac14\langle\pa_{\vf}\rangle^\tb(AB)\langle D\rangle^\frac14}^\sharp(s)  \lesssim_{\tb,k_0} \\
		& \quad \quad  \Big(  \fM_{\langle D\rangle^\frac14\langle\pa_{\vf}\rangle^\tb A\langle D\rangle^\frac14}^\sharp(s) \fM_{\langle D\rangle^\frac14 B \langle D\rangle^\frac14}^\sharp(s_0) +  \fM_{\langle D\rangle^\frac14\langle\pa_{\vf}\rangle^\tb A\langle D\rangle^\frac14}^\sharp(s_0) \fM_{\langle D\rangle^\frac14 B \langle D\rangle^\frac14}^\sharp(s)  \notag\\
		&  \quad \quad 
		+ \fM_{\langle D\rangle^\frac14 A \langle D\rangle^\frac14}^\sharp(s) \fM_{\langle D\rangle^\frac14\langle\pa_{\vf}\rangle^\tb B\langle D\rangle^\frac14}^\sharp(s_0)+\fM_{\langle D\rangle^\frac14 A \langle D\rangle^\frac14}^\sharp(s_0) \fM_{\langle D\rangle^\frac14\langle\pa_{\vf}\rangle^\tb B\langle D\rangle^\frac14}^\sharp(s) \Big)  \, . 
	\end{align*}
	%The same statement holds for  matrix operators $\bA$, $\bB$ as in \eqref{C_transformed}.
\end{lem}
%\begin{proof}
%	It follows by applying Lemma 2.25 in \cite{BM} to $\langle D\rangle^{\frac14} R \langle D\rangle^{\frac14}$, $R=A,B,\braket{\pa_{\vf}}^\tb A,\braket{\pa_{\vf}}^\tb B$, noting for instance that
%	$ \fM_{\langle D\rangle^{\frac14} A B \langle D\rangle^{\frac14}}^\sharp(s) \leq \fM_{\langle D\rangle^{\frac14} A \langle D\rangle^{\frac14}\circ \langle D\rangle^{\frac14}  B \langle D\rangle^{\frac14} }^\sharp(s)\,.$
%\end{proof}

From the proof of Lemma 2.22 in \cite{BKM},
%	applied to $\langle D\rangle^{\frac14} A \langle D\rangle^{\frac14}$ and $\langle D\rangle^{\frac14} \langle\pa_\vf\rangle^\tb A \langle D\rangle^{\frac14}$, 
%noting that $\fM_{\langle D\rangle^{\frac14}A^n\langle D\rangle^{\frac14}}^\sharp(s) \leq \fM_{(\langle D\rangle^{\frac14} A \langle D\rangle^{\frac14})^n}^\sharp(s)$ for any $n\in\N$, 
we deduce the following lemma.

\begin{lem}{\bf (Exponential)} \label{modulo_expo-12}
	Let $A$, $\braket{\pa_\vf}^\tb A$ be $\cD^{k_0}$-$(-\frac12)$-modulo-tame and assume 
	 $\fM_{\langle D\rangle^{\frac14} A \langle D\rangle^\frac14}^\sharp(s_0) \leq 1  $. Then 
	 $e^{\pm A}-{\rm Id}$ and $\braket{\pa_\vf}^\tb( e^{\pm A}- {\rm Id})$ are $\cD^{k_0}$-$(-\frac12)$-modulo-tame  with 
	\begin{equation*}
		\begin{aligned}
			&\fM_{\langle D\rangle^{\frac14}(e^{\pm A} -{\rm Id})\langle D\rangle^{\frac14}}^\sharp (s) \lesssim_{k_0} \fM_{\langle D\rangle^{\frac14} A \langle D\rangle^{\frac14}}^\sharp (s)\,, \\
			&
			\fM_{\langle D\rangle^{\frac14}\langle\pa_\vf\rangle^\tb (e^{\pm A}-{\rm Id})\langle D\rangle^{\frac14}}^\sharp(s) \lesssim_{k_0,\tb} \fM_{\langle D\rangle^{\frac14}\langle\pa_\vf\rangle^\tb A \langle D\rangle^{\frac14}}^\sharp (s)  + \fM_{\langle D\rangle^{\frac14} A \langle D\rangle^{\frac14}}^\sharp(s)\fM_{\langle D\rangle^{\frac14}\langle\pa_\vf\rangle^\tb A \langle D\rangle^{\frac14}}^\sharp (s_0) \,.
		\end{aligned}
	\end{equation*}
\end{lem}

Given a linear operator $A$ acting as  in \eqref{A_expans}, we define the 
\emph{smoothed operator}
$\Pi_N A$, $N\in\N $ whose matrix elements are
\begin{equation}\label{PiNA}
(\Pi_N A)_j^{j'}(\ell-\ell') := \begin{cases}
A_j^{j'}(\ell-\ell') & \text{if } \braket{\ell-\ell'} \leq N \\
0 & \text{otherwise} \, . 
\end{cases}
\end{equation}
We also denote $\Pi_N^\perp:= {\rm Id}-\Pi_N$. 
Arguing as in Lemma 2.27 in \cite{BM}, we have that 	
\begin{equation}\label{modulo_smooth}
\fM_{\langle D\rangle^{\frac14}\Pi_N^\perp A \langle D\rangle^{\frac14}}^\sharp(s) \leq N^{-\tb} \fM_{\langle D\rangle^{\frac14}\langle\pa_{\vf}\rangle^\tb A \langle D\rangle^{\frac14}}^\sharp(s)\,, \ \ \fM_{\langle D\rangle^{\frac14}\Pi_N^\perp A \langle D\rangle^{\frac14}}^\sharp(s)\leq \fM_{\langle D\rangle^{\frac14}A \langle D\rangle^{\frac14}}^\sharp(s) \,.
\end{equation}

In the next lemma we provide a 
sufficient condition 
%Under suitable conditions, $\cD^{k_0}$-tame operators  can be also 
for an operator to be $\cD^{k_0}$-$(-\frac12)$-modulo-tame.

\begin{lem}\label{mod-to-tame}
	Let the operators $\langle D\rangle^{\frac14} R \langle D\rangle^{\frac14}$, $\langle D\rangle^{\frac14} [R,\pa_x] \langle D\rangle^{\frac14}$, $\langle D\rangle^{\frac14} \pa_{\vf_m}^{s_0} R \langle D\rangle^{\frac14}$, $\langle D\rangle^{\frac14} [\pa_{\vf_m}^{s_0} R,\pa_x] \langle D\rangle^{\frac14}$ and $\langle D\rangle^{\frac14} \pa_{\vf_m}^{s_0+\tb} R \langle D\rangle^{\frac14}$, $\langle D\rangle^{\frac14} [\pa_{\vf_m}^{s_0+\tb}R,\pa_x] \langle D\rangle^{\frac14}$, with $m=1,..,\nu$, be $\cD^{k_0}$-tame. Set
	\begin{align}
		&\wt{\mathbb{M}}(s) :=\max \big\{ \fM_{\langle D\rangle^{\frac14} R \langle D\rangle^{\frac14}}(s), \fM_{\langle D\rangle^{\frac14} [ R,\pa_x] \langle D\rangle^{\frac14}}(s),
		\\
		& \quad \quad \quad \quad \quad  \quad \fM_{\langle D\rangle^{\frac14} \pa_{\vf_m}^{s_0} R, \langle D\rangle^{\frac14}}(s), \fM_{\langle D\rangle^{\frac14} [\pa_{\vf_m}^{s_0} R,\pa_x] \langle D\rangle^{\frac14}}(s) \,:\, m=1,...,\nu  \big\}\,, \notag \\
		&\wt{\mathbb{M}}(s,\tb) := \max \big\{  \fM_{\langle D\rangle^{\frac14} \pa_{\vf_m}^{s_0+\tb} R \langle D\rangle^{\frac14}}(s), \fM_{\langle D\rangle^{\frac14} [\pa_{\vf_m}^{s_0+\tb} R,\pa_x] \langle D\rangle^{\frac14}}(s) \,:\, m=1,...,\nu\big\}\,, \notag \\
		&\wt\fM(s,\tb) := \max \big\{  \wt{\mathbb{M}}(s),\wt{\mathbb{M}}(s,\tb) \big\}\,. \label{fM0-12}
	\end{align}
	Then $R$ and $\langle\pa_{\vf} \rangle^\tb R$ are $\cD^{k_0}$-$(-\tfrac12)$-modulo-tame, with
	$$ \fM_{\langle D\rangle^{\frac14} R\langle D\rangle^{\frac14}}^\sharp(s) \, , \
	\fM_{\langle D\rangle^{\frac14}\langle\pa_{\vf} \rangle^\tb R \langle D\rangle^{\frac14}}^\sharp (s) \lesssim_{s_0} \wt\fM(s,\tb)\,. $$
\end{lem}
\begin{proof}
	The matrix elements of $\langle D\rangle^\frac14[ \langle \pa_{\vf} \rangle^\tb R,\pa_x]\langle D\rangle^\frac14 $ are given, for any  $\ell, \ell'\in\Z^\nu$, $j,j'\in\Z$, by $\langle j \rangle^\frac14 \langle \ell-\ell'  \rangle^{\tb} \im(j-j') \langle j' \rangle^\frac14 R_j^{j'}(\ell-\ell') $. From Definition \ref{Dk_0sigma} with $\sigma=0$, we have, for any $|k|\leq k_0$, $\ell'\in\Z^\nu$, $j\in\Z$,
	$$ \upsilon^{2|k|} \sum_{\ell,j} \langle \ell,j \rangle^{2s}|(\pa_\lambda^k R)_j^{j'}(\ell-\ell')|^2 \leq 2 (\fM_{R}(s))^2 \langle \ell',j' \rangle^{2s_0} + 2(\fM_{R}(s_0))^2 \langle \ell',j' \rangle^{2s}\,. $$
	Using the inequality $\langle \ell-\ell' \rangle^{2(s_0+\tb)}\langle j-j' \rangle^2\lesssim_{s_0+\tb} 1 + |\ell-\ell'|^{2(s_0+\tb)}+|j-j'|^2+|\ell-\ell'|^{2(s_0+\tb)}|j-j'|^2$, 
	we therefore obtain, for any $\ell\in\Z^\nu$, $j\in\Z$, recalling \eqref{fM0-12},
	\begin{equation*}
		\begin{aligned}
			&\upsilon^{2|k|}\sum_{\ell,j} \langle \ell,j \rangle^{2s} \langle j \rangle^\frac12 \langle \ell-\ell' \rangle^{2(s_0+\tb)}\langle j-j' \rangle^2\langle j' \rangle^\frac12 |(\pa_\lambda^k R)_j^{j'}(\ell-\ell')|^2  \\
			& \quad \lesssim_{s_0+\tb} (\wt\fM(s,\tb))^2 \langle\ell',j' \rangle^{2s_0} + (\wt\fM(s_0,\tb))^2 \langle\ell',j' \rangle^{2s}\,.
		\end{aligned}
	\end{equation*}
	For any $s_0 \leq s \leq S$ and any $|k|\leq k_0$, by Cauchy-Schwartz inequality, we finally deduce
	\begin{equation*}
		\begin{aligned}
			& \| \langle D \rangle^\frac14 |\langle\pa_{\vf}\rangle^\tb\pa_\lambda^k R| \langle D\rangle^\frac14 h \|_s^2 \leq \sum_{\ell,j}\langle \ell,j \rangle^{2s} \Big(  \sum_{\ell',j'} \langle j \rangle^\frac14 \langle \ell-\ell'  \rangle^{\tb} \langle j' \rangle^\frac14 | (\pa_\lambda^k R)_j^{j'}(\ell-\ell') | |h_{\ell',j'}| \Big)^2 \\
			& \quad = \sum_{\ell,j}\langle \ell,j \rangle^{2s} \Big(  \sum_{\ell',j'} \langle j \rangle^\frac14 \langle \ell-\ell'  \rangle^{s_0+\tb}\langle j-j'\rangle \langle j' \rangle^\frac14 | (\pa_\lambda^k R)_j^{j'}(\ell-\ell') | |h_{\ell',j'}| \frac{1}{\langle \ell-\ell' \rangle^{s_0}\langle j-j' \rangle } \Big)^2\\
			& \lesssim_{s_0} \sum_{\ell,j}\langle \ell,j \rangle^{2s}\sum_{\ell',j'} \langle j \rangle^\frac12 \langle \ell-\ell'  \rangle^{2(s_0+\tb)}\langle j-j'\rangle^2 \langle j' \rangle^\frac12 | (\pa_\lambda^k R)_j^{j'}(\ell-\ell') |^2 |h_{\ell',j'}|^2\\
			& \lesssim_{s_0,\tb} \upsilon^{-2|k|}\sum_{\ell',j'} |h_{\ell',j'}|^2  \big( (\wt\fM(s,\tb))^2 \langle\ell',j' \rangle^{2s_0} + (\wt\fM(s_0,\tb))^2 \langle\ell',j' \rangle^{2s} \big)\,.
		\end{aligned}
	\end{equation*}
	This proves that $\langle\pa_\vf\rangle^\tb R$ is $\cD^{k_0}$-$(-\frac12)$-modulo-tame, with $\fM_{\langle D\rangle^{\frac14}\langle\pa_{\vf} \rangle^\tb R \langle D\rangle^{\frac14}}^\sharp (s) \lesssim_{s_0} \wt\fM(s,\tb)$. 
	% The proof for $R$ goes similarly.
\end{proof}

\subsection{Hamiltonian, Reversible and Momentum preserving operators}\label{subsec:ham_st}

Along the paper we 
exploit in a crucial way several algebraic properties of the water waves equations: 
the Hamiltonian and the reversible structure as well as the invariance under space translations. 
We characterize these properties following \cite{BFM}.

\begin{defn}{\bf (Hamiltonian and Symplectic operators)}\label{def:HS}
A matrix operator $\cR $ as in \eqref{real_matrix} is 
\begin{enumerate}
\item
{\it Hamiltonian} if 
the matrix $ J^{-1} \cR $
is self-adjoint, namely   
$ B^* = B $, $ C^*=C $, $  A^* = - D $
and $ A, B, C, D $ are real; 
\item {\it symplectic} 
 if    $ {\cal W} ( \cR  u, \cR   v) =
{\cal W} (u, v) $ for any  $ u,v \in L^2 (\T_x, \R^2) $, 
where the symplectic 2-form $ {\cal W} $ is defined in \eqref{sympl-form-st}. 
\end{enumerate}
\end{defn}

Let $\cS$ be an involution  as in \eqref{rev_invo} acting on the real variables
$ (\eta, \zeta) \in \R^2 $, 
or 
as in \eqref{rev_aa} acting on the action-angle-normal variables $ (\theta, I, w ) $, or as in
\eqref{inv-complex} acting in the $ (z, \bar z )$ complex variables introduced in \eqref{C_transform}.  
\begin{defn}{\bf (Reversible and reversibility preserving operators)}\label{rev_defn}
	A $ \vf $-dependent family of operators $\cR (\vf) $, $ \vf \in \T^\nu $, is  \emph{reversible} if $\cR(-\vf) \circ\cS = -\cS \circ \cR(\vf)$ for all $\vf\in\T^\nu$. 
	It is \emph{reversibility preserving} if $\cR(-\vf)\circ \cS = \cS \circ \cR(\vf)$ for all $\vf\in\T^\nu$. 
\end{defn}

Since in the complex coordinates $(z,\bar z) $  
the involution $\cS$ defined in \eqref{rev_invo} reads as in 
\eqref{inv-complex}, 
an operator $\bR (\vf)$  as in \eqref{C_transformed} is 
reversible, respectively anti-reversible,  if, for any $i=1,2$,
\begin{equation}\label{Ri-RAR}
\cR_{i} (- \vf) \circ \cS = - \cS \circ \cR_{i} (\vf)  \, , \quad {\rm resp.} \  \
\cR_{i} (- \vf) \circ \cS =  \cS \circ \cR_{i} (\vf)  \, , 
\end{equation}
where, with a small abuse of notation, we still denote $ (\cS u)(x) = \overline{u(-x)}$. 
Moreover, recalling that in the Fourier coordinates such involution reads as in \eqref{inv-zj}, 
we obtain the following lemma (cfr. Lemmata 3.18 and 
3.19 of \cite{BFM}).

\begin{lem}\label{rev_defn_C}
	A $ \vf $-dependent family of operators $\bR (\vf)$, $ \vf \in \T^\nu $,  as in \eqref{C_transformed} is   reversible if, for any $ i = 1, 2 $, 
		\begin{equation}\label{rev:Fourier}
		\left( \cR_{i} \right)_j^{j'}(-\vf) = - \bar{ \left( \cR_{i} \right)_{j}^{j'}(\vf) } \quad \forall\,\vf\in\T^\nu \, , \ \ i.e. \ \left( \cR_{i} \right)_j^{j'}(\ell) = - 
		\bar{ \left( \cR_{i} \right)_{j}^{j'}(\ell) } \,  \quad \forall\,\ell\in\Z^\nu \,;
		\end{equation}
it is reversibility preserving if, for any $ i = 1, 2 $, 
		$$
		\left( \cR_{i} \right)_j^{j'}(-\vf) = \bar{ \left( \cR_{i} \right)_{j}^{j'}(\vf) } \ \  \forall\,\vf\in\T^\nu   \, , \ \ i.e.   \ \left( \cR_{i} \right)_j^{j'}(\ell) = \bar{ \left( \cR_{i} \right)_{j}^{j'}(\ell) } 
		\,  \ \  \forall\,\ell\in\Z^\nu \,.
$$
	A pseudodifferential operator 
	$  \Op(a(\vf, x, \xi))$ is reversible, respectively reversibility 
	preserving, if and only if its symbol
	satisfies
	$
	a(- \vf, - x, \xi) = - \overline{a(\vf, x, \xi)} $,
	resp. $	a(- \vf, - x, \xi) =  \overline{a(\vf, x, \xi)} $. 
	\end{lem}

Note that the composition of a reversible operator with a reversibility preserving operator is reversible.
The flow 
generated by a reversibility preserving operator is reversibility preserving. 
If $ \cR (\vf) $ is reversibility preserving, then $ (\omega \cdot \pa_\vf \cR) (\vf) $
is reversible.

We shall say that a linear operator of the form $  \omega\cdot \pa_\vf + A(\vf)$ is reversible if $A(\vf)$ is reversible.  Conjugating the linear operator $
\omega\cdot\pa_\vf+A(\vf)$ by a family of invertible linear maps $\Phi(\vf)$, we get the transformed operator
\begin{equation}\label{trasf-op}
\begin{aligned}
&  \Phi^{-1}(\vf) \circ \big(  \omega\cdot \pa_\vf + A(\vf) \big) \circ  \Phi(\vf) = \omega\cdot\pa_\vf + A_+(\vf)\,, \\
&   A_+(\vf)  := \Phi^{-1}(\vf)\left( \omega\cdot \pa_\vf\Phi(\vf) \right) + \Phi^{-1}(\vf) A(\vf) \Phi(\vf)\,.
\end{aligned}
\end{equation}
The conjugation of a reversible operator with a reversibility preserving operator  is reversible. 

	A function $ u (\vf, \cdot) $ is called
	{\it reversible} if $ \cS u (\vf, \cdot ) = u(-\vf, \cdot ) $
	  % ({\rm cfr.} \eqref{rev:soluz})
	   and 
	{\it antireversible} if 	$ - \cS u (\vf, \cdot ) = u(-\vf, \cdot) $.
	The same definition holds in the action-angle-normal variables $ (\theta, I, w ) $
	with the involution  $ \vec \cS $ defined in \eqref{rev_aa} and in the $ (z, \bar z )$ complex variables with the involution in \eqref{inv-complex}.

A reversibility preserving operator maps reversible, respectively anti-reversible, functions into 
reversible, respectively anti-reversible, functions, see Lemma 3.22 in \cite{BFM}. 

We also remark that, 
if  $ X $ is a reversible vector field, according to \eqref{revNL},
	and $  u(\vf, x)  $ is a  reversible quasi-periodic function, 
	then the linearized operator  $ \di_u X( u(\vf, \cdot) ) $  is reversible, according to
	Definition \ref{rev_defn} (see e.g. Lemma 3.22 in \cite{BFM}). 

\smallskip

Finally we recall  that the projections $\Pi^\intercal_{\S^+, \Sigma}$, $\Pi^\angle_{\S^+, \Sigma}$ 
of Section \ref{sec:decomp} 
	are reversibility  preserving.

\begin{lem}\label{proj_rev} {\bf (Lemma 3.23 in \cite{BFM})}
	The projections $\Pi^\intercal_{\S^+, \Sigma}$, $\Pi^\angle_{\S^+, \Sigma}$ 
	defined in Section \ref{sec:decomp} 
	commute with the involution 
	$ \cS $ defined in \eqref{rev_invo}, i.e. 
	are reversibility  preserving.
	The orthogonal projectors $\Pi_{\S}$ and $\Pi_{\S_0}^\bot $ 
	commute with the involution in \eqref{inv-complex}, i.e.  
	are reversibility  preserving. 
\end{lem}
Next we define momentum preserving operators.
\begin{defn}	\label{def:mom.pres}
	{\bf (Momentum preserving operators)}	
	A  $ \vf $-dependent family of linear operators 
	$A(\vf) $, $ \vf \in \T^\nu $,  is  {\em momentum preserving} if
	$$
	A(\vf - \ora \jmath \vs )  \circ \tau_\vs = \tau_\vs \circ A(\vf ) \,  , \quad 
	\forall \,\vf \in \T^\nu \, , \ \vs \in \R \,  ,
$$
	where the translation operator $\tau_\vs$ is defined in \eqref{trans}.
	A linear matrix operator $\bA(\vf ) $ of the form \eqref{real_matrix} or \eqref{C_transformed} is  
	{\em momentum preserving} if each of its components is momentum preserving.
\end{defn}

If $ X $ is a vector field translation invariant, i.e. \eqref{eq:mom_pres} holds, and
	$ u   $ is a quasi-periodic traveling wave, 
	then the linearized operator  $ \di_u X( u(\vf, \cdot) ) $  is momentum preserving. 

Momentum preserving operators are closed under several operations
 (cfr. Lemma 3.25 in \cite{BFM}):
\begin{lem}\label{lem:mom_prop} 
	Let $A(\vf), B(\vf)$ be  momentum preserving operators. Then the composition
$A (\vf) \circ B (\vf) $  and
the adjoint $ (A(\vf))^*$ are momentum preserving.
 If $A(\vf)$ is invertible, then $A(\vf)^{-1}$ is momentum preserving.
 Assume that 
		$
		\partial_{t} \Phi^t (\vf) = A (\vf) \Phi^t (\vf) $,
$\Phi^0 (\vf) = {\rm Id} $, 
		has a unique  propagator $\Phi^t (\vf) $, $ t\in[0,1] $. 
		Then $\Phi^t ( \vf ) $ is  momentum preserving.
\end{lem}

We shall say that a linear operator of the form $  \omega\cdot \pa_\vf + A(\vf)$ is momentum preserving if $A(\vf)$ is momentum preserving.  
In particular, conjugating a momentum preserving operator 
$ \omega\cdot\pa_\vf+A(\vf) $ 
by a family of invertible linear momentum preserving maps $\Phi(\vf)$, 
we obtain  the transformed operator 
$ \omega\cdot\pa_\vf + A_+(\vf) $ in
\eqref{trasf-op} which  is momentum preserving. 

	Given a   momentum preserving linear operator $ A(\vf)$
	and a quasi-periodic traveling wave $u$, according to Definition \ref{QPTW}, then
$A(\vf) u $ is a quasi-periodic traveling wave.
\smallskip

The characterizations of 
the momentum preserving property, in Fourier space and for a  pseudo-differential 
operator, is given below (see Lemmata 3.28 and 3.29 in \cite{BFM}).

\begin{lem}
	\label{lem:mom_pseudo} 
	Let $ \vf $-dependent family of operators $ A(\vf) $, $ \vf \in \T^\nu $, 
	is momentum preserving  if and only if 
	the matrix elements of $A(\vf)$, defined by \eqref{A_expans},  fulfill
	$$
	A_j^{j'}(\ell ) \neq 0 \quad \Rightarrow  \quad \ora{\jmath}\cdot \ell + j-j' = 0 \, , \quad
	\forall\, \ell \in\Z^\nu , \ \ j,j'\in\Z \, . 
	$$
	A pseudodifferential operator 
	$ \Op(a(\vf, x, \xi))$ is momentum preserving if and only if
	 its symbol	satisfies
	$
	a(\vf - \ora{\jmath}\vs, x, \xi) =  a(\vf, x+\vs, \xi) $ for any $ \vs \in \R $. 
\end{lem}

We finally  note that the symplectic projections $ \Pi^\intercal_{\S^+, \Sigma}$, $ \Pi^\angle_{\S^+, \Sigma}$,
	 are momentum preserving.

\begin{lem}\label{lem:proj.momentum}
{\bf (Lemma 3.31 in \cite{BFM})}
	The symplectic projections $ \Pi^\intercal_{\S^+, \Sigma}$, $ \Pi^\angle_{\S^+, \Sigma}$,
	the $ L^2 $-projections  $ \Pi^{L^2}_\angle  $
	and  $\Pi_{\S}$, $\Pi_{\S_0}^\bot $ 
	defined in Section \ref{sec:decomp} 
	commute with the translation operators 
	$ \tau_\vs $ defined in \eqref{trans}, i.e. are momentum preserving.
\end{lem}

\paragraph{Quasi-periodic traveling waves in action-angle-normal coordinates.}

We now discuss  how the momentum preserving condition reads in the 
coordinates $(\theta, I, w)$ introduced in \eqref{aacoordinates}. 
Recalling \eqref{vec.tau},
if $u(\vf,x)$ is a quasi-periodic traveling wave with action-angle-normal components $(\theta(\vf), I(\vf), w(\vf, x))$, the condition $\tau_\vs u =  u(\vf - \ora{\jmath} \vs, \cdot)$ becomes 
$$
\begin{pmatrix}
\theta(\vf) - \ora{\jmath} \vs \\
I(\vf) \\
\tau_\vs w(\vf, \cdot)
\end{pmatrix}  = 
\begin{pmatrix}
\theta(\vf -  \ora{\jmath} \vs) \\
I(\vf- \ora{\jmath} \vs) \\
w(\vf - \ora{\jmath} \vs, \cdot)
\end{pmatrix} \, , \quad \forall\, \vs \in \R \, . 
$$
As we look for $\theta(\vf)$ of the form $\theta(\vf) = \vf + \Theta(\vf)$, with 
a $ (2 \pi)^\nu $-periodic function $ \Theta : \R^\nu \mapsto \R^\nu $, 
$ \vf \mapsto \Theta (\vf ) $, the traveling wave condition becomes
\begin{equation}
\label{mompres_aa1}
\begin{pmatrix}
\Theta(\vf)  \\
I(\vf) \\
\tau_\vs w(\vf, \cdot)
\end{pmatrix}  = 
\begin{pmatrix}
\Theta(\vf -  \ora{\jmath} \vs) \\
I(\vf- \ora{\jmath} \vs) \\
w(\vf - \ora{\jmath} \vs, \cdot)
\end{pmatrix} \, , \quad \forall\, \vs \in \R \, . 
\end{equation}

\begin{defn} {\bf(Traveling wave variation)}\label{trav-vari}
	We call a traveling wave variation 
	$ g(\vf) = (g_1(\vf), g_2(\vf), g_3(\vf, \cdot)) \in 
	\R^\nu \times \R^\nu \times \acca_{\S^+,\Sigma}^\angle $ a function satisfying \eqref{mompres_aa1}, i.e.
	$ g_1(\vf) = g_1(\vf - \ora{\jmath}\vs)$,  
	$ g_2(\vf) = g_2(\vf - \ora{\jmath}\vs)$, 
$	\tau_\vs g_3(\vf) = g_3(\vf - \ora{\jmath}\vs)$ for any $ \vs \in \R $, 
	or equivalently  $D \vec \tau_\vs g(\vf) = g(\vf - \ora{\jmath}\vs) $ for any  $\vs \in \R $, 
	where $D \vec \tau_\vs$ is the differential of $ \vec \tau_\vs$, namely
	$$
	D\vec \tau_\vs
	\begin{pmatrix}
	\Theta \\
	I \\
	w
	\end{pmatrix} 
	=
	\begin{pmatrix}
	\Theta  \\
	I \\
	\tau_\vs w
	\end{pmatrix} \, , \quad \forall\, \vs \in \R \, . 
	$$
\end{defn}

According to  Definition \ref{def:mom.pres}, 
a linear operator acting in $ \R^\nu \times \R^\nu \times \acca_{\S^+,\Sigma}^\angle $ is momentum preserving if 
$$
A(\vf - \ora \jmath \vs)\circ D\vec \tau_\vs  = D\vec \tau_\vs \circ A(\vf) 
\, , \quad \forall\, \vs \in \R \, . 
$$
If $A(\vf)$ is a momentum preserving linear operator acting on $\R^\nu \times \R^\nu \times \acca_{\S^+,\Sigma}^\angle$ and 
	$g \in \R^\nu \times \R^\nu \times \acca_{\S^+,\Sigma}^\angle$ is  a traveling wave variation, then   $ A(\vf) g(\vf)$ is a traveling wave variation.

\section{Transversality of linear frequencies}\label{sec:deg_kam}

In this section we extend the  KAM theory approach used in \cite{BaBM, BM,  BBHM, BFM} 
in order to deal with the linear 
frequencies $ \Omega_j(\gamma) $  defined in \eqref{def:Omegajk}, 
of the pure gravity water waves with constant vorticity. We use the vorticity as a parameter. 
In the proof of the key transversality 
Proposition \ref{prop:trans_un}, it is necessary to exploit  the momentum condition 
 for avoiding resonances. %  as in \cite{BFM}.
We shall also exploit that the  tangential sites 
$ \S:=\{ \,\bar{\jmath}_1, \ldots ,\bar{\jmath}_\nu \} \subset \Z\setminus\{0\}$  
defined  in \eqref{def.S}, have  all distinct 
modulus $ | \bar{\jmath}_a | = \bar n_a $, see assumption \eqref{Splus}.

We first introduce the following  definition of non-degenerate function. 

\begin{defn}\label{def:non-deg}
	A function $f=(f_1,\dots,f_N):[\gamma_1,\gamma_2]\rightarrow\R^N$ is 
	\emph{non-degenerate} if, for any $c\in \R^N\setminus\{0\}$, the scalar function $f\cdot c$ is not identically zero on the whole interval $[\gamma_1,\gamma_2]$.
\end{defn}

From a geometric point of view, the function $ f $ is non-degenerate if and only if the image   curve $ f([\gamma_1,\gamma_2]) \subset \R^N $ is not contained in any  hyperplane of $ \R^N $. 

\smallskip

We shall use in the sequel 
that the maps $  \gamma \mapsto \Omega_j (\gamma) $ are analytic in $ [\gamma_1, \gamma_2] $. 
For any $j\in\Z\setminus\{0\}$,  we decompose the linear frequencies $\Omega_j(\gamma) $ as   
\begin{equation}\label{Om-om}
\Omega_j(\gamma) = \omega_j(\gamma) + \frac{\gamma}{2}\frac{G_j(0)}{j}\,, \quad \omega_j (\gamma) := \sqrt{ g \,G_j(0) + \Big( \frac{\gamma}{2}\frac{G_j(0)}{j} \Big)^2 } \, , 
\end{equation}
where $G_j(0) $ is the Dirichlet-Neumann operator defined in \eqref{def:Gj0}. 
Note that $ j \mapsto 
\omega_j (\gamma) $ is even in $ j $, whereas the component   due to the vorticity 
$ j \mapsto \gamma \frac{G_j(0)}{j} $ is odd.  
%Moreover this term is, in view of \eqref{def:Gj0}, uniformly bounded in $ j $. 

\begin{lem}\label{lem:non_deg_vectors}
	{\bf (Non-degeneracy-I)}
	The following frequency vectors are non-degenerate:
	\begin{enumerate}
		\item $\ora{\Omega}(\gamma) := ( \Omega_j (\gamma) )_{j \in \S} \in\R^\nu$;
		\item $\big( \ora{\Omega}(\gamma),1 \big)\in\R^{\nu+1}$; 
		%$\big( \ora{\Omega}(\gamma),\gamma,1 \big)\in\R^{\nu+2}$;
		\item $\big( \ora{\Omega}(\gamma),\Omega_j(\gamma) \big) \in\R^{\nu+1}$, for any  $j\in\Z\setminus\left(\{ 0 \}\cup\S\cup (-\S)\right)$;
		\item $\big( \ora{\Omega}(\gamma),\Omega_j(\gamma) ,\Omega_{j'}(\gamma) \big)\in\R^{\nu+2}$, for any  $j, j'\in\Z\setminus\left( \{0\}\cup\S\cup(-\S) \right)$ and  $|j| \neq |j'|$.
	\end{enumerate}
\end{lem}

\begin{proof}
		We first compute the  jets %derivatives
		of the functions $ \gamma \mapsto 
		\Omega_j (\gamma) $ at $ \gamma = 0 $. Using that $G_j(0)=G_{|j|}(0)> 0 $, see \eqref{def:Gj0}, 
	we write \eqref{Om-om} as 
	\begin{equation}\label{not_poly}
	\Omega_j(\gamma) = \sqrt{g\,G_{|j|}(0)} \left( \sqrt{1+ \gamma^2
	\tc_j^2 }  + \gamma \sgn(j)  \tc_j  \right)\,, \quad 
	% \wt\gamma_j:=\wt\gamma_j(\gamma)
	\tc_j :=  \frac{1}{2|j|}\,  \sqrt{\frac{G_{|j|}(0)}{g}} \, , 
	\end{equation}
	for any $ j \in\Z\setminus\{0\} $.  
%		\begin{equation}\label{not_poly}
%	\Omega_j(\gamma) = \sqrt{g\,G_{|j|}(0)} \left( \sqrt{1+ \wt\gamma_j^2 }  + \sgn(j) \wt\gamma_j \right)\,, \quad \wt\gamma_j:=\wt\gamma_j(\gamma):=  \frac{\gamma}{2} \, \frac{1}{\sqrt{g\,|j|}}\,  \sqrt{\frac{G_{|j|}(0)}{|j|}} \, . 
%	\end{equation}
Each function $ \gamma \mapsto \sqrt{1+ \gamma^2
	\tc_j^2 }  + \gamma \,  \sgn(j)  \tc_j  $ is real analytic  on the whole real line $\R$, and in a neighborhood 
 of $ \gamma = 0 $, it admits the power series expansion
	\begin{equation}\label{pwe}
		\begin{aligned}
		\Omega_j(\gamma) & = \sqrt{g\,G_{|j|}(0)}\Big( 1  + \
		\sum_{n\geq 1} a_n ( \gamma^2 \tc_j^2 )^{n} + \gamma \, \sgn(j) \tc_j \Big) \\
		& = \sqrt{g\,G_{|j|}(0)} + \frac{\sgn(j)}{2} \frac{G_{|j|}(0)}{|j|} \gamma + \sum_{n\geq 1} \frac{a_{n}}{g^{n-\frac12}2^{2n}} \frac{(G_{|j|}(0))^{n+ \frac12}}{|j|^{2n}} 
		 \gamma^{2n} 
		\end{aligned}
	\end{equation}
where $a_n := \binom{1/2}{n} \neq 0$ for any $ n \geq 1$ are  binomial coefficients. 
From \eqref{pwe}, we deduce that, for any $j\in\Z\setminus\{0\}$, for any $n \geq 1 $,
	\begin{equation}\label{deriv_2n}
		\pa_\gamma^{2n} \Omega_j(0)  
		% (2n)! \frac{a_{n}}{g^{n-\frac12}2^{2n}} \frac{(G_{|j|}(0))^{n+ \frac12}}{|j|^{2n}}
		 =  b_{2n} g_j\Big(\frac{G_{|j|}(0)}{|j|^2}\Big)^n  
		 \quad {\rm with} \quad
		g_j:=\sqrt{g\,G_{|j|}(0)}>0\,, \   b_{2n}:= \frac{(2n)! \, a_{n}}{g^n 2^{2n}} \neq 0   \, . 
	\end{equation}
		We now 
	prove that, for any $N$ and integers 
	%$\Omega_{j_1}(\gamma),...,\Omega_{j_N}(\gamma)$ with 
	$1\leq |j_1|<|j_2|< \ldots < |j_N|$, the function $[\gamma_1,\gamma_2]\ni\gamma \mapsto (\Omega_{j_1}(\gamma),...,\Omega_{j_N}(\gamma))\in\R^N $ is non-degenerate according to Definition \ref{def:non-deg}.
	Suppose, by contradiction, that  $(\Omega_{j_1}(\gamma),...,\Omega_{j_N}(\gamma))$ is degenerate, i.e. there exists $c\in\R^N\setminus\{0\}$ such that
	\begin{equation}\label{nd_contra}
		c_1 \Omega_{j_1}(\gamma) + ... + c_N \Omega_{j_N}(\gamma) = 0  \quad \forall\,\gamma\in [\gamma_1, \gamma_2] \, ,
	\end{equation}
hence, by analyticity, it is identically zero for any $ \gamma \in \R$.
Differentiating \eqref{nd_contra}  we get
	\begin{equation*}
		\begin{cases}
		c_1 (\pa_\gamma^{2}\Omega_{j_1})(\gamma) + ... + c_N (\pa_\gamma^{2}\Omega_{j_N})(\gamma) = 0 \\
		.... \\
		c_1 (\pa_\gamma^{2N}\Omega_{j_1})(\gamma) + ... + c_N (\pa_\gamma^{2N}\Omega_{j_N})(\gamma) = 0 \,.
		\end{cases}
	\end{equation*}
As a consequence the $N\times N$ matrix
	\begin{equation}\label{Agam}
		\cA(\gamma) := \begin{pmatrix}
		(\pa_\gamma^{2}\Omega_{j_1})(\gamma) & \cdots & (\pa_\gamma^{2}\Omega_{j_N})(\gamma) \\
		\vdots & \ddots & \vdots \\
		(\pa_\gamma^{2N}\Omega_{j_1})(\gamma) & \cdots & (\pa_\gamma^{2N}\Omega_{j_N})(\gamma)
		\end{pmatrix}
	\end{equation}
	is singular for any $\gamma\in\R$ and 
	% the analytic function $\det\cA(\gamma)$ 	is identically zero,
	\begin{equation}\label{absurd_det}
		\det\cA(\gamma) = 0 \quad \forall\,\gamma\in\R \ . 
	\end{equation}
 In particular, at $\gamma=0 $ we have $\det\cA(0)=0$. On the other hand, by \eqref{deriv_2n} and the multi-linearity of the determinant, we compute
	\begin{equation}\label{vande}
		\det\cA(0) = 
		b_{2}...b_{2N} \prod_{a=1}^{N} g_{j_a} f(j_a)    \, 
		 \det \begin{pmatrix}
		1 & \cdots & 1 \\
		f(j_1) & \cdots & f(j_N) \\
		\vdots & \ddots & \vdots \\
		f(j_1)^{N-1} & \cdots & f(j_N)^{N-1}
		\end{pmatrix}\,, \  f(j):=  \frac{G_{|j|}(0)}{|j|^2}\,.	
	\end{equation}
	This is a Vandermonde determinant, which is therefore given by
	\begin{equation}\label{vande_det}
		\det\cA(0) = b_{2}...b_{2N} \prod_{a=1}^{N} g_{j_a} f(j_a)  
		\,  \prod_{1 \leq p < q \leq N} (f(j_q) - f(j_p)) \,.
	\end{equation}
	Note that the function $f(j)=|j|^{-2}G_{|j|}(0) > 0 $ is even in $j \in \Z\setminus\{0\}$. We claim that the function $f(j)$ is monotone for any $j >0$, from which, together with \eqref{deriv_2n} and the assumption $1\leq |j_1|<...<|j_N|$, we obtain $\det\cA(0)\neq 0$, in contradiction with \eqref{absurd_det}.
	
We now prove the monotonicity of the function $f:(0,+\infty)\to(0,+\infty) $,  
	\begin{equation*}
		f(y) := y^{-2} G_y(0) \stackrel{\eqref{def:Gj0}}{=} \begin{cases}
		y^{-1} \tanh(\tth y) & \text{ if }\tth<+\infty \\
		y^{-1} & \text{ if } \tth=+\infty \,.
		\end{cases}
	\end{equation*}
	For $\tth=+\infty$ the function $f(y)=y^{-1}$ is trivially monotone decreasing. 
	We then consider the case $\tth<+\infty$, when $f(y)= y^{-1}\tanh(\tth y)$. We compute
	\begin{equation*}
		\pa_y f(y) = y^{-2}\big( -\tanh(\tth y) + \tth y(1-\tanh^2(\tth y)) \big) = y^{-2} g(\tth y)\,,
	\end{equation*}
	where $g(x):=-\tanh(x) + x(1-\tanh^2(x))$. Then $\pa_y f(y)<0$ for any $y>0$ if and only if $g(x)<0$ for any $x>0$. We note that
	$\lim_{x\to 0^+}g(x) = 0 $, $ \lim_{x\to+\infty}g(x) = -1 $ and
	$g(x)$ is monotone decreasing for $x>0$ because
	$
	\pa_x g(x) = -2x \tanh(x)(1-\tanh^2(x)) < 0 $, $ \forall\,x>0 $.

We have proved items 1, 3, 4 of the Lemma.
	We show now item $2$, proving that the function $[\gamma_1,\gamma_2]\ni\gamma\mapsto(1,\Omega_{\bar\jmath_1}(\gamma),...,\Omega_{\bar\jmath_\nu}(\gamma))$ is non-degenerate according to Definition \ref{def:non-deg}. 
	By contradiction, suppose that there exists $c=(c_0,c_1,...,c_\nu)\in\R^{\nu+1}\setminus\{0\}$ such that
	\begin{equation}\label{NDcontra}
		c_0 + c_1 \Omega_{\bar\jmath_1}(\gamma) + ... + c_\nu \Omega_{\bar\jmath_\nu}(\gamma) = 0 \quad \forall\,\gamma\in[\gamma_1,\gamma_2] \, , %\R \,.
	\end{equation}
	and thus, by analyticity, for all $ \gamma \in \R $.  
Differentiating  \eqref{NDcontra} with respect to $\gamma$  we find that the $(\nu+1)\times(\nu+1)$-matrix
	\begin{equation}\label{Bgam}
		\cB(\gamma) := \begin{pmatrix}
		1 & \Omega_{\bar\jmath_1}(\gamma) & \cdots & \Omega_{\bar\jmath_\nu}(\gamma) \\
		0  &(\pa_\gamma^{2}\Omega_{\bar\jmath_1})(\gamma) & \cdots & (\pa_\gamma^{2}\Omega_{\bar\jmath_\nu})(\gamma) \\
		\vdots & \vdots & \ddots & \vdots \\
		0 & (\pa_\gamma^{2\nu}\Omega_{\bar\jmath_1})(\gamma) & \cdots & (\pa_\gamma^{2\nu}\Omega_{\bar\jmath_\nu})(\gamma)
		\end{pmatrix}
	\end{equation}
	is singular for all $\gamma\in\R$, and so $\det\cB(\gamma)=0$ for all $\gamma\in\R$. By the structure of the matrix \eqref{Bgam}, we get that $\det\cB(\gamma) = \det\cA(\gamma)$, where the matrix $\cA(\gamma)$ is given in \eqref{Agam}, with $N=\nu$ and $j_p=\bar\jmath_p$ for any $p=1,..,\nu$. We have already proved that $\det\cA(0)\neq 0$ and this gives the claimed contradiction.
\end{proof}

Note that in items 3 and 4 of Lemma \ref{lem:non_deg_vectors} we require
that $ j $ and $ j' $ do not belong to 
$   \{0\} \cup {\mathbb S} \cup (-{\mathbb S}) $. In order to
deal  in Proposition \ref{prop:trans_un} when $ j $ and $  j' $  belong to $-{\mathbb S}$, 
we need also the following lemma. It is actually a direct 
consequence of the proof of
Lemma \ref{lem:non_deg_vectors},  noting that
$\Omega_j (\gamma) - \omega_j (\gamma)$ is linear in $\gamma$ (cfr. \eqref{Om-om})
and its derivatives of order higher than two identically vanish.

\begin{lem}
	\label{rem:omega-nondeg}
	{\bf (Non-degeneracy-II)}
	Let $\ora{\omega}(\gamma):= \left( \omega_{\bar{\jmath}_1}(\gamma),\ldots,\omega_{\bar{\jmath}_\nu}(\gamma) \right)$. The following  vectors are non-degenerate:
	\begin{enumerate}
		\item 
		%$(\ora{\omega}(\gamma),1)$,
		 $(\ora{\omega}(\gamma),\gamma)\in\R^{\nu+1}$; 
		\item 
		%$\left( \ora{\omega}(\gamma),\omega_j(\gamma),1 \right)$, 
		$\left( \ora{\omega}(\gamma),\omega_j(\gamma),\gamma \right)\in\R^{\nu+2}$  for any  $j\in\Z\setminus\left(\{ 0 \}\cup\S\cup (-\S)\right)$.
	\end{enumerate}
\end{lem}

For later use, we provide the following asymptotic estimate of the linear frequencies.
\begin{lem}{\bf (Asymptotics)}\label{rem:exp_omegaj_gam}
	For any $j\in\Z\setminus\{0\}$ we have
	\begin{equation}\label{expa_om_gam}
		\omega_j(\gamma) = \sqrt g |j|^\frac12 + \frac{c_j(\gamma)}{\sqrt g |j|^\frac12} \,,
	\end{equation}
	where, for any $n\in\N_0$, there exists a constant $C_{n,\tth}>0$ such that
	\begin{equation}\label{bound_cj_gam}
		\sup_{j\in\Z\setminus\{0\}\atop \gamma\in[\gamma_1,\gamma_2]} |\pa_\gamma^n c_j(\gamma) | \leq C_{n,\tth} \,.
	\end{equation}
\end{lem}
\begin{proof}
	By \eqref{Om-om}, we deduce \eqref{expa_om_gam} with
	\begin{equation}\label{cj_zero}
		c_j(\gamma):= \frac{ g |j| 
		\big(\frac{G_{|j|}(0)}{|j|}-1\big) + 
		\big(\frac{\gamma}{2}\frac{G_{|j|}(0)}{|j|}\big)^2}{
		1+	\sqrt{\frac{G_{|j|}(0)}{|j|} + \frac{1}{g|j|}\Big(\frac{\gamma}{2}\frac{G_{|j|}(0)}{|j|}\Big)^2}} \,.
	\end{equation}
	The bounds \eqref{bound_cj_gam} follow 
	exploiting that $ \frac{G_{|j|}(0)}{|j|} - 1 = -  \frac{2}{1+e^{2 \tth |j|}}$, 
	% $ \inf_{j \in \Z \setminus \{0\}} G_{|j|}(0)/|j| \geq \tanh ( \tth ) > 0 $ 
	% is uniformly bounded in $j$, 
	%(both for finite and infinite depth) the quantity $G_{|j|}(0)/|j|$ is uniformly bounded in $j$, see
	by  \eqref{def:Gj0}.
\end{proof}

The next proposition is the main result of the section.  
We remind that 
$\ora{\jmath} = ( \bar{\jmath}_1, \ldots ,\bar{\jmath}_\nu ) $ 
denotes the vector in $  \Z^\nu \setminus \{0 \}$ 
of tangential sites, cfr. \eqref{def:vecj} and \eqref{def.S}.  We also recall
that $ \S_0^c = \Z \setminus (\S \cup \{0\}) $.

\begin{prop} {\bf (Transversality)} \label{prop:trans_un}
	There exist $m_0\in\N$ and $\rho_0>0$ such that, for any $\gamma\in[\gamma_1,\gamma_2]$, the following hold:
	\begin{align}
	&\max_{0\leq n \leq m_0}
	| \partial_\gamma^n \ora{\Omega}(\gamma)\cdot \ell | \geq \rho_0\braket{\ell} \,, \quad \forall\,\ell\in\Z^\nu\setminus\{0\} \,; \label{eq:0_meln}\\
	&\begin{cases}
	\max\limits_{0\leq n \leq m_0}|\partial_\gamma^n\,( \ora{\Omega}(\gamma)\cdot \ell + \Omega_j(\gamma) ) | \geq \rho_0 	\braket{\ell} \\
	\ora{\jmath}\cdot \ell + j = 0 \,, \quad \ell\in\Z^\nu\,, \ j\in\S_0^c \, ; 
	\end{cases} \label{eq:1_meln}\\
	&\begin{cases}
	\max\limits_{0\leq n \leq m_0}| \partial_\gamma^n\,( \ora{\Omega}(\gamma)\cdot \ell + \Omega_j(\gamma)-\Omega_{j'}(\gamma) ) | \geq \rho_0 	\braket{\ell} \\
	\ora{\jmath}\cdot \ell + j -j'= 0 \,, \quad \ell\in\Z^\nu\,, \ j,j'\in\S_0^c \,, \ (\ell,j,j')\neq (0,j,j) 
	\, ; 
	\end{cases}  \label{eq:2_meln-}\\
	&\begin{cases}
	\max\limits_{0\leq n \leq m_0}| \partial_\gamma^n\,( \ora{\Omega}(\gamma)\cdot \ell + \Omega_j(\gamma)+\Omega_{j'}(\gamma) ) | \geq \rho_0 	\braket{\ell} \\
	\ora{\jmath}\cdot \ell + j + j' = 0 \,, \ \ell\in\Z^\nu\,, \ j, j' \in\S_0^c  \, . 
	\end{cases} \label{eq:2_meln+}
	\end{align}
	We call $\rho_0$ the {\em amount of non-degeneracy} and $m_0$ the {\em index of non-degeneracy}.
\end{prop}
\begin{proof}
	We prove separately  \eqref{eq:0_meln}-\eqref{eq:2_meln+}. 
	We set for brevity
	$ \Gamma := [ \gamma_1, \gamma_2]$. 
	\\
	{\bf Proof of \eqref{eq:0_meln}}. By contradiction, assume that
	for any $m\in\N$ there exist $\gamma_m\in \Gamma $ and $\ell_m\in\Z^\nu\setminus\{0\}$ such that
	\begin{equation}\label{eq:0_abs_m}
	\Big| \partial_\gamma^n \ora{\Omega}(\gamma_m) \cdot \frac{\ell_m}{\braket{\ell_m}}  \Big| < \frac{1}{\braket{m}} \,, \quad \forall\,0\leq n\leq m \, .
	\end{equation}
	The sequences $(\gamma_m)_{m\in\N}\subset\Gamma$ and $(\ell_m/\braket{\ell_m})_{m\in\N}\subset \R^\nu\setminus\{0\}$ are both bounded. By compactness, up to subsequences
	$\gamma_m\to \bar\gamma\in\Gamma$ and $\ell_m/\braket{\ell_m}\rightarrow\bar c\neq 0$. 
	Therefore, for any $n\in\N_0$, passing to the limit
	for $m\rightarrow + \infty$ in  \eqref{eq:0_abs_m}, we get $\partial_\gamma^n \ora{\Omega}(\bar\gamma)\cdot \bar c = 0$.
	By the analyticity of $ \ora{\Omega}(\gamma)$, we deduce 
	that the function $ \gamma \mapsto \ora{\Omega}(\gamma)\cdot \bar c$ is identically zero on $\Gamma$, which contradicts Lemma \ref{lem:non_deg_vectors}-1, since 
	$ \bar c \neq 0 $.
	\\[1mm]
	{ \bf Proof of \eqref{eq:1_meln}}. 
	By contradiction,  assume that, for any $m\in\N$, there exist $\gamma_m\in\Gamma$, $\ell_m\in\Z^\nu$ and $j_m\in\S_0^c$, 
	such that, for any $ n\in\N_0$ with $n\leq m$,
	\begin{equation}\label{eq:1_abs_m}
	\begin{cases}
	\big| \partial_\gamma^n \big( \ora{\Omega}(\gamma)\cdot\frac{\ell_m}{\braket{\ell_m}}+\frac{1}{\braket{\ell_m}}\Omega_{j_m}(\gamma) \big)_{|\gamma= \gamma_m} \big|  < \frac{1}{\braket{m}} \\
	\ora{\jmath}\cdot \ell_m + j_m = 0 \, . 
	\end{cases} 
	\end{equation}
	Up to subsequences 
	$\gamma_m\rightarrow\bar\gamma\in\Gamma$ and $\ell_m/\braket{\ell_m}\rightarrow \bar c\in\R^\nu$. \\
	{\sc Step 1.} We consider first the case when the sequence 
	{\it $(\ell_m)_{m\in\N}\subset \Z^\nu$ is bounded}. Up to subsequences, we have definitively that 
	$\ell_m=\bar \ell\in\Z^\nu $. Moreover, since $j_m$ and $\ell_m$ satisfy 
	the momentum restriction $ \ora{\jmath}\cdot \ell_m + j_m = 0 $
	also the sequence $(j_m)_{m\in\N}$ is bounded and, up to subsequences,  definitively  $ j_m = \bar \jmath \in\S_0^c $. Therefore, for any $ n \in \N_0 $,  
	taking  $m\rightarrow \infty$ in \eqref{eq:1_abs_m} we obtain
	\begin{equation*}
	\partial_\gamma^n\big( \ora{\Omega}(\gamma)\cdot \bar \ell + \Omega_{\bar \jmath}(\gamma) \big)_{|\gamma = \bar \gamma} = 0 \ , \ \forall\, n\in\N_0  \, , \quad 
	\ora{\jmath}\cdot \bar \ell + \bar \jmath = 0  \, . 
	\end{equation*}
	By analyticity this implies 
	\begin{equation}\label{eq:1_limit_b}
	\ora{\Omega}(\gamma)\cdot \bar \ell + \Omega_{\bar \jmath}(\gamma) = 0 \, , 
	\ \forall \, \gamma \in \Gamma \, ,  \quad 
	\ora{\jmath}\cdot \bar \ell + \bar \jmath = 0  \, . 
	\end{equation}
	We distinguish two cases:
	\begin{itemize}
		\item  Let $\bar \jmath \notin -\S $. By \eqref{eq:1_limit_b} the vector  
		$\big( \ora{\Omega}(\gamma),\Omega_{\bar \jmath}(\gamma) \big) $ 
		is degenerate according to Definition \ref{def:non-deg} 
		with $c:=(\bar \ell, 1)\neq 0$. This contradicts Lemma \ref{lem:non_deg_vectors}-3.
		\item Let $\bar \jmath \in -\S$. With no loss of generality suppose $\bar \jmath = - \bar{\jmath}_1$. Then, denoting $\bar \ell = (\bar{\ell}_1, \ldots ,\bar{\ell}_\nu)$, and 
		\eqref{eq:1_limit_b} reads, for any $ \gamma \in \Gamma $,  
		$$ 	(\bar{\ell}_1+1)\omega_{\bar{\jmath}_1}(\gamma) + \sum_{a=2}^{\nu}\bar{\ell}_a\omega_{\bar{\jmath}_a}(\gamma) + \frac{\gamma}{2}\Big( (\bar{\ell}_1-1)\frac{G_{\bar{\jmath}_1}(0)}{\bar{\jmath}_1} + \sum_{a=2}^\nu \bar{\ell}_a \frac{G_{\bar{\jmath}_a}(0)}{\bar{\jmath}_a} \Big) = 0 \,. $$
		By Lemma \ref{rem:omega-nondeg}-1 the   vector   $(\ora{\omega}(\gamma),\gamma) $ 
		is non-degenerate. Therefore $ \bar \ell_1 = - 1 $ and $ \bar \ell_a = 0 $
		for any $ a = 2, \ldots, \nu  $, and $  - 2 \frac{G_{\bar{\jmath}_1}(0)}{\bar{\jmath}_1}  = 0  $,
		which is a contradiction.
	\end{itemize}
	{\sc Step 2.} We consider now the case when the sequence {\it $(\ell_m)_{m\in\N}$ is unbounded}. Up to subsequences
	$ \ell_m \rightarrow \infty$ as $m\rightarrow\infty$ and  $  \lim_{m\rightarrow\infty}\ell_m/\braket{\ell_m} =: \bar c  \neq 0 $. By  \eqref{Om-om},  
	Lemma \ref{rem:exp_omegaj_gam},  \eqref{def:Gj0}, 
and since the momentum condition implies 
$ | j_m |^\frac12 = | \vec \jmath \cdot \ell_m |^{\frac12} \leq C |\ell_m |^{\frac12}$, we deduce,
	for any $ n \in \N_0 $,
	\begin{equation*}
	\begin{split}
	\partial_\gamma^n \frac{1}{\braket{\ell_m}}\Omega_{j_m}(\gamma_m) & 
	=  \partial_\gamma^n \Big( \frac{1}{\braket{\ell_m}}\sqrt{g}\abs{j_m}^\frac12 + \frac{c_{j_m}(\gamma)}{\braket{\ell_m}\sqrt{g}\abs{j_m}^\frac12} + \frac{\gamma}{2\braket{\ell_m}}\frac{G_{j_m}(0)}{j_m} \Big)_{|\gamma= \gamma_m} \to 0  
	\end{split}
	\end{equation*}
	 for $ m \rightarrow \infty $.
	Therefore
	\eqref{eq:1_abs_m}  becomes, in the limit $m\rightarrow\infty $,   
	$	\partial_\gamma^n  \ora{\Omega}(\bar \gamma) \cdot \bar c  \, 
	= 0  $ for any $ n\in\N_0 $. 
	By analyticity,   this implies $\ora{\Omega}(\gamma)\cdot \bar{c}  = 0 $ for any $ \gamma\in\Gamma$, contradicting 
	 the non-degeneracy of $ \ora{\Omega}(\gamma)$ in Lemma \ref{lem:non_deg_vectors}-1, since $ \bar c \neq 0$.
	\\[1mm]
	{\bf Proof of \eqref{eq:2_meln-}}. % We split  the proof in 3 steps.
We  assume $ j_m \neq j_m' $ because the case $ j_m = j_m' $ is included in \eqref{eq:0_meln}. By contradiction, we assume that, for any $m\in\N$,  there exist $\gamma_m\in\Gamma$, $\ell_m\in\Z^\nu$ and $j_m,j_m'\in\S_0^c$, 
$(\ell_m, j_m, j_m') \neq (0, j_m, j_m) $, 	
 such that, for any $0\leq n\leq m$,
	\begin{equation}\label{eq:2_abs_m}
	\begin{cases}
	\big| \partial_\gamma^n\big( \ora{\Omega}(\gamma) \cdot \frac{\ell_m}{\braket{\ell_m}}+ \frac{1}{\braket{\ell_m}}\big( \Omega_{j_m}(\gamma)-\Omega_{j_m'}(\gamma) \big) \big)_{|
		\gamma= \gamma_m} \big| < \frac{1}{\braket{m}} \\
	\ora{\jmath} \cdot \ell_m + j_m -j_m' =0 \, . 
	\end{cases} 
	\end{equation}
	We have that  $  \ell_m \neq 0 $, otherwise, by the momentum condition 
	$ j_m = j_m' $. %  and $ (\ell_m, j_m, j_m') = (0, j_m, j_m) $.  
	Up to subsequences $\gamma_m\rightarrow\bar\gamma\in\Gamma$ and $\ell_m/\braket{\ell_m}\rightarrow \bar c\in\R^\nu \setminus \{0\}  $. \\
	{\sc Step 1.} We start with the case when {\it  $(\ell_m)_{m\in\N}\subset\Z^\nu$ is bounded}. Up to subsequences, we have definitively that  $\ell_m = \bar \ell\in\Z^\nu \setminus \{0\}$.
	The sequences $(j_m)_{m\in\N}$ and $(j'_m)_{m\in\N}$ may be bounded or unbounded. Up to subsequences, we consider the different cases:
	\\[1mm]
	\indent {\it Case (a). $|j_m|,|j'_m|\to+\infty$ for $m\to\infty$.}
	We have that $ j_m\cdot j'_m >0$, because, otherwise, 
	$ | j_m - j_m' | = |j_m| + |j_m'| \to + \infty $ contradicting that 
	 $| j_m - j_m' |  = |\ora{\jmath}\cdot\ell_m| \leq C $. 
	 Recalling  \eqref{def:Gj0} we have, for any $ j \cdot j' > 0 $, that  
	 \begin{equation}\label{ossGjsgn}
		\Big| \frac{G_j(0)}{j} - \frac{G_{j'}(0)}{j'} \Big| \leq |\sgn(j)-\sgn(j')| + \frac{2e^{-2\tth |j|}}{1+e^{-2\tth|j|}} + \frac{2e^{-2\tth |j'|}}{1+e^{-2\tth|j'|}} \leq C_\tth\Big( \frac{1}{|j|^\frac12} + \frac{1}{|j'|^\frac12} \Big) \,.
	\end{equation}
	Moreover, by the momentum condition 
	$ \ora{\jmath} \cdot \ell_m + j_m -j_m' =0 $, we deduce
\begin{equation}\label{diffsqrt}
	|\sqrt{|j_m|}  - \sqrt{|j_m'|}| = \frac{||j_m|  - |j_m'||}{\sqrt{|j_m|}  + \sqrt{|j_m'|}} 
	\leq \frac{|j_m  - j_m'|}{\sqrt{|j_m|}  + \sqrt{|j_m'|}} 
	\leq \frac{ C |\ell_m|}{\sqrt{|j_m|}  + \sqrt{|j_m'|}}  \, .
\end{equation}
By \eqref{Om-om}, Lemma \ref{rem:exp_omegaj_gam}, 
	$ j_m\cdot j'_m >0$, \eqref{ossGjsgn}, \eqref{diffsqrt},
  %\eqref{def:Omegajk},  \eqref{def:Gj0},
	% and the condition $j_m\neq j'_m$,
	 we conclude that 
	\begin{align*}
	\pa_\gamma^n (\Omega_{j_m}(\gamma)-\Omega_{j_m'}(\gamma)) 
	& = 
	\sqrt{g}  \partial_\gamma^n \big( 
	\sqrt{|j_m|}  - \sqrt{|j_m'|} \big) \\
	&  + 
	 \partial_\gamma^n \Big(    \frac{c_{j_m}(\gamma)}{\sqrt g |j_m|^\frac12} -
	 \frac{c_{j_m'}(\gamma)}{\sqrt g |j_m'|^\frac12} +
	 \frac{\gamma}{2} \Big(\frac{G_{j_m}(0)}{j_m} -  \frac{G_{j_m'}(0)}{j_m'}
	 \Big) \Big) \to 0 
	\end{align*}
	as $	 m \to + \infty $. 
	Passing to the limit in \eqref{eq:2_abs_m}, we obtain
	$ \pa_\gamma^n \{ \ora{\Omega}(\gamma)\cdot\bar\ell 
 \}_{|\gamma=\bar\gamma} = 0 $ for any $ n\in\N_0 $. 
	Hence the analytic function $\gamma\mapsto\ora{\Omega}(\gamma)\cdot\bar \ell $ is identically zero, contradicting 
	Lemma \ref{lem:non_deg_vectors}-1, since $ \bar \ell \neq 0 $.
	\\[1mm]
	\indent {\it Case (b). $(j_m)_{m\in\N}$ is bounded and $|j_m'| \to \infty$} (or viceversa): this case is excluded by the momentum condition 
	$ \ora{\jmath} \cdot \ell_m + j_m -j_m' =0 $  in \eqref{eq:2_abs_m}  and  
	since $ (\ell_m) $ is bounded.
	\\[1mm]
	\indent {\it Case (c). Both $(j_m)_{m\in\N}$, $(j_m')_{m\in\N}$ are bounded:}
	we have definitively that  $j_m= \bar \jmath$ and 
	$ j_m'=\bar{\jmath}' $, 
	with $\bar \jmath, \bar{\jmath}'\in\S_0^c$  and, since $ j_m \neq j_m' $, 
	\begin{equation}
	\label{elljjneq0}
	\bar \jmath \neq  \bar{\jmath}' \, . 
	\end{equation}
	Therefore \eqref{eq:2_abs_m} becomes,  in the limit $m\rightarrow\infty $,  
	\begin{equation*}
	\partial_\gamma^n\big( \ora{\Omega}(\gamma) \cdot \bar \ell + \Omega_{\bar \jmath}(\gamma)-\Omega_{\bar{\jmath}'}(\gamma)  \big)_{| \gamma= \bar \gamma }  = 0 \, , \ \forall\,n\in\N_0\, , 
	\quad 
	\ora{\jmath}\cdot \bar \ell + \bar \jmath - \bar{\jmath}' = 0 \, . 
	\end{equation*}
	By analyticity, we obtain that
	\begin{equation}\label{eq:2-_limit_b}
	\ora{\Omega}(\gamma) \cdot \bar \ell + \Omega_{\bar \jmath}(\gamma) - \Omega_{\bar{\jmath}'}(\gamma) = 0   \quad  \forall\,\gamma\in\Gamma \, , \quad 
	\ora{\jmath}\cdot \bar \ell + \bar \jmath-\bar{\jmath}' = 0 \, . 
	\end{equation}
	We distinguish several cases:
	\begin{itemize}
		\item Let $\bar \jmath,\bar{\jmath}' \notin -\S$ and 
		$|\bar{\jmath}|\neq |\bar{\jmath}'|$. By \eqref{eq:2-_limit_b} the vector $(\ora{\Omega}(\gamma),\Omega_{\bar \jmath}(\gamma),\Omega_{\bar{\jmath}'}(\gamma))$ is degenerate with $c:= (\bar \ell,1,-1)\neq 0$, contradicting  Lemma \ref{lem:non_deg_vectors}-4.
		\item  Let $\bar \jmath,\bar{\jmath}' \notin-\S$ and $\bar{\jmath}'=- \bar{\jmath}$. 
		In view of  \eqref{Om-om},  the first equation in \eqref{eq:2-_limit_b} becomes
		$$ \ora{\omega}(\gamma)\cdot \bar \ell + \frac{\gamma}{2} \Big( \sum_{a=1}^{\nu}\bar{\ell}_a\frac{G_{\bar{\jmath}_a }(0)}{\bar{\jmath}_a} +2 \frac{G_{\bar \jmath}(0)}{\bar \jmath} \Big) = 0   \quad \forall \gamma \in \Gamma \, . $$
		By Lemma \ref{rem:omega-nondeg}-1 the vector $(\ora{\omega}(\gamma),\gamma)$ is non-degenerate, thus $ \bar \ell = 0 $ and $ 2 \frac{G_{\bar \jmath}(0)}{\bar \jmath}  = 0 $,  which is a contradiction. 
		\item Let $\bar{\jmath}'\notin -\S$ and $\bar{\jmath}\in-\S$. With no loss of generality 
		suppose  $\bar \jmath= -\bar{\jmath}_1 $. In view of  \eqref{Om-om}, 
		the first equation in \eqref{eq:2-_limit_b} implies that, for any $\gamma\in \Gamma $,
		$$ (\bar{\ell}_1+1)\omega_{\bar{\jmath}_1}(\gamma)  + \sum_{a=2}^{\nu}\bar{\ell}_a\omega_{\bar{\jmath}_a}(\gamma)-\omega_{\bar{\jmath}'}(\gamma)
		+  \frac{\gamma}{2}\Big( (\bar{\ell}_1-1)\frac{G_{\bar{\jmath}_1}(0)}{\bar{\jmath}_1} + \sum_{a=2}^\nu \bar{\ell}_a \frac{G_{\bar{\jmath}_a}(0)}{\bar{\jmath}_a} -\frac{G_{\bar{\jmath}'}(0)}{\bar{\jmath}'} \Big) = 0\,. $$
		By Lemma \ref{rem:omega-nondeg}-2
		the vector $\big( \ora{\omega}(\gamma),\omega_{\bar{\jmath}'}(\gamma),\gamma \big)$ is non-degenerate, which is a contradiction.
		 
		\item Last, let $\bar \jmath,\bar{\jmath}'\in -\S$ and $\bar{\jmath}\neq \bar{\jmath}'$, by \eqref{elljjneq0}.
		With no loss of generality  suppose  $\bar{\jmath} =- \bar{\jmath}_1$ and $\bar{\jmath}' = -\bar{\jmath}_2$.
		Then the first equation in \eqref{eq:2-_limit_b} reads, for any $\gamma\in\Gamma$,
		\begin{equation*}
			\begin{aligned}
			&(\bar{\ell}_1+1)\omega_{\bar{\jmath}_1}(\gamma) + \left( \bar{\ell}_2 - 1 \right)\omega_{\bar{\jmath}_2} +\sum_{a=3}^{\nu}\bar{\ell}_a\omega_{\bar{\jmath}_a}(\gamma)\\
			 &\ \  \ \ +  \frac{\gamma}{2}\Big( (\bar{\ell}_1-1)\frac{G_{\bar{\jmath}_1}(0)}{\bar{\jmath}_1}+ (\bar{\ell}_2+1 )\frac{G_{\bar{\jmath}_2}(0)}{\bar{\jmath}_2} + \sum_{a=3}^\nu \bar{\ell}_a \frac{G_{\bar{\jmath}_a}(0)}{\bar{\jmath}_a}  \Big) = 0\,   .  
			\end{aligned}
		\end{equation*}
Since the vector $(\ora{\omega}(\gamma),\gamma)$ is non-degenerate by Lemma \ref{rem:omega-nondeg}-1, it implies $\bar{\ell}_1=-1$, $\bar{\ell}_2 = 1$, $\bar{\ell}_3=\ldots=\bar{\ell}_\nu=0$. Inserting these values in  the momentum condition in \eqref{eq:2-_limit_b} we obtain $-2\bar{\jmath}_1+ 2 \bar{\jmath}_2 = 0$. This contradicts $\bar{\jmath}\neq \bar{\jmath}'$.
	\end{itemize}
	{\sc Step 2.} We finally consider the case when {\it $(\ell_m)_{m\in\N}$ is unbounded}. Up to subsequences 
	$ \ell_m \rightarrow \infty$ as $m\rightarrow\infty$ and 
	$	 \lim_{m\to\infty}\ell_m/\braket{\ell_m} =: \bar c  \neq 0 $.  
	%In addition, by \eqref{eq:2_restr}, up to subsequences
	%\begin{equation}
	%\label{eq:d3}
	%\lim_{m \to \infty} \frac{|j_m|^\frac12-|j_m'|^\frac12}{\braket{\ell_m}}=  \bar d_3 \in \R \, . 
	%\end{equation}
	By \eqref{Om-om}, Lemma \ref{rem:exp_omegaj_gam}, % \eqref{diffsqrt}, 
	\eqref{ossGjsgn}, 
	we have, for any $ n \geq 1 $, 
	\begin{align*}
	 \partial_\gamma^n\frac{1}{\braket{\ell_m}} 
	\Big( \Omega_{j_m}(\gamma)-\Omega_{j_m'}(\gamma) \Big)_{|\gamma= \gamma_m}  
	& =
	\partial_\gamma^n \Big( 
	 \frac{1}{\braket{\ell_m}\sqrt{g}}\Big( \frac{c_{j_m}(\gamma)}{|j_m|^\frac12}- \frac{c_{j_m'}(\gamma)}{|j_m'|^\frac12} \Big) \\ 
	 & \qquad + \frac{\gamma}{2\braket{\ell_m}}\Big( \frac{G_{j_m}(0)}{j_m}-\frac{G_{j_m'}(0)}{j_m'} \Big)_{|\gamma= \gamma_m}   \Big) 
	\to 0   \nonumber 
	\end{align*}
	  as $ m  \to \infty$. 
	Therefore, for any $ n \geq 1 $, taking $m\rightarrow\infty $
	in \eqref{eq:2_abs_m} we get 
	$
	\partial_\gamma^n\big( \ora{\Omega}(\gamma)\cdot \bar c  \big)
	_{|\gamma = \bar \gamma} = 0 $. % $  \forall n \geq 1  $. 
	By analyticity this implies 
	$\ora{\Omega}(\gamma)\cdot\bar c  =  \bar d $, for all  $ \gamma \in  \Gamma $,
	contradicting Lemma \ref{lem:non_deg_vectors}-2, since $ \bar c \neq 0 $. 
%	VERSIONE 2 {\sc Step 2.} We finally consider the case when {\it $(\ell_m)_{m\in\N}$ is unbounded}. Up to subsequences 
%	$ \ell_m \rightarrow \infty$ as $m\rightarrow\infty$ and 
%	$	 \lim_{m\to\infty}\ell_m/\braket{\ell_m} =: \bar c  \neq 0 $.  
%	In addition, by \eqref{diffsqrt}, up to subsequences
%	\begin{equation}
%	\label{eq:d3}
%	\lim_{m \to \infty} \frac{|j_m|^\frac12-|j_m'|^\frac12}{\braket{\ell_m}}=  \bar d  \in \R \, . 
%	\end{equation}
%	By \eqref{Om-om}, Lemma \ref{rem:exp_omegaj_gam}, \eqref{diffsqrt}, \eqref{ossGjsgn}, 
%	we have, for any $ n \in \N_0 $, 
%	\begin{align*}
%	& \partial_\gamma^n\frac{1}{\braket{\ell_m}} 
%	\Big( \Omega_{j_m}(\gamma)-\Omega_{j_m'}(\gamma) \Big)_{|\gamma= \gamma_m}  =
%	\partial_\gamma^n \Big( \frac{\sqrt{g}}{\braket{\ell_m}}\big( |j_m|^\frac12-|j_m'|^\frac12 \big)  \\
%	&  + \frac{1}{\braket{\ell_m}\sqrt{g}}\Big( \frac{c_{j_m}(\gamma)}{|j_m|^\frac12}- \frac{c_{j_m'}(\gamma)}{|j_m'|^\frac12} \Big) + \frac{\gamma}{2\braket{\ell_m}}\Big( \frac{G_{j_m}(0)}{j_m}-\frac{G_{j_m'}(0)}{j_m'} \Big)_{|\gamma= \gamma_m}   \Big) 
%	\to \bar d  (\partial_\gamma^n \sqrt{g})   \nonumber 
%	\end{align*}
%	  as $ m  \to \infty$. 
%	Therefore \eqref{eq:2_abs_m} becomes, in the limit $m\rightarrow\infty $,  
%	$
%	\partial_\gamma^n\big( \ora{\Omega}(\gamma)\cdot \bar c  +
%	\bar d  \sqrt{g}  \big)
%	_{|\gamma = \bar \gamma} =  $, $ \forall n\in\N_0 $. 
%	By analyticity this implies 
%	$\ora{\Omega}(\gamma)\cdot\bar c + \bar d  \sqrt{g} =  0$, 
%	for all  $ \gamma \in  \Gamma $,
%	contradicting Lemma \ref{lem:non_deg_vectors}-2, since $ \bar c \neq 0 $. 	
	\\[1mm]
	{\bf Proof of \eqref{eq:2_meln+}}. It follows as  \eqref{eq:2_meln-} and we omit it.	
\end{proof}

\begin{rem}\label{acca.param}
	For the irrotational gravity water waves equations \eqref{ww} with $\gamma=0$, quasi-periodic traveling waves solutions exist for most values of the \emph{depth} $\tth \in [\tth_1,\tth_2]$. 
	%It is enough to adapt the arguments of Section 3 in \cite{BBHM} with space periodic boundary conditions and using the invariance by space translations. 
	In detail, the non-degeneracy of the linear frequencies 
	with respect to the parameter $\tth$ as in Lemma \ref{lem:non_deg_vectors}
	 is proved precisely in Lemma 3.2 in \cite{BBHM}, whereas the transversality properties hold by restricting the bounds in Lemma 3.4 in \cite{BBHM} to the Fourier sites satisfying the momentum conditions. We are not able to use $\tth$ as a parameter for any value of
	 $ \gamma \neq 0 $ (in this case we do not know if  the non-degeneracy properties of  Lemma \ref{lem:non_deg_vectors} hold with respect to $ \tth $). 
\end{rem}

\section{Proof of Theorem \ref{thm:main0}}\label{sec:NM}

Under the 
rescaling $ (\eta,\zeta)\mapsto (\varepsilon \eta, \varepsilon\zeta) $, the Hamiltonian  system \eqref{eq:Ham_eq_zeta} transforms into the Hamiltonian system generated by 
\begin{equation}\label{cHepsilon}
\cH_\varepsilon(\eta,\zeta) := \varepsilon^{-2} \cH (\varepsilon\eta,\varepsilon\zeta) = \cH_L(\eta,\zeta) + \varepsilon P_\varepsilon(\eta,\zeta) \,,
\end{equation}
where $ \cH  $ is the water waves Hamiltonian \eqref{Ham-Wal} expressed in the
Wahl\'en coordinates \eqref{eq:gauge_wahlen}, 
$\cH_L $ is defined in \eqref{lin_real} and, denoting $ \cH_{\geq 3} := \cH - \cH_L   $ 
the cubic part of the Hamiltonian,  % the terms at least cubic, 
$$
P_\varepsilon(\eta,\zeta):= \varepsilon^{-3} \cH_{\geq 3}  (\varepsilon\eta,\varepsilon\zeta  ) \, . 
$$
%& \ \frac{1}{2\varepsilon}\int_\T 
%\left( \zeta+\frac{\gamma}{2}\pa_x^{-1} \eta\right) 
%( G(\varepsilon\eta)-G(0) ) \left( \zeta+\frac{\gamma}{2}\pa_x^{-1} \eta\right) \wrt x \\
%& +  \frac{\gamma}{2}\int_\T\left( 
%- \left( \zeta+\frac{\gamma}{2}\pa_x^{-1} \eta\right)_{\! x} \eta^2 +\frac{\gamma}{3}\eta^3 \right)\wrt x \,. \notag \end{align*}
We now study the Hamiltonian system  generated by the Hamiltonian $\cH_\varepsilon(\eta,\zeta) $, 
in  the action-angle and normal 
coordinates $ (\theta, I, w) $ defined in Section \ref{sec:decomp}. 
Thus we consider the Hamiltonian $H_\varepsilon (\theta, I, w )$ defined by 
\begin{equation}\label{Hepsilon}
H_\varepsilon := \cH_\varepsilon \circ A = \varepsilon^{-2} \cH \circ \varepsilon A
\end{equation}
where  
$A$ is the map defined in \eqref{aacoordinates}. 
The associated symplectic form is given in \eqref{sympl_form}. 

By   \eqref{zero_term} (see also \eqref{QFH}, \eqref{ajbjAA}), 
in the variables $ (\theta, I, w) $
the quadratic Hamiltonian $\cH_L $ 
defined in \eqref{lin_real}  simply reads, up to a constant,
$$
\cN:= \cH_L\circ A = \ora{\Omega}(\gamma)\cdot I + \tfrac12 \left( \b\Omega_W w,w \right)_{L^2} 
$$
where $\ora{\Omega}(\gamma) \in \R^\nu $ is defined in \eqref{Omega-gamma} and $\b\Omega_W$ in \eqref{eq:lin00_wahlen}.
Thus the Hamiltonian $H_\varepsilon$ in \eqref{Hepsilon} is
\begin{equation}\label{cNP}
H_\varepsilon =\cN + \varepsilon P  \qquad {\rm with}
\qquad P:= P_\varepsilon \circ A \, . 
\end{equation}
We look for an embedded invariant torus
$$
i :\T^\nu \rightarrow \R^\nu \times \R^\nu \times \acca_{\S^+,\Sigma}^\angle \,, \quad \vf \mapsto i(\vf):= ( \theta(\vf), I(\vf), w(\vf)) \, , 
$$
of the Hamiltonian vector field 
$ X_{H_\varepsilon} := ( \pa_I H_\varepsilon , -\pa_\theta H_\varepsilon, \Pi_{\S^+,\Sigma}^\angle J \nabla_{w} H_\varepsilon ) $ 
filled by quasi-periodic solutions with  frequency 
vector $\omega\in\R^\nu$.

\subsection{Nash-Moser theorem of hypothetical conjugation}

Instead of looking directly for quasi-periodic solutions of  $ X_{H_\varepsilon} $
we look for quasi-periodic solutions 
of the family of modified Hamiltonians, where $\alpha\in \R^\nu$ are additional parameters,
\begin{equation}\label{Halpha}
H_\alpha := \cN_\alpha + \varepsilon P \,, \quad \cN_\alpha:= \alpha \cdot I + \tfrac12 \left( w, \b\Omega_W w \right)_{L^2} \, .
\end{equation}
We consider   the nonlinear operator
\begin{align}\label{F_op}
\cF(i,\alpha) & := \cF(\omega,\gamma,\varepsilon;i,\alpha) := \omega\cdot\pa_\vf i(\vf) - X_{H_\alpha}(i(\vf)) \notag \\
& = \begin{pmatrix}
\omega\cdot \pa_\vf \theta(\vf) & - \alpha -\varepsilon \pa_I P(i(\vf)) \\
\omega\cdot \pa_\vf I(\vf) &+ \varepsilon \pa_\theta P(i(\vf)) \\
\omega\cdot \pa_\vf w(\vf) &  - \, \Pi_{\S^+,\Sigma}^\angle J ( \b\Omega_W w(\vf) +\varepsilon \nabla_{w} P(i(\vf)) )  
\end{pmatrix} \, . 
\end{align}
If $\cF(i,\alpha)=0$, then 
the embedding $\vf\mapsto i(\vf)$ is an invariant torus for the Hamiltonian vector field $X_{H_\alpha}$, filled with quasi-periodic solutions with frequency $\omega$.

Each Hamiltonian $H_\alpha$ in \eqref{Halpha} is invariant under the
involution $\vec \cS$   
and the translations $\vec \tau_\vs$, $\vs \in \R$, defined respectively in \eqref{rev_aa} and in  \eqref{vec.tau}:
\begin{equation}
\label{Halpha.symm}
H_\alpha\circ \vec \cS = H_\alpha \, , \qquad
H_\alpha\circ \vec\tau_\vs = H_\alpha \, ,  \quad \forall\, \vs \in \R \, . 
\end{equation}
We look for a reversible   traveling torus embedding 
$  i (\vphi) = $ $ ( \theta(\vf), I(\vf), w(\vf)) $, namely 
satisfying  
\begin{equation}\label{RTTT}
\vec \cS i(\vf)=  i(-\vf) \, , \qquad 
\vec \tau_\vs i(\vf) = i(\vf - \ora{\jmath} \vs) \, , \quad \forall \,\vs \in \R \, . 
\end{equation}
Note that, by \eqref{F_op} and \eqref{Halpha.symm}, 
	the operator $ \cF ( \cdot , \alpha) $ maps a reversible, respectively traveling, wave 
	into an anti-reversible,  respectively traveling, wave variation, according to Definition \ref{trav-vari}.

The norm of the periodic components of the embedded torus
\begin{equation}\label{ICal}
\fI (\vf):= i(\vf)-(\vf,0,0) := \left( \Theta(\vf), I(\vf), w(\vf) \right)\,, \quad \Theta(\vf):= \theta(\vf)-\vf\,,
\end{equation}
is $
\norm{ \fI }_s^{k_0,\upsilon} := \norm{\Theta}_{H_\vf^s}^{k_0,\upsilon} + \norm{I}_{H_\vf^s}^{k_0,\upsilon} + \norm{w}_s^{k_0,\upsilon} $, 
where
\begin{equation}\label{k0_def}
k_0:= m_0 + 2
\end{equation}
and $m_0 \in \N $ is the index of non-degeneracy provided by Proposition \ref{prop:trans_un}, which only depends on the linear unperturbed frequencies. We will often omit to write the dependence of the various constants with respect to $k_0$, which
is considered as an absolute constant. We look for quasi-periodic solutions of frequency $\omega$ belonging to a $\delta$-neighbourhood (independent of $\varepsilon$)
$$
\t\Omega := \big\{ \omega \in \R^\nu \ : \
\dist \big( \omega, \ora{\Omega}[\gamma_1,\gamma_2] \big) < \delta \big\} \,, \quad \delta >0 \,,
$$
of the curve
$\ora{\Omega}[\gamma_1,\gamma_2]$ defined by \eqref{Omega-gamma}.

The next theorem, whose proof is based on an implicit function iterative 
scheme of Nash-Moser type, 
provides, for $ \varepsilon $ small enough, 
 a solution $ (i_\infty, \alpha_\infty)(\omega,\gamma; 	\varepsilon) $ of the nonlinear operator
 $ {\cal F}(\varepsilon, \omega, \gamma; i,\alpha) = 0 $ for all the values of
 $ (\omega, \gamma) $ in the Cantor like set   $ {\cal C}_\infty^{\upsilon} $   below.
%  in \eqref{dioph0}-\eqref{2meln+}. 

\begin{thm}  {\bf (Theorem of hypothetical conjugation)}
\label{NMT}
	There exist positive constants ${\rm a_0},\varepsilon_0, C$ depending on $\S$, $k_0 $ and $\tau\geq 1$ such that, for all $\upsilon = \varepsilon^{\rm a}$, ${\rm a}\in (0,{\rm a}_0)$ and for all $\varepsilon\in (0,\varepsilon_0)$, there exist 
	\begin{enumerate}
		\item 
		a $k_0$-times differentiable function of the form 
		$ \alpha_\infty :  \, \t\Omega \times [\gamma_1,\gamma_2] \mapsto \R^\nu $, 
		\begin{align}\label{alpha_infty}
		& \alpha_\infty(\omega,\gamma) := \omega + r_\varepsilon(\omega,\gamma)  \quad \text{ with } \quad  |r_\varepsilon|^{k_0,\upsilon} \leq C \varepsilon \upsilon^{-1} \, ;  
		\end{align}
		\item
		a family of embedded reversible traveling tori $i_\infty (\vf) $ (cfr. \eqref{RTTT}), 
		defined for all $(\omega,\gamma)\in\t\Omega \times[\gamma_1,\gamma_2] $, satisfying
		\begin{equation}\label{i.infty.est}
		\| i_\infty (\vf) -(\vf,0,0) \|_{s_0}^{k_0,\upsilon} \leq C \varepsilon\upsilon^{-1} \, ; 
		\end{equation}
		\item 
		a sequence of $k_0$-times differentiable functions $\mu_j^\infty : \R^\nu \times [\gamma_1,\gamma_2] \rightarrow \R$, $j\in \S_0^c  = \Z\,\setminus\,(\S\cup\{0\})$, of the form
		\begin{equation}\label{def:FE}
		\mu_j^\infty(\omega,\gamma) = \tm_{1}^\infty(\omega,\gamma)j+ \tm_{\frac12}^\infty(\omega,\gamma) \Omega_j(\gamma)  -\tm_{0}^\infty(\omega,\gamma) \sgn(j) + \fr_j^\infty(\omega,\gamma)\,,
		\end{equation}
		with $\Omega_j(\gamma) $ defined in \eqref{def:Omegajk}, satisfying
		\begin{equation}\label{coeff_fin_small}
		 |\tm_1^\infty |^{k_0,\upsilon} \leq C\varepsilon \, , \ 
		  | \tm_{\frac12}^\infty-1 |^{k_0,\upsilon} +  
		|\tm_{0}^\infty|^{k_0,\upsilon} \leq C\varepsilon \upsilon^{-1} \, , \quad   \sup_{j\in\S_0^c} |j|^\frac12  | \fr_j^\infty |^{k_0,\upsilon} \leq C \varepsilon \upsilon^{-3} \, , 
		\end{equation}
	\end{enumerate}
	such that, for all $(\omega,\gamma)$ in the Cantor-like set
	\begin{align}
	\cC_\infty^\upsilon :=
	 & \Big\{ (\omega,\gamma) \in \t\Omega\times[\gamma_1,\gamma_2] \ : \
	 | \omega \cdot \ell | \ge\ 8 \upsilon \langle \ell \rangle^{-\tau} \, , \ 
	 \ \forall\,\ell\in \Z^\nu\setminus\{ 0\} \, , \label{dioph0} \\
&
	  \ \abs{ \omega \cdot \ell -\tm_1^\infty(\omega, \gamma) j } \geq 8 \upsilon \braket{\ell}^{-\tau} \,,  \
	  \forall\, \ell \in\Z^\nu , \, j\in \S_0^c \text{ with } \ora{\jmath}\cdot\ell + j =0 
	  ;  \label{0meln}\\
%	  \ \abs{ (\omega-\tm_1^\infty(\omega, \gamma)\ora{\jmath})\cdot \ell } \geq 8 \upsilon \braket{\ell}^{-\tau} \,, \ \forall\,\ell\in \Z^\nu\setminus\{ 0\}\,;  \label{0meln}\\
	&  \ \abs{ \omega\cdot \ell + \mu_j^\infty(\omega,\gamma) } \geq 4 \upsilon \abs j^{\frac12}\braket{\ell}^{-\tau} \,,  \label{1meln} 
	\forall\, \ell \in\Z^\nu , \, j\in \S_0^c \text{ with } \ora{\jmath}\cdot\ell + j =0  \,;   \\
	& 
	\ \abs{ \omega\cdot \ell + \mu_j^\infty(\omega,\gamma)-\mu_{j'}^\infty(\omega,\gamma) }\geq 4 \upsilon\,
%	\langle | j|^\frac32 - | j'|^\frac32 \rangle 
	\braket{\ell}^{-\tau} \,,  \label{2meln-}\\
	&\  \quad \quad \forall \ell\in\Z^\nu, \, j,j'\in\S_0^c,\, (\ell,j,j')\neq (0,j,j) \text{ with } \ora{\jmath}\cdot \ell + j-j'=0 \, ,   \nonumber \\
	& \  \abs{ \omega\cdot \ell + \mu_j^\infty(\omega,\gamma) +\mu_{j'}^\infty(\omega,\gamma)  } \geq 4\upsilon \,\big(\abs j^\frac12 + |j'|^\frac12 \big) \braket{\ell}^{-\tau}  \,, \label{2meln+}\\
	& \ \quad \quad \forall\,\ell\in \Z^\nu , \, j ,   j'\in\S_0^c \, , 
	\text{ with } \ora{\jmath}\cdot\ell+j+j'=0\, \Big\}  \,,\nonumber
	\end{align}
	the function $i_\infty(\vf):= i_\infty(\omega,\gamma,\varepsilon;\vf)$ is a 
	solution of $\cF(\omega,\gamma,\varepsilon; (i_\infty,\alpha_\infty)(\omega,\gamma))=0$. As a consequence, the embedded torus $\vf\mapsto i_\infty(\vf)$ is invariant for the Hamiltonian vector field $X_{H_{\alpha_\infty(\omega,\gamma)}}$ as it is filled by  quasi-periodic reversible traveling wave solutions  with frequency $\omega$.	
\end{thm}

Note that the Cantor-like set 
  $ {\cal C}_\infty^{\upsilon} $   
  in \eqref{dioph0}-\eqref{2meln+}  is defined in terms of the  functions 
  $ {\mathtt m}_1^\infty (\omega, \gamma) $ and the ``final" perturbed normal 
  frequencies
  $  \mu_j^\infty (\omega, \gamma)$, 
  $ j  \in {\mathbb S}_0^c $, which are defined for {\it all} the values of the
  parameters $ (\omega, \gamma) $. 
This formulation completely decouples the Nash-Moser implicit function theorem 
construction of $ (\alpha_\infty, i_\infty)(\omega, \gamma)  $ (in Sections
\ref{sec:approx_inv}-\ref{sec:NaM})
from the discussion about the measure of the  parameters where all the required
``non-resonance" conditions are verified (Section \ref{subsec:measest}).  
This approach
simplifies considerably the presentation because the measure estimates 
required to build $ ( i_\infty, \alpha_\infty)(\omega, \gamma)  $
are not verified at each step along the Nash-Moser iteration
(the set 
$ {\cal C}_\infty^{\upsilon} $  in \eqref{dioph0}-\eqref{2meln+}  could be empty,
in such a case the functions $ (\alpha_\infty, i_\infty)(\omega, \gamma) $ 
constructed in  Theorem \ref{NMT} are obtained by just
finitely many sums). % obtained in finitely many steps). 
In order to define the  extended functions $ ( i_\infty, \alpha_\infty) $ for all the values of 
 $ (\omega, \gamma) $,  preserving the weighted norm $ \| \ \|^{k_0, \upsilon} $, we use
 % weighted norm $ \| \ \|^{k_0, \upsilon} $. we use 
the  Whitney extension theory reported in Section \ref{sec:FS}.

We also remind that the conditions on the indexes in \eqref{dioph0}-\eqref{2meln+}
(where $ \vec \jmath \in \Z^\nu $ is the vector in \eqref{def:vecj}) are due to the fact that
we look for traveling wave solutions. 
These restrictions are essential to prove the measure estimates of the next section. 

\begin{rem}\label{rem:inclu}
The Diophantine condition \eqref{dioph0} could be weakened requiring only 
$  | \omega \cdot \ell | \ge\ \upsilon \langle \ell \rangle^{-\tau} $
for any 
$ \ell \cdot \vec \jmath = 0 $.  In such a case 
the vector $ \omega $ could admit one non-trivial resonance, i.e.
$ \bar \ell \in \Z^\nu \setminus \{0\} $ such that $ \omega \cdot \bar \ell = 0 $, thus the orbit $ \{ \omega t \}_{t \in \R} $ 
would densely fill a ($\nu - 1$)-dimensional torus, orthogonal to $ \bar \ell $.
In any case $ \vec \jmath  \cdot \bar \ell \neq 0 $ (otherwise 
$ |\omega \cdot \bar \ell| \geq \upsilon \langle \bar \ell \rangle^{-\tau} > 0 $, contradicting
that $ \omega \cdot \bar \ell = 0  $) and then
the closure of the set $ \{ \omega t - \vec \jmath x \}_{t \in \R, x \in \R} $ is dense in $ \T^\nu $.
This is the natural minimal requirement 
to look for traveling quasi-periodic solutions 
$ U( \omega t - \vec \jmath x )$ (Definition \ref{QPTW}). 
\end{rem}

The next goal is to deduce Theorem \ref{thm:main0} from Theorem \ref{NMT}. 

\subsection{Measure estimates: proof of Theorem \ref{thm:main0}}\label{subsec:measest}

We now want to prove the existence of quasi-periodic solutions of the original Hamiltonian system $ H_\varepsilon $ in \eqref{cNP}, which is equivalent after a rescaling to  \eqref{eq:Ham_eq_zeta},
 and not of just of the Hamiltonian system generated by the  modified Hamiltonian 
 $ H_{\alpha_\infty} $. We proceed as follows. 
By \eqref{alpha_infty}, the function $\alpha_\infty(\,\cdot\,,\gamma)$ from $\t\Omega$ into its image $\alpha_\infty(\t\Omega,\gamma)$ is invertible and 
\begin{equation}\label{inv_alpha}
\begin{aligned}
& \beta = \alpha_\infty(\omega,\gamma) = \omega+r_\varepsilon(\omega,\gamma) \ \Leftrightarrow \\
&   \omega = \alpha_\infty^{-1}(\beta,\gamma) = \beta+\breve{r}_\varepsilon(\beta,\gamma) \, , \quad \abs{ \breve{r}_\varepsilon }^{k_0,\upsilon} \leq C\varepsilon\upsilon^{-1} \,.
\end{aligned}
\end{equation}
%We underline that the function $\alpha_\infty^{-1}(\,\cdot\,,\gamma)$ is the inverse of $\alpha_\infty(\,\cdot\,,\gamma)$ at any fixed value $\gamma$ in $[\gamma_1,\gamma_2]$. 
Then, for any $\beta\in\alpha_\infty(\cC_\infty^\upsilon)$, Theorem \ref{NMT} proves the existence of an embedded invariant torus filled by quasi-periodic solutions with Diophantine frequency $\omega=\alpha_\infty^{-1}(\beta,\gamma)$ for the Hamiltonian
\begin{equation*}
H_\beta = \beta \cdot I+  \tfrac12(w,\b\Omega_W w)_{L^2} + \varepsilon P \,.
\end{equation*}
Consider the curve of the unperturbed tangential frequency vector 
$ \ora{\Omega}(\gamma) $
in \eqref{Omega-gamma}. In Theorem \ref{MEASEST} below we prove that for "most" values of $\gamma\in[\gamma_1,\gamma_2]$ the vector $(\alpha_\infty^{-1}(\ora{\Omega}(\gamma),\gamma) ,\gamma )$ is in $\cC_\infty^\upsilon$,
obtaining
%  Hence, for such values of $\gamma$ we have found 
an embedded torus for the Hamiltonian $H_\varepsilon$ in  \eqref{Hepsilon}, filled by quasi-periodic solutions with Diophantine frequency vector 
$\omega = \alpha_\infty^{-1}(\ora{\Omega}(\gamma),\gamma) $, denoted  $ \wt \Omega $ in Theorem \ref{thm:main0}.
Thus $\varepsilon A(i_\infty(\wt \Omega t))$,  where $A$ is defined in \eqref{aacoordinates},
is a quasi-periodic traveling wave solution of the water waves equations 
\eqref{eq:Ham_eq_zeta} written in the Wahl\'en  variables. %  $(\eta, \zeta)$.
Finally, going back to the original Zakharov variables via \eqref{eq:gauge_wahlen} 
% $ (\eta, \psi) $ 
we obtain solutions of \eqref{ww}. 
This proves Theorem \ref{thm:main0} together with the following measure estimates.

\begin{thm} {\bf (Measure estimates)}\label{MEASEST}
	Let
	\begin{equation}\label{param_small_meas}
	\upsilon = \varepsilon^{\rm a} \,, \quad  0 <{\rm a}<\min\{ {\rm a}_0,1/(4 m_0^2) \}\,, \quad \tau > m_0 (2m_0\nu + \nu+2) \,,
	\end{equation}
	where $m_0$ is the index of non-degeneracy given in Proposition \ref{prop:trans_un} and $k_0:= m_0+2$. Then, for $ \varepsilon \in (0, \varepsilon_0) $ small enough,  the measure of the set
	\begin{equation}\label{Gvare}
	\cG_\varepsilon := \big\{ \gamma \in [\gamma_1,\gamma_2] 
	\ : \  \big( \alpha_\infty^{-1}( \ora{\Omega}(\gamma),\gamma ),\gamma \big) 
	\in \cC_\infty^\upsilon \big\}
	\end{equation}
	satisfies $ | \cG_\varepsilon | \rightarrow\gamma_2-\gamma_1$ as 
	$\varepsilon\rightarrow 0$.
\end{thm}

The rest of this section is devoted to prove Theorem \ref{MEASEST}. By \eqref{inv_alpha}
we have
\begin{equation}\label{Om-per}
\ora{ \Omega}_\varepsilon (\gamma):= \alpha_\infty^{-1}(\ora{\Omega}(\gamma),\gamma) =
\ora{\Omega}(\gamma) +\ora{r}_\varepsilon \,,
\end{equation}
where $\ora{r}_\varepsilon(\gamma) 
:= \breve{r}_\varepsilon(\ora{\Omega}(\gamma),\gamma) $
satisfies 
\begin{equation}\label{eq:tang_res_est}
|\pa_\gamma^k {\vec r}_\varepsilon (\gamma)| \leq C \varepsilon\upsilon^{-(1+k)} \,, \quad \forall\,\abs k \leq k_0 \,, \ \text{uniformly on } [\gamma_1,\gamma_2] \,.
\end{equation}
We also denote, with a small abuse of notation, for all $j\in\S_0^c$, 
\begin{equation}\label{eq:final_eig_kappa}
\mu_j^\infty(\gamma):= \mu_j^\infty \big( \ora{\Omega}_\varepsilon(\gamma),\gamma \big) := \tm_1^\infty(\gamma) j +\tm_{\frac12}^\infty(\gamma)\Omega_j(\gamma)- \tm_{0}^\infty(\gamma) \sgn(j) + \fr_j^\infty(\gamma)\,,
\end{equation}
where 
$\tm_1^\infty (\gamma) :=\tm_1^\infty (\ora{\Omega}_\varepsilon(\gamma),\gamma) $, $\tm_{\frac12}^\infty (\gamma) :=\tm_{\frac12}^\infty (\ora{\Omega}_\varepsilon(\gamma),\gamma) $, $\tm_{0}^\infty (\gamma) :=\tm_{0}^\infty (\ora{\Omega}_\varepsilon(\gamma),\gamma) $ and $\fr_j^\infty (\gamma) :=\fr_j^\infty (\ora{\Omega}_\varepsilon(\gamma),\gamma) $.
By \eqref{coeff_fin_small} and \eqref{eq:tang_res_est} 
we get the estimates   
\begin{align}
&  |\pa_\gamma^k\tm_1^\infty(\gamma)|
\leq C \varepsilon\upsilon^{-k} \, , \, 
\big| \pa_\gamma^k\big( \tm_{\frac12}^\infty(\gamma)-1 \big) \big| + 
| \pa_\gamma^k\tm_{0}^\infty(\gamma) |  \leq C \varepsilon\upsilon^{-k-1}, \label{small_coeff_k}\\
&\sup_{j\in\S_0^c }|j|^\frac12\abs{ \pa_\gamma^k \fr_j^\infty(\gamma) } \leq C \varepsilon\upsilon^{-3-k}\,, \quad \forall\, 0\leq k\leq k_0 \,. \label{small_rem_k}
\end{align}
Recalling \eqref{dioph0}-\eqref{2meln+}, the Cantor set in \eqref{Gvare} becomes
\begin{align*}
\cG_\varepsilon  := & \Big\{ 
\gamma \in [\gamma_1,\gamma_2] \ :  
 \ | \ora{\Omega}_\varepsilon(\gamma) \cdot \ell | \geq 8 \upsilon \braket{\ell}^{-\tau} \,, 
\ \forall\,\ell\in \Z^\nu\setminus\{ 0\}\,;  \\
& \ \  | (\ora{\Omega}_\varepsilon(\gamma)  -\tm_1^\infty(\gamma) \vec \jmath ) 
\cdot \ell | \geq 8 \upsilon \braket{\ell}^{-\tau} \,, 
\  \forall\, \ell \in\Z^\nu \setminus \{0\} \, ; \\
& \  \  | \ora{\Omega}_\varepsilon(\gamma)\cdot \ell + \mu_j^\infty(\gamma) | \geq 4 \upsilon 
|j|^{\frac12}\braket{\ell}^{-\tau} \,, \  \forall\, \ell \in\Z^\nu \, ,  \, j\in\S_0^c \, , \text{ with } \ora{\jmath}\cdot\ell + j =0  \,;  \notag \\
& \ \ | \ora{\Omega}_\varepsilon(\gamma)\cdot \ell + \mu_j^\infty(\gamma)-\mu_{j'}^\infty(\gamma) | \geq 4 \upsilon \, 
%\langle | j|^\frac32 - | j'|^\frac32 \rangle
\braket{\ell}^{-\tau} \,,  \\
& \ \  \forall \ell\in\Z^\nu, \, j,j'\in\S_0^c,\, (\ell,j,j')\neq (0,j,j) \text{ with } \ora{\jmath}\cdot \ell + j-j'=0  \, ; \notag \\
& \ \ | \ora{\Omega}_\varepsilon(\gamma)\cdot \ell + \mu_j^\infty(\gamma) +\mu_{j'}^\infty(\gamma)  | \geq 4\upsilon \,\big(|j|^\frac12 + |j'|^\frac12 \big) \braket{\ell}^{-\tau}  \,, \\
&  \ \  \forall\,\ell\in \Z^\nu , \, j,  j'\in\S_0^c \text{ with } \ora{\jmath}\cdot\ell+j+j'=0 \Big\} \, .  \notag 
\end{align*}
We estimate the measure of the complementary set
\begin{align}\label{union_gec}
\cG_\varepsilon^c & := [\gamma_1,\gamma_2] \setminus\cG_\varepsilon  \\
& = \left( \bigcup_{\ell\neq 0} R_{\ell}^{(0)} \cup R_{\ell}^{(T)} \right) \cup \left( \bigcup_{\ell\in\Z^\nu, \, j\in\S_0^c\atop \ora{\jmath}\cdot\ell+j=0} R_{\ell,j}^{(I)} \right) 
\cup \left( \bigcup_{(\ell,j,j')\neq(0,j,j), j \neq j' 
	\atop \ora{\jmath}\cdot\ell+j-j'=0 } R_{\ell,j,j'}^{(II)} \right)
\cup \left( \bigcup_{\ell\in\Z^\nu, j,  j'\in\S_0^c \, , 
	\atop \ora{\jmath}\cdot\ell+j+j'=0} Q_{\ell,j,j'}^{(II)} \right) \,,\notag
\end{align}
where the ``nearly-resonant sets" are, recalling the notation $ \Gamma = [\gamma_1, \gamma_2 ] $, 
\begin{align}
R_{\ell}^{(0)}  := R_{\ell}^{(0)} (\upsilon,\tau) := & \big\{ \gamma\in \Gamma \, : \, | \ora{\Omega}_\varepsilon(\gamma) \cdot \ell | < 8 \upsilon \braket{\ell}^{-\tau}  \big\} \,, \label{R00d}\\
R_{\ell}^{(T)} := R_{\ell}^{(T)} (\upsilon,\tau) := & \big\{ \gamma\in \Gamma \, : \, | (\ora{\Omega}_\varepsilon(\gamma) -\tm_1^\infty(\gamma) \vec \jmath) \cdot \ell | < 8 \upsilon \braket{\ell}^{-\tau}  \big\} \,, \label{R00}\\
R_{\ell,j}^{(I)} := R_{\ell,j}^{(I)} (\upsilon,\tau) :=  & \big\{ \gamma\in \Gamma  \, : \, | \ora{\Omega}_\varepsilon(\gamma)\cdot \ell + \mu_j^\infty(\gamma) | < 4 \upsilon | j|^\frac12 \braket{\ell}^{-\tau}  \big\} \,, \label{RI0}\\
R_{\ell,j,j'}^{(II)} := R_{\ell,j,j'}^{(II)} (\upsilon,\tau) := & \big\{ \gamma\in \Gamma  \, : \, | \ora{\Omega}_\varepsilon(\gamma)\cdot \ell +\mu_j^\infty(\gamma)-\mu_{j'}^\infty(\gamma) | < 4 \upsilon \,
\braket{\ell}^{-\tau}  \big\} \, ,  \label{RII0} \\
Q_{\ell,j,j'}^{(II)} := Q_{\ell,j,j'}^{(II)} (\upsilon,\tau)  :=  & \Big\{ \gamma\in \Gamma \,  : \, | \ora{\Omega}_\varepsilon(\gamma)\cdot \ell + \mu_j^\infty(\gamma)+\mu_{j'}^\infty(\gamma) | < 
\frac{4 \upsilon \big(|j|^\frac12 +|j'|^\frac12 \big)}{ \braket{\ell}^{\tau}}  \Big\} \, .  \label{QII0}
\end{align}
Note that in the third union in \eqref{union_gec} we may require $ j \neq j' $ because
$R_{\ell,j,j}^{(II)} \subset  R_\ell^{(0)}$. 
In the sequel we shall always suppose the momentum conditions on the 
indexes $ \ell, j, j' $  written in \eqref{union_gec}. 
Some of the above sets are empty.
\begin{lem}\label{lem:emptysets}
	For $\varepsilon\in(0,\varepsilon_0)$ small enough, 
	 if $Q_{\ell,j,j'}^{(II)}\neq \emptyset$ then $ |j|^\frac12 + |j'|^\frac12\leq C \braket{\ell}$.
\end{lem}
\begin{proof}
	If $ Q_{\ell,j,j'}^{(II)}\neq \emptyset$
	then there exists $\gamma\in[\gamma_1,\gamma_2]$ such that
	\begin{equation}\label{RII1}
	\abs{ \mu_j^\infty(\gamma) + 
	\mu_{j'}^\infty(\gamma)} < \frac{4\upsilon \big(|j|^\frac12 +|j'|^\frac12 \big) }{\braket{\ell}^\tau} + 	C | \ell | \, .
	\end{equation}
	By \eqref{eq:final_eig_kappa} we have 
	$$
	\mu_{j}^\infty(\gamma) + \mu_{j'}^\infty(\gamma) = \tm_1^\infty(\gamma)(j+j') + 
	\tm_{\frac12}^\infty(\gamma) ( \Omega_{j}(\gamma) + \Omega_{j'}(\gamma) )
	- \tm_{0}^\infty(\gamma)( \sgn(j)+\sgn(j')) 
	+ \fr_j^\infty(\gamma)+ \fr_{j'}^\infty(\gamma) \, .
	$$
	Then, by \eqref{small_coeff_k}-\eqref{small_rem_k} with $ k = 0 $, 
	Lemma \ref{rem:exp_omegaj_gam} and
	the momentum condition $j+j'=-\ora{\jmath}\cdot\ell $,
	we deduce, for $\varepsilon$ small enough,
	\begin{align}
	|\mu_{j}^\infty(\gamma) + \mu_{j'}^\infty(\gamma)|  & 
	  \geq - C\varepsilon 
	| \ell| +\tfrac{\sqrt g}{2}\,\big| | j|^\frac12 + |j'|^\frac12 \big|  - C' -C\varepsilon\upsilon^{- 3} \, .  \label{RII2}
	\end{align}
	Combining \eqref{RII1} and \eqref{RII2}, 
	we deduce 
	$ |  |j|^{\frac12} + |j'|^\frac12 | \leq C\braket{\ell}$, for  $ \varepsilon $ small enough. 
\end{proof}

In order to estimate the measure of the sets \eqref{R00d}-\eqref{QII0},  
the key point is to prove that the perturbed frequencies satisfy 
transversality properties similar to the ones 
\eqref{eq:0_meln}-\eqref{eq:2_meln+} satisfied by the unperturbed frequencies. 
By Proposition \ref{prop:trans_un}, 
 \eqref{Om-per}, and the estimates \eqref{eq:tang_res_est}, 
\eqref{small_coeff_k}-\eqref{small_rem_k} we deduce 
the following  lemma
(cfr. Lemma 5.5 in \cite{BFM}). 

\begin{lem} {\bf (Perturbed transversality)} \label{lem:pert_trans}
	For $\varepsilon\in(0,\varepsilon_0)$ small enough and for all $\gamma\in[\gamma_1,\gamma_2]$, 
	\begin{align}
	& \max_{0\leq n \leq m_0}
	| \partial_\gamma^n \ora{\Omega}_\varepsilon (\gamma)\cdot \ell | \geq \frac{\rho_0}{2} \braket{\ell} \,, \quad \forall\,\ell\in\Z^\nu\setminus\{0\} \,; \label{eq:00_meln}\\
	&
	\max_{0\leq n \leq m_0} | \partial_\gamma^n 
	(\ora{\Omega}_\varepsilon(\gamma)
	 -\tm_1^\infty(\gamma) \vec \jmath) \cdot \ell  | \geq \frac{\rho_0}{2}\braket{\ell} \, , 
	  \quad \forall \ell\in\Z^\nu \setminus \{0\} \label{eq:0_meln_pert}\\	
	& 
	\begin{cases}
	\max_{0\leq n \leq m_0}| \partial_\gamma^n( \ora{\Omega}_\varepsilon(\gamma)\cdot \ell + \mu_j^\infty(\gamma) ) | \geq \frac{\rho_0}{2}	\braket{\ell} \, , \\
	\ora{\jmath}\cdot \ell + j = 0 \,, \quad \ell\in\Z^\nu\,, \ j\in\S_0^c \,;
	\end{cases} 
 \label{eq:1_meln_pert}\\
	&\begin{cases}
	\max_{0\leq n \leq m_0}| \partial_\gamma^n( \ora{\Omega}_\varepsilon(\gamma)\cdot \ell + \mu_j^\infty(\gamma)-\mu_{j'}^\infty(\gamma) ) | \geq \frac{\rho_0}{2}	\braket{\ell} \\
	\ora{\jmath}\cdot \ell + j -j'= 0 \,, \quad \ell\in\Z^\nu\,, \ j,j'\in\S_0^c \,, \ (\ell,j,j')\neq (0,j,j) \, ; 
	\end{cases} \label{eq:2_meln-_pert}\\
	&\begin{cases}
	\max_{0\leq n \leq m_0}| \partial_\gamma^n ( \ora{\Omega}_\varepsilon(\gamma)\cdot \ell + \mu_j^\infty(\gamma)+\mu_{j'}^\infty(\gamma) ) | \geq \frac{\rho_0}{2} 	\braket{\ell} \\
	\ora{\jmath}\cdot \ell + j + j'= 0 \,, \quad \ell\in\Z^\nu\,, \ j,   j'\in\S_0^c  \, . 
	\end{cases} \label{eq:2_meln+_pert}
	\end{align}
\end{lem}

The transversality estimates \eqref{eq:00_meln}-\eqref{eq:2_meln+_pert} and an application of R\"ussmann Theorem 17.1 in \cite{Russ}, directly 
imply the following bounds for the sets in \eqref{union_gec}
	% \eqref{QII0} 
	 (cfr. Lemma 5.6 in \cite{BFM}).

\begin{lem} {\bf (Estimates of the resonant sets)} \label{lem:meas_res}
	The measure of the sets \eqref{union_gec}-
	\eqref{QII0} satisfy
	\begin{align*}
	| R_\ell^{(0)} |, | R_\ell^{(T)} | \lesssim ( \upsilon\braket{\ell}^{-(\tau+1)} )^{\frac{1}{m_0}}  \,, & \quad | R_{\ell,j}^{(I)} |\lesssim \big( \upsilon |j|^{\frac12}\braket{\ell}^{-(\tau+1)} \big)^{\frac{1}{m_0}} \,,\\
	| R_{\ell,j,j'}^{(II)} |\lesssim \big( \upsilon 
%	\,	\langle | j|^\frac32 - | j'|^\frac32 \rangle
	\braket{\ell}^{-(\tau+1)} \big)^{\frac{1}{m_0}}\,, & \quad | Q_{\ell,j,j'}^{(II)} |\lesssim \big( \upsilon\,\big( |j|^\frac12 + |j'|^{\frac12} \big)\braket{\ell}^{-(\tau+1)} \big)^{\frac{1}{m_0}} \, .
	\end{align*}
\end{lem}

	We now estimate the measure of all the sets in  \eqref{union_gec}. 
	By Lemma \ref{lem:meas_res}, and the choice of  
$ \tau $ in	\eqref{param_small_meas}, we have 
	\begin{align}
	\label{R0est}
	&\Big| \bigcup_{\ell \neq 0} R^{(0)}_\ell \cup R^{(T)}_\ell
	\Big| \leq \sum_{\ell \neq 0} | R^{(0)}_\ell| + | R^{(T)}_\ell|
	 \lesssim \sum_{\ell\neq 0} \Big( \frac{\upsilon}{\braket{\ell}^{\tau +1}} \Big)^{\frac{1}{m_0}} \lesssim \upsilon^{\frac{1}{m_0}} \,,\\
	\label{R1est}
	&\abs{\bigcup_{\ell \neq 0,  j = - \ora{\jmath}\cdot\ell}R_{\ell,j}^{(I)}} \leq 
	\sum_{\ell \neq 0} | R_{\ell,- \ora{\jmath}\cdot\ell}^{(I)}| 
	\lesssim  \sum_{\ell} \Big( \frac{\upsilon}{\braket{\ell}^{\tau + \frac12}} \Big)^{\frac{1}{m_0}} \lesssim \upsilon^{\frac{1}{m_0}} \,, 
\end{align}
	and using also Lemma \ref{lem:emptysets},   
\begin{align}
	\label{Q2est}
	&\abs{\bigcup_{\ell, \, j,  j'\in\S_0^c\atop \ora{\jmath}\cdot \ell +j+j'=0} Q_{\ell,j,j'}^{(II)}} \leq \sum_{\ell, \abs j \leq C \braket{\ell}^2, \atop j'=
	-\ora{\jmath}\cdot \ell - j } |Q_{\ell,j,j'}^{(II)}|
	\lesssim
	\sum_{\ell, \abs j \leq C \braket{\ell}^2 } \left( \frac{\upsilon}{\braket{\ell}^{\tau}} \right)^{\frac{1}{m_0}}   \lesssim \upsilon^{\frac{1}{m_0}} \, . 	\end{align}
	We are left with estimating the measure of 
	\begin{align}\label{union_gec_RII}
	\bigcup_{(\ell,j,j')\neq(0,j,j), j \neq j' \atop \ora{\jmath}\cdot\ell+j-j'=0 } 
	\!\!\!\!\!\!\!\! R_{\ell,j,j'}^{(II)} & = 
	\left( \bigcup_{ j \neq j'  \,,\ j\cdot j' <0 \atop \ora{\jmath}\cdot\ell+j-j'=0 } R_{\ell,j,j'}^{(II)} \right)\cup 
\left( \bigcup_{j \neq j' \,, \ j\cdot j'  >0 \atop \ora{\jmath}\cdot\ell+j-j'=0 } R_{\ell,j,j'}^{(II)}\right)  := \tI_1+ \tI_2 \, . 
	\end{align}
	We first estimate the measure of $ \tI_1 $. 
	For $ j\cdot j' < 0 $, the momentum condition reads
	$ j-j' = \sgn(j)(|j|+|j'|) = - \ora{\jmath}\cdot\ell $, 
	thus  $|j|, |j'| \leq C \la \ell \ra$.
	Hence, by Lemma \ref{lem:meas_res} and  the choice of  
$ \tau $ in	\eqref{param_small_meas}, we have 
	\begin{align}
	\label{tI1}
	|\tI_1|
	\leq
	\sum_{ \ell, |j| \leq C \la \ell \ra,   j' = 
	j + \ora{\jmath}\cdot\ell } |R_{\ell,j,j'}^{(II)}|
	\lesssim 
	\sum_{\ell, \abs{j} \leq C \braket{\ell}} \left( \frac{\upsilon}{\braket{\ell}^{\tau+1}} \right)^{\frac{1}{m_0}}   \lesssim  \upsilon^{\frac{1}{m_0}} \, . 
	\end{align}
	Then we estimate the measure of  $ \tI_2 $ in \eqref{union_gec_RII}. 
	The key step is given in the next lemma. Remind 
	the definition of the sets
	$ R_{\ell,j, j'}^{(II)} $ and  
$ R_{\ell}^{(T)}  $
 in \eqref{RI0} and \eqref{RII0}.
	
\begin{lem}\label{lemma:inWave}
Let  $\upsilon_0 \geq \upsilon$ and $ \tau \geq \tau_0\geq 1 $.  
There is a constant $ C_1 > 0 $ such that, for $ \varepsilon $ small enough, 
for any $ \ora{\jmath}\cdot\ell+j-j'=0 $, $ j \cdot j' >  0 $,  % $ |j| \neq |j'| $.
\begin{equation}\label{incluwave}
\min \{|j|,|j'|\} \geq C_1 \upsilon_0^{-2}\braket{\ell}^{2(\tau_0+1)} 
\quad  \Longrightarrow \quad 
R_{\ell,j, j'}^{(II)} (\upsilon,\tau) \subset 
\bigcup_{\ell \neq 0} R_{\ell}^{(T)}(\upsilon_0,\tau_0) \, . 
\end{equation}
\end{lem}

\begin{proof}
If $  \gamma \in [ \gamma_1, \gamma_2] \setminus 
\bigcup_{\ell \neq 0} R_{\ell}^{(T)}(\upsilon_0,\tau_0) $, then	
\begin{equation}\label{1mely}
	|(\ora{\Omega}_\varepsilon(\gamma)-\tm_1^\infty(\gamma)\ora{\jmath})\cdot \ell|\geq 8\upsilon_0\braket{\ell}^{- \tau_0} \, , \quad  \forall \ell\in\Z^\nu\setminus\{0\} \, . 
\end{equation}
%	Note that, if $ j\cdot j' >0$, i.e. $\sgn(j)=\sgn(j')$, then 
%	\begin{equation}\label{ossGjsgn}
%		\Big| \frac{G_j(0)}{j} - \frac{G_{j'}(0)}{j'} \Big| \leq |\sgn(j)-\sgn(j')| + \frac{2e^{-2\tth |j|}}{1+e^{-2\tth|j|}} + \frac{2e^{-2\tth |j'|}}{1+e^{-2\tth|j'|}} \leq C_\tth\Big( \frac{1}{|j|^\frac12} + \frac{1}{|j'|^\frac12} \Big) \,.
%	\end{equation}
	Then, by \eqref{eq:final_eig_kappa}, the momentum condition $j-j'=-\ora{\jmath}\cdot\ell$, \eqref{small_coeff_k}, \eqref{small_rem_k},
	Lemma \ref{rem:exp_omegaj_gam},
	the condition $ j\cdot j' > 0 $, \eqref{ossGjsgn},  and  \eqref{1mely}, we deduce that 
	\begin{equation*}
		\begin{aligned}
	& 	|\ora{\Omega}_\varepsilon(\gamma)\cdot\ell  + \mu_{j}^\infty(\gamma) - \mu_{j'}^\infty(\gamma)| \geq |\ora{\Omega}_\varepsilon(\gamma)\cdot\ell + \tm_1^\infty (j-j') |- |\tm_{\frac12}^\infty||\Omega_{j}(\gamma)-\Omega_{j'}(\gamma) |- |\fr_j^\infty(\gamma)-\fr_{j'}^\infty(\gamma)| \\
		& \geq  | (\ora{\Omega}_\varepsilon(\gamma) - \tm_1^\infty \vec \jmath ) \cdot\ell  |
		  -(1-C \varepsilon \upsilon^{-1})\big||j|^\frac12 -|j'|^\frac12\big| - C\Big( \frac{1}{|j|^\frac12} + \frac{1}{|j'|^\frac12} \Big) -C 
		\frac{\varepsilon}{\upsilon^3}
		\Big( \frac{1}{|j|^\frac12} + \frac{1}{|j'|^\frac12} \Big)\\
		& \geq  \frac{8\upsilon_0}{\braket{\ell}^{\tau_0}} - \frac12 \frac{|j-j'|}{|j|^\frac12 +|j'|^\frac12} - C\Big( \frac{1}{|j|^\frac12} + \frac{1}{|j'|^\frac12} \Big) 
		 \geq \frac{8\upsilon_0}{\braket{\ell}^{\tau_0}} - C \Big( \frac{\braket{\ell}}{|j|^\frac12} + \frac{\braket{\ell}}{|j'|^\frac12} \Big) \\
		 & \geq \frac{4\upsilon_0}{\braket{\ell}^{\tau_0}} 
		\end{aligned}
	\end{equation*}
for any	
$ 	|j|,|j'| > C_1 \upsilon_0^{-2}\braket{\ell}^{2(\tau_0+1)} $, for $ C_1 > C^2 / 64 $. Since
$\upsilon_0 \geq \upsilon$ and $ \tau \geq \tau_0 $ we deduce that	
$$
|\ora{\Omega}_\varepsilon(\gamma)\cdot\ell  + \mu_{j}^\infty(\gamma) - \mu_{j'}^\infty(\gamma)| 
\geq 4\upsilon \braket{\ell}^{-\tau} \, , 
$$
namely   that  $\gamma \not\in  R_{\ell,j,j'}^{(II)} (\upsilon, \tau)$.
\end{proof}	

Note that the set of indexes $ (\ell, j, j' )$ such that 
$ \ora{\jmath}\cdot\ell+j-j'=0 $ 
and 
$ \min \{|j|,|j'|\} < C_1 \upsilon_0^{-2}\braket{\ell}^{2(\tau_0+1)} $ 
is included, for $ \upsilon_0 $ small enough, into 
 the set 
\begin{equation}\label{defIl}
 {\cal I}_\ell := \Big\{
 (\ell,j,j') \ : \, \ora{\jmath}\cdot\ell+j-j'=0  \, , \ 
 |j|, |j'|  \leq \upsilon_0^{-3}  \langle  \ell \rangle^{2(\tau_0+1)} \Big\} 
\end{equation}
because
$
\max \{ |j| , |j'| \} \leq \min \{ |j| , |j'| \} + |j - j' |  <  
C_1 \upsilon_0^{-2} \langle \ell \rangle^{2(\tau_0+1)} + 
C \langle \ell \rangle \leq \upsilon_0^{-3}  \langle \ell \rangle^{2(\tau_0+1)} $.
As a consequence, by Lemma \ref{lemma:inWave} we deduce that 
\begin{align}
& 
\tI_2 = 
\bigcup_{ j\neq j'\,,\ j\cdot j' > 0 \atop \ora{\jmath}\cdot\ell+j-j'=0} R_{\ell,j,j'}^{(II)} (\upsilon, \tau)   \subset
\Big(\bigcup_{\ell  \neq 0} R_{\ell}^{(T)} (\upsilon_0, \tau_0)  \Big) 
\bigcup
 \Big(\bigcup_{(\ell,j,j') \in {\cal I}_\ell} 
 R_{\ell,j,j'}^{(II)} (\upsilon, \tau)  \Big)   \, .  \label{primainc}
 \end{align}
 
 \begin{lem}\label{espotau0}
 Let $ \tau_0 := m_0 \nu $ and  
$ \upsilon_0 = \upsilon^{\frac{1}{4m_0}} $.
% and $ \tau \geq m_0(4(\tau_0+1)+\nu+1)> m_0(4m_0 \nu +\nu + 5) $. 
Then 
\begin{equation} \label{tI2}
|\tI_2|\leq  C \upsilon^{\frac{1}{4 m_0^2}} \, . 
\end{equation}
 \end{lem}
 
 \begin{proof}
 By \eqref{R0est} (applied with $ \upsilon_0, \tau_0 $ instead of $ \upsilon, \tau $), 
 and $\tau_0 = m_0 \nu$, the measure of 
\begin{equation}\label{stima.wtR}
\Big| \bigcup_{\ell \neq 0} R^{(T)}_\ell (\upsilon_0, \tau_0)  \Big|
	 \lesssim \upsilon_0^{\frac{1}{m_0}}  \lesssim \upsilon^{\frac{1}{4 m_0^2}} \, . 
\end{equation}
Moreover, recalling \eqref{defIl}, 
	\begin{equation}\label{stima.altra}
\Big| \bigcup_{(\ell,j,j') \in {\cal I}_\ell} 
 R_{\ell,j,j'}^{(II)} (\upsilon, \tau) \Big|  	
 %\sum_{|j|,|j'|\leq  C_1 \upsilon_0^{-2} \braket{\ell}^{2(\tau_0+1)} } |R_{\ell,j,j'}^{(II)}|  
 \lesssim \sum_{\ell \in \Z^\nu \atop |j|\leq  C_1 \upsilon_0^{-3} \braket{\ell}^{2(\tau_0+1)} } \left( \frac{\upsilon}{\braket{\ell}^{\tau+1}} \right)^{\frac{1}{m_0}}  \lesssim \sum_{\ell\in\Z^\nu} \frac{\upsilon^{\frac{1}{m_0}}\upsilon_0^{-3}}{\braket{\ell}^{\frac{\tau +1}{m_0}-2(\tau_0+1)}} \leq C \upsilon^{\frac{1}{4m_0}} \,,
	\end{equation}
	by the choice of $ \tau $ in \eqref{param_small_meas} and $ \upsilon_0 $. 
	The bound 
	\eqref{tI2} follows by \eqref{stima.wtR} and \eqref{stima.altra}.
	\end{proof}

\begin{proof}[Proof of Theorem \ref{MEASEST} completed.]
	By  \eqref{union_gec}, \eqref{R0est}, \eqref{R1est}, \eqref{Q2est}, \eqref{union_gec_RII}, \eqref{tI1} and \eqref{tI2} we deduce that 
$$
	\abs{\cG_\varepsilon^c} \leq C \upsilon^{\frac{1}{4 m_0^2}} \,.
	$$
	For $\upsilon= \varepsilon^\ta$ as in \eqref{param_small_meas}, we get 
	$| \cG_\varepsilon | \geq \gamma_2-\gamma_1 - C \varepsilon^{\ta/ 4 m_0^2} $. The proof of Theorem \ref{MEASEST} is concluded.
\end{proof}

\begin{rem}
We have actually imposed in Lemma \ref{espotau0} the stronger non-resonance 
condition \eqref{0meln}  with $\upsilon_0  = \upsilon^{\frac{1}{4 m_0}} > \upsilon $.
Since it has no significant importance for Lemma \ref{conju.tr}
%straighthening result of Section  \ref{subsec:straight}, 
we keep $ \upsilon $. 
\end{rem}

\section{Approximate inverse}\label{sec:approx_inv}

In order to implement a convergent Nash-Moser scheme that leads to a solution of $\cF(i,\alpha)=0$, where $ \cF (i, \alpha) $ is the nonlinear operator  defined in \eqref{F_op},  
we construct an \emph{almost approximate right inverse} of the linearized operator
\begin{equation*}
\di_{i,\alpha}\cF(i_0,\alpha_0)[\whi,\wh\alpha] = \omega\cdot \pa_\vf \whi - \di_i X_{H_\alpha}\left( i_0(\vf) \right)[\whi] - \left(\wh\alpha,0,0\right) \,.
\end{equation*}
Note that $\di_{i,\alpha}\cF(i_0,\alpha_0)=\di_{i,\alpha}\cF(i_0)$ is independent of $\alpha_0$.
We assume that the torus $ i_0 (\vf) = ( \theta_0 (\vf), I_0 (\vf), w_0 (\vf)) $ 
is  reversible and traveling,  according to  \eqref{RTTT}.

In the sequel we shall assume  the smallness condition,  
for some $\tk := \tk (\tau,\nu)>0$, 
$ \varepsilon\upsilon^{-\tk} \ll 1 $.

We closely follow the  strategy presented in \cite{BB} and implemented for the water waves 
equations in \cite{BM,BBHM, BFM}. As shown in \cite{BFM} this construction preserves
the momentum preserving properties needed for the search of traveling waves and the estimates are very similar. Thus we are short.

\smallskip

First of all, we state  tame estimates for the composition operator induced by the Hamiltonian vector field $X_{P}= ( \pa_I P , - \pa_\theta P, \Pi_{\S^+,\Sigma}^\angle J \nabla_{w} P )$ in \eqref{F_op} (see Lemma 6.1 of \cite{BFM}).
\begin{lem}{\bf (Estimates of the perturbation $P$)} \label{XP_est}
	Let $\fI(\vf)$ in \eqref{ICal} satisfy $\norm{ \fI }_{3 s_0 + 2 k_0 + 5}^{k_0,\upsilon}\leq 1$. Then, for any $ s \geq s_0 $, 
	$	\norm{ X_{P}(i) }_{s}^{k_0,\upsilon} \lesssim_s 1 + \norm{ \fI }_{s+2 s_0 + 2 k_0 + 3}^{k_0,\upsilon} $, 
	and, for all $\whi:= (\wh\theta,\whI,\whw)$,
	\begin{align*}
	\norm{ \di_i X_{P}(i)[\whi] }_{s}^{k_0,\upsilon} &\lesssim_s \norm{\whi}_{s+1}^{k_0,\upsilon} + \norm{ \fI }_{s+2 s_0 + 2 k_0 + 4}^{k_0,\upsilon}\norm{ \whi}_{s_0+1}^{k_0,\upsilon} \,, \\
	\norm{ \di_i^2 X_{P}(i)[\whi,\whi] }_{s}^{k_0,\upsilon} &\lesssim_s \norm{\whi}_{s+1}^{k_0,\upsilon}\norm{\whi}_{s_0+1}^{k_0,\upsilon} + \norm{ \fI }_{s+2 s_0 + 2 k_0 + 5}^{k_0,\upsilon} ( \norm{\whi}_{s_0+1}^{k_0,\upsilon} )^2 \,.
	\end{align*}
\end{lem}
%\begin{proof}
%	The proof goes as in Lemma 5.1 of \cite{BM}, using also the 
%	estimates of the Dirichlet-Neumann operator in Lemma \ref{DN_pseudo_est}.
%\end{proof}

Along this section, we assume the following hypothesis, which is verified by the approximate solutions obtained at each step of the Nash-Moser Theorem \ref{NASH}.
\begin{itemize}
	\item {\sc ANSATZ.} The map $(\omega,\gamma)\mapsto \fI_0(\omega,\gamma) = i_0(\vf;\omega,\gamma)- (\vf,0,0)$ is $k_0$-times differentiable with respect to the parameters $(\omega,\gamma)\in \R^\nu\times [\gamma_1,\gamma_2]$ and, for some $\mu:=\mu(\tau,\nu)>0$, $\upsilon\in (0,1)$,
	\begin{equation}\label{ansatz}
	\norm{\fI_0}_{s_0+\mu}^{k_0,\upsilon} + \abs{ \alpha_0-\omega }^{k_0,\upsilon} \leq C \varepsilon \upsilon^{-1} \,.
	\end{equation} 
\end{itemize}
%As in \cite{BB,BM,BBHM, BFM}, 
We first modify the approximate torus $i_0 (\vf) $ to obtain a nearby isotropic torus $i_\delta (\vf) $, namely such that the pull-back 1-form  $i_\delta^*\Lambda $  is closed, 
where $\Lambda$ is the Liouville 1-form defined in 
\eqref{liouville}.  
Consider the pull-back $ 1$-form 
\begin{align}\label{ak}
i_0^*\Lambda & = \sum_{k=1}^{\nu} a_k(\vf) \di \vf_k \, , \quad
a_k(\vf) := -\big( [ \pa_\vf \theta_0(\vf) ]^\top I_0(\vf) \big)_k +\tfrac12 
\big( J_\angle^{-1} w_0(\vf), \pa_{\vf_k} w_0(\vf) \big)_{L^2} \, , 
\end{align} 
and define $A_{kj}(\vf)  := \pa_{\vf_k} a_j(\vf) - \pa_{\vf_j}a_k(\vf) $.
The next Lemma follows as in Lemma 5.3 in  \cite{BBHM} and 
Lemma 6.2 in \cite{BFM}. 
Let 
$ Z(\vf)   
:= \cF(i_0,\alpha_0)(\vf) = \omega\cdot \pa_\varphi i_0(\vf) - X_{H_{\alpha_0}}(i_0(\vf)) $. 

\begin{lem} {\bf (Isotropic torus)} \label{torus_iso}
	The torus $i_\delta(\vf):= ( \theta_0(\vf),I_\delta(\vf),w_0(\vf) )$, defined by
	$$
	I_\delta(\vf):= I_0(\vf) + [ \pa_\vf \theta_0(\vf) ]^{-\top}\rho(\vf) \,, 
	\quad \rho = (\rho_j)_{j=1, \ldots,\nu} \, , \quad \rho_j(\vf)
	:= \Delta_\vf^{-1} \sum_{k=1}^{\nu}\pa_{\vf_k}A_{kj}(\vf)\,,
	$$
	is isotropic. Moreover, there is $\sigma:= \sigma(\nu,\tau)$ such that, for all 
	$ s \geq s_0 $, 
	\begin{align}
	\norm{ I_\delta-I_0 }_s^{k_0,\upsilon} &\lesssim_s \norm{\fI_0}_{s+1}^{k_0,\upsilon}     \label{ebb1} \, , \quad 
	\norm{ I_\delta-I_0 }_s^{k_0,\upsilon}  \lesssim_s \upsilon^{-1}
	\big( \norm{Z}_{s+\sigma}^{k_0,\upsilon} +\norm{Z}_{s_0+\sigma}^{k_0,\upsilon} \norm{ \fI_0 }_{s+\sigma}^{k_0,\upsilon} \big)  \\
	\norm{ \cF(i_\delta,\alpha_0) }_s^{k_0,\upsilon} &\lesssim_s  \norm{Z}_{s+\sigma}^{k_0,\upsilon} +\norm{Z}_{s_0+\sigma}^{k_0,\upsilon} \norm{ \fI_0 }_{s+\sigma}^{k_0,\upsilon} \, , \quad 
	\norm{ \di_i(i_\delta)[\whi] }_{s_1} 
	\lesssim_{s_1} \norm{ \whi }_{s_1+1}  \, ,   \label{ebb3} 
	\end{align}
	for $  s_1 \leq s_0 + \mu $ (cfr. \eqref{ansatz}).
	Furthermore  $i_\delta(\vf)$
	is  a reversible and traveling torus,  cfr.  \eqref{RTTT}. 
\end{lem}

We first find an  approximate inverse of the linearized operator $\di_{i,\alpha}\cF(i_\delta)$. We introduce the symplectic diffeomorphism $G_\delta:(\phi,y,\tw) \rightarrow (\theta,I,w)$ of the phase space $ \T^\nu\times \R^\nu \times \acca_{\S^+,\Sigma}^\angle$, 
\begin{equation}\label{Gdelta}
\begin{pmatrix}
\theta \\ I \\ w
\end{pmatrix} := G_\delta \begin{pmatrix}
\phi \\ y \\ \tw
\end{pmatrix} := \begin{pmatrix}
\theta_0(\phi) \\ 
I_\delta(\phi) + \left[ \pa_\phi \theta_0(\phi) \right]^{-\top}y + \left[(\pa_\theta\wtw_0)(\theta_0(\phi))  \right]^\top J_\angle^{-1} \tw \\
w_0(\phi) + \tw
\end{pmatrix}\,,
\end{equation}
where $\wtw_0(\theta):= w_0(\theta_0^{-1}(\theta))$.
It is proved in Lemma 2 of \cite{BB} that $G_\delta$ is symplectic, because the torus $i_\delta$ is isotropic (Lemma \ref{torus_iso}). In the new coordinates, $i_\delta$ is the trivial embedded torus $(\phi,y,\tw)=(\phi,0,0)$.
Moreover	the diffeomorphism $G_\delta$ in \eqref{Gdelta} is reversibility and momentum preserving, 
	in the sense that  (Lemma 6.3 in \cite{BFM})
$
	\vec \cS \circ G_\delta = G_\delta \circ \vec  \cS  $, 
$	\vec \tau_\vs \circ G_\delta = G_\delta \circ \vec \tau_\vs $, 
$ \forall \,\vs \in \R $, 
	where $\vec \cS $ and $ \vec \tau_\vs $ are defined respectively in \eqref{rev_aa}, \eqref{vec.tau}. 
	
Under the symplectic diffeomorphism $G_\delta $, the Hamiltonian vector field $X_{H_\alpha}$ changes into
$$
X_{K_\alpha} = \left(DG_\delta  \right)^{-1} X_{H_\alpha} \circ G_\delta 
\qquad {\rm where} \qquad K_\alpha := H_\alpha \circ G_\delta 
$$
is reversible and momentum preserving, in the sense that 
$ K_\alpha\circ \vec \cS = K_\alpha $,
$ K_\alpha\circ \vec \tau_\vs = K_\alpha $,  $ \forall\, \vs \in \R  $.

The Taylor expansion of $K_\alpha$ at the trivial torus $(\phi,0,0)$ is
\begin{equation}\label{taylor_Kalpha}
\begin{aligned}
K_\alpha(\phi,y,\tw) =& \ K_{00}(\phi,\alpha) + K_{10}(\phi,\alpha) \cdot y + 
( K_{01}(\phi,\alpha),\tw )_{L^2} + \tfrac12 K_{20}(\phi) y\cdot y \\
& + ( K_{11}(\phi)y,\tw )_{L^2} + \tfrac12 ( K_{02}(\phi)\tw,\tw )_{L^2} + K_{\geq 3}(\phi,y,\tw)\,,
\end{aligned}
\end{equation}
where $K_{\geq 3}$ collects all terms at least cubic in the variables $(y,\tw)$.
% By \eqref{Halpha} and \eqref{Gdelta}, the only Taylor coefficients that depend on $\alpha$ 
Here $K_{00}\in \R$, $K_{10}\in\R^\nu$, $K_{01}\in \acca_{\S^+,\Sigma}^\angle $, whereas $K_{20} $ is a $ \nu \times \nu $ symmetric matrix, 
$K_{11}\in\cL ( \R^\nu,\acca_{\S^+,\Sigma}^\angle )$ and 
$ K_{02} $ is a self-adjoint operator  acting on $\acca_{\S^+,\Sigma}^\angle $.

%Differentiating the identities in \eqref{Ka.prop} at $(\phi,0,0)$, we have
%(recalling \eqref{rev_aa}) 
%\begin{equation}\label{Ka.rev}
%\begin{aligned}
%& K_{00}(-\phi) = K_{00}(\phi)\,, \quad  K_{10}(-\phi) = K_{10}(\phi)\,, \quad K_{20}(-\phi) = K_{20}(\phi)\,, \\
%&  \cS \circ K_{01}(-\phi)  = K_{01}(\phi)\,, \quad \cS \circ K_{11}(-\phi) = K_{11}(\phi)\,,  \quad K_{02}(-\phi)\circ \cS =  \cS \circ K_{02}(\phi)\,,
%\end{aligned}
%\end{equation}
%and, recalling \eqref{vec.tau} and 
%using that $\tau_\vs^\top = \tau_{-\vs } = \tau_\vs^{-1} $, for any $\vs \in \R $, 
%\begin{equation}
%\label{K.symm.2}
%\begin{aligned}
%&K_{00}(\phi- \ora{\jmath} \vs) = K_{00}(\phi) \,, \quad  K_{10}(\phi- \ora{\jmath} \vs) = K_{10}(\phi)\,, \quad K_{20}(\phi- \ora{\jmath} \vs) = K_{20}(\phi)\,, \\
%&K_{01}(\phi- \ora{\jmath} \vs) =  \tau_\vs K_{01}(\phi) \,, \quad K_{11}(\phi- \ora{\jmath} \vs) =  \tau_\vs K_{11}(\phi) \,, \quad 
%K_{02}(\phi- \ora{\jmath} \vs) \circ \tau_\vs = \tau_\vs \circ K_{02}(\phi) \, . 
%\end{aligned}
%\end{equation}
The Hamilton equations associated to \eqref{taylor_Kalpha} are
\begin{equation}\label{hameq_Kalpha}
\begin{cases}
\dot\phi =   K_{10}(\phi,\alpha) + K_{20}(\phi)y + [K_{11}(\phi)]^\top \tw + \pa_y K_{\geq 3}(\phi,y,\tw) \\
\dot y =   - \pa_\phi K_{00}(\phi,\alpha) - [\pa_\phi K_{10}(\phi,\alpha)]^\top y - [\pa_\phi K_{01}(\phi,\alpha)]^\top \tw  \\
\ \ \ \ \ - \pa_\phi\left( \tfrac12 K_{20}(\phi)y\cdot y + \left( K_{11}(\phi)y,\tw \right)_{L^2} + \tfrac12 \left( K_{02}(\phi)\tw,\tw \right)_{L^2} + K_{\geq 3}(\phi,y,\tw) \right) \\
\dot\tw =  J_\angle \, \left( K_{01}(\phi,\alpha)+ K_{11}(\phi)y + K_{02}(\phi)\tw + \nabla_{\tw} K_{\geq 3}(\phi,y,\tw) \right) \, , 
\end{cases} 
\end{equation}
where $\pa_\phi K_{10}^\top $ is the $\nu\times\nu$ transposed matrix and $\pa_\phi K_{01}^\top , K_{11}^\top: \acca_{\S^+,\Sigma}^\angle\rightarrow\R^\nu$ are defined by the duality relation 
$ (\pa_\phi K_{01}[\wh\phi],\tw  )_{L^2}=\wh\phi\cdot [\pa_\phi K_{01} ]^\top \tw $ 
for any $\wh\phi\in\R^\nu$, $\tw\in\acca_{\S^+,\Sigma}^\angle$. 
%The transpose $	K_{11}^\top (\phi) $ is defined similarly. 

The terms
$K_{00}, K_{01}$, $K_{10} -  \omega$   in the Taylor expansion \eqref{taylor_Kalpha} vanish at an exact solution: % (that is $Z=0$): 
indeed,  
%More precisely, 
arguing as  in Lemma 5.4 in \cite{BBHM},  
	there is $ \sigma := \sigma (\nu, \tau) > 0 $, such that, for all $ s \geq s_0 $, 
	\begin{equation}\label{Kcoeff_est}
	\norm{ \pa_\phi K_{00}(\cdot , \alpha_0) }_s^{k_0,\upsilon} + \norm{ K_{10}(\cdot ,\alpha_0)-\omega}_s^{k_0,\upsilon} + \norm{ K_{01}( \cdot ,\alpha_0) }_s^{k_0,\upsilon} \lesssim_s  \norm{Z}_{s+\sigma}^{k_0,\upsilon} + \norm{Z}_{s_0+\sigma}^{k_0,\upsilon} \norm{ \fI_0 }_{s+\sigma}^{k_0,\upsilon} \, .  
%	&\norm{\pa_\alpha K_{00}}_s^{k_0,\upsilon} + \norm{ \pa_\alpha K_{10}-{\rm Id} }_s^{k_0,\upsilon} + \norm{ \pa_\alpha K_{01} }_s^{k_0,\upsilon} \lesssim_s \norm{ \fI_0 }_{s+\sigma}^{k_0,\upsilon} \, . 
%	& \norm{ K_{20} }_s^{k_0,\upsilon}\lesssim_s \varepsilon ( 1 + \norm{ \fI_0 }_{s+\sigma}^{k_0,\upsilon} )\,,  \\
%	& \norm{ K_{11}y }_s^{k_0,\upsilon} \lesssim_s \varepsilon ( \norm{ y }_s^{k_0,\upsilon}+ \norm{y }_{s_0}^{k_0,\upsilon}\norm{\fI_0 }_{s+\sigma}^{k_0,\upsilon} )\,,  \ 
%	\norm{ K_{11}^\top \tw  }_s^{k_0,\upsilon} \lesssim_s \varepsilon ( \norm{\tw}_{s}^{k_0,\upsilon} + \norm{\tw}_{s_0}^{k_0,\upsilon}\norm{\fI_0}_{s+\sigma}^{k_0,\upsilon}) \, . 
	\end{equation}
Under the linear change of variables
\begin{equation*}
DG_\delta(\vf,0,0)\begin{pmatrix}
\wh\phi \\ \why \\ \wh\tw
\end{pmatrix}:= \begin{pmatrix}
\pa_\phi \theta_0(\vf) & 0 & 0 \\ \pa_\phi I_\delta(\vf) & [\pa_\phi \theta_0(\vf)]^{-\top} &  [(\pa_\theta\wtw_0)(\theta_0(\vf))]^\top  J_\angle^{-1} \\\pa_\phi w_0(\vf) & 0 & {\rm Id}
\end{pmatrix}\begin{pmatrix}
\wh\phi \\ \why \\ \wh\tw
\end{pmatrix} \,,
\end{equation*}
the linearized operator $\di_{i,\alpha}\cF(i_\delta)$ is approximately transformed into the one obtained when one linearizes the Hamiltonian system \eqref{hameq_Kalpha} at $(\phi,y,\tw) = (\vf,0,0)$, differentiating also in $\alpha$ at $\alpha_0$ and changing $\pa_t \rightsquigarrow \omega\cdot \pa_\vf$, namely
\begin{equation}\label{lin_Kalpha}
\begin{pmatrix}
\widehat \phi  \\
\widehat y    \\ 
\widehat \tw \\
\widehat \alpha
\end{pmatrix} \mapsto
\begin{pmatrix}
\omega\cdot \pa_\vf \wh\phi - \pa_\phi K_{10}(\vf)[\wh\phi] - \pa_\alpha K_{10}(\vf)[\wh\alpha] - K_{20}(\vf)\why - [K_{11}(\vf)]^\top \wh\tw  \\
\omega\cdot \pa_\vf\why + \pa_{\phi\phi}K_{00}(\vf)[\wh\phi]+ \pa_\alpha\pa_\phi K_{00}(\vf)[\wh\alpha] + [\pa_\phi K_{10}(\vf)]^\top \why + [\pa_\phi K_{01}(\vf)]^\top  \wh\tw  \\
\omega\cdot \pa_\vf \wh\tw - J_\angle \,  \big( \pa_\phi K_{01}(\vf)[\wh\phi] + \pa_\alpha K_{01}(\vf)[\wh\alpha] + K_{11}(\vf) \why + K_{02}(\vf) \wh\tw \big)   
\end{pmatrix}. 
\end{equation}
In order to construct an ``almost approximate" inverse of \eqref{lin_Kalpha}, we need that
\begin{equation}\label{Lomegatrue}
\cL_\omega := \Pi_{\S^+,\Sigma}^\angle \left( \omega\cdot \pa_\vf - J K_{02}(\vf) \right)|_{\acca_{\S^+,\Sigma}^\angle}
\end{equation}
is "almost invertible" (on traveling waves) up to remainders of size $O(N_{\tn-1}^{-{\ta}})$, where, for $\tn\in\N_0$
\begin{equation}\label{scales}
N_\tn:= K_\tn^p \,, \quad K_\tn: = K_0^{\chi^\tn} \,, \quad \chi= 3/2\, .
\end{equation}
The $ (K_\tn)_{\tn \geq 0} $ is the scale used in the nonlinear Nash-Moser iteration of Section \ref{sec:NaM} and $ (N_\tn)_{\tn \geq 0} $ is the one in Lemma \ref{conju.tr}
% \eqref{scala.strai} 
and Theorem \ref{iterative_KAM}.
Let $H_\angle^s(\T^{\nu+1}):= H^s(\T^{\nu+1})\cap \acca_{\S^+,\Sigma}^\angle$.
\begin{itemize}
	\item[(AI)]  {\bf Almost invertibility of $\cL_\omega$}: 
	{\it There exist positive real numbers 
		$  \sigma $, $ \mu(\tb) $, $ \ta $, $ p  $, $ K_0 $ and 
		a subset $\t\Lambda_o \subset \tD\tC(\upsilon,\tau)\times [\gamma_1,\gamma_2]$ such that, for all $(\omega,\gamma) \in \t\Lambda_o$, the operator $\cL_\omega$ may be decomposed as
		\begin{equation}\label{Lomega}
		\cL_\omega = \cL_\omega^< + \cR_\omega + \cR_\omega^\perp \,,
		\end{equation}
		where,
		for any traveling wave function 
		$ g\in H_\angle^{s+\sigma}(\T^{\nu+1},\R^2)$ and 
		for any $  (\omega,\gamma) \in \t\Lambda_o $, there is a 
		traveling wave solution $ h \in H_\angle^{s}(\T^{\nu+1},\R^2) $ of 
		$ \cL_\omega^< h = g$ satisfying, for all $s_0\leq s\leq S - \mu(\tb)-\sigma$,  
		\begin{equation}\label{almi4}
		\norm{ (\cL_\omega^<)^{-1}g }_s^{k_0,\upsilon} \lesssim_S \upsilon^{-1}
		\big( \norm g_{s+\sigma}^{k_0,\upsilon}+ \norm g_{s_0+\sigma}^{k_0,\upsilon}\norm{ \fI_0}_{s+\mu({\tb})+\sigma}^{k_0,\upsilon} \big) \,.
		\end{equation}
		In addition, if $ g $ is anti-reversible, then $ h $ is reversible. Moreover, 
		for any $s_0\leq s \leq S- \mu(\tb)-\sigma$, 
		for any traveling wave $ h \in \acca_{\S^+,\Sigma}^\angle    $, the operators $\cR_\omega, \cR_\omega^\perp$ satisfy the estimates
		\begin{align}
		\norm{ \cR_\omega h }_s^{k_0,\upsilon} & \lesssim_S  \varepsilon \upsilon^{-3}N_{\tn-1}^{- \ta} \big( \norm h_{s+\sigma}^{k_0,\upsilon}+ \norm h_{s_0+\sigma}^{k_0,\upsilon} \norm{ \fI_0}_{s+\mu(\tb)+\sigma}^{k_0,\upsilon} \big) \,, 
		\label{almi1}
		\\
		\norm{ \cR_\omega^\perp h}_{s_0}^{k_0,\upsilon} & \lesssim_S K_\tn^{-{\rm b}} \big( \norm h_{s_0+{\rm b}+\sigma}^{k_0,\upsilon} + \norm h_{s_0+\sigma}^{k_0,\upsilon}\norm{\fI_0}_{s_0+\mu(\tb)+\sigma+{\rm b}}^{k_0,\upsilon} \big) \,, \ \forall\, {\rm b}>0 \,,
		\label{almi2}
		\\
		\norm{ \cR_\omega^\perp h}_s^{k_0,\upsilon} & \lesssim_S \norm h_{s+\sigma}^{k_0,\upsilon}+ \norm h_{s_0+\sigma}^{k_0,\upsilon}\norm{\fI_0}_{s+\mu(\tb)+\sigma}^{k_0,\upsilon} \label{almi3}
		\,.
		\end{align}	
	}
\end{itemize}
This assumption shall be verified by Theorem \ref{almo.inve} 
at each %$\tn$-th 
step of the Nash-Moser  iteration.

In order to find an almost approximate inverse of the linear operator in \eqref{lin_Kalpha} (and so of $\di_{i,\alpha}\cF(i_\delta)$), it is sufficient to  invert the operator
\begin{equation}\label{Dsys}
\D\big[ \wh\phi,\why,\wh\tw,\wh\alpha \big]:=\begin{pmatrix}
\omega\cdot \pa_\vf\wh\phi - \pa_\alpha K_{10}(\vf)[\wh\alpha] - K_{20}(\vf)\why- K_{11}^\top(\vf)\wh\tw \\
\omega\cdot \pa_\vf \why +\pa_\alpha\pa_\phi K_{00}(\vf)[\wh\alpha] \\
\cL_\omega^< \wh\tw -  J_\angle  \left( \pa_\alpha K_{01}(\vf)[\wh\alpha] + K_{11}(\vf)\why \right)
\end{pmatrix}
\end{equation}
obtained neglecting in \eqref{lin_Kalpha} the terms $\pa_\phi K_{10}$, $\pa_{\phi\phi}K_{00}$, $\pa_\phi K_{00}$, $\pa_\phi K_{01}$ (they vanish at an exact solution by  \eqref{Kcoeff_est}) and the small remainders $\cR_\omega$, $\cR_\omega^\perp$ appearing in \eqref{Lomega}. 

As in  section 6 of \cite{BFM} we have  the following result, where we denote
$ \normk{(\phi,y,\tw,\alpha)}{s}:= \max \big\{ \normk{(\phi,y,\tw)}{s},\abs{\alpha}^{k_0,\upsilon} \big\} $ (see \cite[Proposition 6.5]{BFM}):
\begin{prop}\label{Dsystem}
	Assume \eqref{ansatz} (with $\mu=\mu({\tb})+\sigma$) and {\rm (AI)}. Then, for all $(\omega,\gamma)\in\t\Lambda_o $, for any anti-reversible traveling wave 
	variation $ g =(g_1,g_2,g_3)$, %  (i.e. satisfying \eqref{g_revcond}-\eqref{g.mom.thm}),
	% system \eqref{Dsys_tos} 
	there exists a unique solution $\D^{-1}g:= ( \wh\phi,\why,\wh\tw,\wh\alpha)$
	of  $\D ( \wh\phi,\why,\wh\tw,\wh\alpha) = g $
	where $( \wh\phi,\why,\wh\tw)$ is a reversible traveling wave variation.  
Moreover, for any $s_0\leq s\leq S- \mu(\tb)-\sigma$, 
$
	\normk{\D^{-1}g}{s} \lesssim_{S} \upsilon^{-1}\big( \normk{g}{s+\sigma}+\normk{\fI_0}{s+\mu({\tb})+\sigma}\normk{g}{s_0+\sigma} \big) $. 
\end{prop}

Finally we conclude that the operator
\begin{equation}\label{bT0}
\bT_0 := \bT_0(i_0):=  ( D\wtG_\delta )(\vf,0,0) \circ \D^{-1} \circ (D G_\delta ) (\vf,0,0)^{-1}
\end{equation}
is an almost approximate right inverse for $\di_{i,\alpha}\cF(i_0)$, where
$	\wtG_\delta(\phi,y,\tw,\alpha) := \left( G_\delta(\phi,y,\tw),\alpha \right) $
is the identity on the $\alpha$-component. 
Arguing exactly as in  Theorem 6.6 in \cite{BFM} we deduce the following. 

\begin{thm} {\bf (Almost approximate inverse)} \label{alm.approx.inv}
	Assume {\rm (AI)}.  Then there is $\bar\sigma :=\bar\sigma(\tau,\nu,k_0)>0$ such that, if \eqref{ansatz} holds with $\mu=\mu(\tb)+\bar\sigma$, then, for all $(\omega,\gamma)\in\t\Lambda_o$ and for any anti-reversible traveling 
	wave variation $g:=(g_1,g_2,g_3)$,  
	%(i.e. satisfying \eqref{g_revcond}-\eqref{g.mom.thm}),
	 the operator $\bT_0$ defined in \eqref{bT0} satisfies, for all $s_0 \leq s \leq S - \mu(\tb)- \bar\sigma$,
	\begin{equation}\label{tame-es-AI}
	\normk{\bT_0 g}{s} \lesssim_{S} \upsilon^{-1} \big( \normk{g}{s+\bar\sigma} +\normk{\fI_0}{s+\mu(\tb)+\bar\sigma}\normk{g}{s_0+\bar\sigma}  \big)\,.
	\end{equation}
	Moreover, the first three components of $\bT_0 g $  form a reversible 
	traveling wave variation. 
	Finally, $\bT_0$ is an almost approximate right inverse of $\di_{i,\alpha}\cF(i_0)$, namely
	\begin{equation*}
	\di_{i,\alpha}\cF(i_0) \circ \bT_0 - {\rm Id} = \cP(i_0) + \cP_\omega(i_0)+\cP_\omega^\perp(i_0)\,,
	\end{equation*}
	where, for any traveling wave variation $ g $,  for all $s_0 \leq s \leq S-\mu(\tb)-\bar\sigma$,
	\begin{align}
	\normk{\cP g}{s} & \lesssim_{S} \upsilon^{-1}
	\Big(  \normk{\cF(i_0,\alpha_0)}{s_0+\bar\sigma}\normk{g}{s+\bar\sigma}  \label{pfi0} \\
	& \qquad + \,   \big(  \normk{\cF(i_0,\alpha_0)}{s+\bar\sigma}+\normk{\cF(i_0,\alpha_0)}{s_0+\bar\sigma}\normk{\fI_0}{s+\mu(\tb)+\bar\sigma}  \big)\normk{g}{s_0+\bar\sigma}  \Big)\, ,  \\
	\normk{\cP_\omega g}{s} & \lesssim_{S} \varepsilon\upsilon^{-4} N_{\tn-1}^{-\ta} \big( \normk{g}{s+\bar\sigma}+ \normk{\fI_0}{s+\mu(\tb)+\bar\sigma}\normk{g}{s_0+\bar\sigma}  \big)\, , \label{pfi1}  \\
	\normk{\cP_\omega^\perp g}{s_0} & \lesssim_{S,b} \upsilon^{-1} K_\tn^{-b} \left( \normk{g}{s_0+\bar\sigma+b}+\normk{\fI_0}{s_0+\mu(\tb)+b+\bar\sigma}\normk{g}{s_0+\bar\sigma} \right)\,, \quad  \forall\,b>0\,,   \label{pfi2} \\
	\normk{\cP_\omega^\perp g}{s} & \lesssim_{S} \upsilon^{-1}\big(  \normk{g}{s+\bar\sigma}+ \normk{\fI_0}{s+\mu(\tb)+\bar\sigma}\normk{g}{s_0+\bar\sigma} \big) \,. 
	\label{pfi3} 
	\end{align}
\end{thm}

\section{The linearized operator in the normal subspace}\label{sec:linnorm}

We now write an explicit expression of the linear operator $\cL_\omega$ defined in \eqref{Lomegatrue}. 
As in Lemma 7.1 in \cite{BFM}, since 
the diffeomorphism $ G_\delta $ in \eqref{Gdelta} is just a
 translation   along
the infinite dimensional normal variable  $ w $,  
we have the following structural result. 

\begin{lem}\label{lem:K02}
	The Hamiltonian operator $\cL_\omega$ defined in \eqref{Lomegatrue}, acting on 
	the normal subspace $ \acca_{\S^+,\Sigma}^\angle $,  has the form
	\begin{equation}\label{cLomega_again}
	\cL_\omega = \Pi_{\S^+,\Sigma}^\angle 
	(\cL -\varepsilon J R)|_{\acca_{\S^+,\Sigma}^\angle}  \,,
	\end{equation}
	where: 
	\\[1mm]
	1. $ \cL $ is the Hamiltonian operator 
		\begin{equation}\label{cL000}
		\cL := \omega\cdot \pa_\vf  - J \pa_u\nabla_u \cH(T_\delta(\vphi)) \, ,
		\end{equation}
		where  
		$\cH$  is the water waves Hamiltonian 
		in the Wahl\'en variables defined in \eqref{Ham-Wal}, evaluated at 
		the reversible traveling wave
		\begin{equation}\label{Tdelta}
		T_\delta(\phi):=  \varepsilon A ( i_\delta (\phi) ) =  \varepsilon A\left( \theta_0(\phi),I_\delta(\phi),w_0(\phi) \right) = \varepsilon v^\intercal\left( \theta_0(\phi),I_\delta(\phi) \right) + \varepsilon w_0(\phi)\,,
		\end{equation}
		the torus  $i_\delta(\vf):= ( \theta_0(\vf),I_\delta(\vf),w_0(\vf) )$ is defined in Lemma 
		\ref{torus_iso} 
		and  $A(\theta,I,w) $, $ v^\intercal(\theta,I)$ in \eqref{aacoordinates};
		2. 
		$ R (\phi) $ has the `finite rank" form 
		\begin{equation}\label{finite_rank_R}
		R(\phi)[h] = \sum_{j=1}^\nu \left( h,g_j \right)_{L^2} \chi_j \,, \quad \forall\, h\in\acca_{\S^+,\Sigma}^\angle\,,
		\end{equation}
		for functions $g_j,\chi_j \in \acca_{\S^+,\Sigma}^\angle$ which satisfy, for some $\sigma:= \sigma(\tau,\nu, k_0) > 0 $,  for all $ j = 1, \ldots, \nu $, for all $s\geq s_0$, 
		\begin{equation}
		\label{gjchij_est}
		\begin{aligned}
		\norm{ g_j }_s^{k_0,\upsilon} + \norm{ \chi_j }_s^{k_0,\upsilon} & \lesssim_s 1 + \norm{ \fI_\delta }_{s+\sigma}^{k_0,\upsilon} \,, \\
		\norm{ \di_i g_j [\whi]}_s + \norm{\di_i \chi_j [\whi]}_s & \lesssim_s \norm{ \whi }_{s+\sigma} + \norm{ \whi }_{s_0+\sigma} 
		\norm{ \fI_\delta }_{s+\sigma} \,.
		\end{aligned}
		\end{equation}
	The operator $ \cL_\omega $ is reversible and momentum preserving.
\end{lem}

In order to compute  $ d X $
we use the "shape derivative" formula,  see e.g. \cite{Lan05},
\begin{equation}\label{shapeder}
G'(\eta)[\wh\eta] \psi := \lim_{\epsilon\rightarrow 0} \tfrac{1}{\epsilon}
\big(  G(\eta+\epsilon\wh\eta)\psi -G(\eta)\psi \big) = - G(\eta)(B \wh\eta) -\pa_x(V \wh\eta) \,,
\end{equation}
where
\begin{equation}\label{BfVf}
B(\eta,\psi):= \frac{G(\eta) \psi +\eta_x \psi_x }{1+\eta_x^2} \,, 
\quad V(\eta,\psi):= \psi_x - B(\eta, \psi) \eta_x \,.
\end{equation}
%It turns out that $ (V,B) = ( \Phi_x, \Phi_y ) $ is the gradient of the generalized
%velocity potential defined  in \eqref{dir}, evaluated at the free surface $ y = \eta (x) $. 
Then, recalling \eqref{Ham-Wal}, \eqref{eq:gauge_wahlen}, \eqref{ham1} and 
\eqref{shapeder} the operator $ \cL $ in \eqref{cL000} is given by 
\begin{equation}\label{Linea10}
\begin{aligned}
\cL = \omega\cdot \pa_\vf   
&+ \begin{pmatrix}
\pa_x\wtV + G(\eta)B  & -G(\eta)   \\
g + B\wtV_x + B G(\eta) B& \wtV\pa_x - B G(\eta) 
\end{pmatrix}  \\
& +\frac{\gamma}{2}\begin{pmatrix}
-G(\eta)\pa_x^{-1} & 0 \\   \pa_x^{-1}G(\eta)B- BG(\eta)\pa_x^{-1} -\frac{\gamma}{2}\pa_x^{-1}G(\eta)\pa_x^{-1} &-\pa_x^{-1}G(\eta) 
\end{pmatrix} \,,
\end{aligned}
\end{equation}
where    
\begin{equation}\label{tc}
\wtV:= V - \gamma \eta  \, , 
\end{equation}
and the functions
$ B := B(\eta,\psi) $, $ V := V(\eta,\psi)$ in \eqref{Linea10}-\eqref{tc} are evaluated  at
the reversible traveling wave $ (\eta,\psi) := W T_\delta(\vf) $ where $ T_\delta(\vf) $ is 
defined in \eqref{Tdelta}. 	
\\[1mm]
{\bf Notation.} In \eqref{Linea10} and hereafter 
the function $ B $ is identified with the corresponding multiplication operators $h \mapsto B h$, and, where there is no parenthesis, composition of operators is understood. For example 
$ B G(\eta) B $ means $ B \circ G(\eta) \circ B $.

\begin{rem}\label{phase-space-ext}
	We consider the operator $ \cL $ in \eqref{Linea10} acting 
	on (a dense subspace of) the whole $  L^2 (\T) \times  L^2 (\T) $. In particular we extend 
	the operator $ \pa_x^{-1} $ to act on the whole $ L^2 (\T) $ as in \eqref{pax-1}. 
\end{rem}

	The following algebraic properties are a direct consequence of the reversible and space-invariance properties of the water waves equations 
	explained in Section 
	\ref{ham.s} and the fact that the approximate solution 
	$ (\eta, \zeta) = T_\delta(\vf) $ is a reversible traveling wave
		(cfr. Lemma 7.3 in \cite{BFM}).

\begin{lem}
	\label{BVtilde.mom}
	The functions  $ (\eta, \zeta) = T_\delta(\vf) $ and $B,  \wt V  $ defined  in \eqref{BfVf}, 
	\eqref{tc} are  quasi-periodic traveling waves.
	The functions $(\eta,\zeta)= T_\delta(\vf)$  are $ ( \even(\vf,x),\odd(\vf,x))$, 
	$B$ is $\odd(\vf,x)$ and  $\wtV$ is $\even(\vf,x)$.  
	The Hamiltonian operator $ \cL $ is  
	reversible and momentum preserving.
\end{lem}

For the sequel we will always assume the following ansatz (satisfied by the approximate solutions obtained along the nonlinear  Nash-Moser iteration of Section \ref{sec:NaM}): for some constants $\mu_0 :=\mu_0(\tau,\nu)>0$, $\upsilon\in (0,1)$, (cfr. 
Lemma \ref{torus_iso})
\begin{equation}\label{ansatz_I0_s0}
\norm{ \fI_0 }_{s_0+\mu_0} ^{k_0,\upsilon} \, ,  \ 
\norm{ \fI_\delta }_{s_0+\mu_0}^{k_0,\upsilon}  
\leq 1 
\,.
\end{equation} 
In order to estimate the variation of the eigenvalues with respect to the approximate invariant torus, we need also to estimate the variation with respect to the torus $i(\vf)$ in another low norm $\norm{ \ }_{s_1}$ for all Sobolev indexes $s_1$ such that
\begin{equation}\label{s1s0}
s_1+\sigma_0 \leq s_0 +\mu_0 \,, \quad \text{ for some } \ \sigma_0:=\sigma_0(\tau,\nu)>0 \,. 
\end{equation}
Thus, by \eqref{ansatz_I0_s0}, we have
$ \norm{ \fI_0 }_{s_1+\sigma_0} ^{k_0,\upsilon} $,
$  \norm{ \fI_\delta }_{s_1+\sigma_0}^{k_0,\upsilon} \leq 1 $.
The constants $\mu_0$ and $\sigma_0$ represent the \emph{loss of derivatives} accumulated along the reduction procedure of the next sections. What is important is that they are independent of the Sobolev index $s$. In the following sections we shall denote by $\sigma:=\sigma(\tau,\nu,k_0)>0 $, 
$ \sigma_N(\tq_0) := \sigma_N(\tq_0,\tau,\nu,k_0) $, 
$ \sigma_M:= \sigma_M(k_0,\tau,\nu)>0 $, $ \aleph_M (\alpha ) $
constants (which possibly increase from lemma to lemma) representing  losses 
of derivatives along the finitely many steps of the reduction procedure.

\begin{rem}
	In the next sections 
	$ \mu_0 :=\mu_0(\tau,\nu, M, \alpha) > 0 $  will depend
	also on indexes $ M, \alpha  $, whose 
	maximal values will be fixed depending only on $ \tau $ and $ \nu $ (and $ k_0 $ which is however considered an absolute constant along the paper). 
	In particular $ M $ is fixed in \eqref{M_choice}, whereas the maximal 
	value of $ \alpha $ depends on 
	$ M $, as explained in Remark \ref{fix:alpha}. 
\end{rem}

\begin{rem}
	Starting from Section \ref{subsec:straight}, we introduce in the estimates upper bounds on the regularity $s \geq s_0$. We shall control the terms in Sobolev spaces $H^s$ with $s_0 \leq s \leq S- \sigma$, where $\sigma$ %  or similia,
	 denotes a loss of derivatives of the finitely many steps of the reduction
	 (possibly increasing along the steps), whereas $ S>s_0+k_0$ is  
	 any finite Sobolev index. The index $ S$ has to be taken finite in view of 
	 Lemma \ref{conju.tr} (see also Appendix \ref{app:FGMP}). The largest 
	 regularity index $S$ will be fixed in \eqref{sigma1}. In particular, it is compatible with the condition \eqref{s1s0}, namely $s_1 +\sigma_0 \leq s_0 + \mu_0 < S $. 
\end{rem}

As a consequence of Moser composition Lemma \ref{compo_moser} and \eqref{ebb1}, 
the Sobolev norm of the function $u=T_\delta(\vf)$ defined in \eqref{Tdelta} satisfies for all $s\geq s_0$
\begin{equation}\label{uI0}
\norm u_s^{k_0,\upsilon} = \norm \eta_s^{k_0,\upsilon} + \norm \zeta_s^{k_0,\upsilon} \leq \varepsilon C(s)\big( 1 + \norm{\fI_0}_s^{k_0,\upsilon} \big)
\end{equation}
(the map $A$ defined in \eqref{aacoordinates} is smooth). Similarly, using \eqref{ebb3}, 
\begin{equation*}\label{Delta12}
\norm{ \Delta_{12}u }_{s_1} \lesssim_{s_1} \varepsilon \norm{i_2-i_1}_{s_1} \,, \quad \text{ where } \quad \Delta_{12}u:=u(i_2)-u(i_1) \, . 
\end{equation*}
We finally recall that $\fI_0 = \fI_0(\omega,\gamma)$ is defined for all $(\omega,\gamma)\in\R^\nu\times [\gamma_1,\gamma_2]$ and that the functions $B,\wtV$ and $c$ appearing in $\cL$ in \eqref{Linea10} 
are $\cC^\infty$ in $(\vf,x)$, as $u=(\eta,\zeta)=T_\delta(\vf)$ is.

\smallskip

In Sections \ref{subsec:good}-\ref{sec:order0}  
	we are going to make several transformations, whose aim is to conjugate 
the operator $ {\cal L} $ in \eqref{Linea10}
to a constant coefficients Fourier multiplier, up to a pseudo-differential operator 
	of order $ - 1 / 2 $ plus a remainder that satisfies tame estimates,  
	% both small in size, 
	see $ {\cal L}_8 $ in  \eqref{cL8}.  
	Finally, in Section \ref{sec:conc} we shall  conjugate the restricted operator 
	$ {\cal L}_\omega $ in \eqref{cLomega_again}.

\subsection{Linearized good unknown of Alinhac}\label{subsec:good}

The first step is to 
conjugate the linear operator $ \cL $ in \eqref{Linea10}
by the  symplectic (Definition \ref{def:HS}) multiplication matrix  operator
\begin{equation*}\label{good_unknown}
\cZ := \left( \begin{matrix}
{\rm Id} & 0 \\ B & {\rm Id}
\end{matrix}\right) \ , \qquad \cZ^{-1}=\left( \begin{matrix}
{\rm Id} & 0 \\ - B & {\rm Id}
\end{matrix} \right) \,,
\end{equation*}
obtaining 
\begin{equation}\label{cL0}
\begin{aligned}
\cL_1 & := \cZ^{-1} \cL \cZ  =   \omega\cdot \pa_\vf 
+  \begin{pmatrix}
\pa_x \wtV & -G(\eta) \\
a   & \wtV \pa_x
\end{pmatrix} 
-    \frac{\gamma}{2}\begin{pmatrix}
G(\eta)\pa_x^{-1} & 0 \\ \frac{\gamma}{2}\pa_x^{-1}G(\eta)\pa_x^{-1} & \pa_x^{-1}G(\eta)
\end{pmatrix} \,,  
\end{aligned}
\end{equation}
where $a$ is the function
\begin{equation}\label{ta}
a := g+ \wtV B_x + \omega\cdot\pa_\vf B \, .
\end{equation}
The matrix $\cZ$  amounts to introduce, as in \cite{Lan05} and
\cite{BM,BBHM}, a linearized version of the ``good unknown of Alinhac".

\begin{lem}\label{lem:good_unknwon}
	The maps $\cZ^{\pm 1}-{\rm Id}$ are $\cD^{k_0}$-tame with tame constants satisfying, for some $\sigma:=\sigma(\tau,\nu,k_0)> 0 $, for all $ s \geq s_0 $,	
	\begin{equation*}
	\label{lem:ga1}
	\fM_{\cZ^{\pm 1}-{\rm Id}}(s)\, , \ \fM_{(\cZ^{\pm 1}-{\rm Id})^*}(s) \lesssim_s \varepsilon\big( 1 + \norm{ \fI_0 }_{s+\sigma}^{k_0,\upsilon} \big) \,.
	\end{equation*} 
	The function $a$ in \eqref{ta} is a quasi-periodic traveling wave $\even(\vf,x)$. There is  $\sigma:= \sigma(\tau,\nu,k_0)>0$ such that, for all $ s \geq  s_0 $,
	\begin{equation} \label{lem:ga2}
	\norm{a-g}_s^{k_0,\upsilon} + \| \wtV \|_s^{k_0,\upsilon} + \norm{ B }_s^{k_0,\upsilon} \lesssim_s \varepsilon \big( 1 + \normk{\fI_0}{s+\sigma} \big) \,.
	\end{equation}
	Moreover, for any $s_1$ as in \eqref{s1s0},
	\begin{align}
		&\norm{\Delta_{12}a}_{s_1}+
	\| \Delta_{12}\wtV \|_{s_1}+\norm{\Delta_{12}B}_{s_1}\lesssim_{s_1} \varepsilon \norm{ i_1-i_2 }_{s_1+\sigma} \,, \label{lem:ga4} \\ 
	& \| \Delta_{12} (\cZ^{\pm 1})h \|_{s_1}, \| \Delta_{12}
	(\cZ^{\pm 1})^* h\|_{s_1} \lesssim_{s_1} \varepsilon \norm{i_1-i_2}_{s_1+\sigma} \norm{h}_{s_1} \,. \label{lem:ga6}
	\end{align}
	The operator $\cL_1$ is Hamiltonian, reversible and momentum preserving.
\end{lem}

\begin{proof}
By the expressions of $ B, \widetilde V, a $ in 
\eqref{BfVf}, \eqref{tc}, \eqref{ta} the composition estimates of 
Lemma \ref{compo_moser},  
 \eqref{prod} 
and the bounds for the Dirichlet-Neumann operator in Lemma \ref{DN_pseudo_est}. 
Since $ B $ is an $ \odd (\vphi, x)$ quasi-periodic traveling wave,
the matrix operator $\cZ$  is reversibility and momentum preserving (Definitions \ref{rev_defn} and \ref{def:mom.pres}).
\end{proof}

\subsection{Almost-straightening of the first order transport operator}\label{subsec:straight}
	
We now write the operator $ \cL_1 $ in \eqref{cL0} as 
\begin{equation}
\label{cL00}
\cL_1  = \omega\cdot \pa_\vf +
\begin{pmatrix}
\pa_x \wt V & 0 \\
0 & \wt V \pa_x 
\end{pmatrix} + 
\begin{pmatrix}
-\frac{\gamma}{2}G(0)\pa_x^{-1} & -G(0) \\
 a - \left(\frac{\gamma}{2}\right)^2\pa_x^{-1}G(0)\pa_x^{-1}& 
-\frac{\gamma}{2}\pa_x^{-1}G(0)
\end{pmatrix} 
 + \bR_1\,,
\end{equation}
where, using the decomposition \eqref{DNGeta}
of the Dirichlet-Neumann operator,   
\begin{equation}
\label{bRG}
\bR_1 :=- \begin{pmatrix}
\frac{\gamma}{2}\cR_G(\eta)\pa_x^{-1} & \cR_G(\eta) \\ \left(\frac{\gamma}{2}\right)^2\pa_x^{-1}\cR_G(\eta)\pa_x^{-1} & \frac{\gamma}{2}\pa_x^{-1}\cR_G(\eta)
\end{pmatrix}  
\end{equation}
is a small remainder in $ \Ops^{-\infty}$. % {uI0}. 
	The aim of this section is to conjugate the variable  coefficients quasi-periodic 
transport operator
\begin{equation}\label{LTR}
	\cL_{\rm TR} := \omega\cdot\pa_\vf +  \begin{pmatrix}
	\pa_x \wtV & 0 \\ 0 & \wtV \pa_x
	\end{pmatrix}
\end{equation}	
to a constant coefficients transport operator 
$\omega\cdot\pa_\vf + \tm_{1,\bar\tn}\,\pa_y $,
up to an exponentially small 
remainder, see \eqref{R2perp}-\eqref{stime.pn.not.w},
where $\tn\in\N_0$ and the 
scale $(N_{\tn})_{\tn\in\N_0}$ is defined, for $N_0>1$, by
\begin{equation}\label{scala.strai}
	N_{\tn}:=N_0^{\chi^{\tn}}\,, \quad \chi=3/2\,, \quad N_{-1}:=1\,.
\end{equation} 
Such small remainder  is left 
because we assume  only finitely many 
non-resonance conditions, see \eqref{tDtCn}. 
This enables  
to deduce  Lemma \ref{inclu.fgmp}, and then 
to formulate the non-resonance condition \eqref{0meln}, 
stated in terms of the ``final" function $ \tm_1^\infty (\omega, \gamma)$, which implies
\eqref{tDtCn}  at any step of the nonlinear Nash-Moser 
iteration of Section \ref{sec:NaM}.

\smallskip

In the next lemma we conjugate $ \cL_{\rm TR} $ by a \emph{symplectic} 
(Definition \ref{def:HS}) transformation
\begin{equation}\label{defcE}
\cE:= \begin{pmatrix}
	(1+\beta_x(\vf,x)) \circ \cB & 0 \\ 0 & \cB
\end{pmatrix} \, , \quad 
\cE^{-1}:=\begin{pmatrix}
	\cB^{-1}\circ (1+\beta_x(\vf,x))^{-1} & 0 \\ 0 & \cB^{-1}
\end{pmatrix}
\end{equation}
where the composition operator 
\begin{equation}\label{defcB}
	(\cB u)(\vf, x) := 
	u\left(\vf,x+ \beta(\vf, x)\right) 
	%\quad (\cB^{-1} u)(\vf, y) := 	u(\vf,y+ \breve\beta(\vf, y)) \, , 
\end{equation}
is induced by  a $\vf$-dependent diffeomorphism $y=x+\beta(\vf,x)$ 
of the torus $\T_x$, for some small quasi-periodic traveling wave 
$\beta:\T_\vf^\nu\times \T_x\to \R $, $\odd(\vf,x)$. 
%We % with $\|\beta_x\|_{L^\infty}<1/2$.
%have denoted in \eqref{defcB} the inverse diffeomorphism by
%$ x = y + \breve\beta(\vf, y) $.
%In the next key lemma we choose  
%$ \beta (\vphi, x ) $. 
Let 
	\begin{equation}\label{tbta}
		\tb:= [\ta] + 2 \in \N \,, \quad \ta:=3(\tau_1+1)  \geq 1 \,, \quad \tau_1:= k_0 +(k_0+1)\tau\,.
	\end{equation}

\begin{lem}\label{conju.tr}
{\bf (Almost-Straightening of the transport operator)}
There exists $\tau_2(\tau,\nu)> \tau_1(\tau,\nu)+1 + \ta $ 
	such that, for all $S >s_0+ k_0  $, there are $N_0:=N_0(S,\tb)\in\N$ 
	and $\updelta := \updelta(S,\tb) \in (0,1) $ such that, if
$ N_0^{\tau_2} \varepsilon \upsilon^{-1} < \updelta    $
	the following holds true. For any $\bar \tn \in \N_0 $: 
\\[1mm]
1.  There exist a  constant $\tm_{1,\bar\tn} := \tm_{1,\bar\tn}(\omega,\gamma) \in \R $, where $\tm_{1,0}=0$,
		defined for any $(\omega,\gamma)\in\R^\nu\times[\gamma_1,\gamma_2] $,   and a
		quasi-periodic traveling wave $\beta(\vf,x):=\beta_{\bar\tn}(\vf,x)$, $\odd(\vf,x)$, satisfying,
		for some $\sigma=\sigma(\tau,\nu,k_0)>0$, the  estimates
		\begin{equation}\label{beta.FGMP.est}
			|\tm_{1,\bar\tn}|^{k_0,\upsilon}\lesssim\varepsilon \,,\quad  \normk{\beta}{s} \lesssim_{S} \varepsilon\upsilon^{-1} (1+\normk{\fI_0}{s+\sigma+\tb}) \,, \quad  \forall \,
			s_0 \leq s \leq S \, , 
		\end{equation}
		{\rm independently of} $ \bar\tn$;
\\[1mm]
2. For any $(\omega,\gamma) $ in 
		\begin{equation}\label{tDtCn}
			\begin{aligned}
				\tT\tC_{\bar\tn+1}(2\upsilon,\tau) & :=\tT\tC_{\bar\tn+1}(\tm_{1,\bar\tn},2\upsilon,\tau)\\
				&:= \Big\{ (\omega,\gamma)\in\R^\nu\times[\gamma_1,\gamma_2] \,:\, |(\omega-\tm_{1,\bar\tn}\ora{\jmath})\cdot\ell| \geq  2\upsilon\braket{\ell}^{-\tau} \,\ \forall\,0<|\ell|\leq N_{\bar\tn} \Big\}
			\end{aligned}
		\end{equation}
		the operator $\cL_{\rm TR}$ in \eqref{LTR} is conjugated to
		\begin{equation}\label{coniugaLTR}
			\cE^{-1} \cL_{\rm TR} \cE = \omega\cdot\pa_\vf +\tm_{1,\bar\tn}\,\pa_y  + \bP_{2}^\perp\, ,
		\end{equation}
		where 
		\begin{equation}\label{R2perp}
			\bP_{2}^\perp := \begin{pmatrix}
				\pa_y p_{\bar\tn} & 0 \\ 0 &  p_{\bar\tn} \pa_y
			\end{pmatrix}\,,
		\end{equation}
		and	 the real, quasi-periodic traveling wave function $p_{\bar\tn}(\vf,y)$, $\even(\vf,y)$, satisfies, for some $\sigma=\sigma(\tau,\nu,k_0)>0$ and for any $s_0 \leq s \leq S$,
		\begin{equation}\label{stime.pn.not.w}
			\normk{p_{\bar\tn}}{s}\lesssim_{s,\tb}\varepsilon\,N_{\bar\tn-1}^{-\ta}(1+\normk{\fI_0}{s+\sigma+\tb}) \, ;  
			%\quad \normk{p_{\bar\tn}}{s+\tb} \lesssim_{s,\tb} \varepsilon\,N_{\bar\tn-1} (1+\normk{\fI_0}{s+\sigma+\tb})\,;
		\end{equation}
3.	The operators $\cE^{\pm}$ are $\cD^{k_0}$-$(k_0+1)$-tame, the operators $\cE^{\pm 1}-{\rm Id}$, $(\cE^{\pm 1}-{\rm Id})^*$ are $\cD^{k_0}$-$(k_0+2)$-tame
	with tame constants satisfying, for some $\sigma:=\sigma(\tau,\nu,k_0)>0$ and for all $s_0\leq s\leq S - \sigma$,
	\begin{align}
		& \fM_{\cE^{\pm 1}}(s) \lesssim_{S} 1+ \normk{\fI_0}{s+\sigma} \,, \ \fM_{\cE^{\pm 1}-{\rm Id}}(s) + \fM_{\left(\cE^{\pm 1}-{\rm Id}\right)^*}(s) \lesssim_{S}\varepsilon\upsilon^{-1} ( 1+ \normk{\fI_0}{s+\sigma+\tb})\, . \label{step5.est11}			
	\end{align}
4. Furthermore, for any $s_1$ as in \eqref{s1s0},
	\begin{align}
		& | \Delta_{12} \tm_{1,\bar\tn} |  \lesssim 
		\varepsilon \norm{i_1-i_2}_{s_1+\sigma}  \,, \quad \|\Delta_{12}\beta\|_{s_1} \lesssim_{s_1} \varepsilon\upsilon^{-1} \|i_1 -i_2\|_{s_1+\sigma+\tb}
		\label{step5.est7}\,, \\
		& \| \Delta_{12} (\cA) h  \|_{s_1} \lesssim_{s_1} \varepsilon \upsilon^{-1} \norm{i_1-i_2}_{s_1+\sigma+\tb}\norm{h}_{s_1+\sigma +\tb} \,, \quad \cA \in \{ \cE^{\pm 1} , (\cE^{\pm 1})^* \}\,.\label{step5.est13}
	\end{align}
%	The real operator $\cE^{-1}\cL_{\rm TR}\cE$ is Hamiltonian, reversible and momentum preserving.
\end{lem}
\begin{proof}
We apply Theorem \ref{thm:as} and Corollary \ref{cor.as} 
to  the transport operator  $  
X_0 = \omega \cdot \partial_\varphi + \widetilde V \partial_x $, 
which has the form \eqref{defX0} with $ p_0 = \widetilde V $.
By % Lemma \ref{proprieta}-item 1, 
\eqref{lem:ga2} and \eqref{ansatz_I0_s0},
the smallness conditions \eqref{small.V.as.AP} and 
\eqref{small12} 
hold for $ N_0^{\tau_2} 
\varepsilon \upsilon^{-1} $  sufficiently small. 
Therefore there exist a constant $ \tm_{1, \bar \tn} \in \R $ and a 
quasi-periodic traveling wave $\beta(\vf,x):=\beta_{\bar\tn}(\vf,x)$, $\odd(\vf,x)$,
 such that, for any $(\omega,\gamma) $ in $\tT\tC_{\bar\tn+1}(2\upsilon,\tau) \subseteq \t\Lambda_{\bar\tn+1}^{\upsilon,\rm T}\subseteq \t\Lambda_{\bar\tn}^{\upsilon,\rm T} $  (see Corollary \ref{cor.TC}) we have  
$$
 \cB_{\bar \tn}^{-1} (\omega \cdot \partial_\varphi + \widetilde V \partial_x) \cB_{\bar\tn} =  
 	\omega\cdot \pa_\vf + (\tm_{1,\bar\tn} + p_{\bar\tn}(\vf,y)) \pa_y  
$$
where the function  $p_{\bar\tn} $ satisfies \eqref{stime.pn.not.w} by \eqref{stime.pn.w} and \eqref{lem:ga2}. 
The estimates 
\eqref{tm.est}, \eqref{stima.B2}, \eqref{lem:ga2}  imply \eqref{beta.FGMP.est} and \eqref{step5.est11}.
The conjugated operator of $ \cL_{\rm TR}$ in \eqref{LTR} is 
$$
	\cE^{-1}\cL_{\rm TR}\cE = 
	\omega\cdot\pa_\vf +  \begin{pmatrix}
	 A_1 & 0 \\ 0 &  (\tm_{1,\bar\tn} + p_{\bar\tn}) \pa_y  
	\end{pmatrix}
$$
where  $ \omega\cdot\pa_\vf +  A_1 =  \cB^{-1} (1+\beta_x)^{-1}\big( \omega\cdot\pa_\vf + \pa_x \wtV  \big) (1+\beta_x) \cB $. 
Since $\cL_{\rm TR}$ is Hamiltonian (Definition \ref{def:HS}), and the map $\cE$ is symplectic, we have that $\cE^{-1}\cL_{\rm TR}\cE$ is Hamiltonian as well. In particular
 $ A_1 =- ( (\tm_{1,\bar\tn} + p_{\bar\tn}) \pa_y  )^* =
 \tm_{1,\bar\tn} \pa_y  + \pa_y p_{\bar\tn}  $. %  (it could be verified also by a direct calculus). 
 This proves  \eqref{coniugaLTR}-\eqref{R2perp}. The estimates 
 \eqref{step5.est7}-\eqref{step5.est13} follow by \eqref{estp121}-\eqref{estm121},
 the bound for 
 $\|\Delta_{12}\beta_{\bar\tn}\|_{s_1} $ in  Corollary \ref{cor.as} and \eqref{lem:ga4}-\eqref{lem:ga6}.
\end{proof}

\begin{rem}
	Actually, for any $(\omega,\gamma)\in\tT\tC_{\bar\tn+1}(2\upsilon,\tau)$ in \eqref{tDtCn}, Theorem \ref{thm:as} and Corollary \ref{cor.TC} would imply also the conjugation of $\cL_{\rm TR}$ to the operator $\omega\cdot\pa_\vf + \tm_{1, \bar\tn+1}\pa_y$ for some $\tm_{1, \bar\tn+1}\in\R$, up to a remainder $\bP_{2}^\perp = O( \varepsilon N_{\bar\tn}^{-\ta})$. For simplicity we stated only the conjugation in \eqref{coniugaLTR}. We shall use the non-resonance condition in \eqref{tDtCn} also later in Sections \ref{sec:order12}, \ref{sec:order0} .
\end{rem}

The next lemma is needed in order to prove the inclusion of the Cantor sets associated 
to two nearby approximate solutions. 

\begin{lem}\label{inclu.fgmp}
	Let $i_1, i_2$ be close enough and $0<2\upsilon-\rho<2\upsilon<1$. Then
	\begin{equation*}\label{condi.1}
		 \varepsilon C(s_1) N_{\bar\tn}^{\tau+1}\|i_1-i_2\|_{s_1+\sigma} \leq \rho 
	\quad \Rightarrow \quad 
		\tT\tC_{\bar\tn+1}(2\upsilon,\tau)(i_1)\subseteq \tT\tC_{\bar\tn+1}(2\upsilon-\rho,\tau)(i_2) \, . 
	\end{equation*}
\end{lem}
\begin{proof}
	For any $(\omega,\gamma)\in\tT\tC_{\bar\tn+1}(2\upsilon,\tau)(i_1)$, using also \eqref{step5.est7}, we have, for any $\ell\in\Z^\nu\setminus\{0\}$, $|\ell|\leq N_{\bar\tn}$,
	\begin{equation*}
	\begin{aligned}
		|(\omega-\tm_{1,\bar\tn}(i_2)\ora{\jmath})\cdot\ell| 
		& \geq |(\omega-\tm_{1,\bar\tn}(i_1)\ora{\jmath})\cdot\ell| - C|\Delta_{12}\tm_{1,\bar\tn}| |\ell| \\
		& \geq \frac{2\upsilon}{\braket{\ell}^\tau}- C(s_1)\varepsilon N_{\bar\tn}\|i_1-i_2\|_{s_1+\sigma} \geq \frac{2\upsilon-\rho}{\braket{\ell}^\tau}\,.
		\end{aligned}
	\end{equation*}
	 We conclude that $(\omega,\gamma)\in\tT\tC_{\bar\tn+1}(2\upsilon-\rho,\tau)(i_2)$.
\end{proof}

We now conjugate the whole operator $\cL_1$ in \eqref{cL00}-\eqref{bRG} by the operator $\cE$ in \eqref{defcE}. 

We first  compute the conjugation of the matrix
\begin{equation*}
	\begin{aligned}
		\cE^{-1}&\begin{pmatrix}
			-\frac{\gamma}{2}G(0)\pa_x^{-1} & -G(0) \\
			a - \left(\frac{\gamma}{2}\right)^2\pa_x^{-1}G(0)\pa_x^{-1}& 
			-\frac{\gamma}{2}\pa_x^{-1}G(0)
		\end{pmatrix} \cE \\
	& = \begin{pmatrix}
		-\frac{\gamma}{2}\cB^{-1}(1+\beta_x)^{-1}G(0)\pa_x^{-1}(1+\beta_x)\cB & -\cB^{-1}(1+\beta_x)^{-1}G(0)\cB \\
		\cB^{-1}\big(a - \left(\frac{\gamma}{2}\right)^2\pa_x^{-1}G(0)\pa_x^{-1}\big)(1+\beta_x)\cB& 
		-\frac{\gamma}{2}\cB^{-1}\pa_x^{-1}G(0)\cB
	\end{pmatrix} \, . 
	\end{aligned}
\end{equation*}
The multiplication operator for  $a(\vf,x)$ % defined in \eqref{ta} 
is transformed into the multiplication operator for the function
\begin{equation}\label{ta_conj}
\cB^{-1} a(1+\beta_x)\cB = \cB^{-1}\big(a(1+\beta_x)\big)\,. 
\end{equation}
We write the Dirichlet-Neumann operator  $G(0)  $ in \eqref{G(0)} as
\begin{equation}
\label{G.exp}
G(0) = G(0, \tth) = 
\pa_x \cH T(\tth) \,,
\end{equation}
where $ \cH $ is the Hilbert transform defined in \eqref{Hilbert-transf} and 
\begin{equation}
\label{Ttth}
T(\tth) := \begin{cases}
\tanh(\tth |D|) = {\rm Id} +  \Op(r_\tth)  & \text{ if } \tth < + \infty \, , 
\qquad r_{\tth} (\xi) := -\frac{2}{1+ e^{2 \tth |\xi| \chi(\xi)}} \in S^{- \infty} \, ,  \\
{\rm Id}  & \text{ if } \tth = \infty \, .
\end{cases}
\end{equation}
We have the conjugation formula (see  formula (7.42) in \cite{BBHM})  
\begin{equation}\label{conG}
\cB^{-1} G(0) \cB = \left\{\cB^{-1}(1+\beta_x)\right\} G(0) + \cR_1 \,,
\end{equation}
where 
\begin{equation*}
\cR_1:=
\left\{\cB^{-1}(1+\beta_x) \right\} \pa_y \big(
\l \cH \left(\cB^{-1} \Op(r_\tth) \cB - \Op(r_\tth) \right)+
\left( \cB^{-1} \cH \cB - \cH \right) ( \cB^{-1} T(\tth) \cB) \big)  \, . 
\end{equation*}
The operator $\cR_1 $ is in $ \Ops^{-\infty}  $ because both 
$  \cB^{-1} \Op(r_\tth) \cB - \Op(r_\tth) $ and $\cB^{-1}\cH \cB - \cH$ 
are in $ \Ops^{-\infty}  $ and there is $ \s > 0 $
such that, for any $m \in \N$,  $\alpha \in \N_0$ and $s \geq s_0 $, 
\begin{equation}
\label{ht.t}
\begin{aligned}
& \normk{ \cB^{-1} \cH \cB - \cH}{-m, s, \alpha} 
\lesssim_{m, s, \alpha, k_0}  
\normk{\beta}{s+m+\alpha + \sigma} \, , \\
& \normk{ \cB^{-1} \Op(r_\tth) \cB - \Op(r_\tth)}{-m, s, \alpha} 
\lesssim_{m, s, \alpha, k_0}  
\normk{\beta}{s+m+\alpha + \sigma} \, . 
\end{aligned}
\end{equation}
The first estimate is given in Lemmata  2.36 and 2.32 in \cite{BM}, whereas 
the second one follows by that fact that $r_\tth \in S^{-\infty}$ (see \eqref{Ttth}), 
Lemma 2.18 in \cite{BBHM} and Lemmata 2.34 and  2.32  in \cite{BM}.
Therefore by \eqref{conG} we obtain
\begin{equation}\label{conDN0}
\cB^{-1} (1+\beta_x)^{-1} G(0)\cB = \{\cB^{-1}(1+\beta_x)^{-1}\}\cB^{-1}G(0)\cB = G(0) + 
\cR_B \,,
\end{equation}
where
\begin{equation}
\label{wtR1}
\cR_B := \{\cB^{-1} (1+\beta_x)^{-1}  \} \, \cR_1 \,.
\end{equation}
Next we transform $G(0)\pa_x^{-1}$.
By \eqref{G.exp} and using the identities $ \cH \pa_x \pa_x^{-1} = \cH $ and
$ \cH T(\tth) = \partial_y^{-1} G(0) $ on the periodic functions, 
we have that
\begin{equation}\label{G(0)pax-1_ego}
\begin{aligned}
	& \cB^{-1}(1+\beta_x)^{-1}G(0)\pa_x^{-1} (1+\beta_x)\cB = G(0)\pa_y^{-1} + \cR_A \\
	&\cB^{-1} \pa_x^{-1}G(0)\cB    
	= \pa_y^{-1}G(0) + \cR_D\,,
\end{aligned}
\end{equation}
where 
\begin{equation}\label{cR2}
\begin{aligned}
\cR_D & = (\cB^{-1} \cH \cB - \cH )(\cB^{-1}T(\tth)  \cB) +
\cH \big(\cB^{-1} \Op(r_\tth) \cB - \Op(r_\tth) \big)\,, \\
\cR_A & = \{ \cB^{-1}(1+\beta_x)^{-1} \}\big[ \cH T(\tth), \{\cB^{-1}(1+\beta_x)\}-1 \big] \\
& \ \ + \{\cB^{-1}(1+\beta_x)^{-1}\}\cR_D\{\cB^{-1}(1+\beta_x)\}\,.
\end{aligned}
\end{equation}
The operator $\cR_D$ is in $ \Ops^{-\infty} $  by \eqref{ht.t}, \eqref{Ttth}.
Also $ \cR_A $ is in $ \Ops^{-\infty} $
using % and since the commutator $[\cH, a] $ with smooth function $a$   is in $\Ops^{-\infty}$ (see
that, by Lemma 2.35 of \cite{BM} and \eqref{Ttth},  there is $ \s > 0 $
such that, for any $m \in \N$, $s \geq s_0 $,  and $\alpha \in \N_0$, 
\begin{equation}
\label{Ht.comm}
\normk{[ \cH T(\tth),  \wt a] }{-m, s, \alpha} 
\lesssim_{m, s, \alpha, k_0}  
\normk{\wt a}{s+m+\alpha + \sigma} \, .
\end{equation}
Finally we conjugate $\pa_x^{-1} G(0)\pa_x^{-1}$.
By  the Egorov Proposition \ref{egorov} applied to $ \pa_x^{-1} $, 
we have that, for any $ N \in\N $, 
\begin{equation}\label{pax-1_ego}
\cB^{-1} \pa_x^{-1}(1+\beta_x) \cB=\cB^{-1}\pa_x^{-1}\cB\, \{\cB^{-1}(1+\beta_x)\} = \pa_y^{-1} 
+ P^{(1)}_{-2,N} (\varphi, x, D)
+ \tR_N\,,
\end{equation}
where $ P^{(1)}_{-2,N} (\varphi, x, D) \in \Ops^{-2}$ is  given by
$$
P^{(1)}_{-2,N} (\varphi, x, D) := 
\big[\{\cB^{-1}(1+\beta_x)^{-1}\},  \pa_y^{-1} \big]\{\cB^{-1}(1 + \beta_x)\}   +
\sum_{j=1}^N p_{-1-j}\pa_y^{-1-j}  
\{ \cB^{-1}(1 + \beta_x) \}  
$$
with functions $ p_{-1-j}(\lambda; \varphi, y)$, $ j = 0, \ldots, N $, satisfying \eqref{norm-pk}  
and $\tR_N$ 
is a regularizing operator satisfying the estimate \eqref{stima resto Egorov teo astratto}. 
So, using \eqref{G(0)pax-1_ego} and \eqref{pax-1_ego}, we obtain
\begin{equation}\label{termine_non_bello}
\begin{aligned}
	\cB^{-1}\pa_x^{-1}G(0)\pa_x^{-1}(1+\beta_x) \cB & = (\cB^{-1}\pa_x^{-1}G(0)\cB)(\cB^{-1}\pa_x^{-1}(1+\beta_x)\cB) \\
	& = \pa_y^{-1} G(0) \pa_y^{-1} +P_{-2,N}^{(2)}  +\tR_{2,N}
\end{aligned}
\end{equation}
where 
\begin{equation}\label{tildecP-1}
\begin{aligned}
P_{-2,N}^{(2)}  &:= \pa_y^{-1} G(0)   P^{(1)}_{-2,N} (\varphi, x, D) \in \Ops^{-2}
\end{aligned}
\end{equation}
and $\tR_{2,N} $ is the regularizing operator 
\begin{equation}\label{cR3}
\tR_{2,N}:= \cR_D (\cB^{-1}\pa_x^{-1}(1+\beta_x)\cB) +  G(0) \pa_y^{-1}\tR_N  \, .
\end{equation}
In conclusion, by Lemma \ref{conju.tr}, \eqref{ta_conj}, \eqref{conDN0}, \eqref{G(0)pax-1_ego} and \eqref{termine_non_bello} 
we obtain the following lemma, which summarizes the main result of this section.

\begin{lem}\label{lemma:riassu} Let $N\in\N$. For any $ \bar \tn \in \N_0 $ and for all $ (\omega, \gamma) \in \tT\tC_{\bar\tn+1}(2\upsilon,\tau) $, the operator $\cL_1$ in \eqref{cL00} is conjugated to the real, Hamiltonian, reversible and momentum preserving operator
\begin{equation}
\label{cL2}
\begin{aligned}
\cL_2  := \cE^{-1} \cL_1 \cE  
= \omega\cdot\partial_\vf + \tm_{1,\bar\tn} \pa_y \,+  & 
\begin{pmatrix}
-\frac{\gamma}{2} G(0) \pa_y^{-1} & -  G(0) \\
 a_1 - \left(\frac{\gamma}{2}\right)^2 \pa_y^{-1} G(0) \pa_y^{-1} & -\frac{\gamma}{2} \pa_y^{-1} G(0)
\end{pmatrix} \\ & +\begin{pmatrix}
0 & 0 \\
 - \left(\frac{\gamma}{2}\right)^2 P_{-2, N}^{(2)} &  0
\end{pmatrix}
+ \bR_{2}^\Psi + \bT_{2,N} + \bP_{2}^\perp\,,
\end{aligned}
\end{equation}
defined for any $(\omega,\gamma)\in\R^\nu\times[\gamma_1,\gamma_2]$, where:
\\[1mm]
1. The constant $\tm_{1,\bar\tn}=\tm_{1,\bar\tn}(\omega,\gamma) \in \R $  satisfies
	$ |\tm_{1,\bar\tn}|^{k_0,\upsilon}\lesssim\varepsilon  $, independently on $ \bar \tn $;
\\[1mm]
2. The real quasi-periodic traveling wave
	$   a_1  :=  \cB^{-1}\big(a(1+\beta_x)\big) $, $\even(\vf,x) $, 
	satisfies, for some $\sigma:= \sigma(k_0,\tau, \nu)>0$ and for all $s_0 \leq s \leq S-\sigma$,
	\begin{equation}\label{AS.est1}
		\normk{a_1 - g}{s} \lesssim_{s} \varepsilon \upsilon^{-1}(1+\normk{\fI_0}{s+\sigma})\,;
	\end{equation}
3. The operator $P_{-2, N}^{(2)}$ 
	is a pseudodifferential operator in $\Ops^{-2}$,  reversibility and momentum preserving,  and satisfies, for some $\sigma_N:=\sigma_N(\tau,\nu,N)>0$, for finitely many 
	$ 0 \leq \alpha \le \alpha (M)  $ (fixed in Remark \ref{fix:alpha}) 
	and for all $s_0\leq s \leq S-\sigma_N-\alpha$,
	\begin{equation}\label{AS.est2}
		\normk{P_{-2, N}^{(2)}}{-2,s,\alpha} \lesssim_{ s,N, \alpha} \varepsilon \upsilon^{-1}(1+\normk{\fI_0}{s+\sigma_N+\alpha})\,;
	\end{equation}
4. For any $ \tq \in \N^\nu_0 $ with $ |\tq| \leq \tq_0$, 
	$n_1, n_2 \in \N_0 $  with $ n_1 + n_2  \leq N -(k_0 + \tq_0)  +2 $,  the  
	operator $\langle D \rangle^{n_1}\partial_{\vphi}^\tq (\bR_{2}^\Psi(\vphi) + \bT_{2,N}(\vf)) \langle D \rangle^{n_2}$ is 
	$\cD^{k_0} $-tame with a tame constant satisfying, for some $\sigma_N(\tq_0) := \sigma_N(\tq_0,k_0,\tau,\nu)>0$ and for any $s_0 \leq s \leq S -\sigma_N(\tq_0)$,   
	\begin{equation}\label{AS.est3}
		{\mathfrak M}_{\langle D \rangle^{n_1}\partial_{\vphi}^\tq (\bR_{2}^\Psi(\vphi) + \bT_{2,N}(\vf)) \langle D \rangle^{n_2}}(s) \lesssim_{S, N, \tq_0} 
		\varepsilon \upsilon^{-1} \big( 1+ \normk{\fI_0}{s+\sigma_N(\tq_0)} \big)\,;
	\end{equation} 
5. The operator $\bP_{2}^\perp$ is defined in \eqref{R2perp}
	and the function $ p_{\bar\tn} $ satisfies \eqref{stime.pn.not.w};
\\[1mm]
6. Furthermore, for any $s_1$ as in \eqref{s1s0}, finitely many 
$ 0 \leq \alpha \le \alpha (M)  $,
$\tq\in\N_0^\nu$, with $\abs\tq\leq \tq_0$, and $n_1,n_2 \in\N_0$, with $n_1+n_2\leq N- \tq_0 + 1 $,
\begin{align}
	&  |\Delta_{12}\tm_{1,\bar \tn} |  \lesssim_{s_1} 
	\varepsilon \norm{i_1-i_2}_{s_1+\sigma} \,,  \ 
	\| \Delta_{12} a_1 \|_{s_1}  \lesssim 
	\varepsilon \upsilon^{-1}\norm{i_1-i_2}_{s_1+\sigma}  \, , 
	\label{AS.est4}\\
	& 	\| \Delta_{12}P_{-2, N}^{(2)}  \|_{-2,s_1,\alpha}  \lesssim_{s_1,N,\alpha} \varepsilon \upsilon^{-1}\norm{ i_1-i_2 }_{s_1+\sigma_N+\alpha}\,,
	\label{AS.est5}\\
	& \norm{\braket{D}^{n_1}\pa_\vf^\tq \Delta_{12} (\bR_{2}^\Psi(\vphi) + \bT_{2,N}(\vf))\braket{D}^{n_2} }_{\cL(H^{s_1})} \lesssim_{s_1, N, \tq_0} \varepsilon \upsilon^{-1} \norm{ i_1-i_2}_{s_1+\sigma_N(\tq_0)} \,.
	\label{AS.est6} 
\end{align}
\end{lem}
\begin{proof}
	Item 1 follows by Lemma \ref{conju.tr}. The function 
	$ a_1   $  % (see \eqref{ta_conj})
	satisfies  \eqref{AS.est1} by \eqref{ta}, \eqref{prod}, \eqref{lem:ga2}, \eqref{step5.est11}, \eqref{beta.FGMP.est}. The estimate \eqref{AS.est2} follows by \eqref{tildecP-1}, Proposition \ref{egorov} and Lemmata \ref{pseudo_compo}, \ref{pseudo_commu}, \ref{product+diffeo}, \ref{conju.tr}. The  operators $\bR_{2}^\Psi $ and $ \bT_{2,N} $ in
	\eqref{cL2} are 	
$$
		\bR_{2}^\Psi:=-\begin{pmatrix}
			\frac{\gamma}{2}\cR_{A} & \cR_{B}  \\ 
			0 & \frac{\gamma}{2}\cR_{D}
		\end{pmatrix} + \cE^{-1}\bR_1\cE \,,  
		\qquad
		\bT_{2,N} := -\left(\frac{\gamma}{2}\right)^2\begin{pmatrix}
			0 & 0\\
			\tR_{2,N} & 0
		\end{pmatrix}\,,
$$
	where $ \cR_B $, $\cR_{A} $, $\cR_{D}$, are  defined in \eqref{wtR1}, \eqref{cR2}, and $\bR_1$, $ {\mathtt R}_{2,N}  $ in \eqref{bRG}, \eqref{cR3}. 
Thus  the estimate \eqref{AS.est3} holds  by Lemmata \ref{tame_compo}, \ref{tame_pesudodiff}, \ref{conju.tr},  \eqref{ht.t}, \eqref{Ht.comm}, Proposition \ref{egorov}, Lemma \ref{DN_pseudo_est}, \eqref{beta.FGMP.est} and Lemmata 2.34, 2.32 in \cite{BM}. The estimates \eqref{AS.est4}-\eqref{AS.est6} are proved similarly.
\end{proof}

\subsection{Symmetrization of the  order $ 1/ 2 $}\label{subsec:princ}

The goal of this section is to symmetrize 
the order $ 1/2 $ of the quasi-periodic Hamiltonian operator 
 $\cL_2$ in \eqref{cL2}. From now on, we neglect the contribution of the operator $ \bP_{2}^\perp$, which will be conjugated in Section \ref{sec:conc}. For simplicity of notation we denote such operator $\cL_2$ as well. 
\\[1mm]
{\bf Step 1:} We first conjugate the operator $\cL_2$ in \eqref{cL2}, where we relabel the space variable $y\rightsquigarrow x$, by the real, symplectic, reversibility preserving and momentum preserving transformation
\begin{equation}\label{map_test_M}
\wt\cM:= \begin{pmatrix}
\Lambda & 0 \\ 0 & \Lambda^{-1}
\end{pmatrix}\,,\quad
\wt\cM^{-1} := \begin{pmatrix}
\Lambda^{-1} & 0 \\ 0 & \Lambda
\end{pmatrix}\,,
\end{equation}
where $\Lambda \in \Ops^{\frac14} $ is the Fourier multiplier
\begin{equation}\label{Lambda}
\Lambda := \tfrac{1}{\sqrt g}\pi_0 + M(D)\,, \quad \text{with inverse} \quad \Lambda^{-1}:= \sqrt g \pi_0 + M(D)^{-1} \in \Ops^{-\frac14}  \,,
\end{equation}
with $\pi_0 $
defined in \eqref{defpi0} and  (cfr. \eqref{eq:T_sym})
\begin{equation}\label{M(D)}
	M(D):=G(0)^\frac14 \big( g - (\tfrac{\gamma}{2})^2\pa_x^{-1}G(0)\pa_x^{-1} \big)^{-\frac14} \in \Ops^{\frac14}\,.
\end{equation}
We have the identities $ \Lambda^{-1} G(0) \Lambda^{-1}  = \omega (\gamma, D) $ and 
\begin{equation}\label{LCL}
\Lambda\big( g   -  \big(\tfrac{\gamma}{2}\big)^2 
\pa_x^{-1}G(0) \pa_x^{-1} \big) \Lambda  = \Lambda^{-1}G(0)\Lambda^{-1} + \pi_0 = \omega(\gamma, D)+\pi_0 \, ,  
\end{equation}
where $\omega(\gamma, D) \in \Ops^{\frac12} $ is defined in \eqref{eq:omega0}.
%In \eqref{Lambda} and \eqref{LCL} we mean that the symbols of $ M(D), M(D)^{-1} $ and $ \omega (\gamma, D) $ are extended to $ 0 $ at $ j = 0 $, multiplying them by the cut-off function $ \chi $ defined  in \eqref{cutoff}.

By \eqref{cL2} we compute
\begin{equation}
\begin{aligned}
\label{cL3}
\cL_3 := \wt\cM^{-1} \cL_2 \wt\cM= & \ \omega\cdot \partial_\vf +\tm_{1,\bar\tn}\pa_x  + 
\begin{pmatrix}
-\frac{\gamma}{2} G(0) \pa_x^{-1}  
&  - \Lambda^{-1} G(0)\Lambda^{-1} \\ 
\Lambda\big(a_1 - (\frac{\gamma}{2})^2\pa_x^{-1}G(0)\pa_x^{-1}\big) \Lambda
&  -\frac{\gamma}{2} G(0) \pa_x^{-1}
\end{pmatrix}    \\
& +
\begin{pmatrix}
0
& 0  \\
-(\frac{\gamma}{2})^2 \Lambda P_{-2,N}^{(2)}\Lambda &
0
\end{pmatrix} + \wt\cM^{-1}\bR_2^\Psi\wt\cM+  \wt\cM^{-1}\bT_{2,N}\wt\cM   \,.
\end{aligned}
\end{equation}
By \eqref{LCL},  \eqref{Lambda} and \eqref{M(D)}, we get
\begin{align}
	&\Lambda\big(a_1 - (\tfrac{\gamma}{2})^2\pa_x^{-1}G(0)\pa_x^{-1}\big) \Lambda
	  = \Lambda\big(g - (\tfrac{\gamma}{2})^2\pa_x^{-1}G(0)\pa_x^{-1}\big) \Lambda + \Lambda(a_1-g)\Lambda  \label{conto2}\\
	& \ \ \ \ \ \ \ = \omega(\gamma,D) + (a_1-g)\Lambda^2 + [\Lambda,a_1]\Lambda + \pi_0 \nonumber\\
	& \ \ \ \ \ \ \ = \big( 1 + \tfrac{a_1-g}{g} \big)\omega(\gamma,D) + \tfrac{a_1-g}{g}\big(g\Lambda^2 -\omega(\gamma,D)\big) + [\Lambda,a_1]\Lambda + \pi_0\nonumber \\
	& \ \ \ \ \ \ \ = a_2^2 \omega(\gamma,D) + \tfrac{a_1-g}{g}(\tfrac{\gamma}{2})^2 M(D)^2 \pa_x^{-1}G(0)\pa_x^{-1} + [\Lambda,a_1]\Lambda + \pi_0 + \tfrac{a_1-g}{g}\pi_0  \nonumber
\end{align}
where $a_2 $ is  the real quasi-periodic traveling wave function
(with $a_1 $ defined in Lemma \ref{lemma:riassu})
\begin{equation}\label{a6}
	a_2 := \sqrt{\tfrac{a_1}{g}} = \sqrt{1+\tfrac{a_1-g}{g}} \, , \quad \even(\vf,x) \,  
	\, .
\end{equation}
Therefore, by \eqref{cL3}, \eqref{LCL}, \eqref{conto2} we obtain
\begin{equation}\label{cL3_final}
	\begin{aligned}
	\cL_3 & = \omega\cdot\pa_\vf +\tm_{1,\bar\tn} \pa_x +\begin{pmatrix}
	-\frac{\gamma}{2}G(0)\pa_x^{-1} & -\omega(\gamma,D) \\ a_2\omega(\gamma,D)a_2 & -\frac{\gamma}{2}\pa_x^{-1}G(0)
	\end{pmatrix} + \begin{pmatrix}
	0 & 0 \\ \pi_0 & 0
	\end{pmatrix} \\
	 & + \begin{pmatrix}
	 0 & 0 \\ C_3  & 0
	 \end{pmatrix} + \bR_{3}^\Psi + \bT_{3,N} \,,
	\end{aligned}
\end{equation}
where
\begin{equation}
	C_3  := a_2 [a_2,\omega(\gamma,D)]+ \tfrac{a_1-g}{g}(\tfrac{\gamma}{2})^2 M(D)^2 \pa_x^{-1}G(0)\pa_x^{-1} + [\Lambda,a_1]\Lambda 
	- (\tfrac{\gamma}{2})^2 \Lambda P_{-2,N}^{(2)}\Lambda
	  \label{C3} 
%	P_{-\frac32,N}^{(3)} & :=\Lambda P_{-2,N}^{(2)}\Lambda \in \Ops^{-\frac32} \,, \label{PSC3}
\end{equation}
is in  $ \Ops^{-\frac12} $ 
% with $P_{-2,N}^{(2)}$ defined in \eqref{tildecP-1}, 
and
\begin{equation}\label{R3T3}
	\bR_3^\Psi := \wt\cM^{-1}\bR_2^\Psi\wt\cM +\begin{pmatrix}
	0 & 0 \\
	(\tfrac{a_1}{g}-1)\pi_0 & 0
	\end{pmatrix}\,, \quad \bT_{3,N}:=  \wt\cM^{-1}\bT_{2,N}\wt\cM\, . 
\end{equation}
The operator $\cL_3$ in \eqref{cL3_final} is Hamiltonian, reversible and momentum preserving.
\\[1mm]
{\bf Step 2:}
We now conjugate the operator  $\cL_3 $  in \eqref{cL3_final}  with the symplectic matrix
of multiplication  operators
\begin{equation*}
\label{Q}
\cQ := \begin{pmatrix}
q & 0 \\
0 & q^{-1}
\end{pmatrix} \ , \qquad
\cQ^{-1} := \begin{pmatrix}
q^{-1} & 0 \\
0 & q
\end{pmatrix}\,,
\end{equation*}
where  $ q $ is a real function, close to $ 1$,  to be determined, see \eqref{q}. 
We have that 
\begin{align}
\label{cL4}
\cL_4  := \cQ^{-1} \cL_3 \cQ  = 
\omega\cdot \partial_\vf +\tm_{1,\bar\tn}\pa_x  + 
 \begin{pmatrix}
A & B \\
C &  D 
\end{pmatrix} 
+  \cQ^{-1} (\bR_3^\Psi + \bT_{3,N}) \cQ \,,
\end{align}
where
(actually $ D = - A^* $, 
% in agreement  with the properties 
% of a Hamiltonian matrix, 
see  Definition \ref{def:HS})
\begin{align}
& A :=-\tfrac{\gamma}{2}  q^{-1}  G(0) \pa_x^{-1}  q +  \tm_{1,\bar\tn}q^{-1}  q_x + 
q^{-1} (\omega \cdot \pa_\vphi q)\,, \label{defA} \\
& B :=  - q^{-1}  \omega(\gamma,D)  q^{-1}  \,,\\
& C := q a_2 \omega(\gamma,D) a_2 q  + q \pi_0 q + q C_3 q   \,,
\\
&D :=  -\tfrac{\gamma}{2} q \pa_x^{-1} G(0)q^{-1}  - \tm_{1,\bar\tn} q^{-1} q_x - q^{-1}(\omega\cdot\pa_\vf q)  \, . \label{defD}
\end{align}
We choose the function $ q $ so that the coefficients 
of the highest order terms of the off-diagonal self-adjoint operators 
$ B $ and $  C $ satisfy $q^{-1} = q  a_2  $, namely 
as the  real quasi-periodic traveling wave, $\even(\vf,x)$ 
\begin{equation}
\label{q}
q(\vf,x) := a_2(\vf,x)^{-\frac12}\, .
\end{equation}
Thus $ \cQ $ is reversibility and momentum preserving. 

In view of \eqref{defA}-\eqref{defD} and \eqref{q}
the operator $\cL_4 $ in \eqref{cL4}  becomes
\begin{equation}
\begin{aligned}
\label{cL4_final}
\cL_4 & =  \omega\cdot \partial_\vf + \tm_{1,\bar\tn} \pa_x  + 
 \begin{pmatrix}
-\frac{\gamma}{2} G(0) \pa_x^{-1}  &  - a_{2}^\frac12  \omega(\gamma,D)a_2^\frac12 \\ a_2^\frac12 \omega(\gamma,D) a_2^\frac12
&  -\frac{\gamma}{2}\pa_x^{-1} G(0) 
\end{pmatrix}  \\
& \ \ \ + \begin{pmatrix}
0 & 0 \\ \pi_0 & 0
\end{pmatrix} +  \begin{pmatrix}
a_3 & 
0 \\
C_4 
 &
-a_3
\end{pmatrix} + \bR_4^\Psi +  \bT_{4,N}  \,,
\end{aligned}
\end{equation}
where $ a_3 $ is  the real quasi-periodic traveling wave function, $\odd(\vf,x)$,
\begin{align}
	a_3  & := \tm_{1,\bar\tn} q^{-1}q_x + q^{-1}(\omega\cdot\pa_\vf q) \label{a7} \, , 
	\quad C_4   := q  C_3   q 	\in\Ops^{-\frac12}\,, 
	\end{align} 
and 
$ \bR_{4}^\Psi , \bT_{4,N}$ are the smoothing remainders
(recall that $ G(0) \partial_x^{-1} = \cH T(\tth) $)
\begin{equation}
\begin{aligned}
& \bR_{4}^\Psi := 
 \begin{pmatrix}
- \frac{\gamma}{2}q^{-1} [ \cH T(\tth), q-1]  & 
0   \\
q\pi_0 q -\pi_0 &  
-\frac{\gamma}{2}[q-1,\cH T(\tth)]q^{-1}
\end{pmatrix}  + \cQ^{-1} \bR_{3}^\Psi\cQ 
\in \Ops^{-\infty}\,  , 
\label{bR4}
\\
& \bT_{4,N}:= \cQ^{-1} \bT_{3,N}\cQ \, . 
%\nonumber \label{T4N} 
\end{aligned}
\end{equation}
% with $\bR_{3}^\Psi$ and $\bT_{3,N}$ defined in \eqref{R3T3}.
The operator $\cL_4$ in \eqref{cL4_final} is Hamiltonian, reversible and momentum preserving.
\\[1mm]
{\bf Step 3:} 
We finally move in complex coordinates, conjugating the operator $\cL_4 $  in \eqref{cL4_final} via the %symplectic 
transformation $\cC$ defined in \eqref{C_transform}. 
The main result of this section is the following lemma. 

\begin{lem}\label{LEMMONE}
	Let $N \in \N $, $ \tq_0 \in \N_0 $. We have that 
	\begin{equation}\label{cL5}
	\begin{aligned}
	\cL_5 := &\, (\wt \cM \cQ \cC)^{-1} \cL_2 \wt \cM \cQ \cC
	\\ = &\,  \omega\cdot\pa_\vf + \tm_{1,\bar\tn} \pa_x + \im\, a_2(\vf,x)\b\Omega(\gamma,D)  + \im\,\b\Pi_0 + a_4 \cH+ \bR_5^{(-\frac12,d)} + \bR_5^{(0,o)} + \bT_{5,N} \,,
	\end{aligned}
	\end{equation}
	where:
\\[1mm]
1. The real 
				quasi-periodic traveling wave $a_2(\vf,x)$  defined in \eqref{a6},
				$ \even(\vf,x)$,  satisfies, 
	for some $\sigma=\sigma(k_0,\tau,\nu)>$ and for any $s_0 \leq s \leq S - \sigma$,
		\begin{equation}\label{m12.stima.f}
	\| a_2 - 1 \|_s^{k_0, \upsilon} \lesssim \varepsilon \upsilon^{-1} (1+\normk{\fI_0}{s+\sigma}) \,;
		\end{equation}
2. $\b\Omega(\gamma, D)$ is the matrix of Fourier multipliers  
		(see \eqref{eq:lin00_ww_C}, \eqref{Omega})
		\begin{equation}
		\label{step5.est1}
		\b\Omega(\gamma, D)=
		\begin{pmatrix}
		\Omega(\gamma, D) &  0\\
		0  & - \bar{\Omega(\gamma, D)} 
		\end{pmatrix}, \quad 
		\Omega(\gamma, D) = \omega (\gamma, D) + \im \,\frac{\gamma}{2}\partial_x^{-1} G(0) \, ; 
		\end{equation}
3. The operator 
\begin{equation*}\label{Pi0}
\b\Pi_0:=
%-\im \,\cC^{-1}\begin{pmatrix}
%0 & 0 \\ \pi_0 & 0
%\end{pmatrix}\cC = 
\frac{1}{2}\begin{pmatrix}
\pi_0 & \pi_0 \\ - \pi_0 & -\pi_0
\end{pmatrix}\,.
\end{equation*}
4. The real quasi-periodic traveling wave 
	$ a_4(\vf,x) := \tfrac{\gamma}{2}(a_2(\vf,x)-1) $, $\even(\vf,x) $, 
	satisfies, 
		for some $\sigma:= $ $ \sigma(k_0,\tau, \nu)>0$ and for all $s_0 \leq s \leq S- \sigma$, 
		\begin{equation}
		\label{step5.est3}
	\normk{a_4}{s}   \lesssim_{s} \varepsilon \upsilon^{-1}  ( 1 + \normk{\fI_0}{s+\sigma} ) \,;
		\end{equation}
5. $\bR_5^{(-\frac12,d)}  \in \Ops^{- \frac12} $ and $ \bR_5^{(0,o)} 
		\in \Ops^{0} $ 
		are pseudodifferential operators % in $\Ops^{0}$ 
		of the form
		\begin{align}\footnotesize
		\label{step5.est4}
		\bR_5^{(-\frac12,d)} :=
		\begin{pmatrix}
		r_5^{(d)}(\vf, x, D)  & 0 \\
		0  &  \bar{r_5^{(d)}(\vf,x, D)}
		\end{pmatrix} , 
		\quad
		\bR_5^{(0,o)} :=
		\begin{pmatrix}
		0 & r_5^{(o)}(\vf,x, D)  \\
		\bar{r_5^{(o)}(\vf,x, D)} & 0 
		\end{pmatrix} \, , \nonumber
		\end{align} 
		reversibility and momentum preserving,  satisfying, 
		for  some $\sigma_N := \sigma(\tau, \nu, N)>0$,  
		for finitely many 
$ 0 \leq \alpha \le \alpha (M)  $ (fixed in Remark \ref{fix:alpha}),  and for all $s_0 \leq s \leq S-\sigma_N - 3\alpha $,
		\begin{equation}
		\label{step5.est5}
	\normk{ \bR_5^{(-\frac12,d)}}{-\frac12,s,\alpha} +	\normk{ \bR_5^{(0,o)} }{0,s,\alpha} \lesssim_{ s,N, \alpha} \varepsilon\upsilon^{-1}  ( 1+ \normk{\fI_0}{s+\sigma_N +3\alpha} )\,;
		\end{equation}
6.		For any $ \tq \in \N^\nu_0 $ with $ |\tq| \leq \tq_0$, 
		$n_1, n_2 \in \N_0 $  with $ n_1 + n_2  \leq N -(k_0 + \tq_0) + \frac32  $,  the  
		operator $\langle D \rangle^{n_1}\partial_{\vphi}^\tq \bT_{5, N}(\vphi) \langle D \rangle^{n_2}$ is 
		$\cD^{k_0} $-tame with a tame constant satisfying, for some $\sigma_N(\tq_0) := \sigma_N(\tq_0,k_0,\tau,\nu)>0$ and for any $s_0 \leq s \leq S -\sigma_N(\tq_0)$,   
		\begin{equation}\label{step5.est6}
	{\mathfrak M}_{\langle D \rangle^{n_1}\partial_{\vphi}^\tq \bT_{5, N}(\vphi) \langle D \rangle^{n_2}}(s) \lesssim_{S, N, \tq_0} 
		\varepsilon\upsilon^{-1}  \big( 1+ \normk{\fI_0}{s+\sigma_N(\tq_0)} \big)\,;
		\end{equation}
7. The operators $\cQ^{\pm 1}$, $\cQ^{\pm 1}-{\rm Id}$, $(\cQ^{\pm 1}-{\rm Id})^*$ are $\cD^{k_0}$-tame
		with tame constants satisfying, for some $\sigma:=\sigma(\tau,\nu,k_0)>0$ and for all $s_0\leq s\leq S - \sigma$,
		\begin{align}
		& \fM_{\cQ^{\pm 1}}(s) \lesssim_{S} 1+ \normk{\fI_0}{s+\sigma} \,, \ \ 
	\fM_{\cQ^{\pm 1}-{\rm Id}}(s) + \fM_{\left(\cQ^{\pm 1}-{\rm Id}\right)^*}(s) \lesssim_{S}\varepsilon \upsilon^{-1}  ( 1+ \normk{\fI_0}{s+\sigma})\label{step5.est12} \, . 
		\end{align}
8. Furthermore, for any $s_1$ as in \eqref{s1s0}, 
		finitely many 
$ 0 \leq \alpha \le \alpha (M)  $, $\tq\in\N_0^\nu$, with $\abs\tq\leq \tq_0$, and $n_1,n_2 \in\N_0$, with $n_1+n_2\leq N- \tq_0 + \frac12 $,
		\begin{align}
		& \| \Delta_{12} (\cA) h  \|_{s_1} \lesssim_{s_1} \varepsilon {\upsilon^{-1}}\norm{i_1-i_2}_{s_1+\sigma}\norm{h}_{s_1+\sigma} \,, \quad \cA \in \{ \cQ^{\pm 1} = (\cQ^{\pm 1})^*  \}\,, \label{step5.est131} \\ 
		&  \|\Delta_{12}a_2 \|_{s_1}  \lesssim_{s_1} 
		\varepsilon {\upsilon^{-1}}\norm{i_1-i_2}_{s_1+\sigma} \,,  \ 
		\| \Delta_{12} a_4 \|_{s_1}  \lesssim 
		\varepsilon{\upsilon^{-1}} \norm{i_1-i_2}_{s_1+\sigma}  \, , 
		\label{step5.est71}\\
		& \| \Delta_{12} \bR_5^{(-\frac12,d)} \|_{-\frac12,s_1,\alpha} + \|\Delta_{12} \bR_5^{(0,o)}\|_{0,s_1,\alpha} \lesssim_{s_1,N,\alpha} \varepsilon{\upsilon^{-1}} \norm{ i_1-i_2 }_{s_1+\sigma_N+2\alpha}\,,\label{step5.est8}\\
		& \norm{\braket{D}^{n_1}\pa_\vf^\tq \Delta_{12}\bT_{5,N}(\vf)\braket{D}^{n_2} }_{\cL(H^{s_1})} \lesssim_{s_1, N, \tq_0} \varepsilon{\upsilon^{-1}}  \norm{ i_1-i_2}_{s_1+\sigma_N(\tq_0)} \,.\label{step5.est9} 
		\end{align}
	The real operator $\cL_5$ is Hamiltonian, reversible and momentum  preserving. 
\end{lem}
\begin{proof}
	By the expression of $ \cL_4 $ in \eqref{cL4_final}  and \eqref{C_transformed} 
	we obtain that $ {\cal L}_5 $ has the form  \eqref{cL5}
	with	
	\begin{equation}
	\begin{aligned}
		r_5^{(d)} & := 
		\tfrac{\gamma}{2}  (a_2 - 1)\cH (T(\tth)-1) + 
		\im\big( \tfrac12 C_4  +a_{2}^\frac12[\omega(\gamma,D),  a_2^\frac12] \big)  \in \Ops^{-\frac12} \,,  \\ 
		 r_5^{(o)} & := a_3 + \tfrac{\im}{2} C_4 \in\Ops^{0}  \label{r5dr50}
	\end{aligned}
	\end{equation}
	(with $  C_4$ given  in \eqref{a7}) and 
$		\bT_{5,N}:=\cC^{-1}(\bR_{4}^\Psi+\bT_{4,N})\cC $. 
	The function $q$ defined in \eqref{q},
	with $ a_2 $  in  \eqref{a6}, 
	satisfies, by \eqref{AS.est1} and Lemma 
	\ref{compo_moser},
	 for all $s_0 \leq s  \leq S-\sigma$,
	\begin{equation}\label{q.est}
	 \normk{q^{\pm 1}-1}{s} \lesssim_{s} \varepsilon{\upsilon^{-1}} ( 1+\normk{\fI_0}{s+\sigma} ) \,.
	\end{equation}
	The estimates \eqref{m12.stima.f} and \eqref{step5.est3} 
	follows by  \eqref{a6} and \eqref{q.est}.
The estimate \eqref{step5.est5} follows by  \eqref{r5dr50}, \eqref{m12.stima.f}, 
	\eqref{q},
	 \eqref{C3}, \eqref{a6}, \eqref{AS.est1}, \eqref{AS.est2}, \eqref{a7},  \eqref{Lambda}, 
	  \eqref{eq:T_sym}, 
	Lemma \ref{lemma:riassu}. 
The estimate \eqref{step5.est6} follows by 
\eqref{bR4}, \eqref{R3T3}, \eqref{Ht.comm}, \eqref{AS.est3}, \eqref{AS.est1}
	Lemmata \ref{tame_compo}, \ref{tame_pesudodiff}, \eqref{q.est}.
	The estimates \eqref{step5.est12} 
	follow by 
	Lemmata \ref{tame_pesudodiff} and \eqref{q.est}.
		The estimates \eqref{step5.est131}-
		\eqref{step5.est9} are proved similarly.
\end{proof}

\subsection{Symmetrization up to smoothing remainders}\label{sec:block_dec}
The goal of this section is to transform the operator $\cL_5$ in \eqref{cL5} into the operator $\cL_{6}$ in \eqref{cL6M} which is block diagonal up to a regularizing remainder. 
From this step we do not preserve any further the Hamiltonian structure, 
but only the reversible and momentum preserving one
(it is  sufficient for proving Theorem \ref{NMT}).   

\begin{lem}\label{block_dec_lemma}
	Fix $ \fm, N \in \N $,  $ \tq_0 \in \N_0$.	There exist  
	real, reversibility  and momentum  preserving operator matrices 
	$\{ \bX_k \}_{k=1}^{\fm}$   of the form
	\begin{equation}\label{geno}
	\bX_k:= \begin{pmatrix}
	0 & \chi_k(\vf,x,D) \\ \bar{\chi_k(\vf,x,D)} & 0   
	\end{pmatrix},
	\qquad
	\chi_k(\vf,x,\xi) \in S^{-\frac{k}{2}} \, ,
	\end{equation}
	such that, conjugating the operator $ \cL_5 $ in \eqref{cL5}  via the map 
	\begin{equation}\label{bPHIM}
	\b\Phi_{\fm}:= e^{\bX_1}\circ \cdots \circ e^{\bX_{\fm}} \, ,
	\end{equation} 
	we obtain the real, reversible and  momentum preserving operator 
	\begin{equation}\label{cL6M}
	\begin{aligned}
	\cL_6 & := \cL_{6}^{(\fm)} := \b\Phi_{\fm}^{- 1} \, \cL_5 \, \b\Phi_{\fm} \\
	&=  \omega\cdot\partial_\vf +\tm_{1,\bar\tn}\pa_x + \im \,a_2 \b\Omega(\gamma, D) +\im\b\Pi_0 + a_4\cH + \bR_{6}^{(-\frac12, d)} + \bR_{6}^{(-\frac{\fm}{2}, o)} + \bT_{6,N}\,,
	\end{aligned}
	\end{equation}
	where:
\\[1mm]
1.	$\bR_{6}^{(-\frac12,d)}  $ is a block-diagonal operator 
	\begin{align*}
	\bR_{6}^{(-\frac12,d)} := \bR_{6,\fm}^{(-\frac12,d)} & := \begin{pmatrix}
	r_{6}^{(d)}(\vf,x,D) & 0 \\ 
	0 &\bar{r_{6}^{(d)}(\vf,x,D)}
	\end{pmatrix} \in \Ops^{-\frac12}  \,,
	\end{align*}
	$ \bR_{6}^{(- \frac{\tm}{2}, o)} $ is a smoothing off diagonal remainder 	
	\begin{align}
	\bR_{6}^{(- \frac{\tm}{2}, o)} := \bR_{6,\fm}^{(-\frac{\fm}{2}, o)}  & := \begin{pmatrix}
	0 & r_{6}^{(o)}(\vf,x,D) \\ 
	\bar{r_{6}^{(o)}(\vf,x,D)} & 0 
	\end{pmatrix} \in \Ops^{- \frac{\fm}{2}} \label{R6o} \, , 
	\end{align}
	satisfying
	for finitely many 
$ 0 \leq \alpha \le \alpha (\fm)  $ (fixed in Remark \ref{fix:alpha}), 
	for  some $ \sigma_N := \sigma_N(k_0,\tau, \nu, N)>0 $,   
	$\aleph_{\fm}(\alpha) > 0 $ and for all 
	$ s_0 \leq s \leq S- \sigma_N-\aleph_{\fm}(\alpha)$, 
	\begin{align}
	&\normk{\bR_{6}^{(-\frac12,d)}}{-\frac12,s,\alpha} + 
	\normk{\bR_{6}^{(-\frac{\fm}{2}, o)}}{-\frac{\fm}{2},s,\alpha} \lesssim_{ s, \fm, N, \alpha} \varepsilon {\upsilon^{-1}} \big( 1+\normk{\fI_0}{s+\sigma_N+\aleph_{\fm}(\alpha)} \big) \, \label{bR6esti1} .
	\end{align}
	Both	$\bR_{6}^{(-\frac12,d)}  $ and $ \bR_{6}^{(- \frac{\fm}{2}, o)} $ are 
	reversible and momentum preserving;
\\[1mm] 
2. 
	For any $ \tq \in \N^\nu_0 $ with $ |\tq| \leq \tq_0$, 
	$n_1, n_2 \in \N_0 $  with $ n_1 + n_2 \leq N -(k_0+\tq_0) + \frac32  $,  the  
	operator $\langle D \rangle^{n_1}\partial_{\vphi}^\tq \bT_{6, N}(\vphi) \langle D \rangle^{n_2}$ is 
	$\cD^{k_0} $-tame with a tame constant satisfying, for some  $\sigma_N(\tq_0) := 
	\sigma_N(k_0,\tau, \nu, \tq_0) $, for 
	any $s_0 \leq s \leq S -\sigma_N(\tq_0) -\aleph_{\fm}(0)$, 
	\begin{equation}\label{block.est2}
	{\mathfrak M}_{\langle D \rangle^{n_1}\partial_{\vphi}^\tq \bT_{6, N}(\vphi) \langle D \rangle^{n_2}}(s) \lesssim_{S, \fm, N, \tq_0} 
	\varepsilon {\upsilon^{-1}}( 1+ \normk{\fI_0}{s+\sigma_N(\tq_0) + \aleph_{\fm}(0)} )\,.
	\end{equation}
3.
	The conjugation map $ \b\Phi_{\fm}$ in \eqref{bPHIM} satisfies, 
	for all $s_0 \leq s \leq S-\sigma_N-\aleph_{\fm}(0)$,
	\begin{equation}\label{block.est6}
	\normk{ \b\Phi_{\fm}^{\pm 1}-{\rm Id} }{0,s,0} + \normk{\left(\b\Phi_{\fm}^{\pm 1}-{\rm Id}\right)^*  }{0,s,0} \lesssim_{s, \fm, N} \varepsilon{\upsilon^{-1}} ( 1+\normk{\fI_0}{s+\sigma_N+ \aleph_{\fm}(0)} )\,.
	\end{equation}
4. 
	Furthermore, for any $s_1$ as in \eqref{s1s0}, finitely many 
$ 0 \leq \alpha \le \alpha (\fm)  $, $\tq\in\N_0^\nu$, with $\abs\tq \leq \tq_0$, and $n_1,n_2\in\N_0$, with $n_1+n_2\leq N- \tq_0 + \frac12$, we have
	\begin{align}
	&\|\Delta_{12} \bR_{6}^{(-\frac12,d)} \|_{-\frac12,s_1,\alpha} +\|\Delta_{12} \bR_{6}^{(-\frac{\fm}{2},o)} \|_{-\frac{\fm}{2}, s_1, \alpha} \lesssim_{ s_1, \fm, N, \alpha} \varepsilon{\upsilon^{-1}} \norm{ i_1-i_2 }_{s_1+\sigma_N+\aleph_{\fm}(\alpha)}   \,, \label{block.est3} \\
	& \| \braket{D}^{n_1} \pa_\vf^\tq \Delta_{12} \bT_{6,N} \braket{D}^{n_2}\|_{\cL(H^{s_1})} \lesssim_{s_1,
		\fm, N, \tq_0} \varepsilon{\upsilon^{-1}} \norm{i_1-i_2}_{s_1+ \sigma_N(\tq_0)+ \aleph_{\fm}(0) }\,,\label{block.est4}\\
	& \|\Delta_{12} \b\Phi_{\fm}^{\pm 1} \|_{0,s_1,0} +\|\Delta_{12} (\b\Phi_{\fm}^{\pm 1})^* \|_{0,s_1,0} \lesssim_{s_1, \fm, N} \varepsilon {\upsilon^{-1}}\norm{ i_1-i_2 }_{s_1+\sigma_N+\aleph_{\fm}(0)} \,. \label{block.est5}
	\end{align}
\end{lem}
\begin{proof}
	The proof is inductive. 
	The operator  $\cL_6^{(0)}:= \cL_5 $ satisfies \eqref{bR6esti1}-\eqref{block.est2}  
	with $	\aleph_0(\alpha) := 3\alpha $, by 
	\eqref{step5.est5}-\eqref{step5.est6}.
	Suppose we have done already $\fm$ steps obtaining an operator
	$ \cL_6^{(\fm)} $ as in \eqref{cL6M} with 
	$  \bR_{6,\fm}^{(-\frac12, d)} := \bR_{6}^{(-\frac12, d)} $ and 
	$  \bR_{6,\fm}^{(-\frac12, o)} := \bR_{6}^{(-\frac{\fm}{2}, o)}  $
	and	the remainder 
	$ {\bf \Phi}_{\fm}^{-1} \bT_{5,N} {\bf \Phi}_{\fm}$, instead of 
	$ \bT_{6,N} $.
	  We now 
	show how to define $ \cL_6^{(\fm+1)} $.
	Let 
	\begin{equation}\label{chiM+1}
	\chi_{\fm+1}(\vf,x,\xi):= -\big(2\im\,  a_2(\vf,x) \omega (\gamma, \xi) \big)^{-1}r_{6,\fm}^{(o)}(\vf,x,\xi) \chi (\xi) 
	\in S^{-\frac{\fm}{2}-\frac12} \, , 
	\end{equation}
	where $ \chi $ is the cut-off function defined in \eqref{cutoff} and 
	$\omega (\gamma, \xi) $ is the symbol 
	(cfr. \eqref{eq:omega0}) 
	$$
	\omega (\gamma, \xi):= 
	\sqrt{  G(0; \xi) \Big(g +  
		\frac{\gamma^2}{4} \frac{G(0; \xi)}{\xi^2}    \Big)} \in S^{\frac12} \, , 
	\ \ G(0; \xi) := 
	\begin{cases}
	\chi (\xi) |\xi| \tanh (\tth |\xi|) \, , \  \tth < + \infty \cr
	\chi (\xi) |\xi| \, , \qquad \qquad 	\ \,  \ \tth = + \infty \,.
	\end{cases}
	$$
	Note that
	$\chi_{\fm +1} $ in \eqref{chiM+1} is well defined because 
	$ \omega (\gamma, \xi) $ is positive on the support of $ \chi (\xi ) $ and
	 $a_2(\vf,x)$ is close to 1. 	
	We conjugate the operator $\cL_6^{(\fm)}  $ in \eqref{cL6M}  by the flow generated by 
	$\bX_{\fm+1}$ of the form \eqref{geno} with	
	$ \chi_{\fm+1} (\vphi, x, \xi) $ defined in \eqref{chiM+1}. 
	By	
	\eqref{bR6esti1} 
	and \eqref{step5.est1}, for suitable constants $\aleph_{\fm+1}(\alpha)>\aleph_{\fm}(\alpha)$, for finitely many $\alpha\in\N_0$ and for any $s_0 \leq s \leq S- \sigma_N - \aleph_{\fm+1}(\alpha)$,
	\begin{equation}\label{chiM+1_est}
	\normk{\bX_{\fm+1}}{-\frac{\fm}{2}-\frac{1}{2},s,\alpha} \lesssim_{s, \fm, \alpha} \varepsilon {\upsilon^{-1}}\big( 1+ \normk{\fI_0}{s+\sigma_N + \aleph_{\fm+1}(\alpha)} \big) \,.
	\end{equation} 
	Therefore, by Lemmata \ref{Neumann pseudo diff}, \ref{pseudo_compo} and
	the induction assumption \eqref{block.est6} for $\b\Phi_{\fm}$, 
	the conjugation map $\b\Phi_{\fm+1}:= \b\Phi_{\fm}e^{\bX_{\fm+1}}$ is well defined 
	and satisfies estimate \eqref{block.est6} with $\fm+1$.
	By the Lie expansion \eqref{lie_abstract} 
	we have
	\begin{align}
	\label{conM+1}
	\cL_6^{(\fm+1)} & := e^{-\bX_{\fm+1}} \,  \cL_6^{(\fm)} \, e^{ \bX_{\fm+1}}  \\
	& =\omega\cdot \partial_\vf +\tm_{1,\bar\tn}\pa_x+ \im a_2 \b\Omega(\gamma,D)+\im\b\Pi_0 +a_4\cH + \bR_{6,\fm}^{(-\frac12,d)} \nonumber\\
	\nonumber
	& - \big[\bX_{\fm+1},\tm_{1,\bar\tn}\pa_x+ \im\,a_2\b\Omega(\gamma,D) \big] + \bR_{6,\fm}^{(-\frac{\fm}{2},o)} + {\bf \Phi}_{\fm+1}^{-1} \bT_{5,N} {\bf \Phi}_{\fm+1}  \\
	\label{geno12}
	& -\int_0^1 \LieTr{\tau \bX_{\fm+1}}{\big[\bX_{\fm+1} \, , \, \omega\cdot \partial_\vf +\im\b\Pi_0+ a_4\cH + \bR_{6,\fm}^{(-\frac12,d)}  \big]}\wrt \tau\\
	\label{geno13}
	& -\int_0^1 \LieTr{\tau\bX_{\fm+1}}{\left[\bX_{\fm+1}, \bR_{6,\fm}^{(-\frac{\fm}{2},o)}  \right]}\wrt \tau \\
	\label{geno14}
	& +\int_0^1(1-\tau)\LieTr{\tau\bX_{\fm+1}}{\left[\bX_{\fm+1},\left[\bX_{\fm+1},\tm_{1,\bar\tn}\pa_x +\im\, a_2\b\Omega(\gamma,D)\right]\right]}\wrt \tau \, .
	\end{align}
	In view of  \eqref{geno}, \eqref{step5.est1} and \eqref{R6o},  we have that
	$$
	-\big[ \bX_{\fm+1},\tm_{1,\bar\tn}\pa_x + \im\, a_2 \b\Omega(\gamma,D) \big] + \bR_{6,\fm}^{(- \frac{\fm}{2},o)}  
	=
	\begin{pmatrix}
	0 & Z_{\fm+1} \\ 
	\bar{Z_{\fm+1}} & 0 
	\end{pmatrix} =: \bZ_{\fm+1} \, , 
	$$
	where, denoting for brevity  $ \chi_{\fm+1} := \chi_{\fm+1}(\vf,x,\xi) $, it results 
	\begin{equation}\label{eq:hom_scalar}
		\begin{aligned}
			Z_{\fm+1} & 
			= 
			\im
			\left(  {\rm Op}(\chi_{\fm+1}) a_2 \, \omega(\gamma,D) + a_2 \, \omega(\gamma,D)
			{\rm Op}(\chi_{\fm+1})
			\right) \nonumber \\
			&	\quad +  \left[ 
			{\rm Op}(\chi_{\fm+1}), - \tm_{1,\bar\tn} \pa_x + a_2 \,   \tfrac{\gamma}{2}\pa_x^{-1}G(0)  \right]		
			+ {\rm Op}(r_{6,\fm}^{(o)}) \, . 
		\end{aligned}
	\end{equation}
	By \eqref{compo_symb}, \eqref{eq:rem_comp_tame} 
	% Lemma \ref{pseudo_compo} 
	and since
	$  \chi_{\fm+1}  \in S^{-\frac{\fm}{2} -\frac12} $ by 
	\eqref{chiM+1}, 
	we have that 
	$$
	{\rm Op}(\chi_{\fm+1}) a_2 \omega(\gamma,D) + a_2\omega(\gamma,D) {\rm Op}(\chi_{\fm+1}) 
	= \Op\big(  2 a_2
	\omega (\gamma, \xi)\chi_{\fm+1} \big) + {\mathtt r}_{\fm+1} \, ,
	$$
	where ${\mathtt r}_{\fm+1} $ 
	is in $ \Ops^{-\frac{\fm}{2}-1} $. By  \eqref{chiM+1} and \eqref{eq:hom_scalar} 
	$$
	Z_{\fm+1}  =  \im {\mathtt r}_{\fm+1}  +
	 \left[ 
			{\rm Op}(\chi_{\fm+1}), - \tm_{1,\bar\tn} \pa_x + a_2  \tfrac{\gamma}{2}\pa_x^{-1}G(0)  \right]
	+ {\rm Op}(r_{6,\fm}^{(o)}(1- \chi (\xi)))  \in \Ops^{-\frac{\fm}{2}-\frac12} \, . 
	$$
	The
	remaining pseudodifferential operators in \eqref{geno12}-\eqref{geno14} 
	have order $ \Ops^{-\frac{\fm+1}{2}} $. 
	Therefore the operator $ \cL_6^{(\fm+1)} $ in \eqref{conM+1} has the form 
	\eqref{cL6M} at $ \fm+ 1 $ with 
	\begin{equation}\label{newRM+1}
	\bR_{6,\fm+1}^{(-\frac12,d)}+\bR_{6,\fm+1}^{(-\frac{\fm+1}{2},o)}
	:= \bR_{6,\fm}^{(-\frac12,d)}+\bZ_{\fm+1}+\eqref{geno12}+\eqref{geno13}+\eqref{geno14}
	\end{equation}
	and a smoothing remainder 
	$ {\bf \Phi}_{\fm+1}^{-1} \bT_{5,N} {\bf \Phi}_{\fm+1}$.
	By 
	Lemmata \ref{pseudo_compo}, \ref{pseudo_commu}, 
	\eqref{bR6esti1},
	\eqref{chiM+1_est}, \eqref{step5.est3}, we conclude that 
	$\bR_{6,\fm+1}^{(-\frac12,d)}$ and $\bR_{6,\fm+1}^{(-\frac{\fm+1}{2},o)}$ satisfy \eqref{bR6esti1} at order $\fm+1$ for
	suitable constants $ \aleph_{\fm+1} (\alpha) >  \aleph_{\fm} (\alpha) $. 
	Moreover the operator $\b\Phi_{\fm+1}^{-1} \bT_{5, N} \b\Phi_{\fm+1} $ satisfies 
	\eqref{block.est2} at order  $\fm+1$ by Lemmata \ref{tame_compo}, \ref{tame_pesudodiff} and \eqref{step5.est6}, \eqref{block.est6}.
	Estimates \eqref{block.est3}-\eqref{block.est5} follow similarly.
	By \eqref{chiM+1}, 
	Lemmata \ref{rev_defn_C}, \ref{lem:mom_pseudo},
	and the induction assumption that $\bR_{6,\fm}^{(-\frac{\fm}{2},o)}$ is reversible and momentum preserving, we get that $\bX_{\fm+1}$ is reversibility and momentum preserving, and so are $e^{\pm\bX_{\fm+1}}$. We deduce that
%	By the induction assumption $\cL_{6}^{(\fm)}$ is reversible and momentum preserving, and so 
	$\cL_{6}^{(\fm+1)}$ is reversible and momentum preserving, in particular
	% the terms  $ \bR_{6,\fm+1}^{(-\frac12,d)}
	$ \bR_{6,\fm+1}^{(-\frac{\fm+1}{2},o)} $
	in \eqref{newRM+1}.
\end{proof}

\begin{rem}\label{rem:after_block}
	The number of regularizing iterations $ \fm \in\N$ will be 
	fixed by the KAM reduction scheme in Section \ref{sec:KAM}, more precisely
we take 	$ \fm = 2 M $ with $ M $ in \eqref{M_choice}. Note that it is independent of the Sobolev index $s$. 
\end{rem}

So far the operator $ \cL_6 $ of  Lemma \ref{block_dec_lemma} depends on 
two indexes $ \fm, N  $ which provide respectively the order 
of the regularizing  off-diagonal remainder $ \bR_{6}^{(- \frac{\fm}{2}, o)} $ and of the smoothing 
tame operator $ \bT_{6,N} $.
From now on we  fix
\begin{equation}\label{M=N}  
\fm  := 2M \, , \ M \in \N \, ,   \quad N = M \, .
\end{equation}

\subsection{Reduction of the order 1/2}\label{sec:order12}
The goal of this section is to transform the operator $\cL_6$ in \eqref{cL6M}
with $ \fm  := 2M $, $N=M$ (cfr. \eqref{M=N}), into the 
operator $\cL_7$ in \eqref{cL7} whose  coefficient
in front of $\b\Omega(\gamma,D)$ is  constant. 
%We first eliminate the $ x$-dependence and then the $ \vf $-dependence.  
First we rewrite % the operator $\cL_{6}$ as % in \eqref{cL6M}, with $N=M$, as
$$
\cL_{6} = \omega\cdot \pa_\vf + \begin{pmatrix}
P_6 & 0 \\ 0 & \bar{P_6}
\end{pmatrix} + \im\b\Pi_0 + \bR_{6}^{(-M,o)} + \bT_{6,M}\, ,
$$
having denoted 
\begin{equation}
\label{P6}
P_6 := P_6(\vf,x,D) := \tm_{1,\bar\tn} \pa_x + \im a_2(\vf,x) \Omega(\gamma, D) + a_4 \cH +  r_6^{(d)}(\vf,x, D) \,.
\end{equation} 
We  conjugate $\cL_{6}$ 
through the real operator 
\begin{equation}\label{Phi1}
\b\Phi(\vf) := \begin{pmatrix}
\Phi(\vf) & 0 \\ 0 & \bar{\Phi}(\vf)
\end{pmatrix} 
\end{equation}
where $\Phi(\vf):=\Phi^\tau(\vf)|_{\tau=1} $ is the time $ 1 $-flow of the PDE
\begin{equation}\label{phi.problem}
\begin{cases}
\partial_\tau \Phi^\tau(\vf) =  \im A(\vf)  \Phi^\tau(\vf) \, , \\
\Phi^0(\vf) = {\rm Id} 	\,  , 
\end{cases}   \qquad A(\vf) := b(\vf, x) |D|^{\frac12} \, , 
\end{equation}
and $b(\vf,x)  $ is a real quasi-periodic traveling wave, $\odd(\vf,x)$,  
chosen later, see \eqref{eqo1}. 
Thus 
$ \im b(\vf, x) |D|^{\frac12}$ is reversibility and momentum preserving as well as 
$	\b\Phi (\vf)  $.
Moreover
$\Phi\pi_0 = \pi_0 = \Phi^{-1}\pi_0$, which implies
\begin{equation}\label{phipi0}
\b\Phi^{-1}\b\Pi_0 \b\Phi = \b\Pi_0\b\Phi \,.
\end{equation}
By the Lie expansion \eqref{lie_abstract} we have  
\begin{equation}
\begin{aligned}
\Phi^{-1} P_6 \Phi  & = 
P_6 - \im [A, P_6] - \frac12 [A, [A, P_6]]+  \sum_{n=3}^{2M+1} \frac{(-\im)^n}{n!} \ad_{A(\vf)}^n(P_6) + T_{M}\, ,\label{lin1} \\
T_{M}& :=  \frac{(- \im)^{2M+2}}{(2M+1)!} \int_0^1 (1 - \tau)^{2M+1} \Phi^{-\tau}(\vf)\,
\ad_{A(\vf)}^{2M+2}(P_6) \,\Phi^\tau (\vf)  \di \tau \, ,   
\end{aligned}
\end{equation}
and, by \eqref{lie_omega_devf}, 
\begin{align}
\Phi^{-1} \circ \omega\cdot \pa_\vf \circ \Phi & = \omega\cdot\pa_\vf + \im 
(\omega\cdot\pa_\vf A) + \frac12 [A,\omega\cdot\pa_\vf A] -  \sum_{n=3}^{2M+1} \frac{(-\im)^n}{n!} \ad_{A(\vf)}^{n-1}(\omega\cdot \pa_\vf A(\vf)) + T_{M}'\,, \notag \\
T_{M}' & := - \frac{(- \im)^{2M+2}}{(2M+1)!} 
\int_0^1 (1 - \tau)^{2M+1} \Phi^{-\tau}(\vf) \,
\ad_{A(\vf)}^{2M+1}(\omega\cdot \pa_\vf A(\vf))\, \Phi^\tau (\vf) \di \tau \, . \label{lin2}
\end{align}
Note that $ \ad_{A(\vf)}^{2M+2}(P_6) $ and $ \ad_{A(\vf)}^{2M+1}(\omega\cdot \pa_\vf A(\vf)) $ are in $ \Ops^{-M} $. 
We now determine the pseudo-differential term of order $ 1/ 2 $
in \eqref{lin1}-\eqref{lin2}. 
We use the expansion of the linear dispersion operator $ \Omega(\gamma,D) $, 
defined by  \eqref{Om-om}, \eqref{def:Gj0}, and, since $j\to c_j(\gamma)\in S^{0}$ (see \eqref{cj_zero}),  
\begin{equation}\label{Omegakx}
	\Omega(\gamma,D) = \sqrt{g}|D|^\frac12 + \im\,\tfrac{\gamma}{2}\cH + r_{-\frac12}(\gamma,D) \,,  \quad % |D|^\frac12=\Op(|\xi|^\frac12\chi(\xi)) \, , \
	r_{-\frac12}(\gamma, D)\in \Ops^{-\frac12}  \, , 
\end{equation}
where $\cH$ is the Hilbert transform in \eqref{Hilbert-transf}.
By  \eqref{P6}, that $ A = b  |D|^{\frac12} $, 
\eqref{eq:moyal_exp}, \eqref{Omegakx} 
we get
	\begin{align}\label{sviP6AP6}
		[A,P_6] & = \big[ b|D|^\frac12, \tm_{1, \tn}\pa_x + \im\,\sqrt{g} a_2 |D|^\frac12 +(a_4-\tfrac{\gamma}{2}a_2)\cH + r_6^{(d)}(x,D) + \im\,a_2 r_{-\frac12}(\gamma,D)\big]
		  \notag \\
		& = -\tm_{1,\bar\tn} b_x |D|^\frac12 -\im\tfrac{\sqrt g}{2}(b_x a_2 -(a_2)_x b)\cH + \Op(r_{b,-\frac12}) \,,
	\end{align}
where $ r_{b,-\frac12} \in S^{-\frac12} $ is small with $b$.
As a consequence,
the contribution at order $\frac12$  of the operator 
 $\im \,\omega\cdot\pa_\vf A + P_6 - \im [A, P_6]  $
is 
 $  \im\big(\omega\cdot\pa_\vf b + \tm_{1,\bar\tn}  b_x + \sqrt{g}\, a_2) |D|^{\frac12} $.
We choose $b(\vf,x)$ as the solution of
\begin{equation}\label{solu12}
(\omega\cdot\pa_\vf  + \tm_{1,\bar\tn} \pa_x )b  + 
\sqrt{g} \, \Pi_{N_{\bar\tn}}  \,a_2 
 = \sqrt{g}\, \tm_\frac12 
\end{equation}
where $\tm_{\frac12}  $ is the average (see \eqref{def:avera})
\begin{equation}
\label{tm12}
	\tm_{\frac12} 
	:= \braket{a_2}_{\vf,x} 
	\,.
\end{equation}
We define $ b(\vf,x) $ to be the real, $\odd(\vf,x)$,
quasi-periodic traveling wave 
\begin{equation}\label{eqo1}
		b(\vf,x) := - \sqrt{g} (\omega\cdot\pa_\vf + \tm_{1,\bar\tn}\pa_x)_{\rm ext}^{-1}\big(\Pi_{N_{\bar\tn}} a_2(\vf,x) - \tm_{\frac12} \big) 
\end{equation}
recall \eqref{paext}. 
Note that $b(\vf,x) $ and $ \tm_{\frac12} $ are 
defined for any $(\omega,\gamma)\in\R^\nu\times[\gamma_1,\gamma_2]$
and that, 
for any $(\omega,\gamma)\in\tT\tC_{\bar\tn+1}(2\upsilon,\tau)$
 defined in \eqref{tDtCn}, it solves \eqref{solu12}.

We deduce by \eqref{lin1}, \eqref{lin2}, \eqref{P6}, \eqref{sviP6AP6}-\eqref{eqo1}, that,
for any $(\omega,\gamma)\in\tT\tC_{\bar\tn+1}(2\upsilon,\tau)$,  
\begin{equation*}
\begin{aligned} 
L_7&:= \Phi^{-1}(\vf)\left( \omega\cdot \pa_\vf + P_6 \right) \Phi(\vf) \\
&= \omega\cdot\pa_\vf + \tm_{1,\bar\tn} \pa_x + \im\,\tm_{\frac12} \Omega(\gamma, D)+ a_5 \cH + \Op(r_7^{(d)}) + T_M + T_M' +\im\sqrt{g}(\Pi_{N_{\bar\tn}}^\perp a_2)|D|^\frac12 \,,
\end{aligned}
\end{equation*}
where $a_5(\vf,x)$ is the real function (using that $ a_4 = \frac{\gamma}{2} (a_2 -1 )$)
\begin{equation}\label{a3d}
\begin{aligned}	 
a_5:= & \,  
  \tfrac{\gamma}{2}( \tm_{\frac12} - 1 )
-\tfrac{\sqrt g}{2}(b_x a_2 -(a_2)_x b) \\
&+\tfrac{\tm_{1,\bar\tn}}{4} \big( b_{xx} b - b_x^2 \big) + \tfrac14\big( b(\omega\cdot\pa_\vf b)_x - (\omega\cdot\pa_\vf b)b_x \big)\,,
\end{aligned}
\end{equation}
and 
\begin{align}
	&\Op(r_7^{(d)}):= 
	\Op( - \im r_{b,-\frac12} + \im\, (a_2 - \tm_{\frac12})
	r_{-\frac12}(\gamma,D) +  r_6^{(d)}) \notag  \\
	&\ \ \ + \tfrac12 \big[ b|D|^\frac12, \im\tfrac{\sqrt g}{2}(b_x a_2 -(a_2)_x b)\cH - \Op(r_{b,-\frac12})\big] + \tfrac12\Op(\wtr_{2}(b|\xi|^\frac12,(\tm_{1,\bar\tn}b_x + \omega \cdot \pa_\vf b) |\xi|^{\frac12})) \notag \\
	& \ \ \ +  \sum_{n=3}^{2M+1} \frac{(-\im)^n}{n!} \ad_{A(\vf)}^n(P_6)- \sum_{n=3}^{2M+1} \frac{(-\im)^n}{n!} \ad_{A(\vf)}^{n-1}(\omega\cdot \pa_\vf A(\vf))  \in \Ops^{-\frac12} \, , \label{r7d}
\end{align}
with  $\wtr_{2}(\,\cdot\,,\,\cdot\,)$ defined  in \eqref{eq:moyal_exp}.  
In conclusion we have the following lemma. 

\begin{lem}\label{red1}
	 Let $ M \in \N $, $ \tq_0 \in \N_0 $. 
	Let  $b(\vf,x)$ 
	be the quasi-periodic traveling wave function $\odd(\vf,x)$, defined in  
	\eqref{eqo1}.
	Then, for any $\bar\tn\in\N_0$,  conjugating $ \cL_6  $ in \eqref{cL6M} 
	via the invertible, real, reversibility and momentum preserving 
	map $ \b\Phi $ defined in \eqref{Phi1}-\eqref{phi.problem}, we obtain,
	for any $ (\omega,\gamma) \in \tT\tC_{\bar\tn+1} (2\upsilon, \tau ) $,   
	the real, reversible and  momentum preserving operator 
	\begin{equation}
	\begin{aligned}\label{cL7}
	\cL_7  & :=   {\bf \Phi}^{-1} \cL_6 {\bf \Phi}\\
	& = \omega\cdot\partial_\vf +\tm_{1,\bar\tn}\,\pa_x + \im \,\tm_{\frac12} \b\Omega(\gamma, D) + a_5\cH + \im\b\Pi_0 + \bR_{7}^{(-\frac12, d)} + \bT_{7,M} + \bQ_7^\perp \,,
	\end{aligned}
	\end{equation}
    defined	for any $(\omega,\gamma)\in\R^\nu\times[\gamma_1,\gamma_2]$,  where:
\\[1mm]
1.		The real constant $ \tm_\frac12$ defined in \eqref{tm12} satisfies 
		$	| \tm_\frac12-1|^{k_0, \upsilon} \lesssim 
		\varepsilon {\upsilon^{-1}}$;
\\[1mm]
2.
		The real, quasi-periodic traveling wave function $ a_5(\vf,x)  $ defined in \eqref{a3d},
		$\even(\vf,x)$,
		satisfies, for some $\sigma= \sigma(\tau,\nu,k_0)>0$, for all $s_0\leq s \leq S-\sigma$,
		\begin{align}\label{a2d.est}
		&  \normk{a_5}{s}\lesssim_{s} \varepsilon {\upsilon^{-2}} ( 1 + \normk{\fI_0}{s+\sigma} )\, , \quad
	|\langle a_5 \rangle_{\varphi,x}|^{k_0, \upsilon} 
	\lesssim \varepsilon{\upsilon^{-1}}  \, ; 
		\end{align}
3. $ \bR_{7}^{(-\frac12,d)} $ is  the block-diagonal operator 
		\begin{align*}
		\bR_{7}^{(-\frac12,d)} & := \begin{pmatrix}
		r_{7}^{(d)}(\vf,x,D) & 0 \\ 
		0 &\bar{r_{7}^{(d)}(\vf,x,D)}
		\end{pmatrix} \in \Ops^{-\frac12} 
		\end{align*}
		with $ r_{7}^{(d)}(\vf,x,D)  $ defined in \eqref{r7d}, 
		that satisfies  for finitely many 
$ 0 \leq \alpha \le \alpha (M)  $ (fixed in Remark \ref{fix:alpha}),  
		for  some  $\sigma_M (\alpha):= \sigma_M(k_0,\tau, \nu, \alpha)>0$ and
		for all $s_0 \leq s \leq S- \sigma_M(\alpha)$, 
		\begin{align}
		&\normk{\bR_{7}^{(-\frac12,d)} }{-\frac12,s,\alpha}  \lesssim_{s, M, \alpha} \varepsilon {\upsilon^{-2}}( 1+\normk{\fI_0}{s+\sigma_M( \alpha)} ) \, \label{bR8.est} \,;
		\end{align}
4.  For any $ \tq \in \N^\nu_0 $ with $ |\tq| \leq \tq_0$, 
		$ n_1, n_2 \in \N_0 $  with $ n_1 + n_2  \leq M - \frac32(k_0+\tq_0)  + \frac32  $,  the  
		operator $\langle D \rangle^{n_1}\partial_{\vphi}^\tq \bT_{7,M}(\vphi) \langle D \rangle^{n_2}$ is 
		$\cD^{k_0} $-tame with a tame constant satisfying, for some 
		$ \sigma_M(\tq_0) := \sigma_M(k_0,\tau, \nu, \tq_0) $,
		for any $s_0 \leq s \leq S - \sigma_M(\tq_0) $, 
		\begin{equation}\label{red1.estT}
		{\mathfrak M}_{\langle D \rangle^{n_1}\partial_{\vphi}^\tq \bT_{7,M}(\vphi) \langle D \rangle^{n_2}}(s) \lesssim_{S, M, \tq_0} 
		\varepsilon {\upsilon^{-2}}( 1+ \normk{\fI_0}{s+\sigma_M(\tq_0)} )\,;
		\end{equation}
5. The operator $\bQ_{7}^\perp$ is % defined by
		\begin{equation}\label{R7perp}
			\bQ_{7}^\perp := \im\sqrt g (\Pi_{N_{\bar\tn}}^\perp a_2)|D|^\frac12 \begin{pmatrix}
			1 & 0 \\ 0 & -1
			\end{pmatrix}\,,
		\end{equation}
		where $a_2(\vf,x)$ is defined in \eqref{a6} and satisfies 
		\eqref{m12.stima.f};
\\[1mm]		
6. The operators $ \b\Phi^{\pm 1} -{\rm Id}$, $(\b\Phi^{\pm 1}-{\rm Id})^*$ are $\cD^{k_0}$-$\frac12(k_0+1)$-tame, with tame constants satisfying, for some $\sigma>0$ and for all $s_0\leq s \leq S-\sigma$,
		\begin{align}
		& \fM_{\b\Phi^{\pm 1} -{\rm Id}}(s) + \fM_{(\b\Phi^{\pm 1}-{\rm Id})^*}(s) \lesssim_{S} \varepsilon{\upsilon^{-2}}( 1 + \normk{\fI_0}{s+ \sigma}) \label{red1.est4}\,.
		\end{align}
7.	Furthermore, for any $s_1$ as in \eqref{s1s0}, 
	 finitely many 
$ 0 \leq \alpha \le \alpha (M)  $, $\tq\in\N_0^\nu$, with $\abs\tq \leq \tq_0$, and $n_1,n_2\in\N_0$, with $n_1+n_2\leq M-\frac32 \tq_0 $, we have
	\begin{align}
	& \| \Delta_{12} a_5 \|_{s_1} \lesssim_{s_1} \varepsilon {\upsilon^{-2}}\norm{i_1-i_2}_{s_1+\sigma} \,, 
	\  | \Delta_{12} \tm_\frac12|  
	\lesssim \varepsilon{\upsilon^{-1}} \norm{i_1-i_2}_{s_0+\sigma}  \, , 
	\label{red1.est1}  \\
	&\|\Delta_{12} \bR_{7}^{(-\frac12,d)} \|_{-\frac12,s_1,\alpha} \lesssim_{s_1, M, \alpha} \varepsilon {\upsilon^{-2}}\norm{ i_1-i_2 }_{s_1+\sigma_M (\alpha)}   \,, \label{red1.est2} \\
	& \| \braket{D}^{n_1} \pa_\vf^\tq \Delta_{12} \bT_{7,M} \braket{D}^{n_2}\|_{\cL(H^{s_1})} \lesssim_{s_1,M, \tq_0} \varepsilon {\upsilon^{-2}}\norm{i_1-i_2}_{s_1+ \sigma_M(\tq_0) }\,,\label{red1.est3} \\
	& \| \Delta_{12} (\cA) h \|_{s_1} \lesssim_{s_1} \varepsilon {\upsilon^{-2}}\norm{i_1-i_2}_{s_1+\sigma} \norm{h}_{s_1+\sigma} \,, \quad \cA \in \{ \b\Phi^{\pm 1}, (\b\Phi^{\pm 1})^*\}\,. \label{red1.est6} 
	\end{align}
\end{lem}
\begin{proof}
	The estimate $ | \tm_{\frac12}-1|^{k_0, \upsilon} \lesssim 
	\varepsilon \upsilon^{-1}$  follows by \eqref{tm12} and \eqref{m12.stima.f}.
	The function $b(\vf,x)$ defined in \eqref{eqo1} satisfies, by 
	\eqref{lem:diopha.eq} and  \eqref{m12.stima.f}, 
	\begin{equation}\label{b1b2rho.est}
		\normk{ b }{s} \lesssim_{s} \varepsilon \upsilon^{-2}( 1 + \normk{\fI_0}{s+\sigma}) 
	\end{equation}
	for some $\sigma>0$ and for all $s_0\leq s \leq S-\sigma$.
	The estimate \eqref{a2d.est} is deduced by \eqref{a3d},
	$ | \tm_{\frac12}-1|^{k_0, \upsilon} \lesssim 
	\varepsilon \upsilon^{-1}$,  
	% \eqref{step5.est3}, 
	\eqref{b1b2rho.est},   \eqref{m12.stima.f}, \eqref{ansatz_I0_s0}. 
	The estimate \eqref{bR8.est}  follows by \eqref{r7d}, \eqref{P6},  Lemmata \ref{pseudo_compo}, \ref{pseudo_commu}, \ref{product+diffeo} and 
	\eqref{b1b2rho.est}, \eqref{bR6esti1}, \eqref{m12.stima.f}, \eqref{step5.est3}.
	The smoothing term $ \bT_{7, M} $  in \eqref{cL7} is, using also \eqref{phipi0},
	$$
	\bT_{7, M}:= 
	\b\Phi^{-1} \bT_{6,M} \b\Phi 
	+ \im\b\Pi_0(\b\Phi -{\rm Id}) + \b\Phi^{-1}\bR_{6}^{(-M,o)} \b\Phi  +
	\begin{pmatrix}
	T_M + T_M' & 0 \\ 0 & \overline{T_M} + \overline{T_M'} 
	\end{pmatrix} 
	$$ 
	with $T_M$ and $T_M'$ defined in \eqref{lin1}, \eqref{lin2}.
	The estimate \eqref{red1.estT} 
	follows by \eqref{P6},   Lemmata \ref{tame_compo}, \ref{tame_pesudodiff}, 
	the tame estimates of $ \b\Phi $ in Proposition 2.37 in \cite{BBHM}, 
	and  \eqref{step5.est3}, \eqref{b1b2rho.est}, \eqref{red1.est4}, \eqref{block.est2}.
	The estimate \eqref{red1.est4} follows by 
	Lemma 2.38 in \cite{BBHM} 	
	and \eqref{b1b2rho.est}.
	The estimates \eqref{red1.est1}, \eqref{red1.est2}, \eqref{red1.est3}, \eqref{red1.est6} are proved in the same fashion, using also \eqref{lem:diopha.eq.12}.
\end{proof}

\subsection{Reduction of the order 0}\label{sec:order0}

The goal of this section is to transform the operator $\cL_7$ in \eqref{cL7} into the operator $\cL_8$ in \eqref{cL8} whose coefficient in front of the Hilbert transform $\cH$ is a real constant. From now on, we neglect the contribution of $\bQ_{7}^\perp$ in \eqref{cL7} 
which will be conjugated in Section \ref{sec:conc}. 
For simplicity of notation we denote such operator $\cL_7 $ as well. 
We first write 
$$
\cL_7 = \omega\cdot \pa_\vf + \begin{pmatrix}
P_7 & 0 \\ 0 & \bar{P_7}
\end{pmatrix} + \im\b\Pi_0 + \bT_{7,M}\,,
$$
where 
\begin{equation}
\label{P7}
P_7 := \tm_{1,\bar\tn}\pa_x + \im \tm_\frac12 \Omega(\gamma, D)  + a_5 (\varphi,x) \cH + \Op ( r_7^{(d)}) \, . 
\end{equation}
We conjugate  $\cL_7$ through the time-$1$ flow 
$\Psi(\vf):= \Psi^\tau(\vf)|_{\tau =1}$ generated by  
\begin{equation}\label{Psi.problem}
\partial_\tau \Psi^\tau(\vf) =  B(\vf)  \Psi^\tau(\vf) \, , \
\Psi^0(\vf) = {\rm Id} 	\,  ,   \quad B(\vf) :=  b_1(\vf,x) \cH \, , 
\end{equation}
where  $b_1(\vf, x)$ is a  real quasi-periodic traveling wave $ {\rm odd}(\vphi,x) $
chosen later 
(see \eqref{eqo2}) and $\cH$ is the Hilbert transform in \eqref{Hilbert-transf}. 
Thus by Lemmata \ref{rev_defn_C}, \ref{lem:mom_pseudo} the operator 
$  b_1(\vf,x) \cH $  is reversibility and momentum preserving and 
so is its flow $ \Psi^\tau (\vphi)$.
Note that, since $ \cH (1) = 0 $,  
% $\Psi\pi_0 = \pi_0 = \Psi^{-1}\pi_0$, so that 
\begin{equation}\label{psipi0}
\Psi (\varphi) \pi_0 = \pi_0 = \Psi^{-1}(\varphi) \pi_0 
%\b\Psi^{-1}\b\Pi_0 \b\Psi = \b\Pi_0\b\Psi 
\,.
\end{equation}
By the Lie expansion in \eqref{lie_abstract} we have 
\begin{equation}
\begin{aligned}
\Psi^{-1} P_7 \Psi  & = 
P_7 -  [B, P_7] +   \sum_{n=2}^{M} \frac{(-1)^n}{n!} \ad_{B(\vf)}^n(P_7) + L_{M} \,,\label{lin3} \\
L_{M} & :=  \frac{(- 1)^{M+1}}{M!} \int_0^1 (1 - \tau)^{M} \Psi^{-\tau}(\vf) \,
\ad_{B(\vf)}^{M+1}(P_7) \, \Psi^\tau (\vf)  \di \tau \, ,  
\end{aligned}
\end{equation}
and, by \eqref{lie_omega_devf}, 
\begin{equation}
\begin{aligned}
\label{lin4}
\Psi^{-1} \circ \omega\cdot\pa_\vf \circ \Psi & = \omega\cdot\pa_\vf +  (\omega\cdot\pa_\vf B(\vf)) - \sum_{n=2}^{M} \frac{(-1)^n}{n!} \ad_{B(\vf)}^{n-1}(\omega\cdot \pa_\vf B(\vf)) + L_{M}'\,, \\
L_{M}' & :=  \frac{(- 1)^{M}}{M!} \int_0^1 (1 - \tau)^{M} \Psi^{-\tau}(\vf) \,
\ad_{B(\vf)}^{M}(\omega\cdot \pa_\vf B(\vf)) \, \Psi^\tau (\vf) \di \tau \, . 
\end{aligned}
\end{equation}
The number $ M $ will be fixed in \eqref{M_choice}.
The contributions at order $0$ come from 
$ (\omega\cdot\pa_\vf B )+ P_7 -  [B, P_7] $.
Since $ B = b_1 \cH $, %  (see \eqref{Psi.problem}), 
by \eqref{P7}, \eqref{eq:moyal_exp} and \eqref{Omegakx} we have
\begin{equation}\label{P8BP8}
	\begin{aligned}
	[B,P_7] & = - \tm_{1,\bar\tn}(b_1)_x\cH + \Op(r_{b_1, - \frac12}) \,,
	\end{aligned}
\end{equation}
where $ \Op(r_{b_1, - \frac12}) \in\Ops^{- \frac12} $ is small with $b_1$.
As a consequence,  the $0$  order term of the operator 
$\omega\cdot\pa_\vf B + P_7 - [B, P_7]   $
is $ \big(\omega\cdot\pa_\vf b_1 + \tm_{1,\bar\tn} (b_1)_x + a_5 \big)\cH $.
We choose $b_1 $ as the solution of 
\begin{equation}\label{b1bom}
(\omega\cdot\pa_\vf b_1 +  \tm_{1,\bar\tn}\pa_x) b_1 +\Pi_{N_{\bar\tn}}  a_5 = \tm_{0}
\end{equation}
where $\tm_{0} $ is the average  (see \eqref{def:avera})
\begin{equation}\label{tm0}
\tm_{0}:= \braket{a_5}_{\vf,x}\,.
\end{equation}
We define $ b_1(\vf,x) $ to be the real, $ {\rm odd}(\vphi,x) $, quasi-periodic traveling wave
\begin{equation}\label{eqo2}
	b_1(\vf,x) := - (\omega\cdot\pa_\vf + \tm_{1,\bar\tn}\pa_x)_{\rm ext}^{-1}\big(\Pi_{N_{\bar\tn}} a_5(\vf,x) - \tm_{0} \big) \,, 
\end{equation}
 recall \eqref{paext}. 
Note that $b_1(\vf,x)$ is  defined for any $(\omega,\gamma)\in\R^\nu\times[\gamma_1,\gamma_2]$ and that, 
for any $(\omega,\gamma)\in\tT\tC_{\bar\tn+1}(2\upsilon,\tau)$  
defined in \eqref{tDtCn}, it solves \eqref{b1bom}.

%which is independent of $x\in\T$ since $a_5(\vf,x)$ is a quasi-periodic traveling wave.
We deduce by \eqref{lin3}-\eqref{lin4} and \eqref{P8BP8}, \eqref{eqo2}, that,
for any $(\omega,\gamma)\in\tT\tC_{\bar\tn+1}(2\upsilon,\tau)$,  
\begin{equation*}\label{defL8}
\begin{aligned} 
L_8:= & \ \Psi^{-1}(\vf)\left( \omega\cdot \pa_\vf + P_7 \right) \Psi(\vf) \\
= & \ \omega\cdot\pa_\vf + \tm_{1,\bar\tn} \pa_x + \im\,\tm_{\frac12} \Omega(\gamma, D)+ \tm_{0} \cH + \Op(r_8^{(d)}) + L_M + L_M' + (\Pi_{N_{\bar\tn}}^\perp a_5)\cH\,,
\end{aligned}
\end{equation*}
where 
\begin{equation}\label{r78d}
\begin{aligned}
\Op(r_8^{(d)}) & :=
\Op( - r_{b_1,-\frac12} +   r_7^{(d)}) \\
&
+  \sum_{n=2}^{M} \frac{(-1)^n}{n!} \ad_{B(\vf)}^n(P_7)- \sum_{n=2}^{M} \frac{(-1)^n}{n!} \ad_{B(\vf)}^{n-1}(\omega\cdot \pa_\vf B(\vf))  \in \Ops^{-\frac12} \, .
\end{aligned}
\end{equation}
%and $ r_{b_1,-\frac12} \in S^{-\frac12} $ is small in $ b_1$.  
In conclusion we have the following lemma. 

\begin{lem}\label{red12}
	Let $ M \in \N $, $ \tq_0 \in \N_0 $. 
	Let  $b_1$ be the % reversibility and momentum preserving 
	quasi-periodic traveling wave defined in \eqref{eqo2}.
	Then, for any $\bar\tn\in\N_0$, 
	 conjugating the operator $ \cL_7  $ in \eqref{cL7} 
	via the invertible, real, reversibility and momentum preserving 
	map  $ \Psi (\vphi)$ (cfr. \eqref{Psi.problem}), we obtain,
	for any $(\omega,\gamma)\in\tT\tC_{\bar\tn+1}(2\upsilon,\tau)$, 
	the real, reversible and  momentum preserving operator 
	\begin{equation}
	\begin{aligned}\label{cL8}
	\cL_{8}  & :=   \Psi^{-1} \cL_7  \Psi \\
	& = \omega\cdot\partial_\vf +\tm_{1,\bar\tn}\pa_x +  \im \,\tm_\frac12 \b\Omega(\gamma, D) + \tm_{0}\cH + \im\b\Pi_0 + \bR_{8}^{(-\frac12, d)} + \bT_{8,M}+\bQ_{8}^\perp\,,
	\end{aligned}
	\end{equation}
	defined for any $(\omega,\gamma)\in\R^\nu\times[\gamma_1,\gamma_2]$, where
\\[1mm]
1.		The constant $ \tm_{0}$ defined in \eqref{tm0} satisfies $| \tm_{0} |^{k_0, \upsilon} \lesssim 	\varepsilon  {\upsilon^{-1}} $;
\\[1mm]
2. $ \bR_{8}^{(-\frac12,d)} $ is  the block-diagonal operator 
		\begin{align*}
		\bR_{8}^{(-\frac12,d)} & = \begin{pmatrix}
		r_{8}^{(d)}(\vf,x,D) & 0 \\ 
		0 &\bar{r_{8}^{(d)}(\vf,x,D)}
		\end{pmatrix} \in \Ops^{-\frac12} \,,
		\end{align*}
		with $ r_{8}^{(d)}(\vf,x,D) $ defined in \eqref{r78d}
		that satisfies, for some $\sigma_M:= \sigma_M(k_0,\tau,\nu)>0$ 
		and  for all $s_0 \leq s \leq S-\sigma_M$, 
		\begin{align}
		&\normk{\bR_{8}^{(-\frac12,d)} }{-\frac12,s,1}  \lesssim_{s,M} {\varepsilon \upsilon^{-3}} ( 1+\normk{\fI_0}{s+\sigma_M} ) \,; \label{bR8} 
		\end{align}
3.  For any $ \tq \in \N^\nu_0 $ with $ |\tq| \leq \tq_0$, 
		$n_1, n_2 \in \N_0 $  with 
		$ n_1 + n_2  \leq M -  \frac32 (k_0+\tq_0)  + \frac32  $,  the  
		operator $\langle D \rangle^{n_1}\partial_{\vphi}^\tq \bT_{8, M}(\vphi) \langle D \rangle^{n_2}$ is 
		$\cD^{k_0} $-tame with a tame constant satisfying, 
		for some $\sigma_M(\tq_0) := \sigma_M(k_0, \tau, \nu, \tq_0) $, for any $s_0 \leq s \leq S - \sigma_M(\tq_0) $, 
		\begin{equation}\label{red12.estT}
		{\mathfrak M}_{\langle D \rangle^{n_1}\partial_{\vphi}^\tq \bT_{8,M}(\vphi) \langle D \rangle^{n_2}}(s) \lesssim_{S, M, \tq_0} 
		{\varepsilon \upsilon^{-3}}( 1+ \normk{\fI_0}{s+\sigma_M(\tq_0)} )\,;
		\end{equation}
4. The operator $\bQ_{8}^\perp$ is 
		\begin{equation}\label{R8perp}
			\bQ_{8}^\perp := (\Pi_{N_{\bar\tn}}^\perp a_5)\cH\begin{pmatrix}
			1 & 0 \\ 0 & 1
			\end{pmatrix}\,,
		\end{equation}
		where $a_5(\vf,x)$ is defined in \eqref{a3d} and satisfies \eqref{a2d.est};
\\[1mm]
5. The operators $ \Psi^{\pm 1} -{\rm Id}$, $(\Psi^{\pm 1}-{\rm Id})^*$ are $\cD^{k_0}$-tame, with tame constants satisfying, for some $\sigma := \sigma (k_0, \tau, \nu ) > 0 $ and for all $s_0 \leq s \leq S-\sigma$,
		\begin{align}
		& \fM_{\Psi^{\pm 1} -{\rm Id}}(s) + \fM_{(\Psi^{\pm 1}-{\rm Id})^*}(s) \lesssim_{s} {\varepsilon\upsilon^{-3}}( 1 + \normk{\fI_0}{s+ \sigma})\, ; \label{red12.est4}
		\end{align}
6. 
	Furthermore, for any $s_1$ as in \eqref{s1s0}, $\tq\in\N_0^\nu$, with $\abs\tq \leq \tq_0$, and $n_1,n_2\in\N_0$, with $n_1+n_2\leq M-  \frac32 \tq_0 $, we have
	\begin{align}
	&\|\Delta_{12} \bR_{8}^{(-\frac12,d)} \|_{-\frac12,s_1,1} \lesssim_{s_1,M} 
	{\varepsilon \upsilon^{-3}}\norm{ i_1-i_2 }_{s_1+\sigma_M}   \,, 
	\  | \Delta_{12} \tm_{0} |  
	\lesssim {\varepsilon \upsilon^{-1}} \norm{i_1-i_2}_{s_0+\sigma} \, , 
	\label{red12.est2} \\
	& \| \braket{D}^{n_1} \pa_\vf^\tq \Delta_{12} \bT_{8,M} \braket{D}^{n_2}\|_{\cL(H^{s_1})} \lesssim_{s_1, M, \tq_0} {\varepsilon\upsilon^{-3}} \norm{i_1-i_2}_{s_1+ \sigma_M(\tq_0) }\,,\label{red12.est3}\\
	& \| \Delta_{12} (\Psi^{\pm 1})h \|_{s_1} + \| \Delta_{12} (\Psi^{\pm 1})^*h \|_{s_1} \lesssim_{s_1} {\varepsilon\upsilon^{-3}} \norm{ i_1 - i_2}_{s_1+\sigma} \norm{h}_{s_1+\sigma} \,. \label{red12.est5}
	\end{align}
\end{lem}
\begin{proof}
	The function $b_1(\vf,x)$ defined in \eqref{eqo2}, satisfies, 
	by 
	\eqref{a2d.est}, \eqref{lem:diopha.eq}, 
	for some $\sigma>0$ and for all $s_0 \leq s \leq S-\sigma$, 
	\begin{equation}\label{b3.est}
	\normk{ b_1 }{s} \lesssim_{s} \varepsilon \upsilon^{-3}(1+ \normk{\fI_0}{s+\sigma}) \, .
	\end{equation}
	The estimate for $\tm_{0}$ follows by \eqref{tm0} and  \eqref{a2d.est}.
	The estimate \eqref{bR8} follows by \eqref{r78d}, \eqref{P7}, Lemmata \ref{pseudo_compo}, 
	\ref{pseudo_commu}, 
	and \eqref{a2d.est}, \eqref{bR8.est}, \eqref{b3.est}.
	Using \eqref{psipi0},
	the smoothing term $  \bT_{8,M}  $ in \eqref{cL8} is
	\begin{equation*}
	\bT_{8,M} := { \Psi}^{-1} \bT_{7,M} {\Psi} + 
	\im \b\Pi_0( \Psi - {\rm Id}) +
	\begin{pmatrix}
	L_M + L_M' & 0 \\ 0 & \overline{L_M} + \overline{L_M'}  
	\end{pmatrix}   	
	\end{equation*}
	with $L_M$ and $L_M'$ introduced in \eqref{lin3}, \eqref{lin4}.
	The estimate \eqref{red12.estT}   follows by Lemmata \ref{tame_compo}, \ref{tame_pesudodiff},  \ref{Neumann pseudo diff},   \eqref{P7}, \eqref{a2d.est}, \eqref{red1.estT}, \eqref{b3.est}, \eqref{red12.est4}.
	The estimate \eqref{red12.est4} follows by Lemmata 
	\ref{Neumann pseudo diff}, 
	\ref{tame_pesudodiff} and  \eqref{b3.est}. 
	The estimates \eqref{red12.est2}, \eqref{red12.est3}, \eqref{red12.est5} are proved in the same fashion.
\end{proof}

\begin{rem}\label{fix:alpha}
	In Proposition \ref{end_redu} we shall 
	estimate $ \| [\pa_x, \bR_8^{(-\frac12,d)} ]\|_{-\frac12,s,0}^{k_0,\upsilon}$ 
	using  \eqref{bR8} and \eqref{eq:comm_tame_AB}. In order to  
	control $ \| \bR_8^{(-\frac12,d)}  \|_{-\frac12,s,1}^{k_0, \upsilon} $  we used the 
	{estimates \eqref{bR8.est}} for finitely many $ \alpha \in \N_0 $, 
	$ \alpha \leq \alpha (M) $, depending on $ M $, as well similar estimates for $\bR_6^{(-\frac12, d)}$, $\bR_5^{(-\frac12, d)}$, etc. 
	In Proposition 
	\ref{end_redu} we shall use 
	\eqref{red12.est2}-\eqref{red12.est3} only  for $ s_1 = s_0 $. 
\end{rem}

\subsection{Conclusion: reduction of $\cL_\omega$}\label{sec:conc}

By Sections \ref{subsec:good}-\ref{sec:order0}, 
 the linear operator $\cL$  in \eqref{Linea10} is conjugated, 
under the map 
\begin{equation}\label{W1W2}
\cW := \cZ \cE  \wt\cM \cQ \cC \b\Phi_{2M} 
\b\Phi \Psi \,,
\end{equation}
for any $ (\omega,\gamma) \in \tT\tC_{\bar\tn+1}(2\upsilon,\tau) $,  $\bar\tn\in\N_0 $, 
into the real, reversible and momentum preserving operator 
\begin{equation}\label{cL9cL}
\cW^{-1} \cL \cW = \cL_{8} - \bQ_{8}^\perp + \bP_{\bar\tn}^\perp + \bQ_{\bar\tn}^\perp \,,
\end{equation}
where $\cL_{8}$ is defined in \eqref{cL8}, 
and
\begin{equation}\label{Rtnperp}
	\bP_{\bar\tn}^\perp:= \big( \wt\cM \cQ \cC \b\Phi_{2 M} 
	\b\Phi \Psi \big)^{-1} \bP_{2}^\perp \wt\cM \cQ \cC \b\Phi_{2M} 
	\b\Phi \Psi\,, \quad \bQ_{\bar\tn}^\perp := \Psi^{-1} \bQ_{7}^\perp \Psi + \bQ_{8}^\perp\,,
\end{equation}
with $\bP_{2}^\perp$, $\bQ_{7}^\perp$ and $\bQ_{8}^\perp$ defined 
respectively in \eqref{R2perp}, \eqref{R7perp} and \eqref{R8perp}; these  
operators  are exponentially small, and will contribute to the remainders estimated in Lemma \ref{RKtn.est}.
Moreover, $\cL_{8} $ is defined for any $ (\omega,\gamma) \in \R^\nu\times[\gamma_1,\gamma_2] $. 

Now we deduce a similar conjugation result for the projected operator 
$ \cL_\omega $ in \eqref{Lomegatrue}, i.e. \eqref{cLomega_again}, which acts
in the normal subspace $ \acca_{\S^+,\Sigma}^\angle $. 
We first introduce some notation. 
We denote  by $ \Pi_{\S^+,\Sigma}^\intercal $ and
$  \Pi_{\S^+,\Sigma}^\angle  $ 
the projections 
on the subspaces $\acca_{\S^+,\Sigma}^\intercal$ and $ \acca_{\S^+,\Sigma}^\angle$ defined in Section \ref{sec:decomp}. 
In view of Remark \ref{phase-space-ext}, we denote, 
with a small abuse of notation, 
$ \Pi_{\S_0^+, \Sigma}^\intercal:= \Pi_{\S^+, \Sigma}^\intercal + \pi_0 $, so that
$ \Pi_{\S_0^+, \Sigma}^\intercal + \Pi_{\S^+,\Sigma}^\angle = {\rm Id}$ on the whole
$ L^2 \times L^2 $. 
We remind that $ \S_0 = \S \cup \{0\} $, where  $ \S $ is the set defined in \eqref{def.S}. 
We denote by $ \Pi_{{\S}_0} := \Pi_\S^\intercal + \pi_0 $, where 
$  \Pi_\S^\intercal $ is defined below  \eqref{def:HS0bot}. 
%together with the  definition of  $ \Pi_{{\mathbb S}_0}^\perp $, so that
We have 
$  \Pi_{{\S}_0} + \Pi_{{\S}_0}^\perp  = {\rm Id}  $. 
Arguing as in Lemma 7.15 in \cite{BFM} we have the following. 

\begin{lem}\label{lemmaWperp}
	Let $ M > 0 $. There is $ \sigma_M > 0 $ (depending also on $ k_0, \tau, \nu $) such that, 
	assuming \eqref{ansatz_I0_s0} with $ \mu_0 \geq \sigma_M $, the following holds: 
	the map  $\cW$ defined in \eqref{W1W2} 
	has the form 
	\begin{equation}\label{Wi}
	\cW = \wt\cM\cC + \cR(\varepsilon) \,, 
	\end{equation}
	where,  for all $s_0 \leq s \leq S- \sigma_M$, 
	\begin{align}
	\normk{\cR(\varepsilon)h}{s} &  \lesssim_{ S, M} {\varepsilon\upsilon^{-3}}\big( \normk{h}{s+\sigma_M} + \normk{\fI_0}{s+\sigma_M} \normk{h}{s_0 +\sigma_M} \big) \, . \label{Wi.est1} 
	\end{align}
	Moreover 
% 	for $\varepsilon \upsilon^{-1} \leq \delta (S) $ small enough, the operator
	\begin{equation}\label{W1W2_proj}
	\cW^\perp:= \Pi_{\S^+,\Sigma}^\angle \cW \Pi_{{\S}_0}^\perp 
	\end{equation}
	is invertible and, 
	for all $s_0 \leq s \leq S-\sigma_M$,
	\begin{equation}
	\begin{aligned}
	\normk{(\cW^\perp)^{\pm 1} h}{s} & \lesssim_{ S, M} \normk{h}{s+\sigma_M} + \normk{\fI_0}{s+\sigma_M} \normk{h}{s_0 +\sigma_M} \,\label{qui.1}  \,, \\  
	\|  \Delta_{12} (\cW^\perp)^{\pm 1} h \|_{s_1} &\lesssim_{ s_1, M} 
	{\varepsilon \upsilon^{-3}}\norm{i_1-i_2}_{s_1 + \sigma_M} \norm{h}_{s_1+\sigma_M} \, . 
	\end{aligned}
	\end{equation}
	The operator $ \cW^\perp $
	maps (anti)-reversible, respectively traveling, waves, into
	(anti)-reversible, respectively traveling, waves.
\end{lem}

For any $ (\omega,\gamma) \in \tT\tC_{\bar\tn+1}(2\upsilon,\tau) $,  $\bar\tn\in\N_0 $, 
the operator $\cL_\omega$ in \eqref{Lomegatrue} (i.e. \eqref{cLomega_again}) is
conjugated under the map $ \cW^\perp $ to
\begin{equation}\label{LomLp}
\cL_\bot :=	(\cW^\perp)^{-1} \cL_\omega \cW^\perp = 
\Pi_{{\S}_0}^\perp \,(\cL_{8}-\bQ_{8}^\perp)\,\Pi_{{\S}_0}^\perp + \bP_{\perp,\bar\tn} + \bQ_{\perp,\bar\tn} +  \cR^f
\end{equation}
where
\begin{equation}\label{bRperptn}
	\bP_{\perp,\bar\tn}:=\Pi_{{\S}_0}^\perp \bP_{\bar\tn}^\perp\Pi_{{\S}_0}^\perp\,, \quad \bQ_{\perp,\bar\tn}:=\Pi_{{\S}_0}^\perp \bQ_{\bar\tn}^\perp\Pi_{{\S}_0}^\perp	
\end{equation}
 and $\cR^f$ is, 
by \eqref{W1W2_proj},
\eqref{cL9cL}, \eqref{Wi} 
% (recall that $ \wt\cM $ is defined in \eqref{map_test_M}-\eqref{Lambda}),  
and 
\eqref{proiez},
\begin{equation} 
\begin{aligned}\label{Rintercal}
\cR^f  & :=  (\cW^\perp)^{-1}\Pi_{\S^+, \Sigma}^\angle \cR(\varepsilon) 
\Pi_{\S_0} \big(\cL_{8} - \bQ_{8}^\perp+\bP_{\bar\tn}^\perp+\bQ_{\bar\tn}^\perp \big)  \Pi_{\S_0}^\bot   \\
& 
- (\cW^\perp)^{-1}\Pi_{\S^+, \Sigma}^\angle 
\cL \Pi_{\S_0^+, \Sigma}^\intercal 
\cR(\varepsilon) \Pi_{\S_0}^\bot 
- \varepsilon (\cW^\perp)^{-1}\Pi_{\S^+, \Sigma}^\angle J R \cW^\perp \, .  
\end{aligned}
\end{equation}

\begin{lem}
	The operator $\cR^f$ in \eqref{Rintercal} has the finite rank form \eqref{finite_rank_R}, \eqref{gjchij_est}. Moreover, 
	let $\tq_0\in\N_0$ and $M \geq \frac32(k_0+\tq_0) +\frac32$. There exists $\aleph(M,\tq_0)>0$ (depending also on $k_0$, $\tau$, $\nu$) such that, for any $n_1, n_2\in\N_0$, with $n_1+n_2 \leq M -  \frac32(k_0+\tq_0)+\frac32 $, and any $\tq\in\N_0^\nu$, with $\abs \tq \leq \tq_0$, the operator $\braket{D}^{n_1} \pa_\vf^\tq \cR^f \braket{D}^{n_2} $ 
	is $\cD^{k_0}$-tame, with a tame constant satisfying
	\begin{align}
	& \fM_{\braket{D}^{n_1} \pa_\vf^\tq \cR^f \braket{D}^{n_2}}(s) \lesssim_{ S, M, \tq_0} {\varepsilon\upsilon^{-3}}(1+\normk{\fI_0}{s+\aleph(M,\tq_0)}) \, , 
	\ \forall s_0 \leq s \leq S-\aleph(M,\tq_0) \, ,  \label{quo.est1}  \\
	&  \| \braket{D}^{n_1} \pa_\vf^\tq\Delta_{12} \cR^f \braket{D}^{n_2} \|_{\cL(H^{s_1})} \lesssim_{s_1, M, \tq_0} {\varepsilon\upsilon^{-3}} \norm{i_1-i_2}_{s_1+\aleph(M,\tq_0)} \, , 
	\label{quo.est2}
	\end{align}
	for any $s_1$ as in \eqref{s1s0}. 
\end{lem}
\begin{proof}
	The first two terms in \eqref{Rintercal} have  the finite rank form 
	\eqref{finite_rank_R} because of the presence of the finite dimensional projectors 
	$ \Pi_{\S_0} $ and $ \Pi_{\S_0^+,\Sigma}^\intercal $. 
	In the last term, the operator $ R $ has  the finite rank form 
	\eqref{finite_rank_R}. 	
	The estimate \eqref{quo.est1} follows by  \eqref{Rintercal}, \eqref{W1W2}, \eqref{W1W2_proj}, \eqref{cL8}, \eqref{finite_rank_R}, \eqref{prod} and \eqref{Wi.est1}, \eqref{qui.1},  \eqref{bR8}, \eqref{red12.estT}, \eqref{gjchij_est}.
	The estimate \eqref{quo.est2} follows similarly.
\end{proof}

\begin{lem}\label{RKtn.est}
	The operators $\bP_{\perp,\bar\tn}$ and $\bQ_{\perp,\bar\tn}$ defined in \eqref{bRperptn}, \eqref{Rtnperp} satisfy, for some $\sigma_M=\sigma_M(k_0,\tau,\nu)>0$,
		for all $ s_0\leq s\leq S - \sigma_M$, 
	\begin{align}
		\normk{\bP_{\perp,\bar\tn}h}{s} & \lesssim_{S} \varepsilon N_{\bar\tn-1}^{-\ta}\big( \normk{h}{s+\sigma_M} + 
		{ \normk{\fI_0}{s+\sigma_M+\tb}\normk{h}{s_0+\sigma_M} } \big)\,, \ \  \forall\,s_0\leq s\leq S - \sigma_M\,, \label{Ptn.est0} \\
		\normk{\bQ_{\perp,\bar\tn}h}{s_0} & \lesssim_{S}
		{\varepsilon\upsilon^{-2}} N_{\bar\tn}^{-{\rm b}}
		\big( {1 + \normk{\fI_0}{s_0+\sigma_M+{\rm b}}\big)\normk{h}{s_0+ \frac12}}   \,, \,  \forall\,{\rm b}>0\,, \label{Qtn.est1} \\
		\normk{\bQ_{\perp,\bar\tn}h}{s} & \lesssim_{S} {\varepsilon\upsilon^{-2}} \big( \normk{h}{s+\frac12} + \normk{\fI_0}{s+\sigma_M}\normk{h}{s_0+\frac12} \big) \, .   \label{Qtn.est2}
	\end{align}
%	Moreover, the operators $\bP_{\perp,\bar\tn}$ and $\bQ_{\perp,\bar\tn}$ are reversible and momentum preserving.
\end{lem}
\begin{proof}
	The estimates \eqref{Ptn.est0}, \eqref{Qtn.est1}, \eqref{Qtn.est2} follow from \eqref{bRperptn}, \eqref{Rtnperp}, \eqref{R2perp}, \eqref{R7perp}, \eqref{R8perp}, using the estimates \eqref{stime.pn.not.w}, \eqref{m12.stima.f}, \eqref{a2d.est}, \eqref{SM12}, \eqref{qui.1}, \eqref{red12.est4}, \eqref{red1.est4}, \eqref{block.est6}, \eqref{step5.est12}. 
	%The operators $\bP_{\perp,\bar\tn}$, $\bQ_{\perp,\bar\tn}$ are reversible and momentum preserving by \eqref{cL9cL} since $\cL_{8}$, $\cL$ are reversible and momentum preserving, whereas $\cW^{\pm 1}$ are reversibility and momentum preserving.
\end{proof}

The next proposition summarizes the main result of this section. 

\begin{prop}\label{end_redu}
	{\bf (Reduction of $ {\cal L}_\omega $ up to smoothing operators)}
	For any $ \bar \tn \in \N_0 $ and  for all $(\omega,\gamma)\in \tT\tC_{\bar\tn+1}(2\upsilon,\tau)$ (cfr. \eqref{tDtCn}), the operator $\cL_\omega$ in \eqref{Lomegatrue} (i.e. \eqref{cLomega_again}) is conjugated as in \eqref{LomLp} to the real, reversible and momentum preserving operator $\cL_\perp $. For all $ (\omega, \gamma) \in \R^\nu \times [\gamma_1, \gamma_2] $ the extended operator defined by the right hand side in \eqref{LomLp}, 
	 has the form  
	\begin{equation}\label{Lperp}
	\cL_\perp = \omega\cdot\pa_\vf \uno_\perp + \im\,\bD_\perp + \bR_\perp +
	\bP_{\perp,\bar\tn} + \bQ_{\perp,\bar\tn}  \,,
	\end{equation}
	where $\uno_\perp$ denotes the identity map of 
	$  \bH_{{\mathbb S}_0}^\bot $ (cfr.  \eqref{def:HS0bot}) and:
	% $ H_{{\mathbb S}_0}^\perp \times H_{- {\mathbb S}_0}^\perp $, 
\\[1mm]
1.  $\bD_\perp$ is the diagonal operator
		\begin{equation*}\label{Dperp}
		\bD_\perp := \begin{pmatrix}
		\cD_\perp & 0 \\ 0 & -\bar{\cD_\perp}
		\end{pmatrix} \,, \quad \cD_\perp:= \diag_{j\in \S_0^c} \mu_j \,, \quad 
		\S_0^c:= \Z\setminus (\S\cup\{0\})  \, , 
		\end{equation*}
		with eigenvalues 
		$ \mu_j :=\tm_{1,\bar\tn}j +
		\tm_{\frac12}\Omega_j(\gamma) - \tm_{0}\,\sgn(j) \in \R\,, $
		where $\Omega_j(\gamma) $ is the dispersion relation  \eqref{def:Omegajk} and 
		the real constants $ \tm_{1,\bar\tn}, \tm_{\frac12}, \tm_{0} $, defined respectively in Lemma \ref{conju.tr}, \eqref{tm12}, \eqref{tm0}, satisfy 
		\begin{equation}\label{const_small}
		\begin{aligned}
		| \tm_{1,\bar\tn}|^{k_0,\upsilon} \lesssim \varepsilon \, , \quad
			| \tm_{\frac12} - 1 |^{k_0,\upsilon}  + 
			|\tm_{0} |^{k_0,\upsilon}  \lesssim \varepsilon  \upsilon^{-1} \,. 
		\end{aligned}
		\end{equation}
		In addition, for some $ \sigma > 0 $, 
		\begin{equation}\label{const_smallV} 
		| \Delta_{12} \tm_{1,\bar\tn}  | \lesssim \varepsilon \norm{i_1-i_2}_{s_0+\sigma} \, , 
		\quad 
		 | \Delta_{12} \tm_{\frac12}| + 
		|\Delta_{12} \tm_{0} | \lesssim \varepsilon \upsilon^{-1} 
		\norm{i_1-i_2}_{s_0+\sigma}  \,;
		\end{equation} 
2. 
% The operator $\bR_\perp := \bR_\perp (\varphi) $  satisfies:
For any $\tq_0\in \N_0$, $M  > \frac32(k_0+\tq_0) +\frac32$, there is a constant $\aleph(M,\tq_0)>0$ (depending also on $k_0$, $\tau$, $\nu$) such that, assuming \eqref{ansatz_I0_s0} 
		with $\mu_0 \geq \aleph(M,\tq_0)$, for any $s_0\leq s \leq S - \aleph(M,\tq_0)$,   $\tq\in\N_0^\nu$, with $\abs\tq \leq \tq_0$, the operators 
		$ \langle D \rangle^\frac14\pa_\vf^\tq\bR_\perp \langle D \rangle^\frac14 $, $ \langle D \rangle^\frac14[\pa_\vf^\tq\bR_\perp, \pa_x]\langle D \rangle^\frac14 $  
		are $\cD^{k_0}$-tame with  tame constants satisfying
		\begin{align}
		&\fM_{\langle D \rangle^\frac14\pa_\vf^\tq\bR_\perp\langle D \rangle^\frac14} (s), \  
		\fM_{\langle D \rangle^\frac14[\pa_\vf^\tq\bR_\perp, \pa_x]\langle D \rangle^\frac14} (s) \lesssim_{ S, M, \tq_0} {\varepsilon\upsilon^{-3} }(1+ \normk{\fI_0}{s+\aleph(M,\tq_0)})\,. \label{fine.rid.1}
		\end{align}
		Moreover,
		for any 
		$\tq\in\N_0^\nu$, with $\abs\tq\leq \tq_0$,
		\begin{equation}
		\| \langle D \rangle^\frac14\pa_\vf^\tq\Delta_{12}\bR_\perp \langle D \rangle^\frac14  \|_{\cL(H^{s_0})} + 
		\| \langle D \rangle^\frac14\pa_\vf^\tq\Delta_{12}[\bR_\perp, \pa_x]  \langle D \rangle^\frac14 \|_{\cL(H^{s_0})} \lesssim_{M} {\varepsilon \upsilon^{-3}} \norm{i_1-i_2}_{s_0+\aleph(M,\tq_0)}\label{fine.rid.2} \,.
		\end{equation}
 The operator $\bR_\perp := \bR_\perp (\varphi) $  is real, reversible and momentum preserving.		
		\\[1mm]
	3. The remainders $\bP_{\perp,\bar\tn}, \bQ_{\perp,\bar\tn}$
	are defined  in \eqref{bRperptn} and satisfy the estimates 
	\eqref{Ptn.est0}-\eqref{Qtn.est2}.  %	of Lemma \ref{RKtn.est}. 
\end{prop}

\begin{proof}
	By \eqref{LomLp} and \eqref{cL8} we deduce \eqref{Lperp}
	with 
	$
	\bR_\perp := \Pi_{{\S}_0}^\perp(\bR_{8}^{(-\frac12,d)}+ \bT_{8,M})\Pi_{{\S}_0}^\perp + \cR^f  $. 
	The estimates \eqref{const_small}-\eqref{const_smallV} follow by Lemmata \ref{LEMMONE}, \ref{red1}, \ref{red12}. 
	The estimate \eqref{fine.rid.1} follows by Lemmata \ref{pseudo_compo},  \ref{pseudo_commu},  \ref{tame_pesudodiff}, 
	\eqref{bR8} and  \eqref{red12.estT}, \eqref{quo.est1}, choosing $(n_1,n_2)=(1,2),(2,1)$.
	The estimate \eqref{fine.rid.2} follows similarly.
%	The operator 
%	$ \cL_\omega $ in \eqref{Lomegatrue} is reversible and momentum preserving
%	(Lemma \ref{lem:K02}).
%	By Sections \ref{subsec:good}-\ref{sec:order12}, the maps 
%	$ \cZ, \cE, \wt \cM, \cQ,\b\Phi_{\color{red} 2M}, \b\Phi, \Psi $ are reversibility and momentum preserving and, 
%	using also Lemmata \ref{proj_rev} and \ref{lem:proj.momentum}, we deduce that 
%	 $ \cL_\bot $ in \eqref{LomLp} is reversible and momentum preserving. 
%	Since $ \im\,\bD_\perp $ is reversible and momentum preserving, we deduce, together with Lemma \ref{RKtn.est}, that
%	$\bR_\perp $   is reversible and momentum preserving. 
\end{proof}

\section{Almost-diagonalization and invertibility of $\cL_\omega$}\label{sec:KAM}

%In Proposition \ref{end_redu} we obtained the operator $ \cL_\perp $
%in \eqref{Lperp}, where we neglected the contribution of the remainders $\bP_{\perp,\bar\tn}$ and $\bQ_{\perp,\bar\tn}$ in \eqref{bRperptn}, \eqref{LomLp}, which is diagonal and constant coefficient up to the  bounded operator $\bR_\perp (\vphi) $. 
In this section we diagonalize the operator 
$\omega\cdot\pa_\vf \uno_\perp + \im \,\bD_\bot + \bR_\perp (\varphi) $
obtained neglecting from  $ \cL_\perp $  in \eqref{Lperp} 
the remainders $\bP_{\perp,\bar\tn}$ and $\bQ_{\perp,\bar\tn}$.
% in \eqref{bRperptn}, \eqref{LomLp},
We implement a KAM iterative scheme.
As starting point, we consider the real, reversible and momentum preserving operator,
acting in $  \bH_{{\mathbb S}_0}^\bot $, 
\begin{equation}\label{bL0}
\bL_0 := \bL_0(i) := % \Pi_{{\S}_0}^\perp \,(\cL_{8}-\bQ_{8}^\perp)\,\Pi_{{\S}_0}^\perp= 
\omega\cdot\pa_\vf \uno_\perp + \im \,\bD_0 + \bR_\perp^{(0)} \,,
\end{equation}
defined for all $(\omega,\gamma)\in \R^\nu \times[\gamma_1,\gamma_2]$, with diagonal part (with respect to the exponential basis)
\begin{equation}\label{D0}
\begin{aligned}
\bD_0&:= \begin{pmatrix}
\cD_0 & 0 \\ 0 & -\bar{\cD_0}
\end{pmatrix} \,,\quad
\cD_0 := \diag_{j \in \S_0^c} \mu_{j}^{(0)} \,, \quad \mu_{j}^{(0)}:= \tm_{1,\bar\tn} j+\tm_{\frac12}\Omega_j(\gamma)-\tm_{0}\,\sgn(j) \,,
\end{aligned}
\end{equation}
where $  \S_0^c = \Z \setminus \S_0 $, $ \S_0 = \S \cup \{0\}  $,  
the real constants $\tm_{1, \bar \tn} $, $ \tm_{\frac12} $,  $ \tm_{0} $ 
satisfy  \eqref{const_small}-\eqref{const_smallV} and 
\begin{equation}\label{Rperp0}
\bR_\perp^{(0)}:= \bR_\perp := \begin{pmatrix}
R_\perp^{(0,d)} & R_\perp^{(0,o)} \\ \bar{R_\perp^{(0,o)}} & \bar{R_\perp^{(0,d)}}
\end{pmatrix}\,, \ \quad R_\perp^{(0,d)}: 
H_{{\mathbb S}_0}^\perp  
\rightarrow H_{{\mathbb S}_0}^\perp  \,, \
R_\perp^{(0,o)} : H_{-{\mathbb S}_0}^\perp  
\rightarrow H_{{\mathbb S}_0}^\perp \, , 
\end{equation}
which is a real, reversible, momentum preserving operator satisfying  \eqref{fine.rid.1}, \eqref{fine.rid.2}. We denote 
$ H_{\pm \S_0}^\bot = \{ h(x) = \sum_{j \not \in \pm \S_0} h_j e^{ \pm \im j x} \in L^2 \} $.  
Note that 
\begin{equation}\label{Dbar}
\bar{\cD_0} :  H_{- {\mathbb S}_0}^\bot \to H_{- {\mathbb S}_0}^\bot  \, , 
\quad   \bar{\cD_0} = {\rm diag}_{j \in - \S_0^c}( \mu_{-j}^{(0)} ) \, . 
\end{equation}
Proposition \ref{end_redu} implies that the operator $ \bR_\perp^{(0)} $ 
satisfies the estimates of Lemma \ref{Rperp_small} below 
by fixing the constant $ M $ large enough
(which means  performing sufficiently many regularizing steps in Section \ref{sec:block_dec}), namely
\begin{equation}\label{M_choice}
M:= \big[  \tfrac32(k_0 + s_0 + \tb) + \tfrac32 \big] +1 \in \N\,,
\end{equation}
where $\tb$ is defined in \eqref{tbta}.
We also set 
\begin{equation}\label{cb.mub}
\mu(\tb):= \aleph(M,s_0+\tb) \, , 
\end{equation}
where the constant  $\aleph(M,\tq_0)$ is given  in Proposition \ref{end_redu}, with $\tq_0=s_0+\tb$.
\begin{lem}\label{Rperp_small}   
	{\bf (Smallness of $ \bR_\perp^{(0)} $)}
	Assume \eqref{ansatz_I0_s0} with $\mu_0 \geq \mu(\tb)$. Then the operators
	$ \langle D \rangle^\frac14\bR_\perp^{(0)} \langle D \rangle^\frac14$, $ \langle D \rangle^\frac14[\bR_\perp^{(0)} , \pa_x] \langle D \rangle^\frac14$, and 
	$ \langle D \rangle^\frac14\pa_{\vf_m}^{s_0} \bR_\perp^{(0)}\langle D \rangle^\frac14 $, $  \langle D \rangle^\frac14[\pa_{\vf_m}^{s_0} \bR_\perp^{(0)}, \pa_x] \langle D \rangle^\frac14$, 
	$\langle D \rangle^\frac14 \pa_{\vf_m}^{s_0+\tb} \bR^{(0)}_\perp \langle D \rangle^\frac14$, 
	\\
	$\langle D \rangle^\frac14[\pa_{\vf_m}^{s_0+\tb} \bR^{(0)}_\perp, \pa_x]\langle D \rangle^\frac14 $, 
	$ m = 1, \ldots, \nu $, 
	are $\cD^{k_0}$-tame. Defining
	\begin{align}
	& \mathbb{M}_0(s):= \max \big\{ \fM_{\langle D \rangle^\frac14\bR^{(0)}_\perp\langle D \rangle^\frac14}(s), \, \fM_{\langle D \rangle^\frac14[\bR^{(0)}_\perp,\pa_x]\langle D \rangle^\frac14}(s), \, \notag \\
	& \quad \quad \quad \quad \quad  \fM_{\langle D \rangle^\frac14\pa_{\vf_m}^{s_0}\bR^{(0)}_\perp\langle D \rangle^\frac14}(s), \,  \fM_{\langle D \rangle^\frac14[\pa_{\vf_m}^{s_0}\bR^{(0)}_\perp,\pa_x]\langle D \rangle^\frac14}(s),  m = 1, \ldots, \nu \big\}
	\label{M0s} \\
	& \mathbb{M}_0(s,\tb):= \max \big\{ \fM_{\langle D \rangle^\frac14\pa_{\vf_m}^{s_0+\tb}\bR^{(0)}_\perp\langle D \rangle^\frac14}(s), \ \fM_{\langle D \rangle^\frac14[\pa_{\vf_m}^{s_0+\tb}\bR^{(0)}_\perp,\pa_x]\langle D \rangle^\frac14}(s) \, , \, m = 1, \ldots, \nu \big\}\,,\label{M0sb}
	\end{align}
	we have, for all $ s_0 \leq s \leq S- \mu(\tb)$,
	\begin{equation}\label{M0ss.est}
	\fM_0(s,\tb):= \max\Set{\mathbb{M}_0(s), \mathbb{M}_0(s,\tb)  } \leq C(S){ \frac{\varepsilon}{\upsilon^3}}(1 + \normk{\fI_0}{s+ \mu (\tb)})  \,, \  \fM_0(s_0,\tb) \leq C(S) {\frac{\varepsilon}{\upsilon^3}} \, . 
	\end{equation}
	Moreover, 
	for all $\tq \in\N_0^\nu$,  with $\abs \tq \leq s_0+\tb$, 
	\begin{equation}\label{Moss.delta}
	\| \langle D \rangle^\frac14\pa_\vf^\tq \Delta_{12} \bR^{(0)}_\perp \langle D \rangle^\frac14 \|_{\cL(H^{s_0})}\,, \ \|\langle D \rangle^\frac14 \Delta_{12}[\pa_\vf^\tq \bR^{(0)}_\perp,\pa_x] \langle D \rangle^\frac14 \|_{\cL(H^{s_0})} \leq C(S) {\varepsilon\upsilon^{-3}}\norm{i_1-i_2}_{s_0+\mu(\tb)} \,.
	\end{equation}
\end{lem}

\begin{proof}
	Recalling \eqref{M0s}, \eqref{M0sb}, 	the bounds \eqref{M0ss.est}-\eqref{Moss.delta} follow
	by \eqref{fine.rid.1}, 
	\eqref{M_choice}, \eqref{cb.mub},  \eqref{fine.rid.2}.
\end{proof}

We perform the almost-reducibility of $\bL_0$ along the scale $(N_{\tn})_{\tn\in\N_0}$, defined in \eqref{scala.strai}.

\begin{thm}\label{iterative_KAM}
	{\bf (Almost-diagonalization of $\bL_0$: KAM iteration)}
	There exists $\tau_2(\tau,\nu)> \tau_1(\tau,\nu)+1 + \ta $ 
	(with $ \tau_1 , \ta $  defined in \eqref{tbta})	such that, for all $S >s_0$, there is $N_0:=N_0(S,\tb)\in\N$ such that, if
	\begin{equation}\label{small_KAM_con}
	N_0^{\tau_2} \fM_0(s_0,\tb)\upsilon^{-1} \leq 1 \,,
	\end{equation}
	then, for all $\bar\tn\in\N_0$, $\tn=0,1,\ldots,\bar\tn$:
	\\[1mm]
	$({\bf S1})_\tn$ There exists a real, reversible and momentum preserving operator
	\begin{equation}\label{bLn}
	\begin{aligned}
	&\bL_\tn := \omega\cdot\pa_\vf\uno_\perp +\im\,\bD_\tn + \bR_\perp^{(\tn)}\,, \\
	&\bD_\tn:= \begin{pmatrix}
	\cD_\tn & 0 \\ 0 & -\bar{\cD_\tn}
	\end{pmatrix}\,, \quad \cD_\tn:= \diag_{j\in\S_0^c} \mu_{j}^{(\tn)} \,,
	\end{aligned}
	\end{equation}
	defined for all $(\omega,\gamma)$ in $\R^\nu\times [\gamma_1,\gamma_2]$,
	where $\mu_{j}^{(\tn)}$ are $k_0$-times differentiable real functions 
	\begin{equation}\label{eigen_KAM}
	\mu_{j}^{(\tn)}(\omega,\gamma):= \mu_{j}^{(0)}(\omega,\gamma) + \fr_j^{(\tn)}(\omega,\gamma)\,, \quad \mu_{j}^{(0)} = \tm_{1,\bar\tn}\,j+\tm_{\frac12}\,\Omega_j(\gamma)-\tm_{0}\,\sgn(j)\,,
	\end{equation}
	satisfying $ \fr_j^{(0)} = 0 $ and, for $ \tn \geq 1 $,  	
	\begin{align}
	\label{rem.eigen.KAM}
	&	|j|^\frac12	| \fr_j^{(\tn)} |^{k_0,\upsilon}  \leq C(S, \tb) \varepsilon\upsilon^{-3}  \, , \quad 	
	|j|^\frac12 | \mu_{j}^{(\tn)} - \mu_{j}^{(\tn-1)} |^{k_0,\upsilon} 
	\leq C(S, \tb) \varepsilon\upsilon^{-3} N_{\tn-2}^{-\ta} \, , 
	\ \   \forall j\in\S_0^c \, \,. 
	\end{align}
	The remainder
	\begin{equation}\label{bRperpn}
	\bR_\perp^{(\tn)}:= \begin{pmatrix}
	R_\perp^{(\tn,d)} & R_\perp^{(\tn,o)}\\
	\bar{R_\perp^{(\tn,o)} } & \bar{R_\perp^{(\tn,d)} } 
	\end{pmatrix} 
	, \ \quad R_\perp^{(\tn,d)}: 
	H_{{\mathbb S}_0}^\perp  
	\rightarrow H_{{\mathbb S}_0}^\perp  \,, \
	R_\perp^{(\tn,o)} : H_{-{\mathbb S}_0}^\perp  
	\rightarrow H_{{\mathbb S}_0}^\perp
	\end{equation}
	is $\cD^{k_0}$-$(-\frac12)$-modulo-tame
	% more precisely, the operators $R_\perp^{(\tn,d)} $, $R_\perp^{(\tn,o)} $, $\braket{\pa_\vf}^\tb R_\perp^{(\tn,d)} $, $\braket{\pa_\vf}^\tb R_\perp^{(\tn,o)} $, are $\cD^{k_0}$-$(-\frac12)$-modulo-tame 
	with a modulo-tame constant 
%	\begin{equation}\label{Mn.sharp}
%	\begin{aligned}
%	&\fM_\tn^\sharp (s):= \fM_{\langle D \rangle^{\frac14}\bR_\perp^{(\tn)}\langle D \rangle^{\frac14}}^\sharp(s):= \max \{ \fM_{\langle D \rangle^{\frac14} R_\perp^{(\tn,d)} \langle D \rangle^{\frac14} }^\sharp(s), \fM_{\langle D \rangle^{\frac14}R_\perp^{(\tn,o)} \langle D \rangle^{\frac14}}^\sharp(s) \} \,, \\
%	& \fM_\tn^\sharp(s,\tb):= \fM_{\langle D \rangle^{\frac14}\langle\pa_\vf\rangle^\tb \bR_\perp^{(\tn)}\langle D \rangle^{\frac14} }^\sharp(s):= \max \{  \fM_{\langle D \rangle^{\frac14}\langle\pa_\vf\rangle^\tb R_\perp^{(\tn,d)} \langle D \rangle^{\frac14}}^\sharp(s), \fM_{\langle D \rangle^{\frac14}\langle\pa_\vf\rangle^\tb R_\perp^{(\tn,o)} \langle D \rangle^{\frac14}}^\sharp(s) \} \, , 
%	\end{aligned}
%	\end{equation} 
\begin{equation}\label{Mn.sharp}
	\fM_\tn^\sharp (s):= \fM_{\langle D \rangle^{\frac14}\bR_\perp^{(\tn)}\langle D \rangle^{\frac14}}^\sharp(s)\,, \quad  \fM_\tn^\sharp(s,\tb):= \fM_{\langle D \rangle^{\frac14}\langle\pa_\vf\rangle^\tb \bR_\perp^{(\tn)}\langle D \rangle^{\frac14} }^\sharp(s)\,,
\end{equation}
	which satisfy, for some constant $C_*(s_0,\tb) > 0 $, for all $s_0\leq s \leq S-\mu(\tb)$,
	\begin{equation}\label{small_scheme}
	\fM_\tn^\sharp (s) \leq C_*(s_0,\tb) \fM_0(s,\tb)
	N_{\tn -1}^{-\ta}   \,, \quad \fM_\tn^\sharp(s,\tb) \leq C_*(s_0,\tb) 
	\fM_0 (s,\tb) N_{\tn-1}   \,.
	\end{equation}
	Define the sets $\t\Lambda_\tn^\upsilon=\t\Lambda_\tn^\upsilon(i)$ by $\t\Lambda_0^\upsilon:=\R^\nu\times [\gamma_1,\gamma_2]$ and, for $ \tn =1,...,\bar\tn $, 
	\begin{equation}\label{tLambdan}
	\begin{aligned}
	\t\Lambda_\tn^\upsilon:= & \big\{ \lambda=(\omega,\gamma) \in \t\Lambda_{\tn-1}^\upsilon \, :  \,  \\
	& \ \ \big| \omega\cdot\ell + \mu_{j}^{(\tn -1)} - \mu_{j'}^{(\tn-1)} \big| \geq \upsilon\, \braket{\ell}^{-\tau}  \\
	& \ \ \forall\,\abs\ell \leq N_{\tn-1}\,,\ j,j'\notin\S_0 \,, \  (\ell,j,j')\neq(0,j,j),\text{ with }\ora{\jmath} \cdot\ell +j-j'=0 \, ,  \\
	&\  \ \big| \omega\cdot\ell + \mu_{j}^{(\tn -1)} + \mu_{j'}^{(\tn-1)} \big| \geq \upsilon \,\big(\abs j^\frac12 + |j'|^\frac12\big) \braket{\ell}^{-\tau}  \\
	& \ \  \forall\,\abs\ell \leq N_{\tn-1}\,, \ j,  j' \notin \S_0
	\text{ with }\ora{\jmath} \cdot\ell +j+j'=0  	
	\big\}\,.
	\end{aligned}
	\end{equation}
	For $\tn\geq 1$ there exists a real, reversibility  and momentum preserving map, defined for all $(\omega,\gamma)\in \R^\nu\times [\gamma_1,\gamma_2]$, of the form
	$$
	\b\Phi_{\tn-1} = e^{\bX_{\tn-1}}\,, \quad \bX_{\tn-1}:=\begin{pmatrix}
	X_{\tn-1}^{(d)} & X_{\tn-1}^{(o)} \\ \bar{X_{\tn-1}^{(o)}} & \bar{X_{\tn-1}^{(d)}}
	\end{pmatrix}\,, \
	\  X_{\tn-1}^{(d)}: 
	H_{{\mathbb S}_0}^\perp  
	\rightarrow H_{{\mathbb S}_0}^\perp  \,, \
	X_{\tn-1}^{(o)} : H_{-{\mathbb S}_0}^\perp  
	\rightarrow H_{{\mathbb S}_0}^\perp \, ,
	$$
	such that, for all $\lambda \in \t\Lambda_{\tn}^\upsilon$, the following conjugation formula holds:
	\begin{equation}\label{ITE-CON}
	\bL_\tn = \b\Phi_{\tn-1}^{-1} \bL_{\tn-1}\b\Phi_{\tn-1}\,.
	\end{equation}
	The operators $\bX_{\tn-1}$, $\braket{\pa_\vf}^\tb \bX_{\tn-1}$, 
	are $\cD^{k_0}$-$(-\frac12)$-modulo-tame 
	with modulo tame constants satisfying, for all $s_0\leq s \leq S - \mu(\tb)$,
	\begin{equation}\label{gen.modulo.KAM}
	\begin{aligned}
	\fM_{\langle D \rangle^{\frac14}\bX_{\tn-1}\langle D \rangle^{\frac14}}^\sharp(s) & \leq
	C(s_0,\tb) \upsilon^{-1} 
	N_{\tn-1}^{\tau_1} N_{\tn-2}^{-\ta} \fM_0(s,\tb)\,,\\
	\fM_{\langle D \rangle^{\frac14}\langle\pa_\vf\rangle^\tb \bX_{\tn-1}\langle D \rangle^{\frac14}}^\sharp (s) & \leq  
	C(s_0,\tb) \upsilon^{-1}  N_{\tn-1}^{\tau_1} N_{\tn-2} \fM_0(s,\tb)\, . 
	\end{aligned}
	\end{equation}
	$({\bf S2})_\tn$ Let $i_1(\omega,\gamma)$, $i_2(\omega,\gamma)$ such that $\bR_\perp^{(\tn)}(i_1)$, $\bR_\perp^{(\tn)}(i_2)$ satisfy \eqref{M0ss.est}, \eqref{Moss.delta}. Then, for all {$(\omega,\gamma)  \in \R^\nu \times \R$}
	\begin{align}
	&\| \langle D \rangle^\frac14|\Delta_{12} \bR_\perp^{(\tn)} |\langle D \rangle^\frac14 \|_{\cL(H^{s_0})}\lesssim_{S,\tb} \varepsilon\upsilon^{-3} N_{\tn-1}^{-\ta}\norm{i_1-i_2}_{s_0+\mu(\tb)} \,, \label{S3_1}\\
	& \|\langle D \rangle^\frac14 | \braket{\pa_\vf}^\tb \Delta_{12} \bR_\perp^{(\tn)}  |\langle D \rangle^\frac14\|_{\cL(H^{s_0})}\lesssim_{S,\tb}\varepsilon\upsilon^{-3} N_{\tn-1}\norm{i_1-i_2}_{s_0+\mu(\tb)} \,.\label{S3_2}
	\end{align}
	Furthermore, for $ \tn \geq 1 $, for all $j\in\S_0^c$,
	\begin{align}
 & |j|^\frac12| \Delta_{12} (\fr_j^{(\tn)} - \fr_j^{(\tn-1)}) | \leq C \| \langle D \rangle^\frac14| \Delta_{12}\bR_\perp^{(\tn)} | \langle D \rangle^\frac14 \|_{\cL(H^{s_0})}  \, , \\
 & |j|^\frac12| \Delta_{12} \fr_j^{(\tn)} | \leq C(S, \tb)\varepsilon\upsilon^{-3} \norm{i_1-i_2}_{s_0+\mu(\tb)} \,. 
	\end{align}
	$({\bf S3})_\tn$ Let $i_1, i_2$ be like in $({\bf S2})_\tn$ and $0 <  \rho < \upsilon/2$. Then
	\begin{equation}\label{INCLPRO}
	\varepsilon\upsilon^{-3}C(S) N_{\tn-1}^{\tau +1}\norm{i_1-i_2}_{s_0+\mu(\tb)}\leq \rho\quad \Rightarrow \quad \t\Lambda_\tn^\upsilon(i_1) \subseteq \t\Lambda_\tn^{\upsilon-\rho}(i_2) \,. 
	\end{equation}
\end{thm}

Theorem \ref{iterative_KAM} implies also that the invertible operator
\begin{equation}\label{bUn}
\bU_0:=\uno_\perp\,, \quad  \bU_{\bar\tn} := \b\Phi_0 \circ \ldots \circ \b\Phi_{\bar\tn-1} \, , \quad  \bar \tn \geq 1 \, , 
\end{equation}
has almost diagonalized $\bL_0$. We have indeed the following corollary.
\begin{thm}\label{KAMRED}
	{\bf (Almost-diagonalization of $ \bL_0 $)}
	Assume \eqref{ansatz_I0_s0} with $\mu_0 \geq \mu(\tb)$. 
	For all $S>s_0$, there exist $N_0=N_0(S,\tb)>0$ and $\delta_0=\delta_0(S)>0$ such that, if the smallness condition
	\begin{equation}\label{KAM_small_cond}
	N_0^{\tau_2}\varepsilon\upsilon^{-4} \leq \delta_0
	\end{equation}
	holds, where $\tau_2=\tau_2(\tau,\nu)$ is defined in Theorem \ref{iterative_KAM}, then, for all $\bar\tn\in\N_0$ and for all $(\omega,\gamma)\in \R^\nu\times [\gamma_1,\gamma_2]$ the operator $\bU_{\bar\tn}$ in \eqref{bUn} is well-defined, the operators $\bU_{\bar\tn}^{\pm 1}- \uno_\perp$ are $\cD^{k_0}$-$(-\frac12)$-modulo-tame with modulo-tame constants satisfying, for all $s_0\leq s \leq S - \mu(\tb)$,
	\begin{equation}\label{flow.sharp.kam}
	\fM_{\langle D \rangle^{\frac14}(\bU_{\bar\tn}^{\pm 1}-\uno_\perp)\langle D \rangle^{\frac14}}^\sharp (s) \lesssim_{S} \varepsilon \upsilon^{-4} N_0^{\tau_1}(1+\normk{\fI_0}{s+\mu(\tb)})\,,
	\end{equation}
	where $\tau_1$ is given by \eqref{tbta}. 
	Moreover $\bU_{\bar\tn}$, $\bU_{\bar\tn}^{-1}$ are real, reversibility and momentum preserving. 
	The operator 
	$ \bL_{\bar\tn} = \omega\cdot\pa_\vf\uno_\perp + \im\,\bD_{\bar\tn} + \bR_\perp^{(\bar\tn)}$, 
	defined in	\eqref{bLn}  
	with $\tn=\bar\tn$ is real, reversible and momentum preserving. 
	The operator $ \bR_\perp^{(\bar\tn)}  $ 
	is $\cD^{k_0}$-$(-\tfrac12)$-modulo-tame with a modulo-tame constant 
	satisfying, for all $ s_0 \leq s \leq S-\mu(\tb) $,
	\begin{equation}\label{stimaRN}
		\fM_{\langle D \rangle^{\frac14}\bR_\perp^{(\bar\tn)}\langle D \rangle^{\frac14}}^\sharp (s) 
		\lesssim_{S} \varepsilon\upsilon^{-3} N_{\bar\tn-1}^{-\ta} (1+\normk{\fI_0}{s+\mu(\tb)}) \,.
	\end{equation}
	Moreover, for all $(\omega,\gamma)$ in
	$	\t\Lambda_{\bar\tn}^\upsilon = \t\Lambda_{\bar\tn}^\upsilon(i) = \bigcap_{\tn=0}^{\bar\tn } \t\Lambda_\tn^\upsilon $, 
	where the sets $\t\Lambda_\tn^\upsilon$ are defined in \eqref{tLambdan}, 
	the  conjugation formula 
	$	\bL_{\bar\tn} := \bU_{\bar\tn}^{-1} \bL_0 \bU_{\bar\tn} $ holds.	
\end{thm}

\subsection*{Proof of Theorem \ref{iterative_KAM}}

The proof of Theorem \ref{iterative_KAM} is  inductive. We first show that $({\bf S1})_\tn$-$({\bf S3})_\tn$ hold when $\tn=0$. 

\paragraph{Proof of $({\bf S1})_0$-$({\bf S3})_0$.}

Properties \eqref{bLn}-\eqref{eigen_KAM}, \eqref{bRperpn}  for $\tn=0$ hold by \eqref{bL0}, \eqref{D0}, \eqref{Rperp0} with $\fr_j^{(0)} =0$. 
% We now prove that also \eqref{small_scheme} for $\tn=0$ holds.
Moreover, by Lemma \ref{mod-to-tame}, we deduce that,  
	for any $s_0 \leq s  \leq S-\mu(\tb)$, we have 
$\fM_0^\sharp(s), \fM_0^\sharp (s,\tb) \lesssim_{s_0,\tb} \fM_0(s,\tb)$
and  \eqref{small_scheme} for $\tn=0$ holds.
The estimates \eqref{S3_1}, \eqref{S3_2}
at $ \tn = 0 $ follows similarly by \eqref{Moss.delta}. % similarly to Lemma \ref{lem:tame}. 
Finally  $({\bf S3})_0$ is trivial since $\t\Lambda_0^\upsilon(i_1)= \t\Lambda_0^{\upsilon-\rho}(i_2) = \R^\nu\times [\gamma_1,\gamma_2]$.

\paragraph{The reducibility step.}
We now describe the generic inductive step, showing how to transform $\bL_\tn$ into $\bL_{\tn+1}$ by the conjugation with $\b\Phi_{\tn}$. For sake of simplicity , we drop the index $\tn$ and we write $+$ instead of $\tn+1$, so that we write $\bL:=\bL_\tn$, $\bL_+:=\bL_{\tn+1}$, $\bR_\perp:= \bR_\perp^{(\tn)}$, $\bR_\perp^{(+)}:= \bR_\perp^{(\tn+1)}$, $N:= N_\tn$, etc. 
We conjugate $\bL$ in \eqref{bLn} by a transformation of the form 
\begin{equation}\label{trasf.KAM}
\b\Phi:= e^{\bX} \,, \quad \bX:= \begin{pmatrix}
X^{(d)} & X^{(o)} \\ \bar{X^{(o)}} & \bar{X^{(d)}} 
\end{pmatrix}, 
\  X^{(d)}: 
H_{{\mathbb S}_0}^\perp  
\rightarrow H_{{\mathbb S}_0}^\perp  \,,	\
X^{(o)} : H_{-{\mathbb S}_0}^\perp  
\rightarrow H_{{\mathbb S}_0}^\perp \, , 
\end{equation}
where $\bX$ is a bounded linear operator, chosen below in \eqref{Xdn.sol}, \eqref{Xon.sol}. By the Lie expansions 
\eqref{lie_abstract}-\eqref{lie_omega_devf} 
we have
\begin{align} \label{bL+}
\bL_+  := \b\Phi^{-1} \bL \b\Phi 
& = 
\omega\cdot \pa_\vf\uno_\perp + \im\,\bD +( (\omega\cdot\pa_\vf \bX) - \im [\bX,\bD] + \Pi_N \bR_\perp ) + \Pi_N^\perp \bR_\perp  \\
& \ \ \nonumber 
- \int_{0}^1 e^{-\tau \bX} [\bX,\bR_\perp] e^{\tau\bX} \wrt\tau 
- \int_{0}^1 (1-\tau)e^{-\tau \bX} [\bX, (\omega\cdot\pa_\vf \bX) - \im [\bX,\bD]]e^{\tau\bX} \wrt\tau 
\end{align}
where $\Pi_N$ is defined in \eqref{PiNA} and $\Pi_N^\perp:= {\rm Id}- \Pi_N$. We want to solve the homological equation
\begin{equation}\label{homeq.KAM}
\omega\cdot\pa_\vf \bX - \im [\bX,\bD] + \Pi_N \bR_\perp = [\bR_\perp]
\end{equation}
where
\begin{equation}\label{ZRperp}
[ \bR_\perp]:= \begin{pmatrix}
[R_\perp^{(d)}] & 0 \\ 0 & [\bar{ R_\perp^{(d)}} ]
\end{pmatrix} \, , \quad [R_\perp^{(d)}] := {\rm diag}_{j \in \S_0^c} (R_\perp^{(d)})_j^j(0) \, . 
\end{equation}
By \eqref{bLn}, \eqref{bRperpn} and \eqref{trasf.KAM}, the homological equation \eqref{homeq.KAM} is equivalent to the two scalar homological equations
\begin{equation}\label{homeq.sys}
\begin{aligned}
& \omega\cdot \pa_\vf X^{(d)} -\im (X^{(d)} \cD - \cD X^{(d)} ) + \Pi_N R_\perp^{(d)} = [R_\perp^{(d)}] \,\\
& \omega\cdot \pa_\vf X^{(o)} + \im(X^{(o)}\bar{\cD} + \cD X^{(o)}) + \Pi_N R_\perp^{(o)} = 0 \,.
\end{aligned}
\end{equation}
Recalling \eqref{bLn} and since $ \bar{\cD} = {\rm diag}_{j \in - \S_0^c}( \mu_{-j}) $, acting in
$ H_{- \S_0}^\bot $ (see \eqref{Dbar})  the solutions of \eqref{homeq.sys} are, for 
all $ (\omega, \gamma) \in \t\Lambda_{\tn+1}^{\upsilon} $ (see \eqref{tLambdan} 
with $ \tn \rightsquigarrow \tn + 1 $)
\begin{align}
& (X^{(d)})_j^{j'}(\ell) := \begin{cases}
-\dfrac{(R_\perp^{(d)})_j^{j'}(\ell)}{\im(\omega\cdot\ell + \mu_j-\mu_{j'})} & \ \text{ if }\begin{cases}
(\ell,j,j')\neq (0,j,j), \ j,j'\in\S_0^c, \ \braket{\ell}\leq N  \\
\ell\cdot\ora{\jmath} + j-j'= 0
\end{cases}\\
0 & \ \text{ otherwise} \, , 
\end{cases}  \label{Xdn.sol} \\
& (X^{(o)})_j^{j'}(\ell) := \begin{cases}
-\dfrac{(R_\perp^{(o)})_j^{j'}(\ell)}{\im(\omega\cdot\ell + \mu_j+\mu_{-j'})} & \  \text{ if }\begin{cases}
\forall \, \ell\in\Z^\nu \ j, - j'\in\S_0^c,  \ \braket{\ell}\leq N  \\
\ell\cdot\ora{\jmath} + j-j'= 0
\end{cases}\\
0 & \  \text{ otherwise} \, . 
\end{cases} \label{Xon.sol}
\end{align}
Note that, since $ - j' \in \S_0^c $, we can apply 
the bounds \eqref{tLambdan} for $ (\omega, \gamma) \in \t\Lambda_{\tn+1}^{\upsilon} $. 

\begin{lem}\label{X_gen.hom}
	{\bf (Homological equations)}
	The real operator $\bX$ defined in \eqref{trasf.KAM}, \eqref{Xdn.sol}, \eqref{Xon.sol}, 
	(which for all $(\omega,\gamma)\in\t\Lambda_{\tn+1}^{\upsilon}$ solves the homological equation \eqref{homeq.KAM}) admits an extension to the whole parameter space 
	$ \R^\nu \times [\gamma_1, \gamma_2] $. Such extended operator 
	is  $\cD^{k_0}$-$(-\tfrac12)$-modulo-tame satisfying, for all $s_0\leq s \leq S-\mu(\tb)$,
	\begin{equation}\label{gen.est}
	\fM_{\langle D \rangle^{\frac14}\bX\langle D \rangle^{\frac14}}^\sharp(s) \lesssim_{k_0} N^{\tau_1}\upsilon^{-1} \fM^\sharp(s)\,, \quad \fM_{\langle D \rangle^{\frac14}\langle\pa_\vf\rangle^\tb \bX\langle D \rangle^{\frac14}}^\sharp(s) \lesssim_{k_0} N^{\tau_1} \upsilon^{-1} \fM^\sharp (s,\tb) \,, 
	\end{equation}
	where $\tau_1:=\tau(k_0+1)+k_0$.
	For all 
	$(\omega,\gamma)\in  \R^\nu \times \R$,
	\begin{align}
	&\| \langle D \rangle^\frac14 |\Delta_{12} \bX| \langle D \rangle^\frac14 \|_{\cL(H^{s_0})} \lesssim \notag \\
	&
	N^{2\tau+1} { \upsilon^{-1}}( \| \langle D \rangle^\frac14|\bR_\perp(i_2)| \langle D \rangle^\frac14\|_{\cL(H^{s_0})}  \norm{i_1-i_2}_{s_0+\mu(\tb)} + \| \langle D \rangle^\frac14 |\Delta_{12} \bR_\perp| \langle D \rangle^\frac14\|_{\cL(H^{s_0})}  )\,, \label{X12.est1} \\
	& \| \langle D \rangle^\frac14 | \langle\pa_\vf\rangle^\tb \Delta_{12} \bX | \langle D \rangle^\frac14\|_{\cL(H^{s_0})} \lesssim \notag \\
	& N^{2\tau+1} { \upsilon^{-1}}( \| \langle D \rangle^\frac14 | \langle\pa_\vf\rangle^\tb\bR_\perp(i_2) |\langle D \rangle^\frac14 \|_{\cL(H^{s_0})}  \norm{i_1-i_2}_{s_0+\mu(\tb)} + \| \langle D \rangle^\frac14 | \langle\pa_\vf\rangle^\tb\Delta_{12} \bR_\perp | \langle D \rangle^\frac14 \|_{\cL(H^{s_0})}  )\,. \label{X12.est2}
	\end{align}
	The operator $\bX$ is  reversibility  and momentum preserving.
\end{lem}
\begin{proof}
	We prove that \eqref{gen.est} holds for $X^{(d)}$. The proof for $X^{(o)}$ holds analogously.	First, we extend the solution in \eqref{Xdn.sol} to all
	$\lambda  $ in $\R^\nu\times [\gamma_1,\gamma_2]$ by setting (without any further relabeling)
	$ (X^{(d)})_j^{j'}(\ell) = \im\,g_{\ell,j,j'}(\lambda) (R_\perp^{(d)})_j^{j'}(\ell) $,
	where 
	$$
	g_{\ell,j,j'}(\lambda) := \frac{\chi(f(\lambda)\rho^{-1})}{f(\lambda)} \,, \quad f(\lambda):= \omega\cdot\ell + \mu_{j}-\mu_{j'} \,, \quad  \rho 
	:= \upsilon \braket{\ell}^{-\tau}\, , 
	$$
	and	$\chi$ is the cut-off function  \eqref{cutoff}. 
	By \eqref{eigen_KAM}, \eqref{rem.eigen.KAM}, \eqref{const_small},  \eqref{tLambdan}, 
	Lemma \ref{rem:exp_omegaj_gam}, together with \eqref{cutoff},
	we deduce that, for any $ k_1 \in\N_0^\nu$, $\abs{k_1}\leq k_0$,
	\begin{equation*}
	\sup_{\abs{k_1}\leq k_0} 
	\big| \pa_\lambda^{k_1} g_{\ell,j,j'} \big|	\lesssim_{k_0} \braket{\ell}^{\tau_1} \upsilon^{-1-\abs{k_1}} \, , \quad \tau_1 = \tau(k_0+1)+k_0 \, , 
	\end{equation*}
	and we deduce,  for all $0\leq \abs k \leq k_0$,
	\begin{align}
	| \pa_\lambda^k (X^{(d)})_j^{j'}(\ell) | 
	& \lesssim_{k_0} 
	\sum_{k_1+k_2=k} 
	|\pa_\lambda^{k_1}g_{\ell,j,j'}(\lambda)| 
	|\pa_\lambda^{k_2} (R_\perp^{(d)})_j^{j'}(\ell)| \nonumber  \\
	& 	\lesssim_{k_0}  \braket{\ell}^{\tau_1} \upsilon^{-1-\abs{k}} \sum_{\abs{k_2}\leq \abs k} \upsilon^{\abs{k_2}} | \pa_\lambda^{k_2} (R_\perp^{(d)})_j^{j'}(\ell) | \,.\label{equa1}
	\end{align}
	By \eqref{Xdn.sol} we have that $ (X^{(d)})_j^{j'}(\ell)= 0 $ 
	for all $ \langle \ell \rangle > N $. 
	Therefore, for all $ |k| \leq k_0 $,  we have
	\begin{align}
	&		\| \langle D \rangle^\frac14| \braket{\pa_\vf}^\tb \pa_\lambda^k X^{(d)}  | \langle D \rangle^\frac14 h \|_{s}^2 \leq \sum_{\ell,j}\braket{\ell,j}^{2s} \Big( \sum_{\braket{\ell-\ell'}\leq N,j'} | \braket{\ell-\ell'}^\tb\langle j \rangle^\frac14 \langle j' \rangle^\frac14 \pa_\lambda^k (X^{(d)})_j^{j'}(\ell-\ell') | | h_{\ell',j'}| \Big)^2 \notag \\
	&  \stackrel{\eqref{equa1}}{\lesssim_{k_0}} N^{2\tau_1}\upsilon^{-2(1+\abs k)} \sum_{\abs{k_2}\leq \abs k}\upsilon^{2\abs{k_2}} \sum_{\ell,j}\braket{\ell,j}^{2s} 
	\Big( \sum_{\ell',j'} | \braket{\ell-\ell'}^\tb \langle j \rangle^\frac14 \langle j' \rangle^\frac14 \pa_\lambda^{k_2} (R_\perp^{(d)})_j^{j'}(\ell-\ell') | | h_{\ell',j'}| \Big)^2 \notag \\
	& \lesssim_{k_0}  N^{2\tau_1} \upsilon^{-2(1+\abs k)} \sum_{\abs{k_2}\leq \abs k} \upsilon^{2\abs{k_2}} \|\langle D \rangle^\frac14 | \braket{\pa_\vf}^\tb \pa_\lambda^{k_2}R_\perp^{(d)} | \langle D \rangle^\frac14|h| \|_s^2 \notag \\
	& \stackrel{\ref{Dk0-modulo-12}, \eqref{Mn.sharp}}{\lesssim_{k_0}} N^{2\tau_1} \upsilon^{-2(1+\abs k)} \big( \fM^\sharp(s,\tb)^2 \norm h_{s_0}^2 + \fM^\sharp (s_0,\tb)^2 \norm h_s^2 
	\big) \,, \nonumber
	\end{align}
	and, by Definition \ref{Dk0-modulo-12}, 
	we conclude that
	$
	\fM_{\langle D \rangle^{\frac14}\langle\pa_\vf\rangle^\tb X^{(d)}\langle D \rangle^{\frac14}}^\sharp (s) \lesssim_{k_0} N^{\tau_1} \upsilon^{-1} \fM^\sharp(s,\tb) $.
	The analogous estimates for $\braket{\pa_\vf}^\tb X^{(o)}$,  $X^{(d)}$,
	$X^{(o)}$ and \eqref{X12.est1}, \eqref{X12.est2} follow similarly.
	By induction, the operator $\bR_\perp$ is  reversible and momentum preserving. Therefore, by \eqref{trasf.KAM}, \eqref{Xdn.sol}, \eqref{Xon.sol} and Lemmata \ref{rev_defn_C}, \ref{lem:mom_pseudo}, it follows that $\bX$ is 
	reversibility and momentum preserving.
\end{proof}

By \eqref{bL+}, \eqref{homeq.KAM},
for all $ \lambda \in\t\Lambda_{\tn+1}^{\upsilon}$, we have 
\begin{equation}\label{DEFL+}
\bL_+ = \b\Phi^{-1} \bL \b\Phi = \omega\cdot\pa_\vf \uno_\perp + \im\,\bD_+ + \bR_\perp^{(+)}\,,
\end{equation}
where
\begin{equation}\label{D+R+}
\begin{aligned}
& \bD_+:= \bD -\im [\bR_\perp] \,, \\
& \bR_\perp^{(+)}:=
\Pi_N^\perp \bR_\perp - \int_0^1 e^{-\tau\bX}[\bX,\bR_\perp] e^{\tau \bX} \wrt\tau + \int_0^1 (1-\tau) e^{-\tau \bX} [\bX, \Pi_N\bR_\perp-[\bR_\perp]] e^{\tau\bX} \wrt\tau \,.
\end{aligned}
\end{equation}
The right hand side of \eqref{DEFL+}-\eqref{D+R+} define an extension of 
$\bL_+$ to the whole parameter space $ \R^\nu \times [\gamma_1, \gamma_2] $,
since $ \bR_\perp $ and $ \bX $ are defined on $ \R^\nu \times [\gamma_1, \gamma_2] $. 

The new operator $\bL_+$ in \eqref{DEFL+} 
has the same form of $\bL$ in \eqref{bLn} with the 
non-diagonal remainder $\bR_\perp^{(+)}$ which is the sum of a 
term $ \Pi_N^\perp \bR_\perp $ supported on high frequencies 
and a quadratic function of $\bX$ and $\bR_\perp $. 
The new normal form $\bD_+$ is diagonal:
\begin{lem}\label{new.normal}
	{\bf (New diagonal part)}
	For all $(\omega,\gamma)\in\R^\nu\times [\gamma_1,\gamma_2]$,  the new normal form is
	\begin{equation*}\label{nn}
	\begin{aligned}
	\im \, \bD_+ = \im \, \bD + [\bR_\perp] = \im \begin{pmatrix}
	\cD_+ & 0 \\ 0 & -\bar{\cD_+}
	\end{pmatrix}\,, \ 
	\cD_+:= \diag_{j\in\S_0^c}\mu_{j}^{(+)}\,, 
	\  \mu_{j}^{(+)} := \mu_{j}+ \tr_j \in \R\,,
	\end{aligned}
	\end{equation*}
	where each $ \tr_j $ satisfies, on $\R^\nu\times[\gamma_1,\gamma_2]$,
	\begin{equation}\label{defrjNF}
	|j|^\frac12 |\tr_j|^{k_0,\upsilon} = |j|^\frac12 | \mu_{j}^{(+)}-\mu_{j} |^{k_0,\upsilon} \lesssim \fM^\sharp(s_0) \, .
	\end{equation}
	Moreover, given tori $i_1(\omega,\gamma), i_2(\omega,\gamma)$, we have  
	$ |j|^\frac12| \tr_j (i_1)- \tr_j(i_2) | \lesssim \| \langle D \rangle^\frac14 | \Delta_{12} \bR_\perp | \langle D \rangle^\frac14\|_{\cL(H^{s_0})} $. 
\end{lem}
\begin{proof}
	Recalling 
	\eqref{ZRperp}, we have that  $ \tr_j:= - \im (R_\perp^{(d)})_j^j(0)$, 
	for all $ j \in \S_0^c $.  By the reversibility of $R_\perp^{(d)} $ and \eqref{rev:Fourier}
	we deduce that  $\tr_j\in\R $.
	Recalling the definition of $\fM^\sharp(s_0)$ in \eqref{Mn.sharp} (with $s=s_0$) and Definition \ref{Dk0-modulo-12}, we have, for all $0\leq \abs k \leq k_0$, $\| \langle D \rangle^\frac14 | \pa_\lambda^k R_\perp^{(d)} | \langle D \rangle^\frac14 h \|_{s_0} \leq 2 \upsilon^{-|k|} \fM^\sharp(s_0) \norm h_{s_0} $, and therefore
	$	|j|^\frac12	|  \pa_\lambda^k (R_\perp^{(d)})_j^j(0) | \lesssim \upsilon^{-\abs k} \fM^\sharp(s_0) \,.
	$
	Hence \eqref{defrjNF} follows. 
	The last bound  for $ |j |^\frac12 | \tr_j (i_1)- \tr_j(i_2) |$ 
	follows analogously.
\end{proof}

\paragraph{The iterative step.}
Assume that the statements $({\bf S1})_{\tn}$-$({\bf S3})_{\tn}$ are true. We now  prove $({\bf S1})_{\tn+1}$-$({\bf S3})_{\tn+1}$. 
For simplicity (as in other parts of the paper) 
we omit to write the dependence on $ k_0 $, which is considered as a  fixed constant.
\\
{\sc Proof of $({\bf S1})_{\tn+1}$}.
The real operator $\bX_{\tn}$ 
defined in Lemma \ref{X_gen.hom} 
is defined for all $(\omega,\gamma)\in\R^\nu\times[\gamma_1,\gamma_2]$ and, by \eqref{gen.est}, \eqref{small_scheme}, 
satisfies the estimates \eqref{gen.modulo.KAM} at the step $\tn+1$. 
The flow maps $\b\Phi_{\tn}^{\pm 1} = e^{\pm \bX_{\tn}} $ are well defined by
Lemma \ref{modulo_expo-12}.  
By \eqref{DEFL+}, for all $ \lambda \in \t\Lambda_{\tn+1}^\upsilon$, the  
conjugation formula \eqref{ITE-CON} holds at the step $\tn+1$. 
The operator $\bX_{\tn}$ is reversibility and momentum preserving, and 
so are the operators $\b\Phi_{\tn}^{\pm 1} = e^{\pm \bX_{\tn}} $.
By Lemma \ref{new.normal}, the operator $\bD_{\tn+1}$ is diagonal with eigenvalues
$ \mu_{j}^{(\tn+1)}:\R^\nu\times[\gamma_1,\gamma_2]\rightarrow \R $,
$ \mu_{j}^{(\tn+1)} = \mu_{j}^{(0)} + \fr_j^{(\tn+1)} $ 
with $ \fr_j^{(\tn+1)} := \fr_j^{(\tn)} + \tr_j^{(\tn)}$ 
satisfying, using also \eqref{small_scheme},  
\eqref{rem.eigen.KAM} 
at the step $\tn+1$. 
The next lemma provides  the estimates of the remainder $ \bR_\perp^{(\tn+1)} = \bR_\perp^{(+)} $
defined in \eqref{D+R+}.
\begin{lem}\label{rema.scheme}
	The operators $\bR_\perp^{(\tn+1)}$ and  $\braket{\pa_\vf}^\tb \bR_\perp^{(\tn+1)}$ 
	are  $\cD^{k_0}$-$(-\frac12)$-modulo-tame with  modulo-tame constants satisfying, for any $s_0 \leq s \leq S-\mu(\tb)$,
	\begin{align}\label{remas.1}
	& \fM_{\tn+1}^\sharp(s)   \lesssim_{s} N_\tn^{-\tb} \fM_\tn^\sharp (s,\tb) + N_\tn^{\tau_1} \upsilon^{-1} \fM_\tn^\sharp (s)\fM_\tn^\sharp (s_0)\,, \\	
	& \label{remas.2}
	\fM_{\tn+1}^\sharp(s,\tb)  \lesssim_{s,\tb} \fM_\tn^\sharp(s,\tb)+ N_\tn^{\tau_1}\upsilon^{-1}
	\big( \fM_\tn^\sharp(s,\tb)\fM_\tn^\sharp(s_0)+\fM_\tn^\sharp(s_0,\tb)\fM_\tn^\sharp(s) \big) \,.
	\end{align}
\end{lem}
\begin{proof}
	The estimates \eqref{remas.1}, \eqref{remas.2} follow by \eqref{D+R+}, Lemmata 
	\ref{modulo_sumcomp-12}, \ref{modulo_expo-12}, \eqref{modulo_smooth}
	and  \eqref{gen.est},
	\eqref{small_scheme}, \eqref{tbta}, \eqref{scala.strai}, \eqref{small_KAM_con}. 
\end{proof}

\begin{lem}\label{QSKAM} Estimates \eqref{small_scheme} holds at the step $\tn+1$.
\end{lem}
\begin{proof}
	It follows by \eqref{remas.1},  \eqref{remas.2}, \eqref{small_scheme} at the step $\tn$, \eqref{tbta}, 
	the smallness condition \eqref{small_KAM_con} with $N_0=N_0(S,s_0,\tb) >0$ large enough and taking $ \tau_2 > \tau_1 +1 + \ta $. 
\end{proof}
Finally  $\bR_\perp^{(\tn+1)}$  is real, reversible and momentum preserving as $\bR_\perp^{(\tn)}$,
since $\bX_\tn$ is real, reversibility and momentum preserving. This concludes the proof of $({\bf S1})_{\tn+1}$.\\
{\sc Proof of $({\bf S2})_{\tn+1}$}. It follows by similar arguments and we omit it.
\\
{\sc Proof of $({\bf S3})_{\tn+1}$}. {Use
%as for $({\bf S4})_{\nu+1}$  of Theorem 7.3 in \cite{BM}, 
 \eqref{eigen_KAM},  \eqref{const_small}-\eqref{const_smallV},   $({\bf S2})_{\tn}$,
and  the momentum conditions in \eqref{tLambdan}.}
% implies  $ | j - j' | \lesssim N_{\tn} $}.

\subsection*{Almost invertibility of $\cL_\omega$}

By \eqref{Lperp}, \eqref{bL0}, \eqref{LomLp} and Theorem \ref{KAMRED},
we obtain
\begin{equation}\label{Lom+kam}
\cL_\omega = \bW_{\bar\tn} \bL_{\bar\tn} \bW_{\bar\tn}^{-1} + \cW^\perp\bP_{\perp,\bar\tn}(\cW^\perp)^{-1} + \cW^\perp\bQ_{\perp,\bar\tn}(\cW^\perp)^{-1} \,, \quad \bW_{\bar\tn}:= \cW^\perp \bU_{\bar\tn}\,, 
\end{equation}
where the operator $\bL_{\bar\tn}$ is defined in \eqref{bLn} with $\tn=\bar\tn $ 
and $\bP_{\perp,\bar\tn}$, $\bQ_{\perp,\bar\tn}$
satisfy the estimates in Lemma \ref{RKtn.est}. 
By \eqref{qui.1} and \eqref{flow.sharp.kam},
we have, for some $\sigma:=\sigma(\tau,\nu,k_0) > 0 $,
for any $ s_0 \leq s \leq S - \mu(\tb) - \sigma$,
\begin{equation}\label{bW.est}
 \normk{\bW_{\bar\tn}^{\pm 1}h}{s} \lesssim_{S} \normk{h}{s+\sigma} + \normk{\fI_0}{s+\mu(\tb)+\sigma}\normk{h}{s_0+\sigma} \, . 
\end{equation}
In order to verify the almost invertibility assumption (AI) of $ \cL_\omega$ 
in Section \ref{sec:approx_inv}, 
we decompose the operator $\bL_{\bar\tn}$ in \eqref{bLn} 
(with $\bar\tn$ instead of $\tn$) as
\begin{equation}\label{bL.dec}
\bL_{\bar\tn} = \bD_{\bar\tn}^{<} + \bQ_\perp^{(\bar\tn)} + \bR_\perp^{(\bar\tn)}
\end{equation}
where $\bR_\perp^{(\bar\tn)}$ satisfies \eqref{stimaRN}, whereas
\begin{equation}\label{bDbQ}
\bD_{\bar\tn}^<  := \Pi_{K_{\bar\tn}}(\omega\cdot \pa_\vf \uno_\perp + \im\, \bD_{\bar\tn}) \Pi_{K_{\bar\tn}} + \Pi_{K_{\bar\tn}}^\perp \, , \quad
\bQ_\perp^{(\bar\tn)}  := \Pi_{K_{\bar\tn}}^\perp(\omega\cdot \pa_\vf \uno_\perp + \im\, \bD_{\bar\tn}) \Pi_{K_{\bar\tn}}^\perp - \Pi_{K_{\bar\tn}}^\perp \, , 
\end{equation}
and  the smoothing operator  $\Pi_{K}$ on the traveling waves is defined in 
\eqref{pro:N}, and $ \Pi_K^\perp := {\rm Id}-\Pi_K  $. The  constants $ K_\tn $ in \eqref{bDbQ} are $ K_\tn := K_0^{\chi^\tn} $, $ \chi= 3/2 $ (cfr. \eqref{scales}), 
and $ K_0 $ will be fixed  in \eqref{param.NASH}. 

\begin{lem}\label{first.meln}
	{\bf (First order Melnikov non-resonance conditions)}
	For all $\lambda=(\omega,\gamma)$ in
	\begin{equation}\label{I.meln}
	\begin{aligned}
	\t\Lambda_{\bar\tn+1}^{\upsilon,I} 
	:= \Big\{ 
	\lambda\in\R^\nu\times[\gamma_1,\gamma_2] \, : \,  | \omega\cdot\ell +\mu_{j}^{(\bar\tn)} | \geq \upsilon \frac{\abs j^\frac12}{ \braket{\ell}^{\tau}} \, , \,  
	\forall \abs\ell\leq K_{\bar\tn}, \, j\in\S_0^c \, , j + 
	\vec \jmath \cdot \ell = 0  \Big\}\,,
	\end{aligned}
	\end{equation}
	on the subspace of the traveling waves  
	$ \tau_\vs g(\vf) = g(\vf - \ora{\jmath}\vs) $, $ \vs \in \R $,  
	such that $ g(\vf, \cdot ) \in \bH_{{\mathbb S}_0}^\bot $, 
	the operator $\bD_{\bar\tn}^<$ in \eqref{bDbQ} is invertible and there exists an extension of the inverse operator (that we denote in the same way) to the whole $\R^\nu\times [\gamma_1,\gamma_2]$ satisfying the estimate
	\begin{equation}\label{D<.inv}
	\normk{(\bD_{\bar\tn}^<)^{-1}g}{s} \lesssim_{k_0} \upsilon^{-1} \normk{g}{s+\tau_1} \,, \quad \tau_1=k_0+\tau(k_0+1) \,.
	\end{equation}
	Moreover $ (\bD_{\bar\tn}^<)^{-1} g $ is a traveling wave. 
\end{lem}

\begin{proof}
	The estimate \eqref{D<.inv} follows arguing as in Lemma \ref{X_gen.hom}. 
\end{proof}
Standard smoothing properties imply that the operator $\bQ_\perp^{(\bar\tn)}$ in \eqref{bDbQ} satisfies, for any traveling wave $ h \in \bH_{{\mathbb S}_0}^\bot $, 
for all $ b>0$, 
\begin{equation}\label{bQ.est}
\normk{\bQ_\perp^{(\bar\tn)}h}{s_0} \lesssim K_{\bar\tn}^{- b} 
\normk{h}{s_0+ b+1} \,, \quad \normk{\bQ_\perp^{(\bar\tn)}h}{s} \lesssim \normk{h}{s+1} \,.
\end{equation}
By the decompositions \eqref{Lom+kam}, \eqref{bL.dec}, Theorem \ref{KAMRED}
(note that \eqref{ansatz} and
Lemma \ref{torus_iso} imply  \eqref{ansatz_I0_s0}), 
Proposition \ref{end_redu}, 
the fact that  $ \bW_{\bar\tn} $
maps (anti)-reversible, respectively traveling, waves, into
(anti)-reversible, respectively traveling, waves (Lemma \ref{lemmaWperp})
and estimates \eqref{bW.est}, \eqref{D<.inv}, \eqref{bQ.est},  \eqref{SM12} 
we deduce the following theorem.
\begin{thm}\label{almo.inve}
	{\bf (Almost invertibility of $ \cL_\omega $)}
	Assume \eqref{ansatz}. Let $ \ta, \tb $ as in \eqref{tbta} and 
	$ M $ as in \eqref{M_choice}.  Let  $S>s_0+k_0$ and assume the smallness condition \eqref{KAM_small_cond}.  Then the almost invertibility assumption (AI) 
	in Section \ref{sec:approx_inv}  holds with $ {\mathtt \Lambda}_o $ replaced by  
	\begin{equation}\label{bLambdan}
	\b\Lambda_{\bar\tn+1}^\upsilon :=  \b\Lambda_{\bar\tn+1}^\upsilon(i)
	:= \t\Lambda_{\bar\tn+1}^\upsilon\cap \t\Lambda_{\bar\tn+1}^{\upsilon,I} \cap \tT\tC_{\bar\tn+1}(2\upsilon,\tau) \,,
	\end{equation}
	(see \eqref{tLambdan}, \eqref{I.meln}, \eqref{tDtCn}) and, with $\mu(\tb)$ defined in \eqref{cb.mub},    
	$$
	\begin{aligned}
	\cL_\omega^{<} := \bW_{\bar\tn} \bD_{\bar\tn}^< \bW_{\bar\tn}^{-1} \,, \quad 
	\cR_\omega := \bW_{\bar\tn} \bR_\perp^{(\bar\tn)} \bW_{\bar\tn}^{-1} + \cW^\perp\bP_{\perp,\bar\tn}(\cW^\perp)^{-1} \,, \\
	   \cR_\omega^\perp  := \bW_{\bar\tn} \bQ_\perp^{(\bar\tn)} \bW_{\bar\tn}^{-1} +\cW^\perp\bQ_{\perp,\bar\tn}(\cW^\perp)^{-1}\, . 
	\end{aligned}
	$$
{ In particular $ \cR_\omega  $, $ \cR_\omega^\perp  $	satisfy 
\eqref{almi1}, \eqref{almi2}, \eqref{almi3}.}
\end{thm}

\section{Proof of Theorem \ref{NMT}} %  of hypothetical conjugation}
\label{sec:NaM}

Theorem \ref{NMT} is a consequence of Theorem \ref{NASH} below. 
In turn Theorem \ref{NASH} is deduced,  in a by now standard way, from the 
almost invertibility result of $ \cL_\omega $ of Theorem \ref{almo.inve}, as in \cite{BM,BBHM,BFM}. 
{Remark that the estimates \eqref{pfi0}, \eqref{pfi1}, \eqref{pfi2}, \eqref{pfi3} coincide with  
(5.49)-(5.52) in \cite{BBHM} with $ M = 1 / 2$.} Therefore this section shall be short.

\smallskip

We consider the finite dimensional subspaces of traveling wave variations 
\begin{equation*}
E_\tn := \big\{ 
\fI(\vf)= (\Theta,I,w)(\vf) \	 {\rm such \ that } \ \eqref{mompres_aa1} \ {\rm holds} \ :  \ 
\Theta = \Pi_\tn \Theta\,, \ I=\Pi_\tn I \,, \ w = \Pi_\tn w \big\}
\end{equation*}
where $\Pi_\tn w := \Pi_{K_\tn} w $   
are defined  as in \eqref{pro:N} with  $ K_\tn $ in 
\eqref{scales}, 
and we denote with the same symbol $\Pi_\tn g(\vf)  
:= \sum_{\abs\ell\leq K_\tn} g_\ell e^{\im\ell\cdot\vf}$.  
Note that the  projector $\Pi_{\tn}$ maps (anti)-reversible traveling variations into (anti)-reversible traveling variations.

In view of the Nash-Moser Theorem \ref{NASH} we introduce the constants 
\begin{align}
&\ta_1 := \max\{ 6\sigma_1 + 13, \chi(p(\tau+1) + \mu(\tb)+2\sigma_1)+1 \} \label{a1} \, , \quad
\ta_2 := \chi^{-1} \ta_1 -\mu(\tb)-2\sigma_1  \, , \\
&  \mu_1 := 3(\mu(\tb)+2\sigma_1)+1  \, , \quad 
\tb_1 := \ta_1 + 2\mu(\tb) + 4\sigma_1 + 3 +\chi^{-1}\mu_1\,,  \quad 
\chi = 3/2 \label{b1}\\
& \sigma_1:= \max\{ \bar\sigma, 2 s_0+2k_0+5 \} \,, \quad S - \mu(\tb) - \bar\sigma= s_0 + \tb_1\,,\label{sigma1}
\end{align}
where $\bar\sigma=\bar\sigma(\tau,\nu,k_0)>0$ is defined by Theorem \ref{alm.approx.inv}, 
$ 2 s_0+2k_0+5$ is the largest loss of regularity in the estimates of the Hamiltonian vector field $X_P$ in Lemma \ref{XP_est}, $\mu(\tb)$ is defined in \eqref{cb.mub}, 
and $\tb =[\ta]+2 $ is defined in \eqref{tbta}. The exponent $p$ in \eqref{scales} 
is required to satisfy
\begin{equation}\label{p.cond}
p \ta >
\tfrac12 \ta_1 + \tfrac32 \sigma_1 \,.
\end{equation}
By  \eqref{tbta}, 
and  the definition of $\ta_1$ in \eqref{a1}, there exists $p=p(\tau,\nu,k_0)$ such that \eqref{p.cond} holds, for example  we fix
$ p:=\frac{3(\mu(\tb)+4\sigma_1+1)}{\ta} $. 

%\begin{rem}\label{rem:cond2.param}
%	The constant $\ta_1$ is the exponent in \eqref{P2.2}. The constant $\ta_2$ is the exponent in the second bound in \eqref{P1.2}. The constant $\mu_1$ is the exponent in $(\cP 3)_\tn$. 
%	The conditions on the constants $ \mu_1, \tb_1, \ta_1 $ to  allow the convergence of the Nash-Moser scheme in Theorem \ref{NASH} are	
%	\begin{equation*}
%	\ta_1 >  
%	6\sigma_1+12 \,, \quad \tb_1 > \ta_1 +  2\mu(\tb)+4\sigma_1  +\chi^{-1}\mu_1 \,, \quad p\ta > \tfrac12 \ta_1 + \tfrac32 \sigma_1 \, , 
%	\end{equation*}
%	as well as
%	$\mu_1 >  3(\mu(\tb)+2\sigma_1) $.  
%	In addition, we require
%	$		\ta_1 \geq \chi(p(\tau+1) + \mu(\tb)+2\sigma_1) + 1 $ 
%	so that $\ta_2\geq p(\tau+1) +\chi^{-1} $, which is used in  
%	the proof of Lemma \ref{sub1}.
%\end{rem}
Given a function
$ W = (\fI,\beta) $ where $ \fI  $ is the periodic component of a torus as in
\eqref{ICal} and $ \beta  \in  \R^\nu $, 
we denote $	\normk{W}{s} := \normk{\fI}{s}+\abs\beta^{k_0,\upsilon} $.

\begin{thm}{\bf (Nash-Moser)} \label{NASH}
	There exist $\delta_0, C_*>0$ such that, if
	\begin{equation}\label{param.NASH}
	K_0^{\tau_3} \varepsilon\upsilon^{-4} < \delta_0 \,, \
	\tau_3:= \max\{ p\tau_2, 2\sigma_1+\ta_1+4 \} \,, \
	K_0 := \upsilon^{-1}\,, \
	\upsilon:= \varepsilon^{\rm a}\,, \  0< {\rm a} <(4+\tau_3)^{-1}\,,
	\end{equation}
	where $\tau_2=\tau_2(\tau,\nu)$ is given by Theorem \ref{iterative_KAM}, then, for all $\tn\geq 0$:
	\begin{itemize}
		\item[$(\cP 1)_\tn$] There exists a $k_0$-times differentiable function $\wtW_\tn:\R^\nu\times[\gamma_1,\gamma_2]\rightarrow E_{\tn-1}\times \R^\nu$, $\lambda=(\omega,\gamma)\mapsto \wtW_\tn(\lambda):= (\wt\fI_\tn, \wt\alpha_\tn-\omega)$, for $\tn \geq 1 $, and $\wtW_0:=0$, satisfying
		\begin{equation}\label{P1.1}
		\normk{\wtW_\tn}{s_0+\mu(\tb)+\sigma_1} \leq C_*\varepsilon\upsilon^{-1} \,.
		\end{equation}
		Let $\wtU_\tn:= U_0+\wtW_\tn$, where $U_0:= (\vf,0,0,\omega)$. The difference $\wtH_\tn:= \wtU_\tn-\wtU_{\tn-1}$, for $\tn \geq 1 $, satisfies
		\begin{equation}\label{P1.2}
		\begin{aligned}
		& \normk{\wtH_1}{s_0+\mu(\tb)+\sigma_1}\leq C_* \varepsilon\upsilon^{-1}\,, \quad
		\normk{\wtH_\tn}{s_0+\mu(\tb)+\sigma_1} \leq C_* \varepsilon\upsilon^{-1} K_{\tn-1}^{-\ta_2}\,, \ \forall\, \tn\geq 2 \,.
		\end{aligned}
		\end{equation}
		The torus embedding $ \wti_\tn := (\vf,0,0) + \wt\fI_\tn $ 
		is  reversible and   traveling, 
		i.e.  \eqref{RTTT} holds.
		\item[$(\cP 2)_\tn$] 
		We define
		\begin{equation}\label{P2.1}
		\cG_0:= \t\Omega \times [\gamma_1,\gamma_2]\,, \quad  \cG_{\tn+1}:= \cG_{\tn} \cap \b\Lambda_{\tn+1}^\upsilon(\wti_\tn) \,, \quad \forall\,\tn \geq 0 \,,
		\end{equation}
		where $\b\Lambda_{\tn+1}^\upsilon(\wti_\tn)$ is defined in \eqref{bLambdan}. Then, for all $\lambda \in \cG_{\tn}$ , setting $K_{-1}:=1$, we have
		\begin{equation*}\label{P2.2}
		\normk{\cF(\wtU_\tn)}{s_0} \leq C_* \varepsilon K_{\tn-1}^{-\ta_1} \,.
		\end{equation*}
		\item[$(\cP 3)_\tn$]{\sc (High norms)} 
		For all $\lambda \in \cG_{\tn}$, we have
		$ \normk{\wtW_\tn}{s_0+\tb_1} \leq C_* \varepsilon\upsilon^{-1} K_{\tn-1}^{\mu_1} $. 
	\end{itemize}
\end{thm}

\begin{proof}
	The inductive proof follows exactly as in \cite{BM,BBHM}. The 
	verification that each approximate torus 
	$ \wti_\tn $ is reversible and traveling is given in  \cite{BFM}. 
\end{proof}

Theorem \ref{NMT} is a by now standard corollary of Theorem \ref{NASH}, as in
\cite{BM,BBHM,BFM}.
% \paragraph{Proof of Theorem \ref{NMT}.}
Let $\upsilon = \varepsilon^{\rm a}$, with $0<{\rm a}<{\rm a_0}:= 1/(4+\tau_3)$. Then, the smallness condition in \eqref{param.NASH} is verified for $0<\varepsilon<\varepsilon_0$ small enough and Theorem \ref{NASH} holds. 
By \eqref{P1.2},
the sequence of functions
$	\wtW_\tn = \wtU_\tn - (\vf,0,0,\omega) = 
(\wt\fI_\tn,\wt\alpha_\tn-\omega) $ 
converges to a function
$	W_\infty : \R^\nu\times [\gamma_1,\gamma_2]
\rightarrow H_\vf^{s_0} \times H_\vf^{s_0} \times H^{s_0} \times \R^\nu $, and
we define
$$ 
U_\infty := (i_\infty,\alpha_\infty) := (\vf,0,0,\omega) + W_\infty \, . 
$$ 
The torus $ i_\infty $ is reversible and traveling, i.e. \eqref{RTTT} holds. 
By \eqref{P1.1}, \eqref{P1.2}, we also deduce the bounds
\begin{equation}\label{Uinfty.est}
\begin{aligned}
&\normk{U_\infty-U_0}{s_0+\mu(\tb)+\sigma_1} \leq C_* \varepsilon\upsilon^{-1} \,,
\quad
\normk{U_\infty-\wtU_\tn}{s_0+\mu(\tb)+\sigma_1} \leq C \varepsilon\upsilon^{-1} K_\tn^{-\ta_2} \,, \  \forall\,\tn \geq 1 \, .
\end{aligned}
\end{equation}
In particular \eqref{alpha_infty}-\eqref{i.infty.est} hold.
By Theorem \ref{NASH}-$(\cP 2)_\tn$, we deduce that $\cF(\lambda;U_\infty(\lambda))=0$ for any $\lambda $ in the set 
$$
\bigcap_{\tn\in\N_0} \cG_{\tn} = \cG_0 \cap \bigcap_{\tn \geq 1} \b\Lambda_\tn^\upsilon(\wti_{\tn-1}) \stackrel{\eqref{bLambdan}}{=} \cG_0 \cap \Big[ \bigcap_{\tn \geq 1} \t\Lambda_\tn^\upsilon (\wti_{\tn-1}) \Big] 
\cap \Big[ \bigcap_{\tn \geq 1} \t\Lambda_\tn^{\upsilon,I}(\wti_{\tn-1}) \Big] \cap \Big[ \bigcap_{\tn \geq 1} \tT\tC_{\tn}(2\upsilon,\tau)(\wti_{\tn-1}) \Big]\,,
$$
where $ \cG_0:= \t\Omega\times[\gamma_1,\gamma_2] $. 
To conclude the proof of Theorem \ref{NMT} 
it remains only to define the $  \mu_j^\infty $ in \eqref{def:FE} and prove that the  set $\cC_\infty^\upsilon$ in \eqref{dioph0}-\eqref{2meln+} is contained in $\cap_{\tn \geq 0}\cG_{\tn} $. We first define 
\begin{equation}\label{Ginfty}
\cG_\infty := \cG_0 \cap \Big[ \bigcap_{\tn \geq 1} \t\Lambda_\tn^{2\upsilon} (i_\infty) \Big] \cap \Big[ \bigcap_{\tn \geq 1} \t\Lambda_\tn^{2\upsilon,I}(i_\infty) \Big] \cap \Big[ \bigcap_{\tn \geq 1} \tT\tC_\tn(4\upsilon,\tau)(i_\infty) \Big] \, . 
\end{equation}
Using that the approximate solution 
$ \wtU_\tn $ is exponentially close to the limit $ U_\infty $ 
according to \eqref{Uinfty.est}, 
and relying on the inclusion properties of the set of non-resonant parameters 
stated precisely in Lemmata \ref{inclu.fgmp} and \eqref{INCLPRO}, one
directly deduces   the following lemma (cfr. e.g. Lemma 8.6 in \cite{BM}).

\begin{lem}\label{sub1}
	$\cG_\infty \subseteq \cap_{\tn \geq 0}\cG_{\tn}$, where $\cG_{\tn}$ are defined in \eqref{P2.1}.
\end{lem}

Then we define the $ \mu_{j}^\infty$ in \eqref{def:FE}
with
$  \tm_{1,\tn}^\infty := \tm_{1,\tn}(i_\infty) $, $ \tm_\frac12^\infty=\tm_\frac12(i_\infty)$, 
$ \tm_{0}^\infty= \tm_{0}(i_\infty)$, 
and  $\tm_{1,\tn}, \tm_\frac12, \tm_{0}$ are provided in Proposition \ref{end_redu}.
By \eqref{rem.eigen.KAM}, 
the sequence $(\fr_j^{(\tn)}(i_\infty))_{\tn\in\N}$, 
with $\fr_j^{(\tn)}$  given by Theorem \ref{iterative_KAM}-$({\bf S1})_\tn$ (evaluated at 
$ i = i_\infty $), is a Cauchy sequence in 
$|\,\cdot\, |^{k_0,\upsilon}$. 
Then we define 
$  \fr_j^\infty := \lim_{\tn\to \infty} \fr_j^{(\tn)} (i_\infty) $, for any $ j\in\S_0^c $,  
which satisfies 
$ |j|^\frac12| \fr_j^\infty - \fr_j^{(\tn)}(i_\infty)|^{k_0,\upsilon} \leq 
C \varepsilon\upsilon^{-3} N_{\tn-1}^{-\ta} $ for any $  \tn \geq 0 $. 
Then, recalling $\fr_j^{(0)}(i_\infty) = 0 $ and \eqref{const_small},  
the estimates \eqref{coeff_fin_small} hold 
(here $C= C(S)$ with $S $ fixed in \eqref{sigma1}).
Finally one checks (see e.g.  Lemma 8.7 in \cite{BM})
that the  Cantor set $\cC_\infty^\upsilon$ in \eqref{dioph0}-\eqref{2meln+} satisfies
$\cC_\infty^\upsilon\subseteq \cG_\infty$, with $\cG_\infty$ defined in \eqref{Ginfty}, 
and  Lemma \ref{sub1} implies 
that $\cC_\infty^\upsilon\subseteq \cap_{\tn \geq 0} \cG_{\tn}$.
This concludes the proof of Theorem \ref{NMT}.

\appendix

\section{Almost straightening of  a transport operator}\label{app:FGMP}

The main results of this appendix are Theorem \ref{thm:as} and Corollary \ref{cor.as}.
The goal is to almost-straighten a
linear quasi-periodic transport operator of the form 
\begin{equation}\label{defX0}
X_0 :=\omega\cdot\pa_\vf + p_0 (\vphi, x) \pa_x\,, 
\end{equation}
to a constant coefficient one $\omega\cdot\pa_\vf + \tm_{1,\tn}\pa_x$, up to a small 
term $p_{\tn}\pa_x$, see \eqref{Xtn.gtn} and \eqref{stime.pn.w}. 
%The reason why we do not perform a complete reduction to constant coefficients of the transport operator but we allow
%the small remainder $p_{\tn}\pa_x$ is to have the inclusion  Lemma \ref{inclu.fgmp}.
We follow the scheme of Section 4 in \cite{BM20}.

We first introduce the following weighted-graded Sobolev norm: for any $u=u(\lambda)\in H^{s}(\T^{\nu+1})$,  $s \in \R $,  $k_0$-times differentiable with respect to $\lambda=(\omega,\gamma)\in\R^\nu\times[\gamma_1,\gamma_2]$, 
we define  the weighted graded Sobolev norm 
\begin{equation*}
	\absk{u}{s}:= \sum_{k\in\N^{\nu+1} \atop 0\leq |k|\leq k_0} \upsilon^{|k|}\sup_{\lambda\in\R^\nu\times[\gamma_1,\gamma_2]}\| \pa_\lambda^k u(\lambda) \|_{s-|k|}\,.
\end{equation*}
This norm satisfies usual tame and interpolation estimates. 
The main reason to use this norm is the estimate \eqref{est:compo} for the composition 
 operator where there is no loss of $ k_0 $-derivatives on the highest norm
 $ | u |_{s}^{k_0,\upsilon} $, unlike 
the corresponding estimate \eqref{est:compo-loss} for the $ \| \ \|_{s}^{k_0,\upsilon} $.
This is used in a crucial way to prove   \eqref{ptn.s+b} and then 
deduce the a-priori bound \eqref{stime.pn.w}
for the divergence of the high norms of the functions $ p_{\tn} $. 
 In the following we consider
\begin{equation*}
	\mathfrak{s}_0:=s_0+k_0 >\frac12(\nu+1)+k_0\,.
\end{equation*}
We report the following estimates which can be proved by 
adapting the arguments of  \cite{BM}.
\begin{lem}\label{proprieta}
	The following hold:
\\[1mm]
(i)  For any $s \in \R$, we have $\absk{u}{s}\leq \normk{u}{s}\leq \absk{u}{s+k_0}$.
\\[1mm]
(ii)  For any  $s\geq \mathfrak{s}_0$, 
		we have 
		$
		\absk{uv}{s}\leq C(s)\absk{u}{s}\absk{v}{\mathfrak{s}_0} +C(\mathfrak{s}_0)\absk{u}{\mathfrak{s}_0}\absk{v}{s}.
		$ 
		The tame constant $C(s) := C(s,k_0)$ is monotone in $s\geq \mathfrak{s}_0$.
\\[1mm]
(iii) For $ N \geq 1$ and $ \alpha \geq 0 $ 
		we have $\absk{\Pi_N u}{s}\leq N^{\alpha}\absk{u}{s-\alpha}$  and $\absk{\Pi_N^\perp u}{s}\leq N^{-\alpha}\absk{u}{s+\alpha}$,  $ \forall s \in \R $.
\\[1mm]
(iv) Let $\absk{\beta}{2\mathfrak{s}_0+1}\leq \delta(\mathfrak{s}_0)$ small enough. Then the composition operator $\cB$ defined as in \eqref{defcB} satisfies the tame estimate, for any $s\geq \mathfrak{s}_0 +1 $,
		\begin{equation}\label{est:compo}
			\absk{\cB u}{s} \leq C(s) (\absk{u}{s} + \absk{\beta}{s} \absk{u}{\mathfrak{s}_0+1} ) . 
		\end{equation}
The tame constant $C(s) := C(s,k_0) $ is monotone in $s\geq \mathfrak{s}_0$. 

Moreover the diffeomorphism $x\mapsto x + \beta(\vf,x)$ is invertible and the inverse diffeomorphism $y \mapsto y + \breve{\beta}(\vf,y)$  satisfies, for any $s\geq \mathfrak{s}_0$, 
$\absk{\breve{\beta}}{s}\leq C(s) \absk{\beta}{s} $.
\\[1mm]
%(v) For any $a_2\leq a_1\leq b_1\leq b_2$ such that $a_1+b_1=a_2+b_2$, we have that $$\absk{u}{a_1}\absk{v}{b_1}\leq \absk{u}{a_2}\absk{v}{b_2}
%		+ \absk{u}{b_2}\absk{v}{a_2}
%		\,.$$
(v) For any $\epsilon>0$, $a_0, b_0\geq 0$ and $p,q>0$, there exists $C_\epsilon=C_\epsilon(p,q)>0$, with $C_1<1$, such that $$\absk{u}{a_0+p}\absk{v}{b_0+q}\leq \epsilon\absk{u}{a_0+p+q}\absk{v}{b_0}
+ C_\epsilon\absk{u}{a_0}\absk{v}{b_0+p+q}
\,.$$
\end{lem}

We now state the  almost straightening result of the quasi-periodic transport operator. 
Remind that $ N_{\tn}:=N_0^{\chi^{\tn}} $, $ \chi=3/2 $, $ N_{-1}:=1 $, see \eqref{scala.strai}. 

\begin{thm}[{\bf Almost straightening}]\label{thm:as}
Consider the quasi-periodic transport operator $ X_0 $  in \eqref{defX0}
where  $ p_0 (\vphi,x) $ is a quasi-periodic traveling wave, $ {\rm even} (\vphi,x) $, defined for all $ (\omega, \gamma) \in  \R^\nu \times [\gamma_1, \gamma_2] $. 
For any $S>\mathfrak{s}_0$, there exist $\tau_2> \tau_1 + 1+  \ta$, 
$\updelta := \updelta(S,\mathfrak{s}_0,k_0,\tb) \in (0,1) $ and  
	$N_0 :=N_0(S,\mathfrak{s}_0,k_0,\tb) \in \N $ (with $\tau_1$, $\ta$, $\tb$ defined in \eqref{tbta}) such that, if 
	\begin{equation}\label{small.V.as.AP}
		N_0^{\tau_2} \,  \absk{p_0}{2\mathfrak{s}_0+\tb+1} \,  \upsilon^{-1} \leq \updelta 
		 \,,
	\end{equation}
	then, for any $\bar\tn\in\N_0$, for any $\tn = 0, \ldots, \bar \tn$, the following holds true:
	\\[1mm]
	$ \bf (S1)_\tn$ There exists a linear quasi-periodic transport operator
	\begin{equation}\label{Xtn.gtn}
			X_{\tn} := \omega\cdot \pa_\vf + (\tm_{1,\tn} + p_{\tn}(\vf,x)) \pa_x \,,	
	\end{equation}
defined for all $(\omega, \gamma) \in \R^\nu \times [\gamma_1, \gamma_2]$,
where  $p_\tn (\vf, x) $ is a quasi-periodic traveling wave function, 
		${\rm even}(\vf, x)$, such that,
		for any $\mathfrak{s}_0\leq s \leq S$, 
	\begin{equation}\label{stime.pn.w}
		\absk{p_{\tn}}{s} \leq C(s,\tb) N_{\tn-1}^{-\ta}\absk{p_0}{s+\tb}\,, \quad \absk{p_{\tn}}{s+\tb} \leq C(s,\tb) N_{\tn-1}\absk{p_0}{s+\tb}\,,
	\end{equation}
for some constant $C(s,\tb) \geq 1 $ monotone in $s\in[\mathfrak{s}_0,S]$, and 
$\tm_{1,\tn} $ is a real constant satisfying
	\begin{equation}\label{tm.est}
		|\tm_{1,\tn}|^{k_0,\upsilon} \leq 2 %C(\mathfrak{s}_0, \tb) 
		\,  \absk{p_0}{\mathfrak{s}_0+\tb} \, , \quad 
		|\tm_{1,\tn}-\tm_{1,\tn-1}|^{k_0,\upsilon}  
		\leq C(\mathfrak{s}_0,\tb) N_{\tn-2}^{-\ta} \absk{p_0}{\mathfrak{s}_0+\tb} \, , 
		\, \forall \tn \geq 2 
		 \, . 
		\end{equation}
		Let $ \t\Lambda_{0}^{\rm T}  := \R^\nu \times [\gamma_1, \gamma_2] $, 
		and, for $\tn \geq 1$,  
	\begin{equation}\label{set.nonres.tn}
	\begin{aligned}
		\t\Lambda_{\tn}^{\rm T} & :=	\t\Lambda_{\tn}^{\upsilon,\rm T}(p_0)\\
		&  :=\big\{ (\omega,\gamma)\in\t\Lambda_{\tn-1}^{\rm T} \,:\, |(\omega-\tm_{1,\tn-1}\ora{\jmath})\cdot\ell|\geq \upsilon \braket{\ell}^{-\tau} \ \forall\,\ell\in\Z^{\nu}\setminus\{0\} \,, \ |\ell|\leq N_{\tn-1}  \big\}\, . 
		\end{aligned}
	\end{equation}	
For $\tn \geq 1$, 
there exists  a  quasi-periodic traveling wave function  $g_{\tn-1}(\vf,x)$, ${\rm odd}(\vf, x)$, 
defined for all $ (\omega, \gamma) \in \R^\nu \times [\gamma_1, \gamma_2] $,  fulfilling for any $\mathfrak{s}_0\leq s \leq S$,
	\begin{equation}\label{gtn.est.better}
		\absk{g_{\tn-1}}{s}\leq C(s) N_{\tn-1}^{\tau_1}\upsilon^{-1}\absk{\Pi_{N_{\tn-1}}p_{\tn-1}}{s} \, , 
	\end{equation}
for some  constant $C(s) \geq 1 $ monotone in $s\in[\mathfrak{s}_0,S]$, such that, defining the composition operators 
	\begin{equation*}
		(\cG_{\tn-1} u)(\vf,x) := u(\vf,x+ g_{\tn-1}(\vf,x))\,, \ \  (\cG_{\tn-1}^{-1}u)(\vf,y) := u(\vf,y+\breve{g}_{\tn-1}(\vf,y))\,,
	\end{equation*}
	where $  x=y+\breve{g}_{\tn-1}(\vf,y)$ is the inverse diffeomorphism of 
	$ y = x + g_{\tn-1}(\vf,x) $, 
	the following conjugation formula holds: 	for any $(\omega,\gamma)$ in the set $ \t\Lambda_{\tn}^{\rm T} $ (cfr. \eqref{set.nonres.tn}) we have 
	\begin{equation}\label{Xtn.gtn2}
			X_{\tn} =\cG_{\tn-1}^{-1} \,  X_{\tn-1}  \, \cG_{\tn-1}	 \, . 
	\end{equation}
$ \bf (S2)_\tn$ % Given   $ p_{0}(i_1), p_{0}(i_2) $ let
Let $ \Delta_{12} p_0 := p_{0,1} - p_{0,2}  $. For any $s_1 \in [s_0+1, S] $, there exist $C(s_1)>0$ and 
 $\updelta'(s_1)\in(0,1)$ such that
 if 
\begin{equation}\label{small12}
	N_0^{\tau_2} \sup_{(\omega, \gamma) \in \R^\nu \times [\gamma_1, \gamma_2]}
	\big( \| p_{0,1} \|_{s_1 + \tb}+\| p_{0,2} \|_{s_1+\tb} \big)
	\upsilon^{-1} \leq \updelta'(s_1)  \,,
\end{equation}  
then, for all $(\omega, \gamma) \in \R^\nu \times \R$,
%$(\omega, \gamma) \in  \t\Lambda_{\tn}^{\upsilon, \rm T}(p_{0,1}) \cap \t\Lambda_{\tn}^{\upsilon,\rm T}(p_{0,2}) $, 
\begin{align}
\label{estp121}
&\|\Delta_{12} p_\tn\|_{s_1-1}  \leq C({s_1})   N_{\tn-1}^{-\ta}\|\Delta_{12} p_0\|_{s_1+\tb} \, , 
\quad \|\Delta_{12} p_\tn\|_{s_1+\tb} \leq C({s_1})
N_{\tn-1}\|\Delta_{12} p_0\|_{s_1+\tb} \\
& |\Delta_{12}(\tm_{1, \tn+1}-\tm_{1, \tn})| \leq \|\Delta_{12}p_\tn\|_{s_0}\,, \quad |\Delta_{12} \tm_{1,\tn}| \leq C({s_1}) \|\Delta_{12} p_0\|_{s_0}\,.\label{estm121}
\end{align}	 
Moreover  for any $s \geq s_0$,
\begin{align}
	& \|\Delta_{12} g_\tn\|_{s} \lesssim_{s} \upsilon^{-1}\big( \|\Pi_{N_{\tn}}\Delta_{12}p_\tn\|_{s+\tau} + \upsilon^{-1} |\Delta_{12}\tm_{1, \tn}|\|\Pi_{N_{\tn}} p_{\tn,2} \|_{s+2\tau+1} \big)\,. \label{estg12}
\end{align}
\end{thm}

We deduce the following corollaries.

\begin{cor}\label{cor.TC}
For any $ \bar \tn \in \N_0 $ we have the inclusion 
$ \tT\tC_{\bar \tn+1}(\tm_{1,\bar \tn}, 2\upsilon,\tau) 
\subset \t\Lambda_{\bar \tn+1}^{\upsilon, \rm T} 
$ where the set $ \tT\tC_{\bar \tn+1}(\tm_{1,\bar \tn},2\upsilon,\tau)$ is defined in \eqref{tDtCn}.
\end{cor}

\begin{proof}
When $\bar\tn=0$, by definition we have $\tT\tC_1(2\upsilon,\tau) \subset \t\Lambda_1^{\upsilon,\rm T}$.
Let $(\omega,\gamma)\in\tT\tC_{\bar \tn+1}(\tm_{1,\bar \tn},2\upsilon,\tau)$. 
For any  $ k = 0, \ldots, \bar \tn-1$ we have, by \eqref{tm.est}, 
\begin{equation}\label{differe}
| \tm_{1,\bar \tn} - \tm_{1,k} | \lesssim_{\mathfrak{s}_0,\tb} 
N_{k-1}^{-\ta} \absk{p_0}{\mathfrak{s}_0+\tb}  \, .
\end{equation}
By  \eqref{tDtCn} and \eqref{differe}, for all $ 0 < |\ell | \leq N_{k} $, 
\begin{align*}
| (\omega-\tm_{1,k}\ora{\jmath})\cdot\ell| & \geq  | (\omega-\tm_{1, \bar \tn}\ora{\jmath})\cdot\ell| -  |\tm_{1, \bar \tn} - \tm_{1,k}||\ora{\jmath}| |\ell | \\
& \geq 2\upsilon\braket{\ell}^{-\tau}  -   C N_{k-1}^{-\ta} \absk{p_0}{\mathfrak{s}_0+\tb}   |\ell |
\geq \upsilon\braket{\ell}^{-\tau}
\end{align*}
if $C  N_{k}^{\tau+1} N_{k-1}^{-\ta}\absk{p_0}{\mathfrak{s}_0+\tb} \upsilon^{-1} < 1 $, which is satisfied by \eqref{small.V.as.AP} and \eqref{tbta}.
Thus, recalling \eqref{set.nonres.tn}, 
we have proved that $(\omega,\gamma)\in\t\Lambda_{\bar \tn+1}^{\upsilon,\rm T}$. 
\end{proof}

%Let  $ \cB_{\tn}:= \cG_{0} \circ \cdots \circ \cG_{\tn-1}  $ for $ \tn \in \N $. 

\begin{cor}
\label{cor.as}
For any $\bar\tn\in\N_0$ and $(\omega,\gamma)\in\tT\tC_{\bar\tn+1}(\tm_{1,\bar\tn},2\upsilon,\tau)$
we have the conjugation formula 
$$
X_{\bar\tn} = \cB_{\bar \tn}^{-1}X_0 \cB_{\bar\tn} \qquad \text{ where} \qquad \cB_0:={\rm Id} \,, \quad
\cB_{\bar \tn}:= \cG_{0} \circ \cdots \circ \cG_{\bar\tn-1} \,, \ \bar\tn\geq 1\,,  
$$ 
and  $X_{\bar\tn}$ is given in \eqref{Xtn.gtn} with $\tn=\bar\tn$. Moreover, when $\bar\tn\geq 1$, for 
any $ \tn = 1, \ldots,  \bar \tn $, each $\cB_{ \tn} $ 
is the composition  operator 
induced by the diffeomorphism of the torus $x \mapsto x +\beta_{ \tn}(\vf,x)$,
$ (\cB_{ \tn}u)(\vf, x) = u(\vf,x+\beta_{ \tn}(\vf,x)) $, where
the function $\beta_{ \tn}$
is a quasi-periodic traveling wave, ${\rm odd}(\vf, x)$, satisfying, 
for any $\mathfrak{s}_0\leq s\leq S$, for some  constant $\underline{C}(S) \geq 1 $,  
\begin{equation}\label{stima.B2}
 \absk{\beta_{ \tn}}{s} \leq  \underline{C}(S)\upsilon^{-1} N_{0}^{\tau_1}\absk{p_0}{s+\tb}\, . 
\end{equation}
Furthermore, %  for any $s_1 \geq s_0 + 1 $,  
for $p_{0, 1}, p_{0,2} $ fulfilling \eqref{small12}, we have $\|\Delta_{12}\beta_{\bar\tn}\|_{s_1} \leq \bar C({S})\upsilon^{-1} N_0^{\tau} \|\Delta_{12}p_0\|_{s_1+\tb}$.
\end{cor}

\begin{proof}
Let $\bar\tn\geq 1$ and we argue by induction on $ \tn = 1, \ldots, \bar \tn $.
 For $ \tn=1$ we have that $\beta_1= g_0 $. 
 Hence, using \eqref{gtn.est.better}, we get, for any 
 $ \mathfrak{s}_0\leq s\leq S $,  
\begin{equation}\label{base}
	\absk{\beta_1}{s} \leq   C(S)\upsilon^{-1} N_{0}^{\tau_1 }\absk{p_0}{s}\, ,
\end{equation}
which proves \eqref{stima.B2} for $ \tn = 1 $ and provided $\underline{C}(S) \geq  C(S)$. 
If $ \bar \tn \geq 2 $, for $ \tn = 2 , \ldots, \bar \tn  $ the operator 
$\cB_{ \tn  } =\cB_{ \tn-1}\circ\cG_{\tn-1}$
is the composition operator induced by the diffeomorphism 
\begin{equation}\label{diff.n}
	\beta_{\tn}(\vf,x) = \beta_{\tn-1}(\vf,x) + g_{\tn-1}(\vf,x+\beta_{\tn-1}(\vf,x)) 
	= \beta_{\tn-1}(\vf,x) + \{\cB_{\tn-1} g_{\tn-1}\}(\vf,x)\,.
\end{equation}
Since $ g_0 (\vf,x)  $ is a quasi-periodic traveling wave ${\rm odd}(\vf, x)$ each 
$ \beta_{\tn}(\vf,x) $ is a quasi-periodic traveling wave ${\rm odd}(\vf, x)$. 
We now assume by induction that \eqref{stima.B2} up to $  \tn - 1  $. 
We first prove that, for any $ k = 2,..., \tn  $, we have, 
 for any $\mathfrak{s}_0 \leq s \leq S$,
	\begin{align}
		\absk{\beta_{k}-\beta_{k-1}}{s} & \stackrel{\eqref{diff.n}} 
		= \absk{\cB_{k-1} g_{k-1}}{s} 
		\stackrel{\eqref{est:compo}} 
		{\leq} C(s) \left( \absk{g_{k-1}}{s}+  \absk{\beta_{k-1}}{s}\absk{g_{k-1}}{\mathfrak{s}_0+1}\right) \notag \\
		\notag
		& \stackrel{\eqref{gtn.est.better},\eqref{stime.pn.w},\eqref{stima.B2}} 
		{\leq} C(S,\tb)
		N_{k-1}^{\tau_1} \upsilon^{-1} N_{k-2}^{-\ta}\absk{p_0}{s+\tb}
		\\
		& \qquad \qquad   \qquad 
		 + C(S,\tb) \underline{C}(S) \upsilon^{-2} N_{0}^{\tau_1}  	 N_{k-1}^{\tau_1}N_{k-2}^{-\ta}
		 \absk{p_0}{s+\tb}\absk{p_0}{\mathfrak{s}_0+\tb+1} \notag  \\
		& \stackrel{\eqref{small.V.as.AP}} {\leq}
		C(S, \tb) \, (1 + \underline{C}(S) ) \, 
		 \upsilon^{-1}N_{k-1}^{\tau_1}N_{k-2}^{-\ta} 
		 \absk{p_0}{s+\tb} \, .  \label{stima.B1}
	\end{align}
 By
 \eqref{stima.B1} and  \eqref{base}, we derive, for any $ \tn = 2 , \ldots, \bar \tn  $ and setting  $b:= \ta  -  \frac12 \tau_1  \geq  1  $  (see 
 \eqref{tbta})
\begin{align*}
		\absk{\beta_{\tn}}{s} 
		& \leq \sum_{k=2}^{\tn}\absk{\beta_{k}-\beta_{k-1}}{s} + \absk{\beta_1}{s} \\
	& 	\leq  
		\big(  C(S, \tb) \, (1+ \underline{C}(S)) \, N_0^{-b} + C(S) \big) \upsilon^{-1}N_{0}^{\tau_1}\absk{p_0}{s+\tb}\\
		& \leq \underline{C}(S) \upsilon^{-1}N_{0}^{\tau_1}\absk{p_0}{s+\tb}\, 
\end{align*}
provided $C(S, \tb) \, N_0^{-b} \leq \frac12$ and  $\underline{C}(S) :=  1+ 2 C(S)$.
This proves \eqref{stima.B2} at the step 
$\tn+1$. 

The estimate for $\Delta_{12}\beta_{\bar \tn}$ follows similarly 
%First note that, 
by  \eqref{estp121}--\eqref{estg12},  \eqref{stime.pn.w}, 
% the smallness assumption \eqref{small.V.as.AP} and 
\eqref{small12}.
\end{proof}

\begin{rem}
If the function $ p_0(\vphi,x)$ in \eqref{defX0} is not  a quasi-periodic traveling wave $ p_0(\vphi,x)$, the same kind of conjugation result holds requiring in
\eqref{set.nonres.tn} the non resonance conditions 
$$
		| \omega \cdot \ell + \tm_{1,\tn-1} j |\geq \upsilon \braket{\ell}^{-\tau}, \  \forall\,
		(\ell, j) \in (\Z^{\nu} \times \Z) \setminus\{0\} \,, \ |(\ell,j)|\leq N_{\tn-1} \, . 
$$
\end{rem}

\begin{proof}[Proof of Theorem \ref{thm:as}]
The proof is inductive. 
In Lemma \ref{stime.iter.ptn} 
we prove that
the norms $ |p_\tn|_s^{k_0, \upsilon} $ satisfy 
inequalities typical of
a Nash-Moser iterative scheme, which converges 
under the  smallness low norm condition \eqref{small.V.as.AP}.
\\[1mm]
\noindent{\bf The step $\tn = 0$}. The items $\bf (S1)_0$, $\bf (S2)_0$, %$(S_3)_0$ 
hold with  $\tm_{1,0}:=0$ (the estimates \eqref{tm.est},  \eqref{stime.pn.w} are trivial, as well as \eqref{estp121}-\eqref{estm121}). 
\\[1mm]
\noindent{\bf The reducibility step.} 
We now describe the generic inductive step, showing how to transform
$X_{\tn}$ in \eqref{Xtn.gtn}  into $X_{\tn+1 }$  by  conjugating with the 
composition operator
$\cG_{\tn}$ induced by a diffeomorphism $ x + g_{\tn} (\varphi,x) $ for a periodic
function $ g_{\tn} (\varphi,x) $ (defined in \eqref{gtn.def}).
A direct computation gives 
\begin{equation*}
	\begin{aligned}
	\cG_{\tn}^{-1} \, X_{\tn} \, \cG_{\tn }
&= \omega\cdot\pa_\vf + \{\cG_{\tn}^{-1}\big( \omega\cdot\pa_\vf g_{\tn} + (\tm_{1,\tn}+p_{\tn})(1+(g_{\tn})_x) \big)\} \pa_y \\
& =  \omega\cdot\pa_\vf + 
 \tm_{1,\tn} \pa_y + 
\{\cG_{\tn}^{-1}\big( (\omega\cdot\pa_\vf  + \tm_{1,\tn} \pa_x )g_{\tn} +p_{\tn}+ p_{\tn}(g_{\tn})_x \big)\} \pa_y \, . 
	\end{aligned}
	\end{equation*}
	We choose  $g_{\tn}(\vf,x)$ as the solution of  the homological equation
	\begin{equation}
	\label{he.as}
 (\omega\cdot\pa_\vf + \tm_{1,\tn}\pa_x)	g_{\tn}(\vf,x) +\Pi_{N_{\tn}} p_{\tn} = \braket{p_{\tn}}_{\vf,x}   
	\end{equation}
where $ \braket{  p_{\tn} }_{\vf,x}  $  is the average of $p_{\tn} $ defined as in \eqref{def:avera}. So 
 we define
	\begin{equation}\label{gtn.def}
			g_{\tn}(\vf,x) := - (\omega\cdot\pa_\vf + \tm_{1,\tn}\pa_x)_{\rm ext}^{-1}(\Pi_{N_{\tn}} p_{\tn}- \braket{p_{\tn}}_{\vf,x}) 
	\end{equation}
	where the operator $(\omega\cdot\pa_\vf + \tm_{1,\tn}\pa_x)_{\rm ext}^{-1}$ is introduced 
	in \eqref{paext}.	
The function  $g_{\tn}(\vf,x) $ is defined for all parameters $(\omega, \gamma) \in \R^\nu \times [\gamma_1, \gamma_2]$, it is a quasi-periodic traveling wave, ${\rm odd}(\vf,x)$,  fulfills \eqref{gtn.est.better} at the step $\tn$ (by \eqref{lem:diopha.eq}), 
and for any $(\omega, \gamma)$ in the set $\t\Lambda_{\tn+1}^{\rm T}$ defined in \eqref{set.nonres.tn}, it solves the homological equation \eqref{he.as}. 
By  \eqref{gtn.est.better}  at the step $\tn$, \eqref{stime.pn.w}, \eqref{small.V.as.AP},
$ \ta \geq \chi \tau_1 + 3 $ (see \eqref{tbta}) 
\begin{equation}\label{g.piccolo}
	\absk{g_{\tn}}{2\mathfrak{s}_0+1} \leq C(\mathfrak{s}_0) N_{\tn}^{\tau_1}N_{\tn-1}^{-\ta}\absk{p_0}{2\mathfrak{s}_0+\tb+1} \upsilon^{-1} 
% N_{0}^{\tau_1 - \tau_2}	 
< \delta(\mathfrak{s}_0)  
	\,
\end{equation}
provided $N_0$ is large enough. 
 By Lemma \ref{proprieta}-4 
 % composition operator $\cG_{\tn}$ induced by
the diffeomorphism $ y =  x + g_{\tn} (\vphi,x) $ is invertible 
and its inverse $ x =  y + \breve g_{\tn} (\vphi,y) $ (which induces
the operator  $\cG_{\tn}^{-1}$)   satisfies
% thus defining  $\cG_{\tn}^{-1}$,  with
\begin{equation}\label{stimagbreve}
\absk{\breve g_{\tn}}{s} \leq C(s) \absk{g_{\tn}}{s} \, .
\end{equation}
	For any $(\omega, \gamma)$ in  $\t\Lambda_{\tn+1}^{\rm T}$, the operator   
	$ X_{\tn+1} = \cG_{\tn}^{-1} \, X_{\tn} \, \cG_{\tn} $ takes the form \eqref{Xtn.gtn} at step $\tn+1$ with 
	\begin{equation}\label{mtn.ptn}
		\tm_{1,\tn+1}:= \tm_{1,\tn} + \braket{p_{\tn}}_{\vf,x} \in\R\,, \quad p_{\tn+1}(\vf,y):= \{ \cG_{\tn}^{-1}\big( \Pi_{N_{\tn}}^\perp p_{\tn} + p_{\tn} (g_{\tn})_x \big) \}(\vf,y)\,.
	\end{equation}
This  verifies \eqref{Xtn.gtn2}  at step $\tn+1$.	Note that the constant $\tm_{1,\tn+1} \in \R $ and the function 
$p_{\tn+1}(\vf,y) $ in \eqref{mtn.ptn} are defined for all $(\omega,\gamma)\in\R^\nu\times[\gamma_1,\gamma_2]$.

\smallskip

In order to prove the inductive  estimates \eqref{stime.pn.w}, we first show the following iterative estimates of Nash-Moser type. 

\begin{lem}\label{stime.iter.ptn}
	The function $p_{\tn+1}$ defined in \eqref{mtn.ptn} satisfies, for any $\mathfrak{s}_0\leq s \leq S$,
	\begin{align}
		& \absk{p_{\tn+1}}{s} \leq C_1(s) \big( N_{\tn}^{-\tb} \absk{p_{\tn}}{s+\tb} + N_{\tn}^{\tau_1+1}\upsilon^{-1} \absk{p_{\tn}}{s}\absk{p_{\tn}}{\mathfrak{s}_0}\big) \label{ptn.s} \\
		& \absk{p_{\tn+1}}{s+\tb} \leq  C_2(s,\tb) \big( \absk{p_{\tn}}{s+\tb} + N_{\tn}^{\tau_1+1}\upsilon^{-1} \absk{p_{\tn}}{s+\tb}\absk{p_{\tn}}{\mathfrak{s}_0} \big) \label{ptn.s+b}
	\end{align}
where the positive constants $C_1(s), C_2(s,\tb)$ are monotone in $ \mathfrak{s}_0 \leq s \leq S$.
\end{lem}
\begin{proof}
	 We first show the estimate \eqref{ptn.s}. We write $p_{\tn+1}$ in \eqref{mtn.ptn} as $p_{\tn+1}:= \cG_{\tn}^{-1} F_{\tn} $ with $F_{\tn}:=\Pi_{N_{\tn}}^\perp p_{\tn} + p_{\tn} (g_{\tn})_x $. %By \eqref{gtn.est.better}  at step $\tn$ and 
	 By Lemma \ref{proprieta}-item 2,
	 we get
	\begin{equation}\label{A.26}
			\absk{F_{\tn}}{s}  \leq \absk{\Pi_{N_{\tn}}^\perp p_{\tn}}{s} + C(s) \absk{p_{\tn}}{s}\absk{g_{\tn}}{\mathfrak{s}_0+1}+C(\mathfrak{s}_0)\absk{p_{\tn}}{\mathfrak{s}_0} \absk{g_{\tn}}{s+1}\,.
	\end{equation}
	By \eqref{est:compo},  % \eqref{g.piccolo}, 
	\eqref{stimagbreve},
		 \eqref{A.26},  % \eqref{mtn.ptn}, 
	\eqref{gtn.est.better}  at step $\tn$, Lemma \ref{proprieta} and \eqref{g.piccolo},
	we have
		\begin{align}
			\absk{p_{\tn+1}}{s} &\lesssim_{s} \absk{F_{\tn}}{s} + \absk{g_{\tn}}{s}\absk{F_{\tn}}{\mathfrak{s}_0+1} \notag \\
			& \lesssim_{s} \absk{\Pi_{N_{\tn}}^\perp p_{\tn}}{s}+ \absk{p_{\tn}}{s}\absk{g_{\tn}}{\mathfrak{s}_0+1}+\absk{p_{\tn}}{\mathfrak{s}_0} \absk{g_{\tn}}{s+1}   \notag\\
			& \ \ +\absk{g_{\tn}}{s}\big( \absk{\Pi_{N_{\tn}}^\perp p_{\tn}}{\mathfrak{s}_0+1}+ \absk{p_{\tn}}{\mathfrak{s}_0+1}\absk{g_{\tn}}{\mathfrak{s}_0+1}+\absk{p_{\tn}}{\mathfrak{s}_0} \absk{g_{\tn}}{\mathfrak{s}_0+2} \big) \notag \\
			& \lesssim_{s} \absk{\Pi_{N_{\tn}}^\perp p_{\tn}}{s}+ \absk{p_{\tn}}{s}\absk{g_{\tn}}{\mathfrak{s}_0+1} +\absk{p_{\tn}}{\mathfrak{s}_0} \absk{g_{\tn}}{s+1} 
			+N_{\tn}^{\tau_1+1}\upsilon^{-1}\absk{p_{\tn}}{s-1} \absk{p_{\tn}}{\mathfrak{s}_0+1}   \notag \\
			& \lesssim_{s} N_{\tn}^{-\tb} \absk{p_{\tn}}{s+\tb} + N_{\tn}^{\tau_1+1}\upsilon^{-1} \absk{p_{\tn}}{s}\absk{p_{\tn}}{\mathfrak{s}_0}\,, \label{cont0}
		\end{align}
	which is  \eqref{ptn.s}. The estimate \eqref{ptn.s+b} follows 
	as for \eqref{cont0} (with $ s \rightsquigarrow  s + \tb $).
%	\begin{equation*}
%		\begin{aligned}
%			\absk{p_{\tn+1}}{s+\tb} &\lesssim_{s,\tb} \absk{F_{\tn}}{s+\tb} +  \absk{g_{\tn}}{s+\tb}\absk{F_{\tn}}{\mathfrak{s}_0+1} \\
%			& \lesssim_{s,\tb} \absk{p_{\tn}}{s+\tb}+ \absk{p_{\tn}}{s+\tb}\absk{g_{\tn}}{\mathfrak{s}_0+1} +\absk{p_{\tn}}{\mathfrak{s}_0} \absk{g_{\tn}}{s+\tb+1} \\
%			& \ \ +N_{\tn}^{\tau_1+1}\upsilon^{-1}\absk{p_{\tn}}{s+\tb-1}\big(\absk{p_{\tn}}{\mathfrak{s}_0+1}+\absk{p_{\tn}}{\mathfrak{s}_0}  \big)\\
%			& \lesssim_{s,\tb}  \absk{p_{\tn}}{s+\tb} + N_{\tn}^{\tau_1+1}\upsilon^{-1} \absk{p_{\tn}}{s+\tb}\absk{p_{\tn}}{\mathfrak{s}_0}\,,
%		\end{aligned}
%	\end{equation*}
%	which concludes the proof.
\end{proof}

As a corollary of the previous lemma we deduce the following lemma.
\begin{lem}
The estimates \eqref{stime.pn.w}-\eqref{tm.est}  hold at the  step $ \tn + 1 $.
\end{lem}

\begin{proof}
%5We first prove estimate \eqref{stime.pn.w} at step $\tn+1$.
	By \eqref{ptn.s} and \eqref{stime.pn.w}  we have, for any $\mathfrak{s}_0 \leq s \leq S$,
	\begin{align*}
		\absk{p_{\tn+1}}{s} & \leq C_1(S) \, C(s, \tb) \big( N_{\tn}^{-\tb}N_{\tn-1}^{}\absk{p_0}{s+\tb}+ C(\mathfrak{s}_0,\tb)\upsilon^{-1}N_{\tn}^{\tau_1+1}N_{\tn-1}^{-2\ta}\absk{p_0}{s+\tb}\absk{p_0}{\mathfrak{s}_0+\tb}\big)\\
		&  \leq C(s,\tb) N_{\tn}^{-\ta}\absk{p_0}{s+\tb}
	\end{align*} 
	asking that
$	 C_1(S) 	N_{\tn}^{-\tb}N_{\tn-1}^{}\leq \tfrac12 N_{\tn}^{-\ta} $ and 
$ C_1(S)C(\mathfrak{s}_0,\tb) \upsilon^{-1}N_{\tn}^{\tau_1+1}N_{\tn-1}^{-2\ta}\absk{p_0}{\mathfrak{s}_0+\tb} \leq \tfrac12 N_{\tn}^{-\ta} $, 
	which both follow by \eqref{tbta}, the smallness assumption \eqref{small.V.as.AP} and taking $N_0 := N_0 (S) > 0 $ sufficiently large. This proves the first estimate of \eqref{stime.pn.w} at step $\tn+1$. 
	The second follows in a similar way, eventually increasing $N_0$.

Finally we have, by \eqref{mtn.ptn} and the first inequality in \eqref{stime.pn.w},  
\begin{equation}\label{diffm1m2}
		|\tm_{1,\tn+1}-\tm_{1,\tn}|^{k_0,\upsilon} = | \braket{p_{\tn}}_{\varphi,x}|^{k_0,\upsilon} \leq \absk{p_{\tn}}{\mathfrak{s}_0} \leq C(\mathfrak{s}_0,\tb) N_{\tn-1}^{-\ta} \absk{p_0}{\mathfrak{s}_0+\tb} \, , 
\end{equation}
proving the second estimate  \eqref{tm.est} at step $\tn+1$.
Writing $ \tm_{1,\tn+1} = 
\sum_{j=0}^{\tn}  (\tm_{1,j+1}-\tm_{1,j}) $ and recalling that $ \tm_{1,0} = 0  $, 
	we deduce by \eqref{diffm1m2} the first estimate  \eqref{tm.est} at step $\tn+1$.
\end{proof}

The proof of $\bf (S1)_{\tn+1}$ is complete. The item  $\bf (S2)_{\tn+1}$ follows by similar inductive arguments. 
The proof of Theorem \ref{thm:as} is concluded.
\end{proof}

\paragraph{Acknowledgements.}
We thank Riccardo Montalto for many useful discussions.
The work of the author L.F. is supported by Tamkeen under the NYU Abu Dhabi Research Institute grant CG002.

\begin{footnotesize}
	
\end{footnotesize}

\end{document}